\documentclass[11pt,a4paper,reqno]{amsart}
\usepackage[margin=1.05in]{geometry}
\usepackage{amsmath,amsthm,amssymb,amsxtra,comment,graphicx,psfrag}

\usepackage{enumerate,enumitem}
\usepackage{bm,mathrsfs}
\usepackage{mathtools}
\usepackage{xspace}
\usepackage{color}
\usepackage{cite}
\usepackage{subcaption}
\usepackage{longtable}
\usepackage{color}
\usepackage{hyperref}
\hypersetup{%
  colorlinks = true,
  citecolor=blue, 
  linkcolor  = black
}
\usepackage{multirow}
\usepackage{longtable}
\usepackage{color}
\usepackage{hyperref}
\usepackage{subcaption} % 在导言区添加这个包
\usepackage{amsfonts,bm,hyperref,subcaption}

\allowdisplaybreaks
\allowdisplaybreaks[4]

\theoremstyle{plain}% default
\numberwithin{equation}{section}
\numberwithin{table}{section}
\numberwithin{figure}{section}

\usepackage{titlesec} 
\titleformat{\section}{\vskip10pt\large\bfseries}{\thesection.}{0.5em}{\centering\vspace{5pt}}
\titleformat{\subsection}{\vskip10pt\normalsize\bfseries}{\thesubsection.}{0.5em}{}
\titleformat{\subsubsection}{\vskip10pt\normalsize\bfseries}{\thesubsection.}{0.5em}{}
\usepackage{lipsum,titletoc}

\titlecontents{section} % set formatting for \section -
[2.5em]                 % adjust left margin
{\rmfamily}%{\scshape}             % font formatting
{\contentslabel{2.3em}} % section label and offset
{\hspace*{-2.3em}} {\titlerule*[1pc]{.}\contentspage}

\titlecontents{subsection} % set formatting for \section -
[4.8em]                 % adjust left margin
{\rmfamily}             % font formatting
{\contentslabel{2.3em}} % section label and offset
{\hspace*{-2.3em}} {\titlerule*[1pc]{.}\contentspage}

\allowdisplaybreaks

\hypersetup{%
	colorlinks = true,
	citecolor=blue, 
	linkcolor  = black
}

\allowdisplaybreaks

\newtheorem{theorem}{Theorem}[section]

\newtheorem{lemma}[theorem]{Lemma}

\newtheorem{remark}[theorem]{Remark}
\theoremstyle{definition}

\newtheorem{example}[theorem]{Example}

\def\bfx{{\bf x}}
\def\bfe{{\bf e}}

\def\nu{n}
\def\b{{\bf b}}

\def\d{{\mathrm d}}

\def\R{{\mathbb R}}

\def\ud{\underline{D}}
\def\md{\partial_{t}^\bullet}

\def\ehm{{\hat{e}_h^m}}
\def\eM{{e_h^{m+1}}}
\def\ehM{{\hat{e}_h^{m+1}}}
\def\evm{{e_v^m}}
\def\ehvm{{\hat{e}_v^m}}
\def\ek{{e_{\kappa}^m}}

\def\Gm{{\Gamma^m}}

\def\Ghm{{\Gamma^m_h}}
\def\GhM{{\Gamma^{m+1}_h}}
\def\Ghsm{{\hat\Gamma^m_{h, *}}}
\def\GhsM{{\hat\Gamma^{m+1}_{h, *}}}

\def\nhm{{n^m_h}}

\def\nhsm{{\hat n^m_{h, *}}}
\def\nhsM{{\hat n^{m+1}_{h, *}}}
\def\nbhsm{{\bar n^m_{h, *}}}
\def\nbhsM{{\bar n^{m+1}_{h, *}}}
\def\nsm{{n^m_{*}}}
\def\nsM{{n^{m+1}_{*}}}
\def\nhtm{{\hat n_{h,\theta}^m}}
\def\nahtm{{\bar{n}^m_{h,\theta}}}
\def\nahm{{\bar{n}^m_h}}
\def\nahsm{{\bar{n}^m_{h,*}}}

\def\Thsm{{\hat T^m_{h, *}}}

\def\Tsm{{T^m_{*}}}
\def\TsM{{T^{m+1}_{*}}}
\def\Thtm{{\hat\Gamma_{h,\theta}^{m}}}
\def\Tahm{{\bar{T}^m_h}}
\def\Tbhm{{\bar{T}^m_h}}
\def\Tahsm{{\bar{T}^m_{h,*}}}
\def\Tbhsm{{ \bar{T}^m_{h, *}}}

\def\Nhsm{{\hat N^m_{h, *}}}

\def\Nsm{{N^m_{*}}}
\def\NsM{{N^{m+1}_{*}}}
\def\Nahm{{\bar{N}^m_h}}
\def\Nbhm{{\bar{N}^m_h}}
\def\Nahsm{{\bar{N}^m_{h,*}}}

\def\k{{\kappa}}

\def\b{\color{blue}}

\numberwithin{equation}{section}
\allowdisplaybreaks[4] 

\begin{document}

\title[]{A convergent finite element method with\\
minimal deformation rate for mean curvature flow}

\author[]{Tiantian Huang,\,\,Buyang Li,\,\,and\,\,Rong Tang}
\address{Department of Applied Mathematics, The Hong Kong Polytechnic University,
Hong Kong. 
{\rm Email address: {\tt tiantian7.huang@connect.polyu.hk}, {\tt buyang.li@polyu.edu.hk} and {\tt claire.tang@polyu.edu.hk}
}}


%\date{\today}


\begin{abstract}
We propose and analyze a fully discrete parametric finite element method with minimal deformation rate (MDR) for simulating the mean curvature flow of general closed surfaces in three dimensions. The method is formulated from a coupled system that enforces the mean curvature flow law for the normal velocity while introducing an artificial tangential velocity that minimizes the deformation-rate energy, thereby preserving mesh quality without requiring remeshing or reparametrization. An $L^{2}$-projected averaged normal vector is used in the scheme to facilitate a rigorous convergence analysis. Within the projected--distance framework, we establish the first complete convergence proof for a parametric finite element method that incorporates the MDR tangential motion without relying on evolution equations for the mean curvature or the normal vector, achieving optimal-order error estimates for finite elements of degree $k \ge 3$. Numerical experiments corroborate the theoretical results and demonstrate that the proposed MDR method maintains mesh quality comparable to the Barrett--Garcke--N\"urnberg method, for which convergence has not yet been established.
\end{abstract}

\maketitle

\tableofcontents
\setlength\abovedisplayskip{4pt}
\setlength\belowdisplayskip{4pt}

\section{Introduction}
	
% The evolution of surfaces under geometric curvature flows has attracted significant attention over the past decades, with classical examples including mean curvature flow and higher‐order flows such as Willmore flow and surface diffusion (see, e.g., \cite{barrett2020parametric,deckelnick2005computation,ecker2012regularity,ganesan2017ale,li2021convergence}). These flows model a variety of physical phenomena, from the dynamics of soap films and the morphology of oil droplets on water to solid‐state dewetting processes. In particular, the mean curvature flow, being the gradient flow of the surface area functional, constitutes the most natural geometric evolution of hypersurfaces and has therefore been the object of extensive study in geometric analysis. The development of novel numerical methods and numerical analysis for computing surface evolution in mean curvature flow and related geometric flows is still an active and challenging research area nowadays. 

The evolution of surfaces driven by curvature-dependent dynamics has been a long-standing and active area of research in both geometric analysis and applied mathematics. Among these evolutions, the \emph{mean curvature flow} plays a central role as the most fundamental geometric flow. As the gradient flow of the surface area functional, the mean curvature flow provides a natural mechanism by which a surface evolves to reduce its total area. It also appears in numerous physical and engineering contexts, such as the motion of grain boundaries, the dynamics of soap films, and the morphological evolution of thin films and droplets. Closely related higher-order flows, including surface diffusion and Willmore flow (see, e.g.,~\cite{barrett2020parametric,deckelnick2005computation,ecker2012regularity,ganesan2017ale,li2021convergence}), share similar geometric structures while exhibiting richer analytical and numerical challenges. The development of accurate and robust numerical methods for computing such curvature-driven surface evolutions remains an active and challenging field of research.

\smallskip
Mathematically, the mean curvature flow of a smooth evolving surface $\Gamma(t)$ can be expressed by the normal velocity law
\begin{align*}
  v = -Hn ,
\end{align*}
where $H$ denotes the mean curvature and $n$ the unit normal of~$\Gamma(t)$. A fundamental geometric identity which links the mean curvature vector $-Hn$ to the surface $\Gamma(t)$ is 
\begin{align}\label{con-identity}
  -Hn = \Delta_{\Gamma} \mathrm{id},
\end{align}
where $\mathrm{id} : \mathbb{R}^3 \supset \Gamma \to \mathbb{R}^3$ is the identity map $\mathrm{id}(x) = x$, and $\Delta_\Gamma$ is the surface Laplace--Beltrami operator. Consequently, the velocity of the surface evolving by mean curvature flow satisfies
\begin{align}\label{velocity-MCF}
  v - \Delta_{\Gamma}\mathrm{id} = 0 .
\end{align}

A pioneering numerical approach for simulating such surface evolutions is the \emph{parametric finite element method} (FEM) introduced by Dziuk~\cite{Dziuk}. In this approach, the continuous surface $\Gamma(t_m)$ at time $t_m$ is approximated by a piecewise triangular surface $\Gamma_h^m$. Given $\Gamma_h^m$, the discrete surface at the next time step, $\Gamma_h^{m+1}$, is obtained via a discrete flow map $X_h^{m+1} : \Gamma_h^m \to \mathbb{R}^3$ that describes the evolution from $\Gamma_h^m$ to $\Gamma_h^{m+1}$, i.e.,
\[
  \Gamma_h^{m+1} = X_h^{m+1}(\Gamma_h^m).
\]
Let $S_h(\Gamma_h^m)^3$ denote the vector-valued finite element space defined on $\Gamma_h^m$. The discrete flow map $X_h^{m+1}\in S_h(\Gamma_h^m)^3$ is then determined by the finite element weak formulation of~\eqref{velocity-MCF}:
\begin{align}\label{Dziuk}
  \int_{\Gamma_h^m}\frac{X_h^{m+1}-\mathrm{id}}{\tau}\cdot \chi_h
  + \int_{\Gamma_h^m}\nabla_{\Gamma_h^m}X_h^{m+1}\cdot \nabla_{\Gamma_h^m}\chi_h
  \,=\,0
  \quad \forall\, \chi_h\in S_h(\Gamma_h^m)^3 ,
\end{align}
where $\tau = t_{m+1}-t_m$ denotes the time-step size. This variational framework has subsequently been extended to a variety of higher-order curvature flows, such as surface diffusion, Willmore flow, and Helfrich flow (see, e.g.,~\cite{bonito2010parametric,dziuk2008computational,banesch2005finite}).

The parametric FEM has proven to be highly effective for simulating curvature-driven surface evolutions, particularly when the tangential motion of the surface is moderate. However, when significant tangential deformation arises, the mesh quality may deteriorate rapidly, leading to issues such as node clustering and element distortion that compromise both accuracy and stability. In such cases, remeshing or mesh-regularization techniques are typically required once the mesh quality falls below acceptable thresholds (see, e.g.,~\cite{marchandise2011highquality,remacle2010highquality}). These challenges have motivated the development of improved numerical strategies capable of maintaining accuracy and mesh quality during simulations of mean curvature flow and other curvature-dependent geometric evolutions.

An important approach that avoids frequent remeshing was developed by Barrett, Garcke, and Nürnberg in their seminal works~\cite{barrett2007parametric,barrett2008hypersurfaces,barrett2008willmore}. They proposed a family of finite element weak formulations incorporating an artificial tangential velocity designed to make the one-step flow map $X_h^{m+1}:\Gamma_h^m\rightarrow\Gamma_h^{m+1}$ approximately harmonic, thereby reducing mesh distortion and degeneration. For the mean curvature flow, the Barrett–Garcke–Nürnberg (BGN) formulation can be equivalently written as 
\begin{align}\label{BGN}
  \int_{\Gamma_h^m}^{(h)}\Big(\frac{X_h^{m+1}-\mathrm{id}}{\tau}\cdot n_h^m\Big)(\chi_h\cdot n_h^m)
  + \int_{\Gamma_h^m}\nabla_{\Gamma_h^m}X_h^{m+1}\cdot \nabla_{\Gamma_h^m}\chi_h
  \,=\,0 \quad \forall\, \chi_h\in S_h(\Gamma_h^m)^3,
\end{align}
where $n_h^m$ denotes the piecewise-defined normal vector on $\Gamma_h^m$, and $\int_{\Gamma_h^m}^{(h)}$ indicates that mass lumping is employed to compute the integral on each triangular element. The BGN method has been widely recognized for its ability to preserve mesh quality in practical computations and has since been successfully applied to a variety of problems, including simulations of solid-state dewetting and contact line migration (see, e.g.,~\cite{barrett2013eliminating,barrett2015stable,bao2021structure,bao2022volume,Bao2021,Bao2023,zhao2020parametric}).

Several other approaches incorporating artificial tangential velocities have also been proposed to improve the mesh quality of evolving surfaces. These include the reparametrization techniques of Elliott and Fritz~\cite{elliott2017approximations,elliott2016algorithms}, which introduce a surface reparametrization to eliminate degeneracy in the evolution equations; the minimal-deformation (MD) method~\cite{duan2024new,Gao-Li-2025}, which determines a tangential velocity minimizing the total deformation energy from $\Gamma_h^0$ to $\Gamma_h^m$ for all $m=1,2,\ldots$; and the minimal-deformation-rate (MDR) method~\cite{hu2022evolving,bai2024convergent}, which minimizes the deformation-rate energy at each time level. All of these methods have demonstrated improved mesh quality compared with standard parametric FEMs that do not include artificial tangential motions. Nevertheless, rigorous convergence analyses of such methods remain technically challenging and have progressed only gradually.

For the evolution of one-dimensional curves, convergence of parametric FEMs has been well established~\cite{Dziuk1994,DeckelnickDziuk2009,Bartels2013,ElliottFritz2017,Li2020,YeCui2021}, including results for schemes that incorporate tangential velocities based on the Elliott–Fritz reparametrization technique~\cite{Barrett2017curve,ElliottFritz2017}. More recently, convergence of a stabilized BGN method for planar curve-shortening flow was proved in~\cite{Bai-Li-MCOM2025} using mass lumping and stabilization of the tangential velocity. However, the mass-lumping argument central to that analysis does not extend naturally to triangulated surfaces with higher-order elements, leaving the convergence of BGN-type schemes for general two-dimensional surfaces an open problem. As a result, extending the convergence theory to mean curvature flow or Willmore flow of general closed surfaces in three dimensions remains a major challenge.

For the evolution of two-dimensional surfaces, rigorous convergence has been established mainly for finite element and finite difference discretizations of mean curvature and Willmore flows in graph and axisymmetric geometries; see, for example,~\cite{DeckelnickDziuk1995,DeckelnickDziuk2006} and the more recent works~\cite{BarrettDeckelnickNurnberg2021,DeckelnickNurnberg2021,DeckelnickStyles2022}. Moreover, evolving surface FEMs incorporating tangential velocities via the Elliott–Fritz reparametrization technique have been shown to converge for mean curvature flow on graph and axisymmetric surfaces, as well as for closed toroidal surfaces of genus one~\cite{DeckelnickNurnberg2021,DeckelnickStyles2022,mierswa2020error}. Extending these convergence results to general closed surfaces, however, remains an open and challenging topic of current research.

The convergence of parametric FEMs for the evolution of closed surfaces in mean curvature flow, Willmore flow, and coupled mean–curvature–diffusion dynamics has been established for a class of methods that reformulate these curvature flows in terms of the evolution equations for the mean curvature and the normal vector; see~\cite{kovacs2019convergent,kovacs2021convergent,elliott2022numerical}. Convergence of such schemes incorporating the MDR tangential motion has also been proved in~\cite{hu2022evolving}. However, these approaches, which explicitly use the evolution equations of the mean curvature and normal vector, typically require the surface to possess higher regularity compared with other formulations (such as the BGN method; see the comparison in Figure~\ref{eg:Example3}). Moreover, they may suffer from error accumulation in computing the mean curvature and normal vector, causing the computed geometric quantities to drift from the true surface geometry and therefore require re-initialization to realign them with the underlying geometry (see the discussions in the numerical examples of~\cite{hu2022evolving}).

Among all the methods that do not employ the evolution equations of the mean curvature and normal vector, only Dziuk's original parametric FEM with finite elements of degree $k \ge 3$ has been rigorously proved to converge for the evolution of general closed surfaces under mean curvature flow~\cite{bai2024new}. In contrast, establishing convergence for methods incorporating artificial tangential motions—such as the BGN scheme—to maintain mesh quality during the evolution of two-dimensional surfaces remains an intriguing and challenging open problem. The primary difficulty in proving convergence of the BGN method lies in deriving suitable estimates for the tangential velocity of the numerically computed surface. In particular, the limiting PDE approximated by the BGN method \eqref{BGN}, i.e., 
\begin{align}\label{eq:BGN_limit_0}
(v \cdot n)n - \Delta_{\Gamma}{\rm id} = 0 ,
\end{align}
determines only the normal velocity of the surface and therefore provides no control of the tangential velocity. 

This difficulty is overcome in the present work by introducing a method that solves the following coupled system of equations:
\begin{subequations}\label{kv}
\begin{align}
  v \cdot n - (\Delta_\Gamma \mathrm{id}) \cdot n &= 0, \label{kv_1} \\
  -\Delta_\Gamma v &= \kappa n, \label{kv2}
\end{align}
\end{subequations}
where the first equation corresponds to the normal component of~\eqref{velocity-MCF} or~\eqref{eq:BGN_limit_0}, and the second enforces the MDR tangential motion by minimizing the deformation-rate energy $\int_\Gamma |\nabla_\Gamma v|^2$ under the constraint of the first equation, with $\kappa$ being the Lagrange multiplier associated with this constrained optimization problem. The velocity obtained from~\eqref{kv} differs from that of~\eqref{velocity-MCF} only in its tangential component and thus leaves the underlying surface unchanged.

We consider the following semi-implicit, fully discrete parametric FEM for~\eqref{kv}. Given an approximate surface $\Gamma_h^m$ at time $t_m$, find $(X_h^{m+1},\kappa_h^{m+1}) \in S_h(\Gamma_h^m)^3 \times S_h(\Gamma_h^m)$ such that, for all $(\chi_h,\phi_h) \in S_h(\Gamma_h^m)^3 \times S_h(\Gamma_h^m)$,
\begin{subequations}\label{NBGN}
\begin{align}
  \int_{\Gamma_h^m} \frac{X_h^{m+1}-\mathrm{id}}{\tau}\cdot \bar n_h^m\,\phi_h
  + \int_{\Gamma_h^m} \nabla_{\Gamma_h^m} X_h^{m+1} \cdot \nabla_{\Gamma_h^m}(\bar n_h^m \phi_h) &= 0,
  \quad \forall \phi_h \in S_h(\Gamma_h^m), \label{hm1} \\
  \int_{\Gamma_h^m} \nabla_{\Gamma_h^m} \frac{X_h^{m+1}-\mathrm{id}}{\tau} \cdot \nabla_{\Gamma_h^m} \chi_h
  - \int_{\Gamma_h^m} \kappa_h^{m+1}\,\bar n_h^m \cdot \chi_h &= 0,
  \quad \forall \chi_h \in S_h(\Gamma_h^m)^3, \label{hm2}
\end{align}
\end{subequations}
where $\bar n_h^m = P_{\Gamma_h^m} n_h^m \in S_h(\Gamma_h^m)^3$ is an averaged (projected) normal vector, and $P_{\Gamma_h^m}: L^2(\Gamma_h^m)^3 \to S_h(\Gamma_h^m)^3$ denotes the $L^2$-orthogonal projection onto $S_h(\Gamma_h^m)^3$. In particular,
\begin{align}\label{def-bar-nhm}
  \int_{\Gamma_h^m} \bar n_h^m \cdot \chi_h 
  = \int_{\Gamma_h^m} P_{\Gamma_h^m} n_h^m \cdot \chi_h
  = \int_{\Gamma_h^m} n_h^m \cdot \chi_h,
  \qquad \forall \chi_h \in S_h(\Gamma_h^m)^3.
\end{align}
The surface $\Gamma_h^{m+1}$ at time $t_{m+1}$ is then defined as $\Gamma_h^{m+1}=X_h^{m+1}(\Gamma_h^m)$.

The main contribution of this paper is the rigorous proof of convergence of the MDR method~\eqref{NBGN} for the evolution of general closed surfaces under mean curvature flow. The analysis is built upon two key ingredients:  
(i) the formulation~\eqref{kv} employing the MDR tangential motion, and  
(ii) the definition of the averaged normal vector $\bar n_h^m$ via the $L^2$ projection.  
These two components, combined within the projected–distance framework~\cite{bai2024new}, enable the establishment of stability estimates and hence the full convergence analysis.

The formulation~\eqref{kv}, based on the MDR tangential motion, provides a stable and well-defined description of the tangential velocity, ensuring mesh quality comparable to that of the BGN method while allowing a rigorous stability analysis. 

The use of the $L^2$–projected averaged normal vector $\bar n_h^m$, defined by $\bar n_h^m = P_{\Gamma_h^m} n_h^m$, does not rely on evolution equations for the normal vector. As a result, it avoids the accumulation of geometric errors and drift from the true surface, thereby preserving mesh quality without re-initialization and enabling the method to handle initial surfaces of lower regularity compared with evolution–based MDR schemes~\cite{hu2022evolving,bai2024convergent}; see Figures~\ref{eg:Example2} and~\ref{eg:Example3}. From the analytical perspective, this definition removes the discontinuities of $n_h^m$ across curved element interfaces and eliminates jump terms in integration by parts, thereby overcoming difficulties that typically arise in establishing stability estimates when the piecewise-defined normal vector $n_h^m$ is used. Furthermore, the $L^2$ projection is essential for proving the stability of the numerical scheme: it allows the $L^2$ projection operator in~\eqref{def:dnM} to vanish when~\eqref{def:dnM} is applied in~\eqref{def:M2im}. This makes it possible to integrate by parts and transfer the surface gradient from $\hat X_{h,*}^{m+1} - \hat X_{h,*}^{m}$ to other terms, leading to the desired estimates for the term $M_{22}^m$ defined in \eqref{def:M2im}. In contrast, if a different averaging operator were adopted, the corresponding operator would persist in~\eqref{def:dnM} when substituted into~\eqref{def:M2im}, preventing such integration by parts and thereby obstructing the derivation of the required stability estimates.

The remainder of this article is organized as follows. Section~\ref{sec:settings} presents the main theorem on the convergence of the MDR method for mean curvature flow within the projected–distance framework. In Section~\ref{sec:framework}, we develop the analytical foundation required for the proof, including Lagrange interpolation bounds, induction hypotheses, estimates for the averaged normal, super‐approximation results, Poincaré‐type inequalities, and geometric identities induced by the nodewise distance projection. Section~\ref{consistency-estimates} is devoted to establishing the consistency estimates, while Section~\ref{sec:stability} derives the stability estimates and completes the proof of the main theorem. Section~6 reports numerical experiments that support the theoretical analysis and demonstrate that the proposed scheme preserves high-quality meshes for mean curvature flow. Finally, certain technical details are collected in the Appendix.

\vspace{-1ex}%

\section{Basic settings and main results}\label{sec:settings}

We consider mean curvature flow of a closed surface $\Gamma(t)$ in $\mathbb{R}^3$ with a smooth initial surface $\Gamma^0$ at time $t=0$. The evolution of the surface is governed by equation \eqref{velocity-MCF}. 

\subsection{Distance projection onto the exact surface}
\label{section:projection}

	Let $\delta > 0$ be a sufficiently small constant such that every point $x$ within the $\delta$-neighborhood of the exact surface $\Gamma^m = \Gamma(t_m)$, denoted by 
	\begin{align}\label{def-D_delta}
	D_\delta(\Gamma^m) = \{x \in \mathbb{R}^3 : \mathrm{dist}(x, \Gamma^m) \leq \delta\} ,
	\end{align}
	possesses a unique distance projection onto $\Gamma^m$, denoted by $a^m: D_\delta(\Gamma^m) \to \Gamma^m$. In particular, the distance projection of a point $x\in D_\delta(\Gamma^m)$ onto $\Gamma^m$, denoted by $a^m(x)$, satisfies the following relation:
	\[
	x - a^m(x) = \pm |x - a^m(x)| n^m(a^m(x)),
	\]
where $n^m$ denotes the unit normal vector on $\Gamma^m$, and the vector $x - a^m(x)$ is orthogonal to the tangent plane of $\Gamma^m$ at $a^m(x)$. 

The constant $\delta$ is chosen to be independent of $m$, though it may depend on the final time $T$.

\subsection{Initial triangulation}

Let $\Gamma_{h,{\rm f}}^0$ be a piecewise flat triangular surface which interpolates the exact surface $\Gamma^0$ with a shape-regular and quasi-uniform triangulation. For sufficiently small mesh size $h$, the piecewise flat triangular surface $\Gamma_{h,{\rm f}}^0$ is in the neighborhood $D_\delta(\Gamma^0)$ on which the distance projection $a^0:D_\delta(\Gamma^0)\rightarrow\Gamma^0$ is well defined. For each flat triangle $K_{\mathrm{f}}^0$ on $ \Gamma_{h,{\rm f}}^0$, every point on $K_{\mathrm{f}}^0$ can be projected onto the exact surface $\Gamma^0$ using the distance projection, generating a set of nodes on $\Gamma^0$ which determines a curved triangle $K^0$ with parametrization $F_{K^0}= I_ha^0: K_{\mathrm{f}}^0 \to K^0$, which is a polynomial of degree $k$ defined on the flat triangle $K_{\mathrm{f}}^0$. The curved triangles generated in this way form a piecewise curved triangular surface $\Gamma_h^0$ which interpolates the exact surface $\Gamma^0$. 
%	We denote by $\Gamma_h^0$ the piecewise curved triangular surface which interpolates the exact surface $\Gamma^0$ at some nodes $x_j^0$, $j=1,\dots,J$. Each triangle $K^0$ of $\Gamma_h^0$ is generated as follows: we denote by $K_{\mathrm{f}}^0$ the corresponding flat triangle which shares the three vertices with $K^0$. 
%	The flat triangles $K_{\rm f}^0$ form a piecewise flat triangular surface
%	$$ 
%	\Gamma_{h,{\rm f}}^0=\bigcup_{K^0\subset\Gamma_h^0} K_{\rm f}^0 . 
%	$$

We assume that the initial triangulation is of sufficiently high quality, satisfying the condition: 
	\begin{align}\label{P0}
		\max_{K^0 \subset \Gamma_h^0} 
		\Big( \|F_{K^0}\|_{W^{k,\infty}(K_{\mathrm{f}}^0)} + \|\nabla_{K^0} F_{K^0}^{-1}\|_{L^\infty(K^0)} \Big) 
		\leq \kappa_0,
	\end{align}
where $\kappa_0$ is a constant independent of the mesh size $h$. 
This property is standard for parametric finite elements and ensures the following optimal-order approximation of the exact surface $\Gamma^0$ by the discrete surface $\Gamma_h^0$:
	\begin{align}
		\max_{K^0 \subset \Gamma_h^0} \| a^0 \circ F_{K^0} - F_{K^0} \|_{L^\infty(K_{\mathrm{f}}^0)} 
		\leq C h^{k+1} ,
	\end{align}
where $F_{K^0}$ and $a^0 \circ F_{K^0}$ are local parametrizations of $\Gamma_h^0$ and $\Gamma^0$, respectively. 
%Since $a^0(x)$ is well-defined in a neighborhood of $\Gamma^0$, it is also well-defined on $\Gamma_h^0$ for sufficiently small $h$. 
This guarantees the geometric consistency of the discrete surface with the exact surface. 
	
We denote by $\bfx^{0}=(x_{1}^{0} , \cdots,  x_{J}^{0} )^\top$ the nodal vector which consists of the positions of the nodes of the initial surface $\Gamma_h^0$, where $J$ denotes the number of nodes in the triangulation.

\subsection{The numerical scheme}

Let $t_m = m\tau$, $m = 0, 1, \ldots, N$, be a partition of the time interval $[0, T]$ with a uniform time stepsize $\tau = T/N$. For a given nodal vector $\bfx^{m}$ which determines a piecewise curved triangular surface $\Gamma_h^m$ at time level $t=t_m$, the set of curved triangles which forms $\Gamma_h^m$ is denoted by $\mathcal{K}_h^m$, with each curved triangle $K \in \mathcal{K}_h^m$ being the image of a curved triangle $K^0$ under the flow map $X_h^m:\Gamma_h^0\rightarrow \Gamma_h^m$. The curved triangle has a parametrization $F_K : K_{\mathrm{f}}^0 \to K$ which is a polynomial of degree $k$ defined on the flat triangle $K_{\mathrm{f}}^0$ that shares the three vertices with $K^0$. 
%	
%	
%	corresponding curved triangle $K^0 \subset \Gamma_h^0$ under the discrete flow map $X_h^m: \Gamma_h^0\rightarrow\Gamma_h^m$, which is a finite element function of degree $k$ with nodal vector $\bfx^m$. 
The finite element space on the piecewise curved triangular surface $\Gamma_h^m$ is defined as 	
\begin{align}
	S_h(\Gamma_h^m) = \{v_h \in C(\Gamma_h^m) : v_h \circ F_K \in \mathbb{P}^k(K_{\mathrm{f}}^0) \,\, \text{for all } K \in \mathcal{K}_h^m\},
\end{align}
where $\mathbb{P}^k(K_{\mathrm{f}}^0)$ denotes the space of polynomials of degree $\le k$ on the flat triangle $K_{\mathrm{f}}^0$. The three-dimensional vector-valued finite element space on the piecewise curved triangular surface $\Gamma_h^m$ is denoted by $S_h(\Gamma_h^m)^3$.

We compute the nodal vector $\bfx^{m+1}$, which determines the piecewise curved triangular surface $\Gamma_h^{m+1}$, using the MDR method described in \eqref{NBGN}, where we regard $X_h^{m+1}$ to be the finite element function on $\Gamma_h^m$ with nodal vector $\bfx^{m+1}$ (thus $\Gamma_h^{m+1}$ is the image of $\Gamma_h^m$ under the map $X_h^{m+1}$).

\subsection{Distance projection of $\Gamma_h^m$ onto $\Gamma^m$}
We will study the convergence of proposed MDR method by estimating the distance from the numerically computed surface $\Gamma_h^m$ (determined by the nodal vector $\bfx^{m}$) to the exact surface $\Gamma^m$. To this end, we denote by $\hat\bfx_*^{m}=(\hat x_{1,*}^{m} , \cdots, \hat x_{J,*}^{m} )^\top$ the distance projection of $\bfx^{m}=(x_{1}^{m}, \cdots, x_{J}^{m} )^\top$ onto $\Gamma^m$, with $\hat x_{j,*}^{m} = a^m(x_{j}^{m})$ for $j=1,\dots,J$. This is well defined if $\Gamma_h^m$ is sufficiently close to $\Gamma^m$, say $\Gamma_h^m\subset D_\delta(\Gamma^m)$ for the neighborhood $D_\delta(\Gamma^m)$ defined in \eqref{def-D_delta}. 
Then we define $\hat\Gamma_{h,*}^{m}$ to be the piecewise curved triangular surface which interpolates the exact surface $\Gamma^m$ at the nodes in $\hat\bfx_*^{m}$. Thus $\hat\Gamma_{h,*}^{m}$ is the distance projection of $\Gamma_h^m$ onto the exact surface $\Gamma^m$. 
The error between $\Gamma_{h}^{m}$ and $\hat\Gamma_{h,*}^{m}$ will be estimated in this paper.
%and define $\bfx_*^{m+1}$ to be the nodal vector consisting of the new positions of the nodes in $\hat\bfx_*^{m}$ evolving under mean curvature flow (without tangential motion) from $t_{m}$ to $t_{m+1}$. The piecewise curved triangular surface determined by the nodal vector $\bfx_*^{m+1}$ is denoted by $\Gamma_{h,*}^{m+1}$. 

\subsection{Notation for finite element functions}
\label{section:notation-FE}
To simplify the notation, we use the same symbol for a finite element function defined on different discrete surfaces whenever they share the same nodal vector. Specifically, a finite element function is uniquely characterized by its nodal vector once the underlying discrete surface is specified. For instance, we denote by $X_h^m$ the finite element function with nodal vector $\bfx^m$; the meaning of $X_h^m$ is always clear from the context indicating which surface it is defined on. In particular, when regarded as a function on $\Gamma_h^0$, $X_h^m$ represents the discrete flow map from $\Gamma_h^0$ to $\Gamma_h^m$, while when viewed as a function on $\Gamma_h^m$, it coincides with the identity map, i.e., ${\rm id}(x) = x$. Accordingly, integrals such as $\int_{\Gamma_h^0} X_h^m$ and $\int_{\Gamma_h^m} X_h^m$ are well-defined, as the domain of integration is always specified.

More generally, in expressions such as $ \int_{\hat\Gamma_{h,*}^m} \nabla_{\hat\Gamma_{h,*}^m}X_h^{m+1}\cdot\nabla_{\hat\Gamma_{h,*}^m}v_h $ and $ \int_{\Gamma_h^m} \nabla_{\Gamma_h^m}X_h^{m+1}\cdot\nabla_{\Gamma_h^m}v_h $, the function $ v_h $ is represented by the same nodal vector in both cases, but it is understood to be defined on different surfaces. Similarly, the gradient operators $ \nabla_{\hat\Gamma_{h,*}^m} $ and $ \nabla_{\Gamma_h^m} $ act on $ v_h $ as a function defined on their respective surfaces. Under this convention, the notation for $ v_h $ remains unambiguous as long as the underlying surface is specified for each operation (integration, differentiation, or norm computation).

\subsection{Lift and inverse lift}

	% For a continuous function $v$ which has definition in a neighborhood of $\Ghsm$, we denote by $I_hv$ the interpolated finite element function on $\hat\Gamma_{h,*}^{m}$. 
	
	The lift of a finite element function $v_h$ from $\Ghsm$ onto $\Gm$ is defined by  
	\[
	v_h^l = v_h \circ (a^m |_{\Ghsm})^{-1} .
	\]
	This is well defined if $\Ghsm$ is sufficiently close to $\Gamma^m$, say $\Ghsm\subset D_\delta(\Gamma^m)$ for the neighborhood $D_\delta(\Gamma^m)$ defined in \eqref{def-D_delta}. Conversely, the inverse lift of a function $v \in L^2(\Gm)$ onto $\Ghsm$ is defined as $v^{-l} = v \circ a^m$.
    
%see \cite[Section 3.4]{kovacs2019convergent}. 
The lift of a finite element function $v_h$ from $\Ghm$ to $\Gamma^m$, denoted by $v_h^l$, is defined by firstly identifying $v_h$ as a finite element function with the same nodal vector on the interpolated surface $\Ghsm$ and then lift it onto $\Gm$.

\subsection{Main theoretical result}

The distance from the numerically computed surface \(\Gamma_h^m\) to the smooth surface \(\Gamma^m\) is defined as
\begin{align}\label{def-distance-error-0}
	\hat d^m(x) := \min_{y \in \Gamma^m} |x - y|, \quad \text{for } x \in \Gamma_h^m,
\end{align}
In this paper, we prove that the numerically computed surface \(\Gamma_h^m\) lies within the neighborhood \(D_\delta(\Gamma^m)\) for sufficiently small $h$, using a mathematical induction argument.
Consequently, we define the following distance error: 
\begin{align}\label{def-distance-error}
	\hat e^m(x) := a^m(x) - x \quad \text{for } x \in \Gamma_h^m,
\end{align}
where \(a^m: D_\delta(\Gamma^m) \rightarrow \Gamma^m\) is the distance projection defined in Section~\ref{section:projection}. It then follows that \(\hat d^m(x) = |\hat e^m(x)|\). 
{The lift of \(\hat e^m\) from $\Gamma_h^m$ to \(\Gamma^m\) through distance projection is defined as
\begin{align}\label{lift-lift}
	(\hat e^m)^{\hat l} = \hat e^m \circ (a^m|_{\Gamma_h^m})^{-1} .
    %\quad \text{for all } v \in W^{1,\infty}(\Gamma_h^m).
\end{align}
}

%	where $X_h^{m,l}$ denotes the lift of $X_h^m$ onto $\Gm$ through the interpolated surface $\Ghsm$.
The main theoretical result of this article is stated in the following theorem.

	\begin{theorem}[Convergence of the MDR method]\label{thm:main}
		%Let $X_h^{m, l}\in S_h(\Gm)$ be the lift of the numerical flow $X_h^m\in S_h(\Ghsm)$ computed from the scheme \eqref{eq:BGN} with initial value $X_h^0 = I_h(\Gamma(0))$ where Gauss-Lobatto quadrature nodes is used for the mass lumping. 
		Suppose that the flow map $\phi:\Gamma^0\times[0,T]\rightarrow \R^3$ of the mean curvature flow and its inverse map $\phi(\cdot,t)^{-1}:\Gamma(t)\rightarrow\Gamma^0$ are both sufficiently smooth, uniformly with respect to $t\in[0,T]$, and the initial approximation $\Gamma_h^0$ is sufficiently good, satisfying \eqref{P0}. 
Let $\Gamma_h^m$ be the piecewise curved triangular surface at time level $t_m$ computed by the MDR method in \eqref{NBGN} with initial condition $\Gamma_h^0$. 
		%Then the following results hold:
		%\begin{enumerate}
		%\item 
		Then, for any given constant $c$ (independent of $\tau$ and $h$), there exists a positive constant $h_0$ such that for $\tau \le c h^k$ and $h\le h_0$ the following results hold for finite elements of degree $k\ge 3${\rm:} 
		
\begin{enumerate}
\item
The numerically computed surface $\Gamma_h^m$ is in the neighborhood $D_\delta(\Gamma^m)$. Thus the distance error $\hat e^m$ and its lift $(\hat e^m)^{\hat l}$ onto $\Gm$ are well-defined by \eqref{def-distance-error}-\eqref{lift-lift}. 

\item
The following error estimate holds:
		\begin{align}
			\label{eq:err_est_1}
			\max_{0 \leq m \leq \lfloor {T}/{\tau} \rfloor} \| (\hat{e}^m)^{\hat l} \|_{L^2(\Gamma^m)}^2 
	+ \sum_{m = 0}^{\lfloor {T}/{\tau} \rfloor} \tau \| \nabla_{\Gamma^m} (\hat{e}^m)^{\hat l} \|_{L^2(\Gamma^m)}^2 
	\leq C (\tau^2 + h^{2k}).
		\end{align}
		where the constant $C$ is independent of $\tau$ and $h$ (but may depend on $\kappa_0$ and $T$).
\end{enumerate}
\end{theorem}
	
	%..................................................
	\section{The underlying framework}\label{sec:framework}
	
	In this section, we present the general settings of the underlying framework in which we estimate the distance error from the numerically computed surface $\Gamma_h^m$ to the exact surface $\Gamma^m$.
	
	\subsection{Notaions}\label{sec:notation}
	
	The following notations will be frequently used in this article. They are similar to the notations in \cite[Section 3.1]{bai2024new} and are listed below for the convenience of the readers. 
	
	\begin{longtable}{p{1.1cm}p{13cm}}
		%\caption*{Regular meeting schedule}\\ 
		$\Gamma^m$:
		& 
		The exact smooth surface at time level $t=t_m$.\\
		
		$\Gamma_h^m$:
		&
		The numerically computed surface at time level $t=t_m$.\\
		
		$\bfx^{m}$: 
		&
		The nodal vector $\bfx^m=(x_1^m,\dots,x_J^m)^\top$ consisting of the positions of nodes on $\Gamma_h^m$.\\
		
		$\hat\bfx_*^{m}$: 
		&
		The distance projection of $\bfx^{m}$ onto the exact surface $\Gamma^m$, i.e., $\hat\bfx_*^{m}=(\hat x_{1,*}^m,\dots,\hat x_{J,*}^m)^\top$ with $\hat x_{j,*}^m=a^m(x_j^m)$. \\
		
		$\bfx_*^{m+1}$: 
		&
		The new position of $\hat\bfx_*^{m}$ evolving under mean curvature flow from $t_m$ to $t_{m+1}$. \\
		
		$\Ghsm$: 
		&
		The piecewise triangular surface which interpolates $\Gamma^m$ at the nodes in $\hat\bfx_*^{m}$.\\
		
		$\Gamma_{h,*}^{m+1}$: 
		&
		The piecewise triangular surface which interpolates $\Gamma^{m+1}$ at the nodes in $\bfx_*^{m+1}$.\\
		
		$X_{h}^{m}$: 
		&
		The finite element function with nodal vector $\bfx^m$. It coincides with the identity map when it is considered as a function on $\Gamma_{h}^m$. It coincides with the discrete flow map from $\Gamma_{h}^0$ to $\Gamma_{h}^m$ when it is considered as a function on $\Gamma_{h}^0$. \\
		
		$X_{h}^{m+1}$: 
		&
		The finite element function with nodal vector $\bfx^{m+1}$. 
		When it is considered as a function on $\Gamma_{h}^m$, it represents the local flow map from $\Gamma_{h}^m$ to $\Gamma_{h}^{m+1}$.\\
		
		$\hat X_{h,*}^{m}$: 
		&
		The finite element function with nodal vector $\hat\bfx_*^m$. It coincides with the identity map when it is considered as a function on $\Ghsm$. It coincides with the discrete flow map from $\Gamma_h^0$ to $\Ghsm$ when it is considered as a function on $\Gamma_{h}^0$.\\
		
		$X_{h,*}^{m+1}$: 
		&
		The finite element function with nodal vector $\bfx_*^{m+1}$. 
		When it is considered as a function on $\Ghsm$, it represents the local flow map from $\Ghsm$ to $\Gamma_{h,*}^{m+1}$.\\
		
		$X^{m+1}$: 
		&
		The local flow map from $\Gamma^m$ to $\Gamma^{m+1}$ under mean curvature flow. \\
		
		$\hat e_{h}^m$: 
		&
		The finite element error function with nodal vector $\hat\bfe^m=\bfx^m-\hat\bfx_*^m$.\\ 
		
		$e_{h}^{m+1}$: 
		&
		The auxiliary error function with nodal vector $\bfe^{m+1}=\bfx^{m+1}-\bfx_*^{m+1}$.\\

		$H^m$:
		&
		The mean curvature on $\Gm$.\\
		
		$n^m$: 
		&
		The unit normal vector on $\Gamma^m$. \\
		
		$a^m(x)$:
		&
		The distance projection of $x$ onto $\Gm$. It is well defined in $D_\delta(\Gamma^m)$.\\
		
		$n^m_*$: 
		&
		The unit normal vector of $\Gamma^m$ inversely lifted to the neighborhood $D_\delta(\Gamma^m)$ of $\Gamma^m$, i.e., $n^m_*=n^m\circ a^m$. \\
		
		$n_h^m$: 
		&
		The unit normal vector on $\Gamma_{h}^m$. \\
		
		$\nahm$:
		&
		The averaged normal vector (not necessarily unit) on $\Ghm$, defined in \eqref{def-bar-nhm}.\\
		
		$\hat n_{h,*}^m$: 
		&
		The unit normal vector on $\Ghsm$. \\
		
		$\nahsm$:
		&
		The averaged normal vector (not necessarily unit) on $\Ghsm$, defined in \eqref{def:nahsm}.\\
		
		$\Nsm$: 
		&
		The normal projection operator $\Nsm=n^m_* (n^m_*)^\top$ which is well defined in the neighborhood $D_\delta(\Gamma^m)$ of $\Gamma^m$. It is well defined on the interpolated surface $\hat\Gamma_{h,*}^m$ when $\hat\Gamma_{h,*}^m$ is sufficiently close to the exact surface $\Gamma^m$. \\
		
		$N^m$: 
		&
		The normal projection operator $N^m=n^m (n^m)^\top$ on $\Gamma^m$. 
		Thus $N_*^m$ is the extension of $N^m$ to the neighborhood $D_\delta(\Gamma^m)$ of $\Gm$. \\
		
		$\Nhsm$: 
		&
		The normal projection operator $\Nhsm=\hat n^m_{h,*} (\hat n^m_{h,*})^\top$ on $\Ghsm$. \\
		
		$\Nahsm$: 
		&
		The averaged normal projection operator $\Nahsm= \frac{\bar n^m_{h,*}}{|\bar n^m_{h,*}|} (\frac{\bar n^m_{h,*}}{|\bar n^m_{h,*}|})^\top$ on $\Ghsm$. \\

		$\Nahm$: & The averaged normal projection operator $\Nahm= \frac{\bar n^m_{h}}{|\bar n^m_{h}|} (\frac{\bar n^m_{h}}{|\bar n^m_{h}|})^\top$ on $\Ghm$. \\
		
		$T_*^m$: 
		&
		The tangential projection operator $T_*^m=I - n^m_* (n^m_*)^\top$ which is well defined in the neighborhood $D_\delta(\Gamma^m)$ of $\Gamma^m$. It is well defined on the interpolated surface $\hat\Gamma_{h,*}^m$ when $\hat\Gamma_{h,*}^m$ is sufficiently close to the exact surface $\Gamma^m$. \\
		
		$T^m$: 
		&
		The tangential projection operator $T^m=I - n^m (n^m)^\top$ on $\Gamma^m$. 
		Thus $T_*^m$ is the extension of $T^m$ to the neighborhood $D_\delta(\Gamma^m)$ of $\Gm$. \\
		
		$\Thsm$: 
		&
		The tangential projection operator $\Thsm=I - \nhsm (\nhsm)^\top$ on $\Ghsm$. \\
		
		$\Tbhsm$: 
		&
		The averaged tangential projection operator $\Tbhsm=I - \frac{\bar n^m_{h,*}}{|\bar n^m_{h,*}|} (\frac{\bar n^m_{h,*}}{|\bar n^m_{h,*}|})^\top$ on $\Ghsm$.\\

		$\Tbhm$: 
		&
		The averaged tangential projection operator $\Tbhm=I - \frac{\bar n^m_{h}}{|\bar n^m_{h}|} (\frac{\bar n^m_{h}}{|\bar n^m_{h}|})^\top$ on $\Ghm$.
		
	\end{longtable}

	\subsection{Approximation properties of the interpolated surface $\Ghsm$}
	\label{3.2}
	
	If $K$ is a curved triangle of $\Ghsm$ then we denote by $K^0\subset\Gamma_{h}^0$ the  curved triangle which is mapped to $K$ by the discrete flow map $\hat X_{h,*}^m:\Gamma_{h}^0\rightarrow\Ghsm$, and denote by $F_{K^0}: K_{\rm f}^0 \rightarrow K^0$ the parametrization of the curved triangle $K^0\subset\Gamma_h^0$, where $K_{\rm f}^0$ is the flat triangle which shares the same three vertices with $K^0$. The flat triangle $K_{\rm f}^0$ form a piecewise flat triangular surface $\Gamma_{h,{\rm f}}^0$.
	% $$ 
	% \Gamma_{h,{\rm f}}^0=\bigcup_{K^0\subset\Gamma_h^0} K_{\rm f}^0 . 
	% $$
	As explained in Section \ref{section:notation-FE}, we use the same notation $\hat X_{h,*}^m:\Gamma_{h,{\rm f}}^0\rightarrow\Ghsm$ to denote the unique piecewise polynomial of degree $k$  (with nodal vector $\hat{\bf x}_*^m$ as before) which parametrizes $\Ghsm$ through $\Gamma_{h,{\rm f}}^0$, and denote by $\| \hat X_{h,*}^m\|_{H^j_h(\Gamma_{h,\rm f}^0)}$ and $\| \hat X_{h,*}^m\|_{W^{j,\infty}_h(\Gamma_{h,\rm f}^0)}$ the piecewise Sobolev norms on the piecewise flat triangular surface $\Gamma_{h,\rm f}^0$, i.e., 
	
	\begin{align*}
		\| \hat X_{h,*}^m\|_{H^j_h(\Gamma_{h,\rm f}^0)}  
		:= \Big( \sum_{K_{\rm f}^0\subset \Gamma_{h,\rm f}^0} \| \hat X_{h,*}^m\|_{H^j(K_{\rm f}^0)}^2 \Big)^{\frac12} 
		\quad\mbox{and}\quad 
		\| \hat X_{h,*}^m\|_{W^{j,\infty}_h(\Gamma_{h,\rm f}^0)}  
		:=  \max_{K_{\rm f}^0\subset \Gamma_{h,\rm f}^0} \| \hat X_{h,*}^m\|_{W^{j,\infty}(K_{\rm f}^0)} . 
	\end{align*}
	For the discrete flow maps $\hat X_{h,*}^m:\Gamma_{h,{\rm f}}^0\rightarrow\Ghsm$, $m=0,\cdots, l$, we denote
	\begin{equation}\label{kl*}
		\begin{aligned}
			\kappa_l&:=\max\limits_{0\leq m\leq l}(\|\hat{X}^m_{h,*}\|_{H^{k-1}_h(\Gamma^0_{h,{\rm f}})}+\|\hat{X}^m_{h,*}\|_{W^{k-2,\infty}_h(\Gamma^0_{h,{\rm f}})}+\|(\hat{X}^m_{h,*})^{-1}\|_{W^{1,\infty}(\hat \Gamma^m_{h,*})}),\\
			\kappa_{*,l}&:=\max\limits_{0\leq m\leq l}(\|\hat{X}^m_{h,*}\|_{H^{k}_h(\Gamma^0_{h,{\rm f}})}+\|\hat{X}^m_{h,*}\|_{W^{k-1,\infty}_h(\Gamma^0_{h,{\rm f}})}).
		\end{aligned}
	\end{equation}
	By pulling functions on $\Ghsm$ back to $\Gamma_{h,\rm f}^0$ via the map $\hat X_{h,*}^m:\Gamma_{h,\rm f}^0\rightarrow \Ghsm$ (and vice visa), one can see that the $W^{1,p}$ and $L^p$ norms on $\Ghsm$ and on $\Gamma_{h,\mathrm{f}}^0$ are equivalent, up to constants that depend on $\kappa_l$, i.e., 
	\begin{align}\label{W1p-equiv}
	C_{\kappa_m} ^{-1} \|v\|_{W^{i,p}(\Ghsm)} \le \|v\circ \hat X_{h,*}^m\|_{W^{i,p}(\Gamma_{h,\rm f}^0)} \le C_{\kappa_m} \|v_h\|_{W^{i,p}(\Ghsm)},
	% \quad\mbox{for}\,\,\, 1\le p\le\infty,\,\,\, 0\leq m\leq l . 
	\end{align}
	for $1\le p\le\infty, 0\leq m\leq l,i=0,1$.
This equivalence relation requires $k\ge 3$ according to the definition of $\kappa_l$ in \eqref{kl*}.  Subsequently, the following inverse inequalities hold:
\begin{subequations}\label{inverse-ineq}
	\begin{align}
		\|v\|_{W^{1,p}(\Ghsm)} \le C_{\kappa_m} \|v\circ \hat X_{h,*}^m\|_{W^{1,p}(\Gamma_{h,\rm f}^0)} \le  C_{\kappa_m} h^{-1} \|v\circ \hat X_{h,*}^m\|_{L^p(\Gamma_{h,\rm f}^0)}   \le  C_{\kappa_m} h^{-1} \|v\|_{L^p(\Ghsm)},\\
		\|v\|_{L^p(\Ghsm)} \le C_{\kappa_m} \|v\circ \hat X_{h,*}^m\|_{L^p(\Gamma_{h,\rm f}^0)} \le  C_{\kappa_m} h^{2/p - 2/q} \|v\circ \hat X_{h,*}^m\|_{L^q(\Gamma_{h,\rm f}^0)}   \le  C_{\kappa_m} h^{2/p - 2/q} \|v\|_{L^q(\Ghsm)},
	\end{align}
\end{subequations}
for $1\le p,q\le\infty$, $p\ge q $ and $ 0\leq m\leq l$. 
% Additionally, inverse inequalities hold with constants only depending on $\kappa_l$.

	For a curved triangle $K\subset \Ghsm$, its parametrization $F_K=\hat{X}^m_{h,*}|_{K_{\rm f}^0}: \Gamma_{h,{\rm f}}^0\rightarrow K$ (a polynomial of degree $k$) satisfies the following estimates as a result of \eqref{kl*}: For $m=0,\dots,l$,
	\begin{equation}
		\begin{aligned}
			\Big( \sum_{K\subset \Ghsm} \|F_K\|_{H^{k-1}(K_{\rm f}^0)}^2 \Big)^{\frac12} +\max_{K\subset \Ghsm} \| F_K\|_{W^{k-2,\infty}(K_{\rm f}^0)}+\max_{K\subset \Ghsm} \| F_K^{-1}\|_{W^{1,\infty}(K)}&= \kappa_l,\\
			\Big( \sum_{K\subset \Ghsm} \|F_K\|_{H^k(K_{\rm f}^0)}^2 \Big)^{\frac12} +\max_{K\subset \Ghsm} \| F_K\|_{W^{k-1,\infty}(K_{\rm f}^0)}&=\kappa_{*,l}.
		\end{aligned}
	\end{equation}
	
	For a curved triangle $K\subset \Ghsm$, we denote by $I_{K_{\rm f}^0}$ the interpolation operator onto the flat triangle $K_{\rm f}^0$. Since $F_{K}$ and $a^m\circ F_{K}$ are equal at the nodes of $K_{\rm f}^0$, and $F_{K}$ is a polynomial of degree $k$ on $K_{\rm f}^0$, it follows that $I_{K_{\rm f}^0}[ a^m\circ F_{K} ]=F_{K}$. 
	The interpolation of the distance projection $a^m |_\Ghsm: \Ghsm\rightarrow\Gamma^m$ onto the curved surface $\Ghsm$ is defined as 
	$$
	I_h a^m:=I_{K_{\rm f}^0} [a^m\circ F_{K}] \circ F_{K}^{-1} = {\rm id }
	\quad\mbox{on each curved triangle}\,\, K\subset \Ghsm . 
	$$
	For a smooth  function $f$ on the smooth surface $\Gamma^m$, we denote by $I_hf$ the interpolation of the inversely lifted function $f^{-l}=f\circ a^m$ onto $\Ghsm$, i.e.,
	$$
	I_h f:=I_{K_{\rm f}^0} [f\circ a^m\circ F_{K}] \circ F_{K}^{-1} 
	\quad\mbox{on a curved triangle}\,\, K\subset \Ghsm . 
	$$
	We denote by $(I_hf)^l = (I_hf)\circ (a^m |_\Ghsm)^{-1}$ the lift of $I_hf$ onto $\Gamma^m$. 
	For a piecewise smooth function $f$ on $\Ghsm$ (rather than on $\Gm$), which is continuous on $\Ghsm$, we use the same notation $I_h f$ to denote the following interpolated function on $\Ghsm$:   
	$$
	I_h f:=I_{K_{\rm f}^0} [f\circ F_{K}] \circ F_{K}^{-1} 
	\quad\mbox{on a curved triangle}\,\, K\subset \Ghsm . 
	$$
In \cite[inequality (3.4)]{bai2024new}, it is shown that the parametrization $a^m: \Ghsm\rightarrow\Gamma^m$ of the smooth surface $\Gm$ satisfies $I_ha^m={\rm id}$ on $\Ghsm$ and the following estimates for $m=0,\cdots,l$: 
	\begin{equation}\label{aH}
		\|a^m - I_ha^m \|_{L^2(\Ghsm)} 
		+ h \|a^m - I_ha^m \|_{H^1(\Ghsm)} + h^2\|a^m - I_ha^m \|_{W^{1,\infty}(\Ghsm)}\le C_{\kappa_l}(1+\kappa_{*,l})h^{k+1}.
	\end{equation}
    
If the mesh size \(h\) is sufficiently small so that the following inequality holds (this will be ensured by the mathematical induction hypothesis (3), which guarantees that \(\kappa_{*,l}\) and \(\kappa_{l}\) are bounded independently of \(h\); see the next subsection):
	\begin{equation}\label{cond1}
		(1+\kappa_{*,l})h^{k-2.6}\le C_{\kappa_l}^{-1},
	\end{equation}
	then 
	\begin{align}\label{lift_norm_equiv}
		\|a^m - {\rm id} \|_{W^{1,\infty}(\Ghsm)}\le C_{\kappa_l}(1+\kappa_{*,l})h^{k-1} \le h^{1.6}.	
	\end{align}
Therefore, for sufficiently small $h$, we have $\Ghsm \subset D_\delta(\Gamma^m)$ {and the map $a^m:\Ghsm\rightarrow\Gamma^m$ is invertible (as a small $h^{1.6}$-perturbation of the identity map in the $W^{1,\infty}$ norm). These properties imply that, through change of coordinates, the $W^{1,p}$ and $L^p$ norms of functions on $\Ghsm$ and $\Gamma^m$, related through the map $a^m:\Ghsm\rightarrow\Gamma^m$, are equivalent for $1\le p\le \infty$, i.e.,}
	\begin{align}\label{norm-equiv-lift}
	  C_0 ^{-1} \|v\| _{W^{1,p}(\Ghsm)} \le \|v^l\|_{W^{1,p}(\Gamma^m)} \le C_0 \|v\|_{W^{1,p}(\Ghsm)} \quad \text{for any } v\in W^{1,p}(\Ghsm) ,
	\end{align}
where $C_0$ is independent of $h$, $\tau$, $\k_l$, and $\k_{*,l}$, but may depend on $T$. Moreover, for a smooth function $f$ on the smooth surface $\Gm$, the following approximation estimates hold for $m=0,\dots,l$ (see \cite[inequality (3.5)]{bai2024new}):
	\begin{equation}\label{Ihf}
		\begin{aligned}
			\|f^{-l} - I_hf \|_{L^2(\Ghsm)} 
			+ h\|f^{-l} - I_hf \|_{H^1(\Ghsm)} 
			&\le C_{\kappa_l}(1+\kappa_{*,l})h^{k+1} ,\\
			\|f - (I_hf)^l \|_{L^2(\Gamma^m)} 
			+ h\|f - (I_hf)^l \|_{H^1(\Gamma^m)} 
			&\le C_{\kappa_l}(1+\kappa_{*,l})h^{k+1} .
		\end{aligned}
	\end{equation}
	
	We denote by $n^m$ and $H^m$ the unit normal vector and the mean curvature on $\Gamma^m$, respectively, and denote by $n_*^m=n^m\circ a^m$ and $H_*^m=H^m\circ a^m$ the smooth extensions of $n^m$ and $H^m$ to the neighborhood $D_\delta(\Gm)$ of $\Gm$, where $D_\delta(\Gm)$ is defined in \eqref{def-D_delta}, with  
	\begin{align*}
		\|n_*^m\|_{W^{j,\infty}(D_\delta(\Gm))}+ \|H_*^m\|_{W^{j,\infty}(D_\delta(\Gm))} \le C_{j}\quad\mbox{for all $j\ge 0$} .
	\end{align*}
In particular, $\nsm$ and $H_*^m$ are well defined on $\Ghsm\subset D_\delta(\Gm)$, where the inclusion relation is guaranteed by \eqref{lift_norm_equiv}. 

By using the parametrizations $\hat X_{h,*}^m:\Gamma_{h,\rm f}^0\rightarrow \Ghsm$ and $a^m\circ \hat X_{h,*}^m:\Gamma_{h,\rm f}^0\rightarrow\Gm$, respectively, and the notation $F_K=\hat{X}^m_{h,*}|_{K_{\rm f}^0}$ for a flat triangle $K_{\rm f}^0\subset \Gamma_{h,\rm f}^0$, the normal vectors of $\Ghsm$ and $\Gm$ have the following expressions: 
	\begin{align}\label{nexpre}
		\hat n_{h,*}^m \circ F_K= \frac{\partial_u F_K\times \partial_vF_K}{|\partial_uF_K\times \partial_vF_K|}
		\quad\mbox{and}\quad
		n_*^m \circ F_K=\frac{\partial_u(a^m\circ F_K)\times \partial_v(a^m\circ F_K)}{|\partial_u(a^m\circ F_K)\times \partial_v(a^m\circ F_K)|} , 
	\end{align}
where $(u,v)$ is the local coordinate on the flat triangle $K_{\rm f}^0\subset \Gamma_{h,\rm f}^0$.
The following estimates are shown in \cite[inequalities (3.6) and (3.7)]{bai2024new} for finite elements of degree $k\ge 3$:
	\begin{align}	
		\label{ninfty}
		\|\hat n_{h,*}^m - n_*^m\|_{L^\infty(\Ghsm)} 
		&\le C_{\kappa_l}(1+\kappa_{*,l})h^{k-1},\\ 
		\label{n2}
		\|\hat n_{h,*}^m - n_*^m\|_{L^2(\Ghsm)} 
		&\le C_{\kappa_l}(1+\kappa_{*,l})h^k,\\
		\label{H1}
		\|\hat n_{h,*}^m - n_*^m\|_{H^1(\Ghsm)} 
		&\le C_{\kappa_l}(1+\kappa_{*,l})h^{k-1},\\
		\label{n1infty}
		\|\hat n_{h,*}^m - n_*^m\|_{W^{1,\infty}(\Ghsm)} 
		&\le C_{\kappa_l}(1+\kappa_{*,l})h^{k-2}. 
	\end{align}
	From \eqref{ninfty}--\eqref{n1infty} and the mesh size assumption \eqref{cond1}, the following result holds:
	\begin{equation}\label{nsm}
		\|\hat n_{h,*}^m - n_*^m\|_{L^\infty(\Ghsm)} + \|\hat n_{h,*}^m - n_*^m\|_{H^1(\Ghsm)}  + h\|\hat n_{h,*}^m - n_*^m\|_{W^{1,\infty}(\Ghsm)}
		\le h^{1.6}.
	\end{equation}
	% inverse inequality and the boundedness of $\nsm$ imply the boundedness of $\nhsm$ via the triangle inequality:
	% \begin{align}
	% 	\label{nhsmW}
	% 	\|\hat n_{h,*}^m \|_{W^{1,\infty}(\Ghsm)}\lesssim 1+h^{0.6}\lesssim 1,\\
	% 	\label{nhsmH}
	% 	\|\hat n_{h,*}^m \|_{H^1(\Ghsm)}\lesssim 1+h^{1.6}\lesssim 1.
	% \end{align}
The boundedness of $\kappa_l$ and $\kappa_{*,l}$ will be established through error analysis and mathematical induction. As a result, the condition in \eqref{cond1} can be satisfied by choosing $h$ sufficiently small.

	% In the rest of this article, we denote by $C$ a generic positive constant which may be different at different occurrences, possibly dependent on $\kappa_l$ and $T$, but is independent of $\tau$, $h$, $m$ and $\kappa_{*,l}$. We denote by $C_0$ generic positive constant which is independent of $\kappa_l$. For the simplicity of notation, we denote by $A \lesssim B$ the statement ``$A\le CB$ for some constant $C$''. The statement ``for sufficiently small $h$ ..." means that ``there exists a constant $C$, possibly depending on $\kappa_l$, such that for $h\le C^{-1}$ ...''.
	%..................................................
	
	\subsection{Induction assumptions}\label{induc ass}
	
	We consider mathematical induction on $l$, by assuming that the following conditions hold for $m=0,\dots,l$ (and then prove that these conditions could be recovered for $m=l+1$): 
	\begin{enumerate}
		\item[(1)]
		The numerically computed surface $\Gamma_h^m$ and its distance projection surface $\Ghsm$ are both in a $\delta$-neighborhood of the exact surface $\Gamma^m$. 
		% Therefore, the distance projection of the nodes of $\Gamma_h^m$ onto $\Gamma^m$ are well defined (therefore the interpolated surface $\Ghsm$ is well defined). 
		
		\item[(2)]
		The error $\ehm=X_h^m-\hat X_{h,*}^m$ satisfies the following estimate (with coefficient $1$ on the right-hand side of inequality): 
		\begin{align}\label{cond2}
			\| \ehm \|_{L^2(\Ghsm)} + h\| \ehm \|_{H^1(\Ghsm)} &\le h^{2.6}.
		\end{align}
		
		\item[(3)]
		The constants $\kappa_l$ and $\kappa_{*,l}$ are bounded by some constants $C_{\#}$ and $C_{*}$, respectively, {which are (determined through the proof, in Appendix \ref{appendix_H}) independent of $\tau$ and $h$.} Moreover, the mesh size $h$ is chosen sufficiently small so that condition \eqref{cond1} is satisfied.
	\end{enumerate}
For sufficiently small $h$, these conditions are satisfied for $l=0$. 

In the remainder of this article, we adopt the following notational conventions. {We denote by $C$ a generic positive constant that may vary from one occurrence to another and may depend on $\kappa_l$, $C_\#$ and $T$, yet it remains independent of $\tau$, $h$, $m$, $\kappa_{*,l}$ and $C_*$.} Similarly, we use $C_0$ to denote a generic positive constant that is independent of $\kappa_l$, $C_{\#}$, $\kappa_{l,*}$ and $C_{*}$. For simplicity, we write $A \lesssim B$ to signify that $A \le CB$ for some constant $C$. Moreover, the phrase “for sufficiently small $h$” is interpreted to mean that there exists a constant $C$, which may depend on $\kappa_l$ and $C_\#$, such that the subsequent estimates are valid for $h \le C^{-1}$. 
{
\begin{remark}{\upshape
The facts that the generic constants $C$ in our error estimates are independent of $ \kappa_{*,l} $ and $ C_* $, and the generic constants $C_0$ are independent of $\kappa_l$, $C_{\#}$, $\kappa_{l,*}$ and $C_{*}$, ensure that we can ultimately select constants $C_\#$ and $ C_* $ in Appendix \ref{appendix_H} such that the third induction assumption stated above, once satisfied for $ m=l $, can be recovered for $ m=l+1 $.}
\end{remark}
}

	Based on these induction assumptions, the following results are obtained from \eqref{cond2} by applying inverse inequalities \eqref{inverse-ineq}:
	\begin{align}\label{Linfty-W1infty-hat-em}
		\| \nabla_{\Ghsm} \ehm \|_{L^2(\Ghsm)}\le h^{1.6},
		\quad
		\|  \ehm \|_{L^\infty(\Ghsm)} \lesssim h^{1.6}
		\quad\mbox{and}\quad 
		\| \nabla_\Ghsm \ehm \|_{L^\infty(\Ghsm)} \lesssim h^{0.6},
  	\end{align} 
which are  in accordance with the notational conventions stated above, any dependence of the constants on $\kappa_l$ and $C_\#$ is omitted.

    We define a family of intermediate surfaces between the interpolated surface $\Ghsm$ and the numerical surface $\Ghm$, i.e., 
    $$
    \hat\Gamma_{h,\theta}^m=(1 - \theta)\Ghsm + \theta\Ghm ,\quad  \theta\in [0, 1],
    $$
    which are curved triangulated surfaces determined by the nodal vector $(1-\theta)\hat{\bf x}^m_* + \theta {\bf x}^m$. We begin with a lemma from \cite[Lemma 4.3]{kovacs2017convergence}, which establishes that norms of finite element functions, defined with identical nodal values across a family of surfaces, are equivalent. 
	% This, according to \cite[Lemma 4.3]{kovacs2017convergence}, guarantees the equivalence of $L^p$ and $W^{1,p}$ norms, $1\le p\le \infty$, of finite element functions $v_h$ (with a fixed nodal vector) on the family of intermediate surfaces 
	% $$\hat\Gamma_{h,\theta}^m=(1 - \theta)\Ghsm + \theta\Ghm ,\quad  \theta\in [0, 1],$$
	% between the surface $\Ghsm$ and the numerical surface $\Ghm$.

\begin{lemma}[\!\!{\cite[Lemma 4.3]{kovacs2017convergence}}]
\label{equi-MA}
    Suppose that \( \|\nabla_{\hat{\Gamma}_{h,*}^m} \hat{e}_h^m\|_{L^\infty(\hat{\Gamma}_{h,*}^m)} \le \frac{1}{2} \) for \( \theta \in [0,1] \). Then, for all \( 1 \le p \le \infty \), the following norm equivalences hold:
    \begin{align}\label{equiv-norm-Lp-W1p}
    \begin{aligned}
        \|v_h\|_{L^p(\hat{\Gamma}_{h,*}^m)} &\lesssim \|v_h\|_{L^p(\hat{\Gamma}_{h,\theta}^m)} \lesssim \|v_h\|_{L^p(\hat{\Gamma}_{h,*}^m)}, \\
        \|\nabla_{\hat{\Gamma}_{h,*}^m} v_h\|_{L^p(\hat{\Gamma}_{h,*}^m)} &\lesssim \|\nabla_{\hat{\Gamma}_{h,\theta}^m} v_h\|_{L^p(\hat{\Gamma}_{h,\theta}^m)} \lesssim \|\nabla_{\hat{\Gamma}_{h,*}^m} v_h\|_{L^p(\hat{\Gamma}_{h,*}^m)}.
    \end{aligned}
    \end{align}
\end{lemma}

In particular, it follows from the induction assumption \eqref{Linfty-W1infty-hat-em} that, for sufficiently small $h$, norm equivalence relations in \eqref{equiv-norm-Lp-W1p} hold.

% 	{\b To streamline notation, each finite element function is represented by its nodal vector, a correspondence that becomes unique once the discrete surface is specified. For example, the two integrands of
% \begin{align*}
%     \int_{\hat\Gamma_{h,*}^m} v_h
%     \quad\text{and}\quad
%     \int_{\Gamma_h^m} v_h
% \end{align*} share the same nodal vector but the corresponding domains of definition of this finite element function are different. Similarly, the gradient operators \(\nabla_{\Ghsm}\) and \(\nabla_{\Ghm}\) acting on \(v_h\) are defined over distinct surfaces since the domains of definition corresponding to $v_h$ aligns with the ambient surfaces on which each gradient operator is applied. Once the underlying surface is fixed, the notation \(v_h\) is unambiguous. Since all computations in this paper involve only integrals, differentiations, and norms, this convention guarantees that finite element functions are always uniquely and clearly tied to the surface specified in each integral, differentiation, or norm.
% }
	%.................
	
\subsection{Estimates of the averaged normal vectors}

On the interpolated surface \( \hat{\Gamma}_{h,*}^m \), we define the 
\( L^2 \)-orthogonal projection operator \( P_{\hat{\Gamma}_{h,*}^m}: L^2(\hat{\Gamma}_{h,*}^m)^3 \rightarrow S_h(\hat{\Gamma}_{h,*}^m)^3 \) 
as follows: 
\begin{equation}\label{L2proj}
	\int_{\hat{\Gamma}_{h,*}^m} P_{\hat{\Gamma}_{h,*}^m} v \cdot \chi_h 
	= \int_{\hat{\Gamma}_{h,*}^m} v \cdot \chi_h 
	\quad \forall\, v \in L^2(\hat{\Gamma}_{h,*}^m)^3\,\,\,\mbox{and}\,\,\, \chi_h \in S_h(\hat{\Gamma}_{h,*}^m)^3.
\end{equation}
The averaged normal vector on the interpolated surface \(\Ghsm\) is then defined as 
the \( L^2 \)-projection of the exact normal vector, i.e.,
\begin{align}\label{def:nahsm}
	\bar{n}_{h,*}^m := P_{\hat{\Gamma}_{h,*}^m} \hat{n}_{h,*}^m 
\in S_h(\Ghsm)^3.
\end{align}

The following result on the \( L^p \)-stability of the $L^2$ projection operator \( P_{\hat{\Gamma}_{h,*}^m} \) will be used frequently throughout this paper. We omit the proof here as it is a standard result in the flat space, but provide a detailed proof in Appendix~\ref{appendix_0} for the readers' convenience.

\begin{lemma}\label{stability-L2}
For any \( 1 \le p \le \infty \), the \( L^2 \)-projection operator 
\( P_{\hat{\Gamma}_{h,*}^m} \) satisfies the stability estimate
\begin{align}
	\|P_{\hat{\Gamma}_{h,*}^m} v\|_{L^p(\hat{\Gamma}_{h,*}^m)}
	\le C \|v\|_{L^p(\hat{\Gamma}_{h,*}^m)},
\end{align}
where the constant \(C\) depends on 
\(\|\hat{X}_{h,*}^m\|_{W^{1,\infty}(\Gamma_{h,\mathrm{f}}^0)}\) and 
\(\|(\hat{X}_{h,*}^m)^{-1}\|_{W^{1,\infty}(\Ghsm)}\) (and therefore depends on $\kappa_{m}$).
\end{lemma}

In addition to \( n_h^m \) and \( \hat{n}_{h,*}^m \), which are the unit normal vectors on \( \Gamma_h^m \) and \( \Ghsm \), respectively, we also define \(\hat n_{h,\theta}^{m} \) as the unit normal vector of the intermediate surface $\hat{\Gamma}_{h,\theta}^{m} = (1 - \theta) \hat{\Gamma}_{h,*}^{m} + \theta \Gamma_{h}^{m}$ for $\theta \in [0,1]$. 
Then the difference \( n_h^m - \hat{n}_{h,*}^m \) can be estimated as follows through the expression of $\partial_\theta ^\bullet\hat n_{h,\theta}^{m}$ which can be found in \cite[Lemma 37]{barrett2020parametric}:
% terms of the derivative of \( \hat{e}_h^m = X_h^m - \hat{X}_{h,*}^m \). 
%To this end, we denote by \( \hat{e}_h^{m, \theta} \) the finite element function on \( \Gamma_h^{m,\theta} \) whose nodal values align with those of \( \hat{e}_h^m \).  
%This is achieved through the following relation:
\begin{align*}
    n_h^m \circ X_h^m - \hat{n}_{h,*}^m \circ \hat{X}_{h,*}^m = -\int_0^1 [(\nabla_{\Gamma_{h,\theta}^{m}} \hat{e}_h^{m}) \hat n_{h,\theta}^{m}] \circ \hat{X}_{h,\theta}^m \, d\theta , 
\end{align*}
where $X_h^m$, $\hat{X}_{h,*}^m$ and $\hat{X}_{h,\theta}^m$ are considered as maps from $\Gamma_h^0$ to $\Ghm$, $\Ghsm$ and $\hat\Gamma_{h,\theta}^m$, respectively. 

By employing the inductive assumption in \eqref{Linfty-W1infty-hat-em} and utilizing the equivalence of the \( L^p \) and \( W^{1,p} \) norms, we derive the following bounds:
\begin{align}
    \| n_h^m - \hat{n}_{h,*}^m \|_{L^2(\Ghsm)} &\lesssim \| \nabla_{\Gamma_{h,*}^m} \hat{e}_h^m \|_{L^2(\Ghsm)}, \label{n22} \\
    \| n_h^m - \hat{n}_{h,*}^m \|_{L^\infty(\Ghsm)} &\lesssim \| \nabla_{\Gamma_{h,*}^m} \hat{e}_h^m \|_{L^\infty(\Ghsm)}, \label{ninft}
\end{align}
where \( n_h^m \) is interpreted as \( n_h^m \circ X_h^m \circ (\hat X_{h,*}^m)^{-1} \) defined on $\Ghsm$.

For the simplicity of notation, in the remainder of this paper, any function \( f \) defined on \( \Gamma_h^m \) can be similarly viewed as a function on \( \hat{\Gamma}_{h,*}^m \) through the transformation \( f \circ X_h^m \circ (\hat X_{h,*}^m)^{-1} \).

By utilizing the definitions in \eqref{def-bar-nhm} and \eqref{def:nahsm}, we proceed to estimate the discrepancy between \( \bar{n}_{h,*}^m \) and \( \bar{n}_h^m \), when they are both viewed as finite element functions on $\hat \Gamma_{h,*}^m$. This is formalized in the following lemma.

\begin{lemma}\label{Lem:normalvector}
    The following approximation properties hold:
	\begin{subequations}\label{eq:nsa}
		\begin{align}
			 \| \bar{n}_{h,*}^m - \bar{n}_h^m \|_{L^2(\hat \Gamma_{h,*}^m)} &\lesssim \| \nabla_{\hat \Gamma_{h,*}^m} \hat e_h^m \|_{L^2(\hat \Gamma_{h,*}^m)}, \label{eq:nsa1} \\
			\| \bar{n}_{h,*}^m - n_*^m \|_{L^2(\hat \Gamma_{h,*}^m)} &\lesssim (1 + \kappa_{*,l}) h^k, \label{eq:nsa2}\\
			\| \bar{n}_h^m - n_*^m \|_{L^2(\hat \Gamma_{h,*}^m)} &\lesssim \| \nabla_{\hat \Gamma_{h,*}^m} \hat e_h^m \|_{L^2(\hat \Gamma_{h,*}^m)} + (1 + \kappa_{*,l}) h^k, \label{eq:nsa3}\\
			\| \bar{n}_{h,*}^m - \hat n_{h,*}^m \|_{L^2(\hat \Gamma_{h,*}^m)} &\lesssim (1 + \kappa_{*,l}) h^k, \label{eq:nsa4} \\
			\| \bar{n}_h^m - n_h^m \|_{L^2(\hat \Gamma_{h,*}^m)} &\lesssim \| \nabla_{\hat \Gamma_{h,*}^m} \hat e_h^m \|_{L^2(\hat \Gamma_{h,*}^m)} + (1 + \kappa_{*,l}) h^k. \label{eq:nsa5}
		\end{align}
	\end{subequations}
\end{lemma}

\begin{proof}
%    To proceed, we define the interpolated surface $\hat{\Gamma}_{h,\theta}^{m} = (1 - \theta) \hat{\Gamma}_{h,*}^{m} + \theta \Gamma_{h}^{m}$ for $\theta \in [0,1]$ and let \( \hat n_{h,\theta}^{m} \) denote the normal vector on this intermediate surface \( \hat{\Gamma}_{h,\theta}^{m} \). 
    We begin with the following expression of $\nahm - \nahsm $: 
    \begin{align}\label{eq:dif_nahm_nahsm}
        \nahm - \nahsm = P_{\Ghm} \nhm - P_{\Ghsm} \nhsm 
        &= \int_{0}^{1} \partial_\theta^\bullet (P_{\hat{\Gamma}_{h,\theta}^{m}} \hat n_{h,\theta}^{m}) \circ \hat X_{h,\theta}^m \circ (\hat X_{h,*}^{m})^{-1}\, d\theta.
    \end{align}
    where $\partial_\theta^\bullet (P_{\hat{\Gamma}_{h,\theta}^{m}} \hat n_{h,\theta}^{m}) \circ \hat X_{h,\theta}^m \circ (\hat X_{h,*}^{m})^{-1}$ denotes the pull back of $\partial_\theta^\bullet (P_{\hat{\Gamma}_{h,\theta}^{m}} \hat n_{h,\theta}^{m})$ from $\hat{\Gamma}_{h,\theta}^{m}$ to $\hat{\Gamma}_{h,*}^{m}$, while \( \partial_\theta^\bullet (P_{\hat{\Gamma}_{h,\theta}^m} \hat n_{h,\theta}^{m} ) \) can be formulated explicitly by differentiating the following relation which defines the $L^2$-orthogonal projection on $\hat{\Gamma}_{h,\theta}^m$:
    \begin{equation*}
        \int_{\hat{\Gamma}_{h,\theta}^m} \big(P_{\hat{\Gamma}_{h,\theta}^m} \hat n_{h,\theta}^{m} - \hat n_{h,\theta}^{m}\big) \cdot \chi_{h} = 0,
    \end{equation*}
    which implies that 
    \begin{equation*}
        \int_{\hat{\Gamma}_{h,\theta}^m} \big(\partial_\theta^\bullet (P_{\hat{\Gamma}_{h,\theta}^m} \hat n_{h,\theta}^{m} )- \partial_\theta^\bullet \hat n_{h,\theta}^{m}\big) \cdot \chi_{h} + \big(P_{\hat{\Gamma}_{h,\theta}^m} \hat n_{h,\theta}^{m} - \hat n_{h,\theta}^{m}\big) \cdot \chi_{h} \big(\nabla_{\hat\Gamma_{h,\theta}^m} \cdot \hat{e}_h^{m}\big) = 0,
    \end{equation*}
This gives the following expression of $\partial_\theta^\bullet (P_{\hat{\Gamma}_{h,\theta}^m} \hat n_{h,\theta}^{m})$:
    \begin{equation}\label{eq:Pnhm}
        \partial_\theta^\bullet P_{\hat{\Gamma}_{h,\theta}^m} \hat n_{h,\theta}^{m} = P_{\hat{\Gamma}_{h,\theta}^m} \left[-\left(P_{\hat{\Gamma}_{h,\theta}^m}\hat n_{h,\theta}^{m} - \hat n_{h,\theta}^{m}\right) \nabla_{\hat\Gamma_{h,\theta}^m} \cdot \hat{e}_h^{m}\right] + P_{\hat{\Gamma}_{h,\theta}^m}  \partial_\theta^\bullet \hat n_{h,\theta}^{m}.
    \end{equation}
    Moreover, we have the following relation from \cite[Lemma 3.7]{barrett2020parametric}:
    \begin{equation}\label{relation_material}
        \begin{aligned}
            \partial_\theta^\bullet \hat n_{h,\theta}^{m} &= - (\nabla_{\hat{\Gamma}_{h,\theta}^m} \hat{e}_h^{m}) \hat n_{h,\theta}^{m}.
        \end{aligned}
    \end{equation}
    Substituting \eqref{eq:Pnhm}--\eqref{relation_material} into \eqref{eq:dif_nahm_nahsm}, along with the equivalence of \( L^p \) and \( W^{1,p} \) norms in \eqref{equiv-norm-Lp-W1p}, we derive that 
    \begin{align*}
        \|\nahsm - \nahm\|_{L^2(\Ghsm)} &\leq \left\|\int_{0}^{1} P_{\hat{\Gamma}_{h,\theta}^m} \left[\left(P_{\hat{\Gamma}_{h,\theta}^m} \hat n_{h,\theta}^{m} - \hat n_{h,\theta}^{m}\right) \nabla_{\Gamma_{h,\theta}^m} \cdot \hat{e}_h^{m}\right] d\theta \right\|_{L^2(\Ghsm)} 
        + \|\nabla_{\Ghsm} \ehm\|_{L^2(\Ghsm)} \\
        &\lesssim \|\hat n_{h,\theta}^{m}\|_{L^\infty(\Ghsm)} \|\nabla_{\Ghsm} \ehm\|_{L^2(\Ghsm)} + \|\nabla_{\Ghsm} \ehm\|_{L^2(\Ghsm)} \\
        &\quad \,\,\text{(using the $L^\infty$ stability of the $L^2$ projection)} \\
        &\lesssim \|\nabla_{\Ghsm} \ehm\|_{L^2(\Ghsm)}.
    \end{align*}
Since $n_*^m$ is defined within the neighborhood $D_\delta(\Gamma^m) \supset \Ghsm $, the second result can be obtained in the following way:
    \begin{align*}
        \|\nahsm - \nsm\|_{L^2(\Ghsm)} &= \|P_{\Ghsm} \nhsm - \nsm\|_{L^2(\Ghsm)} \\
        &\lesssim \|P_{\Ghsm}(\nhsm - \nsm)\|_{L^2(\Ghsm)} + \|P_{\Ghsm} \nsm - \nsm\|_{L^2(\Ghsm)} \\
        &\lesssim (1 + \kappa_{*,l}) h^k,
    \end{align*}
    where we have used the $L^2$-stability of the $L^2$ projection and inequalities \eqref{Ihf} and \eqref{n2}.

    For the third, fourth, and fifth results, we employ the triangle inequality to obtain
    \begin{align*}
        \|\nahm - \nsm\|_{L^2(\Ghsm)} &\lesssim \|\nahsm - \nahm\|_{L^2(\Ghsm)} + \|\nahsm - \nsm\|_{L^2(\Ghsm)} \notag \\
        &\lesssim \|\nabla_{\Ghsm} \ehm\|_{L^2(\Ghsm)} + (1 + \kappa_{*,l}) h^k, \\
        \|\nahsm - \nhsm\|_{L^2(\Ghsm)} &\lesssim \|\nahsm - \nsm\|_{L^2(\Ghsm)} + \|\nsm - \nhsm\|_{L^2(\Ghsm)} \notag 
        \lesssim (1 + \kappa_{*,l}) h^k, \\
        \|\nahm - \nhm\|_{L^2(\Ghsm)} &\lesssim \|\nahm - \nsm\|_{L^2(\Ghsm)} + \|\nsm - \nhsm\|_{L^2(\Ghsm)} + \|\nhsm - \nhm\|_{L^2(\Ghsm)} \\
        &\lesssim \|\nabla_{\Ghsm} \ehm\|_{L^2(\Ghsm)} + (1 + \kappa_{*,l}) h^k,
    \end{align*}
where we have applied inequalities \eqref{n2}, \eqref{n22} and \eqref{eq:nsa1}--\eqref{eq:nsa2}.
\end{proof}

Furthermore, by employing an argument similar to that in the proof of Lemma~\ref{Lem:normalvector}, along with applications of the inverse inequality that introduces an additional factor of $ h^{-1} $ in the estimates, we can derive the following $ L^\infty $ and $ W^{1,\infty} $ estimates for the differences between the various definitions of normal vectors.

%  This result is obtained by employing inverse inequality along with the Lagrange interpolation estimate \( \|I_h n_*^m - n_*^m\|_{L^\infty(\hat{\Gamma}_{h,*}^m)} \lesssim (1 + \kappa_{*,l}) h^k \), which follows from \eqref{Ihf}. 

\begin{lemma}\label{Lem:normalvector1}
    The following approximation properties hold:
    \begin{subequations}\label{eq:nsa-infty}
        \begin{align}
            \| \bar{n}_{h,*}^m - \bar{n}_h^m \|_{L^\infty(\hat \Gamma_{h,*}^m)} &\lesssim h^{-1} \| \nabla_{\hat \Gamma_{h,*}^m} \hat e_h^m \|_{L^2(\hat \Gamma_{h,*}^m)}, \label{eq:nsa1-infty} \\
            \| \bar{n}_{h,*}^m - n_*^m \|_{L^\infty(\hat \Gamma_{h,*}^m)} &\lesssim (1 + \kappa_{*,l}) h^{k-1}, \label{eq:nsa2-infty} \\
            \| \bar{n}_h^m - n_*^m \|_{L^\infty(\hat \Gamma_{h,*}^m)} &\lesssim h^{-1} \| \nabla_{\hat \Gamma_{h,*}^m} \hat e_h^m \|_{L^2(\hat \Gamma_{h,*}^m)} + (1 + \kappa_{*,l}) h^{k-1}, \label{eq:nsa3-infty} \\
            \| \bar{n}_{h,*}^m - \hat{n}_{h,*}^m \|_{L^\infty(\hat \Gamma_{h,*}^m)} &\lesssim (1 + \kappa_{*,l}) h^{k-1}, \label{eq:nsa4-infty} \\
            \| \bar{n}_h^m - n_h^m \|_{L^\infty(\hat \Gamma_{h,*}^m)} &\lesssim h^{-1} \| \nabla_{\hat \Gamma_{h,*}^m} \hat e_h^m \|_{L^2(\hat \Gamma_{h,*}^m)} + (1 + \kappa_{*,l}) h^{k-1}. \label{eq:nsa5-infty}
        \end{align}
    \end{subequations}
\end{lemma}

\begin{lemma}\label{Lem:normalvector2}
    The following approximation properties hold:
    \begin{subequations}\label{eq:nsa-infty-1}
        \begin{align}
            \| \bar{n}_{h,*}^m - \bar{n}_h^m \|_{W^{1,\infty}(\hat \Gamma_{h,*}^m)} &\lesssim h^{-2} \| \nabla_{\hat \Gamma_{h,*}^m} \hat e_h^m \|_{L^2(\hat \Gamma_{h,*}^m)}, \label{eq:nsa1-infty-1} \\
            \| \bar{n}_{h,*}^m - n_*^m \|_{W^{1,\infty}(\hat \Gamma_{h,*}^m)} &\lesssim (1 + \kappa_{*,l}) h^{k-2}, \label{eq:nsa2-infty-1} \\
            \| \bar{n}_h^m - n_*^m \|_{W^{1,\infty}(\hat \Gamma_{h,*}^m)} &\lesssim h^{-2} \| \nabla_{\hat \Gamma_{h,*}^m} \hat e_h^m \|_{L^2(\hat \Gamma_{h,*}^m)} + (1 + \kappa_{*,l}) h^{k-2}, \label{eq:nsa3-infty-1} \\
            \| \bar{n}_{h,*}^m - \hat{n}_{h,*}^m \|_{W^{1,\infty}(\hat \Gamma_{h,*}^m)} &\lesssim (1 + \kappa_{*,l}) h^{k-2}, \label{eq:nsa4-infty-1} \\
            \| \bar{n}_h^m - n_h^m \|_{W^{1,\infty}(\hat \Gamma_{h,*}^m)} &\lesssim h^{-2} \| \nabla_{\hat \Gamma_{h,*}^m} \hat e_h^m \|_{L^2(\hat \Gamma_{h,*}^m)} + (1 + \kappa_{*,l}) h^{k-2}. \label{eq:nsa5-infty-1}
        \end{align}
    \end{subequations}
\end{lemma}

The following results, which follow from Lemmas \ref{Lem:normalvector}, \ref{Lem:normalvector1} and \ref{Lem:normalvector2}, offer further insights into the behaviour of the average normal vectors on the surfaces \( \Gamma_h^m \) and \( \hat{\Gamma}_{h,*}^m \).

\begin{remark}\label{remark:normal-vector}\upshape 
    By leveraging the boundedness of $\nsm$, the triangle inequality and the inequality \eqref{eq:nsa2-infty-1}, we arrive at the following estimate:
	\begin{align}\label{nhmW-1}
		\|\nahsm\|_{W^{1,\infty}(\Ghsm)} & \lesssim 1+ \|\bar n_{h,*}^m - n_{*}^m\|_{W^{1,\infty}(\Ghsm)} \notag \\
		&\lesssim 1 + (1 + \kappa_{*,l}) h^{k-2} \lesssim 1,
	\end{align}
	where we have utilized \eqref{cond1} which is guaranteed by the third mathematical induction assumption in Section \ref{induc ass}.
	% onsequently, from inequality \eqref{nhmW-1}, it follows that the definition of $\bar{T}_{h,*}^m$ is well-posed and properly established.

    Additionally, from inequality \eqref{eq:nsa1-infty} in Lemma \ref{Lem:normalvector1} and \eqref{eq:nsa1-infty-1} in Lemma \ref{Lem:normalvector2}, the following results hold:
	\begin{subequations}\label{nhmW-2}
		\begin{align}
			\|\nahm\|_{L^\infty(\Ghsm)} &\lesssim \|\nahsm\|_{L^\infty(\Ghsm)} + \| \nahm - \nahsm \|_{L^\infty(\Ghsm)} \notag \\
			&\lesssim 1 + h^{-1} \|\nabla_{\Ghsm} \ehm\|_{L^2(\Ghsm)} \lesssim 1 + h^{0.6} \lesssim 1, \label{nhmW-2b}\\
			\|\nahm\|_{W^{1,\infty}(\Ghsm)} &\lesssim \|\nahsm\|_{W^{1,\infty}(\Ghsm)} + \| \nahm - \nahsm \|_{W^{1,\infty}(\Ghsm)} \notag \\
			&\lesssim 1 + h^{-2} \|\nabla_{\Ghsm} \ehm\|_{L^2(\Ghsm)}, \label{nhmW-2a}
		\end{align}
	\end{subequations}
	where the induction assumption in \eqref{Linfty-W1infty-hat-em} and the estimate in \eqref{nhmW-1} are used in deriving the last inequalities in \eqref{nhmW-2b} and \eqref{nhmW-2a}.
	Similarly, we establish
    \begin{align}\label{nhmW-3}
        \|\nahm\|_{H^1(\Ghsm)} &\lesssim \|\nahsm\|_{H^1(\Ghsm)} + \| \nahm - \nahsm \|_{H^1(\Ghsm)} \notag \\
        &\lesssim 1 + h^{-1} \|\nabla_{\Ghsm} \ehm\|_{L^2(\Ghsm)} \lesssim 1 + h^{0.6} \lesssim 1,
    \end{align}
    by using \eqref{eq:nsa1}, the induction assumption \eqref{Linfty-W1infty-hat-em} and inverse inequalities \eqref{inverse-ineq}.

    Furthermore, by applying the triangle inequality and Lemma \ref{Lem:normalvector1}, and noting that the magnitude of $n_*^m$ equals 1, we derive
    \begin{subequations}
        \begin{align}
            % \||\nahsm| - 1\|_{L^2(\Ghsm)} &\lesssim (1+\kappa_{*,l})h^{k}, \label{average-0} \\
            \||\nahsm| - 1\|_{L^\infty(\Ghsm)} &\lesssim (1+\kappa_{*,l})h^{k-1} \lesssim h^{1.6}, \label{average-1} \\
            % \||\nahm| - 1\|_{L^2(\Ghsm)} &\lesssim (1+\kappa_{*,l})h^{k} + \|\nabla_{\Ghsm} \ehm\|_{L^2(\Ghsm)}, \label{1nhm2} \\
            \||\nahm| - 1\|_{L^\infty(\Ghsm)} &\lesssim (1+\kappa_{*,l})h^{k-1} + h^{-1} \|\nabla_{\Ghsm} \ehm\|_{L^2(\Ghsm)} \lesssim h^{0.6} , \label{1nhmf}
        \end{align}
    \end{subequations}
where the last inequalities in \eqref{average-1} and \eqref{1nhmf} follow from \eqref{cond1} and \eqref{Linfty-W1infty-hat-em}, which are guaranteed by the mathematical induction assumptions in Section \ref{induc ass}. For sufficiently small $ h $, these estimates ensure that $ |\bar n_{h,*}^m |$ and $| \bar n_{h}^m |$ remain uniformly bounded from above and away from zero. Consequently, the averaged normal projection operators $\bar N_{h,*}^m$ and $\bar N_{h}^m$, as well as the associated tangential projection operators $\bar T_{h,*}^m$ and $\bar T_{h}^m$, are all well-defined. The precise definitions of these operators can be found in Section \ref{sec:notation}.

\end{remark}

	%..................................................
	\subsection{Poincar\'{e} Inequalities for Vector Functions}\label{poincaresec}
    A Poincar\'{e}-type inequality relevant to a closed finite element surface was established in \cite[Lemma 3.4]{hu2022evolving} and is given by:
    \begin{align}\label{poin}
        \int_{\Ghsm} | v |^2 \lesssim \int_{\Ghsm} | v \cdot I_h \nsm |^2 + \int_{\Ghsm} | \nabla_{\Ghsm} v |^2,
    \end{align}
    which essentially implies that the complete $L^2$ norm of a vector field on a closed surface can be bounded by a combination of the $L^2$ norm of its normal component and its $H^1$ semi-norm. Since the various averaged normal vectors are defined throughout this work, we now establish several Poincaré-type inequalities by replacing the interpolated normal vector $I_h\,n_{*}^m$ with the averaged normals $\bar n_{h,*}^m$ and $\bar n_h^m$. To this end, we employ the following super-approximation estimates, which are analogous in spirit to those in \cite[Lemma~4.4]{bai2024new}. The first of these estimates is given in the following lemma, which characterizes super-approximation estimates for the product of two finite element functions on a surface or the product of a finite element function and a $W^{1,\infty}$ function defined on the interpolated surface $\Ghsm$.
	
	% We omit the proof here and refer the reader to Appendix~\ref{appendix_B} in the supplementary material.

    \begin{lemma}[Super-approximation estimates for product of functions]\label{super}
        For any finite element functions $u_h, v_h, w_h \in S_h(\Ghsm)$ and any function $v \in W^{1,\infty}(\Ghsm)$, the following estimates hold:
		\begin{subequations}\label{supersuper}
			\begin{align}
            % \label{super1}&\| (1 - I_h)\Tsm v_h \|_{L^2(\Ghsm)} + h\| (1 - I_h)\Tsm v_h \|_{H^1(\Ghsm)} \lesssim h^2 \| v_h \|_{H^1(\Ghsm)}, \\
            \label{super2}&\| (1 - I_h)(v_h w_h) \|_{L^2(\Ghsm)}+h\| (1 - I_h)(v_h w_h) \|_{H^1(\Ghsm)} \lesssim h^2 \| v_h \|_{W^{1,\infty}(\Ghsm)} \| w_h \|_{H^1(\Ghsm)}, \\
            % \label{super3}\| \nabla_{\Ghsm} (1 - I_h)(v_h w_h) \|_{L^2(\Ghsm)} &\lesssim h \| v_h \|_{W^{1,\infty}(\Ghsm)} \| w_h \|_{H^1(\Ghsm)}, \\
            % \label{super4}\| (1 - I_h)(u_h v_h w_h) \|_{L^2(\Ghsm)} &\lesssim h^2 \| u_h \|_{W^{1,\infty}(\Ghsm)} \| v_h \|_{W^{1,\infty}(\Ghsm)} \| w_h \|_{H^1(\Ghsm)},\\
			\label{super6}&\| (1 - I_h)(u_h v_h w_h) \|_{L^1(\Ghsm)} \lesssim h^2 \| u_h \|_{W^{1,\infty}(\Ghsm)} \| v_h \|_{H^1(\Ghsm)} \| w_h \|_{H^1(\Ghsm)}.
			% \label{super5}&\| (1 - I_h)(v w_h) \|_{L^2(\Ghsm)} + h\| (1 - I_h)(v w_h) \|_{H^1(\Ghsm)} \lesssim h^{2-\epsilon}  \| v\|_{W^{1,\infty}(\Ghsm)}  \| w_h \|_{H^1(\Ghsm)}, \\
			% \label{super7}&\| (1 - I_h)(v w_h) \|_{L^2(\Ghsm)} +h \| (1 - I_h)(v w_h) \|_{H^1(\Ghsm)}
			% \lesssim h \| v\|_{W^{1,\infty}(\Ghsm)} \| w_h \|_{L^2(\Ghsm)}.
        \end{align}
		\end{subequations}
    \end{lemma}
The proof of Lemma~\ref{super} is provided in Appendix~\ref{appendix_B}.

In addition, the following super-approximation estimates hold for the projection operators $\Tsm$ and $\Nsm$ associated with the normal vector $n_*^m$ defined on the neighborhood $D_\delta(\Gamma^m) \supset \Ghsm$, as well as for the projection operators $\bar T_{h,*}^m$, $\bar N_{h,*}^m$, $\bar T_{h}^m$ and $\bar N_{h}^m$ associated with the averaged normal vectors $\bar n_{h,*}^m$ and $\bar n_{h}^m$;  see Section~\ref{sec:notation} for the definitions.
% , and the tangential projection operator $\bar T_{h}^m$ defined in Section \ref{sec:notation}.
% Since $n_*^m$ is smoothly defined on the neighborhood $D_\delta(\Gamma^m)\supset\Ghsm$, and since the averaged normals $\bar n_{h,*}^m$ and $\bar n_{h}^m$, which induce the projection operators, are piecewise polynomial finite‐element functions of degree $k$ on each reference triangle, the following super‐approximation estimates hold by virtue of the Lagrange interpolation approximation properties. 
The proof of these results is provided in Appendix~\ref{appendix_B}.

\begin{lemma}[Super-approximation estimates for projection operators]\label{super-bar-nhm}
	For any finite element function \( v_h\in S_h(\Ghsm)^3 \), the following super-approximation estimates hold:
	\begin{subequations}\label{super-bar-nhmstability}
		\begin{align}	
			\label{super1}\| (1 - I_h)(\Tsm v_h) \|_{L^2(\Ghsm)} + h\| (1 - I_h)(\Tsm v_h) \|_{H^1(\Ghsm)} &\lesssim h^2 \| v_h \|_{H^1(\Ghsm)}, \\
			\label{super0}\| (1 - I_h)(\Nsm v_h) \|_{L^2(\Ghsm)} + h\| (1 - I_h)(\Nsm v_h) \|_{H^1(\Ghsm)} &\lesssim h^2 \| v_h \|_{H^1(\Ghsm)}, \\
			\label{super1_projected-surface}
		\| (1 - I_h)(\bar T_{h,*}^m v_h) \|_{L^2(\Ghsm)}+h\| (1 - I_h)(\bar T_{h,*}^m v_h) \|_{H^1(\Ghsm)} &\lesssim h^2 \| v_h \|_{H^1(\Ghsm)},\\
		\label{super2_projected-surface}
		\| (1 - I_h)(\bar N_{h,*}^m v_h) \|_{L^2(\Ghsm)}+h\| (1 - I_h)(\bar N_{h,*}^m v_h) \|_{H^1(\Ghsm)}  &\lesssim h^2 \| v_h \|_{H^1(\Ghsm)},\\
			\| (1 - I_h)(\bar T_{h}^m v_h) \|_{L^2(\Ghsm)} + h \| (1 - I_h)(\bar T_{h}^m v_h) \|_{H^1(\Ghsm)} &\lesssim h^2\|\bar n_h^m\|_{W^{1,\infty}(\Ghsm)}\| v_h \|_{H^1(\Ghsm)}\notag\\
			& \lesssim h^{1.6}\| v_h \|_{H^1(\Ghsm)}, \label{super-bar-tangent}\\		
			\| (1 - I_h)(\bar N_{h}^m v_h) \|_{L^2(\Ghsm)} + h \| (1 - I_h)(\bar N_{h}^m v_h)\|_{H^1(\Ghsm)}&\lesssim  h^2\|\bar n_h^m\|_{W^{1,\infty}(\Ghsm)} \| v_h \|_{H^1(\Ghsm)}\notag\\
			&\lesssim h^{1.6}\| v_h \|_{H^1(\Ghsm)}. \label{super-bar-normal}
		\end{align}
	\end{subequations}
\end{lemma}

\begin{remark}\upshape
	Following the proof of \eqref{super-bar-tangent} and \eqref{super-bar-normal} in Appendix \ref{appendix_B}, we obtain the following super-approximation estimates:
	\begin{subequations}
		\begin{align}
			& \| (1 - I_h)((\bar T_{h,*}^m - \bar T_{h}^m) v_h) \|_{L^2(\Ghsm)} + h \| (1 - I_h)((\bar T_{h,*}^m - \bar T_{h}^m)v_h) \|_{H^1(\Ghsm)}\notag\\
			&\lesssim h^2\|\bar n_{h,*}^m - \bar n_h^m\|_{W^{1,\infty}(\Ghsm)}\| v_h \|_{H^1(\Ghsm)}, \label{super-bar-tangent-dif}\\[7pt]	
			&\| (1 - I_h)((\bar N_{h,*}^m - \bar N_{h}^m) v_h) \|_{L^2(\Ghsm)} + h \| (1 - I_h)((\bar N_{h,*}^m - \bar N_{h}^m) v_h)\|_{H^1(\Ghsm)}\notag\\
			&\lesssim  h^2\|\bar n_{h,*}^m-\bar n_h^m\|_{W^{1,\infty}(\Ghsm)} \| v_h \|_{H^1(\Ghsm)}. \label{super-bar-normal-dif}
		\end{align}
	\end{subequations}
\end{remark}

Based on the super-approximation estimates discussed above, we can derive the following stability results. The proof of these results is provided in Appendix~\ref{appendix_B}. 

	\begin{lemma}[Stability Estimates]\label{stability}
		For any finite element functions $v_h, w_h \in S_h(\Ghsm)$ and any function $u \in W^{1,\infty}(\Ghsm)$, the following stability estimates hold:
		\begin{subequations}\label{stabilitystability}
			\begin{align}
				&\| I_h(v_h w_h) \|_{L^2(\Ghsm)} \lesssim \| v_h w_h \|_{L^2(\Ghsm)}\lesssim \| v_h \|_{L^\infty(\Ghsm)} \| w_h \|_{L^2(\Ghsm)}, \label{stability2}\\
				&\| I_h(v_h w_h) \|_{H^1(\Ghsm)} \lesssim \| v_h w_h \|_{H^1(\Ghsm)}\lesssim \| v_h \|_{H^1(\Ghsm)} \| w_h \|_{W^{1,\infty}(\Ghsm)}, \label{stability7}\\
				&\|I_h(u v_h)\|_{L^2(\Ghsm)} \lesssim \|u\|_{L^\infty(\Ghsm)}\|v_h\|_{L^2(\Ghsm)}, \label{stability3}\\
				% \|I_h(u v_h)\|_{H^1(\Ghsm)} &\lesssim \|u\|_{L^\infty(\Ghsm)}\|v_h\|_{H^1(\Ghsm)} + h^{1-\epsilon} \|u\|_{W^{1,\infty}(\Ghsm)}\|v_h\|_{H^1(\Ghsm)}, \label{stability5}\\
				&\|I_h(u v_h w_h)\|_{L^2(\Ghsm)} \lesssim \|u\|_{L^\infty(\Ghsm)}\|v_hw_h\|_{L^2(\Ghsm)}, \label{stability4}\\
				&\| I_h (\Tsm v_h) \|_{L^2(\Ghsm)} +\| I_h (\Nsm v_h) \|_{L^2(\Ghsm)} \lesssim \| v_h \|_{L^2(\Ghsm)}, \label{stability1} \\
				&\| I_h (\Tahsm v_h) \|_{L^2(\Ghsm)} +\| I_h (\Nahsm v_h) \|_{L^2(\Ghsm)}\lesssim \| v_h \|_{L^2(\Ghsm)}, \label{stability11} \\
				&\| I_h (\Tahm v_h) \|_{L^2(\Ghsm)} +\| I_h (\Nahm v_h) \|_{L^2(\Ghsm)}\lesssim \| v_h \|_{L^2(\Ghsm)}. \label{stability12}
			\end{align}
		\end{subequations}
	\end{lemma}
	\begin{remark}\upshape
	Following the same arguments as in the proofs of \eqref{stability2} and \eqref{stability7}, and noting that \(u_h v_h w_h\) is a high-order polynomial when \(u_h, v_h, w_h \in S_h(\Ghsm)\), we obtain the following stability estimate for the Lagrange interpolation operator:
	\begin{align}\label{stability-final}
		\|I_h (u_h v_h w_h)\|_{W^{i,p}(\Ghsm)} \lesssim \|u_h v_h w_h\|_{W^{i,p}(\Ghsm)}, \quad \text{for} \quad i=0,1 \quad \text{and} \quad 1 \leq p \leq \infty.
	\end{align}
\end{remark}

    By replacing $I_h \nsm$ with $\nahsm$ in Poincar\'{e} inequality \eqref{poin}, we establish the following Poincar\'{e}-type inequality.

    \begin{lemma}[Poincar\'e Inequality]\label{lemma:poincare}
        For sufficiently small $h$, the following Poincar\'{e}-type inequalities hold for $v_h \in S_h(\Ghsm)^3$:
        \begin{align}
            \label{eq:poincare}
            \int_{\Ghsm} | v_h |^2 &\lesssim \int_{\Ghsm} | v_h \cdot \nahsm |^2 + \int_{\Ghsm} | \nabla_{\Ghsm} v_h |^2, \\
            \label{eq:poincare2}
            \int_{\Ghsm} | I_h (\Tahsm v_h) |^2 &\lesssim \int_{\Ghsm} | \nabla_{\Ghsm} I_h (\Tahsm v_h) |^2, \\
            \label{eq:poincare3}
            \| v_h \|_{L^\infty(\Ghsm)} &\lesssim C_\epsilon h^{-\epsilon} \left(\| v_h \cdot \nahsm \|_{L^2(\Ghsm)} + \| \nabla_{\Ghsm} v_h \|_{L^2(\Ghsm)}\right),
        \end{align}
		where inequality \eqref{eq:poincare3} is valid for any $0<\epsilon<1$.
    \end{lemma}

	\begin{proof}
		Using the Lagrange interpolation approximation estimates \eqref{Ihf} and \eqref{eq:nsa2-infty} along with inverse inequality, we derive 
		\begin{equation*}
			\|\nahsm - I_h \nsm\|_{L^\infty(\Ghsm)} \lesssim (1 + \kappa_{*,l}) h^{k-1} \lesssim h^{1.6} \quad (\text{using the induction assumption \eqref{cond1}}).
		\end{equation*}
		Consequently, starting with \eqref{poin} and using the triangle inequality, we have 
		\begin{align}\label{pflem1}
			\int_{\Ghsm} | v_h |^2 &\lesssim \int_{\Ghsm} | v_h \cdot I_h \nsm |^2 + \int_{\Ghsm} | \nabla_{\Ghsm} v_h |^2 \notag \\
			&\lesssim \int_{\Ghsm} | v_h \cdot (\nahsm - I_h \nsm) |^2 + \int_{\Ghsm} | v_h \cdot \nahsm |^2 + \int_{\Ghsm} | \nabla_{\Ghsm} v_h |^2 \notag \\
			&\lesssim \|\nahsm - I_h \nsm\|_{L^\infty(\Ghsm)}^2 \int_{\Ghsm} | v_h |^2 + \int_{\Ghsm} | v_h \cdot \nahsm |^2 + \int_{\Ghsm} | \nabla_{\Ghsm} v_h |^2 \notag \\
			&\lesssim h^{3.2} \int_{\Ghsm} | v_h |^2 + \int_{\Ghsm} | v_h \cdot \nahsm |^2 + \int_{\Ghsm} | \nabla_{\Ghsm} v_h |^2,
		\end{align}
		where the first term on the right-hand side of \eqref{pflem1} can be absorbed by the left-hand side by choosing a sufficiently small $h$, therefore establishing inequality \eqref{eq:poincare}.
	
		From inequality \eqref{eq:poincare}, we further obtain
		\begin{align}\label{pflem2}
			\int_{\Ghsm} | I_h (\Tahsm v_h) |^2 &\lesssim \int_{\Ghsm} | I_h (\Tahsm v_h) \cdot \nahsm |^2 + \int_{\Ghsm} | \nabla_{\Ghsm} I_h (\Tahsm v_h) |^2 \notag \\
			&= \int_{\Ghsm} | ((I_h - 1) (\Tahsm I_h \Tahsm v_h)) \cdot \nahsm |^2 \quad\text{(by orthogonality of $\Tahsm$ and $\nahsm$)} \notag \\
			&\quad + \int_{\Ghsm} | \nabla_{\Ghsm} I_h (\Tahsm v_h) |^2 \notag \\
			&\lesssim (h^2  \|I_h (\Tahsm v_h)\|_{H^1(\Ghsm)})^2 \|\nahsm\|_{L^\infty(\Ghsm)}^2\quad\text{(by \eqref{super1_projected-surface})} \notag \\
			&\quad + \int_{\Ghsm} | \nabla_{\Ghsm} I_h (\Tahsm v_h)|^2 \notag \\
			&\lesssim h^4 \|I_h (\Tahsm v_h)\|_{H^1(\Ghsm)}^2 + \int_{\Ghsm} | \nabla_{\Ghsm} I_h (\Tahsm v_h) |^2 \quad\text{(by \eqref{nhmW-1})} \notag \\
			&\lesssim h^2 \int_{\Ghsm} | I_h (\Tahsm v_h) |^2 + \int_{\Ghsm} | \nabla_{\Ghsm} I_h \Tahsm v_h |^2\,\,\text{(by inverse inequality \eqref{inverse-ineq})} .
		\end{align}
		Again, the first term on the right-hand side of \eqref{pflem2} can be absorbed by the left-hand side by choosing a sufficiently small $h$, which yields inequality \eqref{eq:poincare2}. 
	
            Finally, inequality \eqref{eq:poincare3} follows by lifting $ v_h $ from $ \Ghsm $ to $ \Gamma^m $, applying the Sobolev embedding inequality, and then using the inverse inequality:
		\begin{align}\label{sobolem}
			\|v_h\|_{L^\infty(\Ghsm)} &\lesssim \|v_h^l\|_{L^\infty(\Gamma^m)}\lesssim C_\epsilon \|v_h^l\|_{W^{1, 2+\frac{2\epsilon}{1-\epsilon}}(\Gamma^m)} \lesssim C_\epsilon\|v_h\|_{W^{1, 2+\frac{2\epsilon}{1-\epsilon}}(\Ghsm)}\notag\\
			& \quad\text{(using the norm equivalences between $\Ghsm$ and $\Gamma^m$, as shown in \eqref{norm-equiv-lift})}\notag\\
			& \lesssim C_\epsilon h^{-\epsilon} \|v_h\|_{H^1(\Ghsm)} \quad \text{(using the inverse inequality)}\notag\\
			& \lesssim C_\epsilon h^{-\epsilon}(\| v_h \cdot\bar n_{h,*}^m \|_{L^2(\Ghsm)} + \| \nabla_{\Ghsm} v_h \|_{L^2(\Ghsm)}) ,
		\end{align}
        where the last inequality follows from applying \eqref{eq:poincare}.
	\end{proof}

	\begin{remark}\upshape 
    Given that \( \nahm \) deviates from \( \nahsm \) by $Ch^{-1}\|\nabla_{\hat{\Gamma}_{h,*}^m}\hat{e}_h^m\|_{L^2(\hat{\Gamma}_{h,*}^m)}$ in the \( L^\infty \) norm, as established in \eqref{eq:nsa1-infty} of Lemma \ref{Lem:normalvector1}, we can leverage the induction assumption in \eqref{Linfty-W1infty-hat-em} to estimate this difference:
    \begin{align*}
        \|\bar{n}_h^m - \bar{n}_{h,*}^m\|_{L^\infty(\hat{\Gamma}_{h,*}^m)} \lesssim h^{-1}\|\nabla_{\hat{\Gamma}_{h,*}^m}\hat{e}_h^m\|_{L^2(\hat{\Gamma}_{h,*}^m)} \lesssim h^{0.6}.
    \end{align*}
    This result allows us to replace \( \nahsm \) with \( \nahm \) in inequality \eqref{eq:poincare} and absorb the residual term into the left-hand side. Consequently, we derive an alternative form of the Poincar\'e inequality expressed in terms of \( \bar n_h^m \):
    \begin{align}
        \label{hmpoin1}
        \int_{\Ghsm} | v_h |^2 &\lesssim \int_{\Ghsm} | v_h \cdot \nahm |^2 + \int_{\Ghsm} | \nabla_{\Ghsm} v_h |^2.
    \end{align}
% It then follows from \eqref{hmpoin1} that
% \begin{align}\label{hmpoin0}
%   \int_{\Ghsm} \bigl|\bar T_{h}^m v_h\bigr|^2
%    \lesssim 
%   \int_{\Ghsm} \bigl|\nabla_{\Ghsm}(\bar T_{h}^m v_h)\bigr|^2.
% \end{align}
Moreover, by applying the super-approximation estimate \eqref{super-bar-tangent} for $\bar n_h^m$ and following the proof of \eqref{pflem2}, we can obtain the following result: 
\begin{align}\label{poincare4}
    \int_{\Ghsm} |I_h \bar T_h^m v_h|^2 
     \lesssim 
    \int_{\Ghsm} |\nabla_{\Ghsm}I_h \bar T_h^m v_h|^2.
\end{align}
\end{remark}

	%...................................................
	
	%...................................................
	\subsection{Geometric Relations}\label{rekl}
The geometric framework adopted in this paper aligns with the setting described in \cite[Section 3.4]{bai2024new}, encompassing the following relations presented in \eqref{eq:geo_rel_1}--\eqref{Xexrelat} as well as Lemma \ref{lemma:T<=N1}.

Firstly, by the definitions of $\ehm$ and $e_h^{m}$ given in Section \ref{sec:notation}, and noting that \( \ehm \) is the distance error onto the exact surface $\Gamma^{m}$ at each node, the following nodal relation can be established using Taylor's expansion:
\begin{align}\label{eq:geo_rel_1}
    \ehm = I_h\big[ (e_h^m\cdot n_*^m) \nsm \big] + f_h^m,
\end{align}
where the correction term $f_h^m$ satisfies
\begin{align}\label{eq:geo_rel_2}
    |f_h^m| \le C_0 |[ I - n_*^{m} (n_*^{m})^\top ] e_h^{m} |^2
    \quad \text{at the nodes of } \Ghsm,
\end{align}
if $e_h^m$ is small.
This implies that \( \ehm \) differs from \( (e_h^m \cdot \nsm) \nsm \) by a higher-order correction term.

Secondly, recall that \( X_{h,*}^{m+1}: \Ghsm \to \Gamma_{h,*}^{m+1} \) denotes the local flow map under which the nodes of \( \Ghsm \) evolve precisely according to mean curvature flow without tangential motion. Let \( X^{m+1}: \Gamma^m \to \Gamma^{m+1} \) denote the local flow map for the exact mean curvature flow. Since the relation \( X_{h,*}^{m+1} - \hat{X}_{h,*}^m = X^{m+1} - \mathrm{id} \) holds at the finite element nodes on \( \Ghsm \), the following identities can be derived:
\begin{subequations}\label{X-id-Hn00}
\begin{align}
    X_{h,*}^{m+1} - \hat{X}_{h,*}^m &= I_h(X^{m+1} - \mathrm{id}) \quad \text{on } \Ghsm, \label{X-id-Hn0} \\
    X^{m+1} - \mathrm{id} &= \tau ( -H^m n^m + g^m) \quad \text{on } \Gamma^m, \label{X-id-Hn}
\end{align}
\end{subequations}
where \( -H^m n^m \) represents the velocity of the exact mean curvature flow (without tangential motion) at time \( t = t_m \), and \( g^m \) denotes the smooth correction term from the Taylor expansion theorem. This correction term satisfies the following estimate:
\begin{align}\label{W1infty-g}
    \|g^m\|_{W^{1,\infty}(\Gamma^m)} \leq C\tau.
\end{align}

Combining these relations, we obtain the following expression for the numerical displacement:
\begin{align}\label{Xexrelat}
    \begin{aligned}
        X_h^{m+1} - X_h^m &= \eM - \ehm + X_{h,*}^{m+1} - \hat{X}_{h,*}^m \\
        &= \eM - \ehm + \tau I_h(-H^m n^m + g^m).
    \end{aligned}
\end{align}
This relation plays a pivotal role in deriving estimates for the numerical displacement \( X_h^{m+1} - X_h^m \).

The definition of \( \ehm \), specifically its orthogonality to \( \Gamma^m \) at the nodal points, ensures that the tangential component of \( \hat{e}_h^m \) (evaluated at points that are not nodes) is significantly smaller than its normal component in both \( L^2 \) and \( H^1 \) norms. Consequently, the total \( L^2 \) and \( H^1 \) norms of \( \hat{e}_h^m \) can be effectively bounded by the corresponding norms of its normal component. These observations are formalized in the following lemma, which plays a crucial role in establishing the \( H^1 \)-monotonicity of mean curvature flow. The proof of this lemma is detailed in \cite[Section 3.5]{bai2024new}.

\begin{lemma}\label{lemma:T<=N1}
    For sufficiently small \( h \), the following estimates hold:
    \begin{align}
        \| (I - n_{*}^m (n_{*}^m)^\top) \hat{e}_h^m \|_{L^2(\Ghsm)} 
        &\lesssim h \| (\hat{e}_h^m \cdot n_{*}^m) n_{*}^m \|_{L^2(\Ghsm)},\label{eq:NT_L2} \\
        \| (I - n_{*}^m (n_{*}^m)^\top) \hat{e}_h^m \|_{H^1(\Ghsm)} 
        &\lesssim h \| (\hat{e}_h^m \cdot n_{*}^m) n_{*}^m \|_{H^1(\Ghsm)}, \label{en} \\
        \| \hat{e}_h^m \|_{L^2(\Ghsm)} 
        &\leq 2 \| (\hat{e}_h^m \cdot n_{*}^m) n_{*}^m \|_{L^2(\Ghsm)}, \label{eq:NT_stab_L2} \\
        \| \hat{e}_h^m \|_{H^1(\Ghsm)} 
        &\leq 2 \| (\hat{e}_h^m \cdot n_{*}^m) n_{*}^m \|_{H^1(\Ghsm)}. \label{eq:NT_stab_H1}
    \end{align}
\end{lemma}

Similar results hold when \( n_{*}^m \) is replaced by the averaged normal vector \( \bar{n}_{h,*}^m \) on \( \Ghsm \), as stated in the following lemma.

\begin{lemma}\label{lemma:T<=N2}
    For sufficiently small \( h \), the following estimates hold:
    \begin{align}
        \| \Tbhsm \hat{e}_h^m \|_{L^2(\Ghsm)} 
        &\lesssim h \| \hat{e}_h^m \cdot \bar{n}_{h,*}^m \|_{L^2(\Ghsm)}, \label{eq:NT_stab_L21} \\
        \| \hat{e}_h^m \|_{L^2(\Ghsm)} 
        &\leq 2 \| \hat{e}_h^m \cdot \bar{n}_{h,*}^m \|_{L^2(\Ghsm)}. \label{eq:NT_stab_L22}
    \end{align}
\end{lemma}

\begin{proof}
    % From the first super-approximation result in Lemma \ref{super}, we have the following estimate:
    % \[
    %     \| (1 - I_h) \Tsm \hat{e}_h^m \|_{L^2(\Ghsm)} \lesssim h \| \hat{e}_h^m \|_{L^2(\Ghsm)}.
    % \]
    Using inequalities \eqref{eq:nsa2-infty}, \eqref{average-1} and the stability result \eqref{stability3}, we derive:
    \begin{align}
        & \quad\,\,\|(\Tahsm - \Tsm) \hat{e}_h^m \|_{L^2(\Ghsm)} \notag\\
        &\lesssim \| \Tahsm - \Tsm \|_{L^\infty(\Ghsm)} \| \hat{e}_h^m \|_{L^2(\Ghsm)} \notag\\
        &\leq \| \bar{n}_{h,*}^m (\bar{n}_{h,*}^m)^\top - n_{*}^m (n_{*}^m)^\top \|_{L^\infty(\Ghsm)} \| \hat{e}_h^m \|_{L^2(\Ghsm)} \notag\\
        &\quad\,\, + \| \bar{n}_{h,*}^m (\bar{n}_{h,*}^m)^\top - \frac{\bar{n}^m_{h,*}}{|\bar{n}^m_{h,*}|} \left( \frac{\bar{n}^m_{h,*}}{|\bar{n}^m_{h,*}|} \right)^\top \|_{L^\infty(\Ghsm)} \| \hat{e}_h^m \|_{L^2(\Ghsm)} \notag\\
        % &\leq \| (\bar{n}_{h,*}^m - n_{*}^m) (\bar{n}_{h,*}^m)^\top \|_{L^\infty(\Ghsm)} \| \hat{e}_h^m \|_{L^2(\Ghsm)} \\
        % &\quad + \| n_{*}^m (\bar{n}_{h,*}^m - n_{*}^m)^\top \|_{L^\infty(\Ghsm)} \| \hat{e}_h^m \|_{L^2(\Ghsm)} \\
        % &\quad + \||\bar{n}_{h,*}^m|^2 - 1 \|_{L^\infty(\Ghsm)} \| \hat{e}_h^m \|_{L^2(\Ghsm)} \\
        &\lesssim (1 + \kappa_{*,l}) h^{k-1} \| \hat{e}_h^m \|_{L^2(\Ghsm)} \notag\\
        &\lesssim h^{1.6} \| \hat{e}_h^m \|_{L^2(\Ghsm)} \quad \text{(using the induction assumption \eqref{cond1}).} \label{eq:NT_L2_inter}
    \end{align}

    From the above inequality \eqref{eq:NT_L2_inter} and the estimate \eqref{eq:NT_L2} in Lemma \ref{lemma:T<=N1}, we obtain:
    \begin{align}\label{eq:nhs_bar_len-proof}
        \| \Tbhsm \hat{e}_h^m \|_{L^2(\Ghsm)} 
        &\lesssim \|  (\Tbhsm - \Tsm) \hat{e}_h^m \|_{L^2(\Ghsm)} + \| \Tsm \hat{e}_h^m \|_{L^2(\Ghsm)} \notag \\
        &\lesssim h^{1.6} \| \hat{e}_h^m \|_{L^2(\Ghsm)}+ h \| \hat{e}_h^m \cdot {n}_{*}^m \|_{L^2(\Ghsm)}\notag\\
		& \lesssim  h \| \hat{e}_h^m \|_{L^2(\Ghsm)}\lesssim h\| \Tbhsm \hat{e}_h^m \|_{L^2(\Ghsm)}  + h \| \bar N_{h,*}^m\hat{e}_h^m \|_{L^2(\Ghsm)}\notag\\
		& \lesssim h\| \Tbhsm \hat{e}_h^m \|_{L^2(\Ghsm)}  + h \|\hat{e}_h^m \cdot \bar n_{h,*}^m \|_{L^2(\Ghsm)} ,
    \end{align}
	where inequality \eqref{average-1} is used in the last inequality. Sin ce the first term on the right-hand side of \eqref{eq:nhs_bar_len-proof} can be absorbed into the left-hand side, we conclude the first result of Lemma \ref{lemma:T<=N2}.
    The second result follows directly from \eqref{eq:NT_stab_L21} and \eqref{average-1}.

    % Furthermore, using stability result \eqref{stability2} and \eqref{average-1}, we have:
    % \begin{align}\label{Nnev}
    %     \| I_h \Nahsm \hat{e}_h^m \|_{L^2(\Ghsm)} 
    %     &= \| I_h \Nahsm \Nahsm \hat{e}_h^m \|_{L^2(\Ghsm)} \notag \\
    %     &\lesssim \| \Nahsm \|_{L^\infty(\Ghsm)} \| \Nahsm \hat{e}_h^m \|_{L^2(\Ghsm)} \notag \\
    %     &\sim \| \hat{e}_h^m \cdot \frac{\nahsm}{|\nahsm|} \|_{L^2(\Ghsm)} \sim \| \hat{e}_h^m \cdot \nahsm \|_{L^2(\Ghsm)}.
    % \end{align}

\end{proof}

	%.................................................
	\section{Consistency estimates}\label{consistency-estimates}
	The consistency errors of the MDR method for mean curvature flow are defined as the residual terms that arise when the numerical surface $\Gamma_h^m$ and numerical solution $X_h^{m+1}$ are replaced by the interpolated surface $\Ghsm$ and interpolated flow map $X_{h,*}^{m+1}$, respectively, i.e., 
    % Since the error between the consecutive interpolated surfaces $\Ghsm$ and $\hat{\Gamma}_{h,*}^{m+1}$ is not readily controlled, the flow map $X_{h,*}^{m+1}$ is introduced to bridge $\Ghsm$ and $\hat{\Gamma}_{h,*}^{m+1}$. This is because the trajectory of each node moving from $\Ghsm$ to $\Gamma_{h,*}^{m+1}$ follows the exact mean curvature flow, which is explicitly known, and the geometric relations \eqref{X-id-Hn0}–\eqref{X-id-Hn} will be employed. The consistency error is thus defined as follows.
\begin{align}
    d_v^m(\phi_h) :=& \int_{\Ghsm} \frac{X_{h,*}^{m+1} - \mathrm{id}}{\tau} \cdot \nahsm \phi_h + \int_{\Ghsm} \nabla_{\Ghsm} X_{h,*}^{m+1} \cdot \nabla_{\Ghsm} (\nahsm \phi_h) \notag \\
    =& \int_{\Ghsm} \frac{X_{h,*}^{m+1} - \mathrm{id}}{\tau} \cdot \nahsm \phi_h + \int_{\Gamma^m} H^m n^m \cdot (\nahsm \phi_h)^l \notag \\
    &\quad - \int_{\Gamma^m} \nabla_{\Gamma^m} \mathrm{id} \cdot \nabla_{\Gamma^m} (\nahsm \phi_h)^l + \int_{\Ghsm} \nabla_{\Ghsm} X_{h,*}^{m+1} \cdot \nabla_{\Ghsm} (\nahsm \phi_h) \notag \\
    =:& \,d_{v1}^m(\phi_h) + d_{v2}^m(\phi_h), \label{er1} \\
    d_\kappa^m(\chi_h) :=& \int_{\Ghsm} \nabla_{\Ghsm} I_h v^m \cdot \nabla_{\Ghsm} \chi_h - \int_{\Ghsm} I_h \kappa^m \nahsm \cdot \chi_h \notag \\
    =& \int_{\Ghsm} \nabla_{\Ghsm} I_h v^m \cdot \nabla_{\Ghsm} \chi_h - \int_{\Gamma^m} \nabla_{\Gamma^m} v^m \cdot \nabla_{\Gamma^m} (\chi_h)^l \notag \\
    &\quad + \int_{\Gamma^m} \kappa^m n^m \cdot (\chi_h)^l - \int_{\Ghsm} I_h \kappa^m \nahsm \cdot \chi_h \notag \\
    =:& \, d_{\kappa1}^m(\chi_h) + d_{\kappa2}^m(\chi_h), \label{er2}
\end{align}
where \( (v^m, \kappa^m) \) denotes the solution of the PDE system \eqref{kv} at the time level \( t = t^m \).

% The consistency estimates in this section rely on the following results. 
% We begin with a lemma from \cite[Lemma 4.3]{kovacs2017convergence}, which establishes that norms of finite element functions, defined with identical nodal values across a family of surfaces, are equivalent. 
% Specifically, consider the family of surfaces
% \[
%     \hat{\Gamma}_{h,\theta}^m = (1 - \theta) \Ghsm + \theta \Ghm, \quad \theta \in [0, 1].
% \]

% \begin{lemma}\label{equi-MA}
%     Suppose that \( \|\nabla_{\hat{\Gamma}_{h,*}^m} \hat{e}_h^m\|_{L^\infty(\hat{\Gamma}_{h,*}^m)} \le \frac{1}{2} \) for \( \theta \in [0,1] \). Then, the following equivalences of norms hold for \( 1 \le p \le \infty \):
%     \begin{align*}
%         \|v_h\|_{L^p(\hat{\Gamma}_{h,*}^m)} &\lesssim \|v_h\|_{L^p(\hat{\Gamma}_{h,\theta}^m)} \lesssim \|v_h\|_{L^p(\hat{\Gamma}_{h,*}^m)}, \\
%         \|\nabla_{\hat{\Gamma}_{h,*}^m} v_h\|_{L^p(\hat{\Gamma}_{h,*}^m)} &\lesssim \|\nabla_{\hat{\Gamma}_{h,\theta}^m} v_h\|_{L^p(\hat{\Gamma}_{h,\theta}^m)} \lesssim \|\nabla_{\hat{\Gamma}_{h,*}^m} v_h\|_{L^p(\hat{\Gamma}_{h,*}^m)}.
%     \end{align*}
% \end{lemma}

The following lemma states that the error between the two integrals, resulting from the perturbation of the surface in the normal direction, is of order $\mathcal{O}(h^{k+1})$.

\begin{lemma}[{\!\!\cite[Lemma 4.2]{bai2024new}}]\label{geometric-perturbation}
	For \(f_1, f_2 \in W^{1,\infty}(\Ghsm)\), \(f \in (W^{1,\infty}(\Ghsm))^3\), and their lifts \(f_1^l, f_2^l \in W^{1,\infty}(\Gamma^m)\), \(f^l \in (W^{1,\infty}(\Gamma^m))^3\), the following perturbation estimates hold:
	\begin{subequations}\label{perturbation}
		\begin{align}
			&\bigg|\int_{\Ghsm} f_1 f_2 - \int_{\Gamma^m} f_1^l f_2^l \bigg|
			\lesssim (1 + \kappa_{*,l})\,h^{k+1}\,\|f_1\|_{L^\infty(\Ghsm)}\,\|f_2\|_{L^2(\Ghsm)},
			\label{perturbation1}\\
			&\bigg|\int_{\Ghsm} \nabla_{\Ghsm} f_1 \cdot \nabla_{\Ghsm} f_2
			- \int_{\Gamma^m} \nabla_{\Gamma^m} f_1^l \cdot \nabla_{\Gamma^m} f_2^l \bigg|\notag\\
			&  \lesssim (1 + \kappa_{*,l})\,h^{k+1}\,\|\nabla_{\Ghsm} f_1\|_{L^\infty(\Ghsm)}\,\|\nabla_{\Ghsm} f_2\|_{L^2(\Ghsm)},
			\label{perturbation2}\\
			&\bigg|\int_{\Ghsm} f \cdot \nabla_{\Ghsm} f_2
			- \int_{\Gamma^m} f^l \cdot \nabla_{\Gamma^m} f_2^l \bigg|
			\lesssim (1 + \kappa_{*,l})\,h^{k}\,\|f\|_{L^\infty(\Ghsm)}\,\|\nabla_{\Ghsm} f_2\|_{L^2(\Ghsm)}.
			\label{perturbation3}
		\end{align}
	\end{subequations}
\end{lemma}

% The consistency errors of the MDR method for mean curvature flow are defined as the remainders when the exact solutions and the exact surface are replaced by the interpolated solutions and the interpolated surface $\Ghsm$, respectively. That is,
% \begin{align}
%     d_v^m(\phi_h) :=& \int_{\Ghsm} \frac{X_{h,*}^{m+1} - \mathrm{id}}{\tau} \cdot \nahsm \phi_h + \int_{\Ghsm} \nabla_{\Ghsm} X_{h,*}^{m+1} \cdot \nabla_{\Ghsm} (\nahsm \phi_h) \notag \\
%     =& \int_{\Ghsm} \frac{X_{h,*}^{m+1} - \mathrm{id}}{\tau} \cdot \nahsm \phi_h + \int_{\Gamma^m} H^m n^m \cdot (\nahsm \phi_h)^l \notag \\
%     &\quad - \int_{\Gamma^m} \nabla_{\Gamma^m} \mathrm{id} \cdot \nabla_{\Gamma^m} (\nahsm \phi_h)^l + \int_{\Ghsm} \nabla_{\Ghsm} X_{h,*}^{m+1} \cdot \nabla_{\Ghsm} (\nahsm \phi_h) \notag \\
%     =:& d_{v1}^m(\phi_h) + d_{v2}^m(\phi_h), \label{er1} \\
%     d_\kappa^m(\chi_h) :=& \int_{\Ghsm} \nabla_{\Ghsm} I_h v^m \cdot \nabla_{\Ghsm} \chi_h - \int_{\Ghsm} I_h \kappa^m \nahsm \cdot \chi_h \notag \\
%     =& \int_{\Ghsm} \nabla_{\Ghsm} I_h v^m \cdot \nabla_{\Ghsm} \chi_h - \int_{\Gamma^m} \nabla_{\Gamma^m} v^m \cdot \nabla_{\Gamma^m} (\chi_h)^l \notag \\
%     &\quad + \int_{\Gamma^m} \kappa^m n^m \cdot (\chi_h)^l - \int_{\Ghsm} I_h \kappa^m \nahsm \cdot \chi_h \notag \\
%     =:&  d_{\kappa1}^m(\chi_h) + d_{\kappa2}^m(\chi_h), \label{er2}
% \end{align}
% where \( (v^m, \kappa^m) \) denotes the solution of the elliptic system \eqref{kv} at the time level \( t = t^m \).

\begin{lemma}\label{consterr}
    Under the conditions of Theorem \ref{thm:main}, the consistency errors satisfy the following estimates:
    \begin{align}
        \label{dv}
        d^m_v(\phi_h) &\lesssim (1 + \kappa_{*,l} h^{k-1}) \tau \|\phi_h\|_{L^2(\Ghsm)} + (1 + \kappa_{*,l}) h^k \|\phi_h\|_{H^1(\Ghsm)}, \\
        \label{dk}
        d^m_\kappa(\chi_h) &\lesssim (1 + \kappa_{*,l}) h^k \|\chi_h\|_{H^1(\Ghsm)}.
    \end{align}
\end{lemma}
The estimate \eqref{dv} is analogous to the arguments presented in \cite[Lemma 4.3]{bai2024new}. The detailed proof is omitted here and can be found in Appendix \ref{appendix_C} of the supplementary material.

	\section{Proof of Theorem \ref{thm:main}}\label{sec:stability}
Let \( \Gamma \subset \mathbb{R}^3 \) be a smooth surface, and let \( u \in C^\infty(\Gamma) \). Denote the components of the surface gradient \( \nabla_\Gamma u \) by \( \ud_i u := (\nabla_\Gamma u)_i \) for \( i = 1, 2, 3 \). The relevant identities for the surface gradient, including the Leibniz rule, chain rule, integration-by-parts formula, commutators, and the evolution equation for the normal vector, are summarized below.

\begin{lemma}\label{lemma:ud}
    Let \( \Gamma \) and \( \Gamma^\prime \) be smooth surfaces, and let \( f, h \in C^\infty(\Gamma) \) and \( g \in C^\infty(\Gamma^\prime; \Gamma) \). The following properties hold:
    \begin{itemize}\upshape
        \item[1.] \textbf{Leibniz rule:} \( \ud_i(fh) = \ud_i f h + f \ud_i h \) on \( \Gamma \).
        \item[2.] \textbf{Chain rule:} \( \ud_i(f \circ g) = (\ud_j f \circ g) \cdot \ud_i g_j \) on \( \Gamma^\prime \).
        \item[3.] \textbf{Integration by parts:} If \( \Gamma \) is closed, then 
        \[
        \int_{\Gamma} \ud_i f = \int_{\Gamma} f H n_i,
        \]
        where \( n \) is the unit normal vector and \( H := \ud_i n_i \) is the mean curvature.
        \item[4.] \textbf{Commutator of second derivatives:} 
        \[
        \ud_i \ud_j f = \ud_j \ud_i f + n_i H_{jl} \ud_l f - n_j H_{il} \ud_l f,
        \]
        where \( H_{ij} := \ud_i n_j = \ud_j n_i \).
        \item[5.] \textbf{Material derivative of gradients:} If \( \Gamma \) evolves under a velocity field \( v \), and \( G_T := \bigcup_{t \in [0, T]} \Gamma(t) \times \{t\} \), then for all \( f \in C^2(G_T) \), 
        \[
        \md(\ud_i f) = \ud_i (\md f) - (\ud_i v_j - n_i n_l \ud_j v_l) \ud_j f,
        \]
        where \( \md \) denotes the material derivative with respect to \( v \).
        \item[6.] \textbf{Time derivative of surface integrals:} For \( f, h \in C^2(G_T) \),
        \[
        \frac{\d}{\d t} \int_{\Gamma} f h = \int_{\Gamma} \md f h + \int_{\Gamma} f \md h + \int_{\Gamma} f h (\nabla_\Gamma \cdot v),
        \]
        where \( \nabla_\Gamma \cdot v := \ud_i v_i \).
        \item[7.] \textbf{Evolution of the unit normal vector:} The unit normal vector \( n \) evolves according to the velocity field \( v \) as
        \[
        \md n_i = -\ud_i v_j n_j.
        \]
    \end{itemize}
\end{lemma}

\begin{proof}
    The proofs of these properties can be found in \cite[Lemma 5.1]{bai2024new} and \cite[Eq. (2.7)]{bai2023erratum}.
\end{proof}

The following formula quantifies the error in \( H^1 \)-bilinear forms due to surface discrepancies. This result can be found in \cite[Lemma 7.1]{kovacs2019convergent}.

\begin{lemma}\label{lemma:A_iden}
    Let \( w_h^\theta \) and \( z_h^\theta \) be families of finite element functions  defined on the intermediate surface \( \hat{\Gamma}_{h,\theta}^m = (1 - \theta) \Ghsm + \theta \Ghm \). The following identity holds:
    \begin{align}\label{intermediate-integral-derivative}
	\frac{\d }{\d  \theta}\int_{\hat \Gamma_{h,\theta}^m} \nabla_{\hat \Gamma_{h,\theta}^m} w_h^\theta \cdot \nabla_{\hat \Gamma_{h,\theta}^m} z_h^\theta &= \int_{\hat{\Gamma}_{h,\theta}^m} \nabla_{\hat{\Gamma}_{h,\theta}^m} w_h^\theta \cdot (D_{\Thtm} \hat{e}_h^m) \nabla_{\hat{\Gamma}_{h,\theta}^m} z_h^\theta \, \d \theta + \int_{\hat \Gamma_{h,\theta}^m}\nabla _{\hat \Gamma_{h,\theta}^m}\partial^\bullet _\theta w_h^\theta \cdot \nabla_{\hat \Gamma_{h,\theta}^m} z_h^\theta \notag\\
	& \quad\, + \int_{\hat \Gamma_{h,\theta}^m}\nabla _{\hat \Gamma_{h,\theta}^m}\partial^\bullet _\theta z_h^\theta \cdot \nabla_{\hat \Gamma_{h,\theta}^m} w_h^\theta,
	\end{align}
    where \( (D_{\Thtm} v)_{rl} := -\ud_l v_r - \ud_r v_l + \delta_{rl} \ud_m v_m \).
\end{lemma}

	%......................................................
    \subsection{The Error Equation}\label{S5.1}
    The error equations are derived by subtracting \eqref{er1}--\eqref{er2} from \eqref{hm1}--\eqref{hm2}, respectively. These yield the following expressions:
    \begin{subequations}
    \begin{align}
		&\int_{\Gamma_h^m} \frac{X_h^{m+1} - X_h^m}{\tau} \cdot \nahm \phi_h 
		- \int_{\Ghsm} \frac{X_{h,*}^{m+1} - \mathrm{id}}{\tau} \cdot \nahsm \phi_h \notag \\
		&\quad + \int_{\Gamma_h^m} \nabla_{\Gamma_h^m} X_h^{m+1} \cdot \nabla_{\Gamma_h^m} (\nahm \phi_h) 
		- \int_{\Ghsm} \nabla_{\Ghsm} X_{h,*}^{m+1} \cdot \nabla_{\Ghsm} (\nahsm \phi_h) \notag \\
		&= -d^m_v(\phi_h), \label{ev} \\[7pt]
		&\int_{\Gamma_h^m} \nabla_{\Gamma_h^m} \frac{X_h^{m+1} - X_h^m}{\tau} \cdot \nabla_{\Gamma_h^m} \chi_h 
		- \int_{\Ghsm} \nabla_{\Ghsm} I_h v^m \cdot \nabla_{\Ghsm} \chi_h \notag \\
		&\quad - \int_{\Gamma_h^m} \kappa_h^m \nahm \cdot \chi_h 
		+ \int_{\Ghsm} I_h \kappa^m \nahsm \cdot \chi_h \notag \\
		&= -d^m_\kappa(\chi_h). \label{ek}
    \end{align}
    \end{subequations}
    
	\subsubsection*{Treatment of the First Two Terms in \eqref{ev}$:$}
	Using relation \eqref{Xexrelat}, and relation $\hat{X}_{h,*}^m={\rm id}$ when $\hat{X}_{h,*}^m$ is considered as a finite element function on $\Ghsm$, the first two terms on the left-hand side of \eqref{ev} can be expressed as:
    \begin{subequations}
    \begin{align}
		&\quad\,\,\int_{\Gamma_h^m} \frac{X_h^{m+1} - X_h^m}{\tau} \cdot \nahm \phi_h 
		- \int_{\Ghsm} \frac{X_{h,*}^{m+1} - \hat{X}_{h,*}^m}{\tau} \cdot \nahsm \phi_h \notag \\
		&= \int_{\Ghsm} \frac{X_h^{m+1} - X_h^m}{\tau} \cdot \nahsm \phi_h 
		- \int_{\Ghsm} \frac{X_{h,*}^{m+1} - \hat{X}_{h,*}^m}{\tau} \cdot \nahsm \phi_h \notag \\
		&\quad \,\,+ \int_{\Gamma_h^m} \frac{X_h^{m+1} - X_h^m}{\tau} \cdot \nahm \phi_h 
		- \int_{\Ghsm} \frac{X_h^{m+1} - X_h^m}{\tau} \cdot \nahsm \phi_h \notag \\
		% &= \int_{\Ghsm} \frac{\eM - \ehm}{\tau} \cdot \nahsm \phi_h \notag \\
		% &\quad \,\,+ \int_{\Gamma_h^m} \frac{X_h^{m+1} - X_h^m}{\tau} \cdot \nahm \phi_h 
		% - \int_{\Ghsm} \frac{X_h^{m+1} - X_h^m}{\tau} \cdot \nahsm \phi_h \notag \\
		&= \int_{\Ghsm} \frac{\eM - \ehm}{\tau} \cdot \nahsm \phi_h + J_1^m(\phi_h), \label{J1v}
    \end{align}
	with 
    \begin{equation}\label{J1}
		J_1^m(\phi_h) := \int_{\Gamma_h^m} \frac{X_h^{m+1} - X_h^m}{\tau} \cdot \nahm \phi_h 
		- \int_{\Ghsm} \frac{X_h^{m+1} - X_h^m}{\tau} \cdot \nahsm \phi_h.
    \end{equation}
    \end{subequations}
    
	\subsubsection*{Treatment of the Third and Fourth Terms in \eqref{ev}$:$} 
	The third and fourth terms on the left-hand side of \eqref{ev} can be written as:
	\begin{align}
		&\quad\,\,\int_{\Gamma_h^m} \nabla_{\Gamma_h^m} X_h^{m+1} \cdot \nabla_{\Gamma_h^m} (\nahm \phi_h) 
		- \int_{\Ghsm} \nabla_{\Ghsm} X_{h,*}^{m+1} \cdot \nabla_{\Ghsm} (\nahsm \phi_h) \notag \\
		&= \int_{\Gamma_h^m} \nabla_{\Gamma_h^m} X_h^{m+1} \cdot \nabla_{\Gamma_h^m} [(\nahm - \nahsm) \phi_h] \notag \\
		&\quad \,\,+ \int_{\Gamma_h^m} \nabla_{\Gamma_h^m} X_h^{m+1} \cdot \nabla_{\Gamma_h^m} (\nahsm \phi_h) 
		- \int_{\Ghsm} \nabla_{\Ghsm} X_{h,*}^{m+1} \cdot \nabla_{\Ghsm} (\nahsm \phi_h) \notag \\
		&=: J_2^m(\phi_h) + J_3^m(\phi_h). \label{J23}
	\end{align}

	In \cite[Section 5.2]{bai2024new}, the expression \( J_3^m(\phi_h) \) is reformulated by introducing a set of bilinear forms defined for any two \(\mathbb{R}^3\)-valued functions \( u \) and \( v \) on a surface \(\Gamma\). These forms are defined as follows:
\begin{align}\label{def-A-AN-AT}
    A_{\Gamma}(u, v) &:= \int_{\Gamma} \nabla_{\Gamma} u \cdot \nabla_{\Gamma} v , \notag\\
    A^N_{\Gamma}(u, v) &:= \int_{\Gamma} \left[(\nabla_{\Gamma} u) n\right] \cdot \left[(\nabla_{\Gamma} v) n\right] , \notag\\
    A^T_{\Gamma}(u, v) &:= \int_{\Gamma} \operatorname{tr}\left[(\nabla_{\Gamma} u) (I - n n^\top) (\nabla_{\Gamma} v )^\top \right] , \notag\\
    B_{\Gamma}(u, v) &:= \int_{\Gamma} (\nabla_{\Gamma} \cdot u)(\nabla_{\Gamma} \cdot v) - \operatorname{tr}(\nabla_{\Gamma} u \nabla_{\Gamma} v) ,
\end{align}
where we have \( A_{\Gamma}(u, v) = A^N_{\Gamma}(u, v) + A^T_{\Gamma}(u, v) \). These bilinear forms are similarly defined on the approximate surfaces \(\Ghsm\), \(\Gamma_h^m\), and \(\hat\Gamma_{h,\theta}^m\).

An important identity involving these forms is provided in \cite[Eq. (5.8)]{bai2024new} and \cite[Eq. (5.9)]{bai2024new}:
\begin{equation}\label{relation-gradid-DuDv}
    \int_{\Gamma} \nabla_{\Gamma} \mathrm{id} \cdot (D_{\Gamma} u) \nabla_{\Gamma} v = -A^T_{\Gamma}(u, v) + B_{\Gamma}(u, v),
\end{equation}
which remains valid when applied to the approximate surfaces \(\Ghsm\), \(\Gamma_h^m\), and \(\hat\Gamma_{h,\theta}^m\).

Furthermore, as shown in \cite[Eq. (2.1)]{bai2023erratum}, if the surface \(\Gamma\) is sufficiently smooth, the symmetric bilinear form \( B_{\Gamma}(u, v) \) can be represented using integration by parts as:
\begin{align}\label{eq:tr_diff}
    B_{\Gamma}(u, v) &= \int_{\Gamma} u_j \partial_i v_i H n_j - \int_{\Gamma} u_j \partial_j v_i H n_i \notag\\
    &\quad + \int_{\Gamma} u_j \partial_k v_i n_i H_{jk} - \int_{\Gamma} u_j \partial_k v_i H_{ik} n_j, \quad \forall\, u, v \in H^1(\Gamma),
\end{align}
where \( H \) represents the mean curvature, \( n \) denotes the unit normal vector, and \( H_{ij} := \partial_i n_j = \partial_j n_i \) is the second fundamental form. This expression for \( B_{\Gamma}(u, v) \) is essential for analyzing geometric properties and curvature-driven flows on the surface.

We define $\hat{X}_{h, \theta}^{m} := (1 - \theta)\hat X_{h, *}^{m} + \theta X_{h}^{m}$ and $X_{h, \theta}^{m+1} := (1 - \theta)X_{h, *}^{m+1} + \theta X_{h}^{m+1}$ in the sense of nodal vectors. Then the second term on the right-hand side of \eqref{J23} can be decomposed as follows (as shown in \cite[Eq. (5.10)]{bai2024new}, or using Lemma \ref{lemma:A_iden} with $z_h$ replaced by $\nahsm\phi_{h}$)
\begin{align}
J_3^m(\phi_h)=&\int_{\Gamma^m_h}\nabla_{\Gamma^m_h}X^{m+1}_h\cdot\nabla_{\Gamma^m_h}(\nahsm\phi_h)-\int_{\Ghsm}\nabla_{\Ghsm}X^{m+1}_{h,*}\cdot\nabla_{\Ghsm}(\nahsm\phi_h)\notag\\
	=&A_{h, *}^N(e_h^{m+1},\nahsm\phi_{h}) + A_{h, *}^T(e_h^{m+1} - \hat{e}_h^m,\nahsm\phi_{h}) + B^m(\hat{e}_h^m,\nahsm\phi_{h}) +K^m(\nahsm\phi_{h}) ,\label{J3}
\end{align}
where we have used the following notations for simplicity:
\begin{subequations}
\begin{align}\label{def-As-AGs}
	A_{h,*}^N(u,v) 
	:\!\!&= A_{\Ghsm}^N(u,v) 
	\quad\mbox{and}\quad
	A_{h,*}^T(u,v) 
	:= A_{\Ghsm}^T(u,v) , \\
	A_{h,*}(u,v) :\!\!&= A_{h,*}^N(u,v)  + A_{h,*}^T(u,v)  
	\quad\mbox{and}\quad
	B^m(u,v) 
	= B_{\Gm}(u^l, v^l) \label{def-Bm} \\
	K^m(v) 
	&= \int_0^1 \big[ A_{\hat\Gamma_{h,\theta}^m}^N(e_h^{m+1},v) - A_{\Ghsm}^N(e_h^{m+1},v) \big] \d\theta  \notag\\
	&\quad + \int_0^1 
	\big[A_{\hat\Gamma_{h,\theta}^m}^T(e_h^{m+1} - \hat e_h^m,v) - A_{\Ghsm}^T(e_h^{m+1} - \hat e_h^m,v) \big] \d\theta   \notag\\
	&\quad+ \int_0^1 \big[ B_{\hat\Gamma_{h,\theta}^m}(\hat e_h^m,v)  - B_{\Ghsm}(\hat e_h^m,v) \big] \d\theta\notag\\
	&\quad + B_{\Ghsm}(\hat e_h^m,v) -B_{\Gm}(\hat e_h^{m,l},v^l) \notag\\
	&\quad+\int_0^1 \int_{ \hat\Gamma_{h,\theta}^m}  \nabla_{\hat\Gamma_{h,\theta}^m} (X_{h, \theta}^{m+1}-\hat X_{h, \theta}^{m}) \cdot D_{\hat\Gamma_{h,\theta}^m} \hat e_{h}^{m} \nabla_{\hat\Gamma_{h,\theta}^m} v\d\theta\notag\\
	& =: K_1^m(v) + K_2^m(v) + K_3^m(v) +K_4^m(v) + K_5^m(v). \label{k_m}
\end{align}
\end{subequations}

\subsubsection*{Treatment of the First and Second Terms in \eqref{ek}$:$} 
	The first two terms on the left-hand side of \eqref{ek} can be written as 
    \begin{subequations}\label{R1}
    \begin{align}
		&\int_{\Gamma_h^m}\nabla_{\Gamma^m_h}\frac{X_h^{m+1}-X_h^m}{\tau}\cdot\nabla_{\Gamma^m_h}\chi_h-\int_{\Ghsm}\nabla_{\Ghsm}I_hv^m\cdot \nabla_{\Ghsm}\chi_h\notag\\
		% =&\int_{\Gamma_h^m}\nabla_{\Gamma^m_h}\Big(\frac{X_h^{m+1}-X_h^m}{\tau}-I_hv^m\Big)\cdot\nabla_{\Gamma^m_h}\chi_h\notag\\
	% &+\int_{\Gamma_h^m}\nabla_{\Gamma^m_h}I_hv^m\cdot\nabla_{\Gamma^m_h}\chi_h-\int_{\Ghsm}\nabla_{\Ghsm}I_hv^m\cdot \nabla_{\Ghsm}\chi_h\notag\\
		=&\int_{\Gamma_h^m}\nabla_{\Gamma^m_h}\Big(\frac{X_h^{m+1}-X_h^m}{\tau}-I_hv^m\Big)\cdot\nabla_{\Gamma^m_h}\chi_h + R_1^m(\chi_h),
    \end{align}
    with
    \begin{align}
R_1^m(\chi_h):=\int_{\Gamma_h^m}\nabla_{\Gamma^m_h}I_hv^m\cdot\nabla_{\Gamma^m_h}\chi_h-\int_{\Ghsm}\nabla_{\Ghsm}I_hv^m\cdot \nabla_{\Ghsm}\chi_h .
    \end{align}
    \end{subequations}
	% \\[-50pt]
	\subsubsection*{Treatment of the Third and Fourth Terms in \eqref{ek}$:$} 
	The third and fourth terms on the left-hand side of \eqref{ek} can be rewritten as
    \begin{subequations}\label{R2}
    \begin{align}
		-\int_{\Gamma_h^m}\kappa_h^m\nahm\cdot \chi_h+\int_{\Ghsm}I_h\kappa^m\nahsm \cdot \chi_h
		% =&-\int_{\Gamma_h^m}\kappa_h^m\nahm\cdot \chi_h+\int_{\Gamma_h^m}I_h\kappa^m\nahm\cdot \chi_h\notag\\
		% &-\int_{\Gamma_h^m}I_h\kappa^m\nahm\cdot \chi_h+\int_{\Ghsm}I_h\kappa^m\nahsm \cdot \chi_h\notag\\
		=&\int_{\Gamma_h^m}(I_h\kappa^m-\kappa_h^m)\nahm\cdot \chi_h + R_2^m(\chi_h) ,
	\end{align}
    with
    \begin{align}
    R_2^m(\chi_h):= -\int_{\Gamma_h^m}I_h\kappa^m\nahm\cdot \chi_h+\int_{\Ghsm}I_h\kappa^m\nahsm \cdot \chi_h . 
    \end{align}
    \end{subequations}
    
    In conclusion, the error equation \eqref{ev} can be reformulated in the following structured form by incorporating the expressions from \eqref{J1v}, \eqref{J23}, and \eqref{J3}:
\begin{align}\label{err}
    &\int_{\Ghsm} \Big(\frac{\eM - \ehm}{\tau} \cdot \nahsm\bigg) \phi_h + A_{h, *}^N(e_h^{m+1}, \nahsm \phi_{h}) + A_{h, *}^T(e_h^{m+1} - \hat{e}_h^m, \nahsm \phi_{h}) + B^m(\hat{e}_h^m, \nahsm \phi_{h}) \notag\\
    &= -J_1^m(\phi_h) - J_2^m(\phi_{h}) - K^m(\nahsm \phi_{h}) - d_v^m(\phi_{h}).
\end{align}

Similarly, the error equation \eqref{ek} can be expressed in the following form by utilizing the expressions given in \eqref{R1} and \eqref{R2}:
\begin{align}\label{ekk}
    &\int_{\Gamma_h^m} \nabla_{\Gamma_h^m} \left( \frac{X_h^{m+1} - X_h^m}{\tau} - I_h v^m \right) \cdot \nabla_{\Gamma_h^m} \chi_h + \int_{\Gamma_h^m} (I_h \kappa^m - \kappa_h^m) \nahm \cdot \chi_h \notag\\
    &= -R_1^m(\chi_h) - R_2^m(\chi_h) - d_\kappa^m(\chi_{h}).
\end{align}

\subsection{Outline of the Proof of Theorem~\ref{thm:main}}\label{outline}
For the readers' convenience, we provide an outline of the proof of Theorem~\ref{thm:main} in this subsection, specifying the test function employed in each step for the respective error equation. Readers may skip this subsection and proceed directly to the detailed proof. 

To complete the proof of Theorem~\ref{thm:main}, we need to define and bound the velocity error of the numerical solution, as well as the modified velocity error defined below: 
\begin{equation}\label{evm}
\begin{aligned}
      \evm&:=\frac{X_h^{m+1}-X_h^m}{\tau}-I_hv^m,\\
      \ehvm&:=\frac{\eM-\ehm}{\tau}-I_h\Tsm v^m,\\[5pt] 
      \ek&:=\kappa_h^m-I_h\kappa^m,
\end{aligned}
\end{equation}
where $ v^m $ is the exact solution of \eqref{kv} at time $ t_m $ (defined on the surface $ \Gamma^m $), and $ I_h $ denotes the Lagrange interpolation operator onto $ \Ghsm $. The nodal vectors of $ I_hv^m $ and $ I_h\Tsm v^m $ are thus well-defined, ensuring that $ \evm $, $ \ehvm $, and $ \ek $ can be interpreted on any surface using these nodal vectors. Using the decomposition $ v^m = -H^m n^m + T_*^m v^m $ (into normal and tangential components) and the relation \eqref{Xexrelat}, we obtain
\begin{equation}\label{xm11}
\begin{aligned}			
      X_h^{m+1} - X_h^{m}  - \tau I_h v^m 
      &= X_h^{m+1} - X_h^{m}  - \tau I_h (-H^m n^m)  - \tau I_h \Tsm v^m \\
      &= \eM - \ehm  - \tau I_h \Tsm v^m + \tau I_h g^m.
\end{aligned}
\end{equation}
Dividing both sides of \eqref{xm11} by $ \tau $, we find that $ \ehvm $ is related to $ \evm $ via 
\begin{equation}\label{xm}
\evm=\ehvm+I_hg^m.
\end{equation}
\begin{enumerate}
\item \textbf{Estimates for \(J_i^m(\phi_h)\) and \(R_i^m(\chi_h)\) with $i=1,2$.}  \\
The bounds for \(J_i^m(\phi_h)\) and \(R_i^m(\chi_h)\) \((i=1,2)\) are established in Section~\ref{S5.2} by the fundamental theorem of calculus together with Lemmas~\ref{Lem:normalvector}--\ref{Lem:normalvector2}. 
%In addition to the position‐error terms \(\hat e_h^m\) appearing on the right‐hand side of the estimates for \(J_i^m(\phi_h)\), one also requires corresponding bounds for the velocity error \(\hat e_v^m\) defined in \eqref{evm}.
\item \textbf{Estimate for tangential component of \(\hat e_v^m\).}\\
% By exploiting the orthogonality \(\bar T_h^m \perp \bar N_h^m\) and the stability estimate from Lemma~\ref{lemma:NT_stab}, and 
% By invoking the relation \eqref{xm}, Lemmas~\ref{Lem:normalvector}--\ref{Lem:normalvector2}, the super-approximation properties in Lemmas \ref{super}-\ref{super-bar-nhm}, and the stability estimates in Lemma \ref{stability}, the various norms of 
% \[
% I_h \bar T_h^m e_v^m
% \quad\text{and}\quad
% I_h \bar N_h^m e_v^m
% \]
% can be reduced to the norms of 
% \[
% I_h \bar T_{h,*}^m \hat e_v^m
% \quad\text{and}\quad
% I_h \bar N_{h,*}^m \hat e_v^m,
% \]
% up to controllable error terms. 
By exploiting the orthogonality \(\bar T_h^m \perp \bar N_h^m\), one obtains the estimate in Lemma~\ref{lemma:NT_stab}, which provides a stability bound for the \(H^1\) bilinear form  
\[
 \int_{\Gamma_h^m} \nabla_{\Gamma_h^m} I_h \Nbhm\evm \cdot  \nabla_{\Gamma_h^m} I_h \Tbhm\evm 
\]
in terms of $\|I_h \Nbhm\evm\|_{L^2(\Ghsm)}$ and $\|\nabla_{\Ghsm}I_h \Tbhm\evm\|_{L^2(\Ghsm)}$, i.e., with one fewer derivative on $I_h \Nbhm\evm$. 
By utilizing this stability bound and testing the error equation \eqref{ekk} with 
\[
\chi_h := I_h \bar T_h^m e_v^m,
\]
we obtain estimates of the following terms in \eqref{I_hTsmL}--\eqref{evmH1}: 
\[
\| \nabla_{\Ghsm} I_h \Tahsm \ehvm \|_{L^2(\Ghsm)} , \quad
\|  \ehvm \|_{L^2(\Ghsm)}	,	\quad
\| \nabla_{\Ghsm} \ehvm \|_{L^2(\Ghsm)} .
\]

Moreover, the $L^2$ norm of the auxiliary variable error $e_\kappa^m$ and the normal component of the velocity error $\hat e_v^m$ appear on the right-hand side of the estimates \eqref{I_hTsmL}--\eqref{evmH1}, and these terms will be controlled in the subsequent steps.

\item \textbf{Estimate for auxiliary variable error $e_\kappa^m$.}  \\
We test the error equation \eqref{ekk} with 
\[
  \chi_h  :=  I_h\bigl(e_\kappa^m\,\bar n_h^m\bigr).
\]
This yields an estimate for $ \|e_\kappa^m\|_{L^2(\hat\Gamma_{h,*}^{m})} $, as stated in \eqref{estekL2} of Lemma~\ref{e_kappa_l2}.

\item \textbf{Estimate for normal component of \(\hat e_v^m\).}  \\
We test the error equation \eqref{err} with
\[
  \phi_h := I_h\bigl(\hat e_v^m \cdot \bar n_{h,*}^m\bigr).
\]
This provides an estimate for $ \bigl\|I_h\bar N_{h,*}^m \hat e_v^m\bigr\|_{L^2(\Ghsm)} $; see \eqref{INe}. By combining this result with the bound on the auxiliary variable $ e_\kappa^m $ from \eqref{estekL2}, we derive refined estimates of the velocity error, which yield bounds for
\[
  \|\hat e_v^m\|_{L^2(\Ghsm)}, 
  \quad
  \bigl\|\nabla_{\Ghsm}\bigl(I_h\bar T_{h,*}^m \hat e_v^m\bigr)\bigr\|_{L^2(\Ghsm)}, 
  \quad
  \|\hat e_v^m\|_{H^1(\Ghsm)},
\]
as stated in \eqref{evmTL}--\eqref{eq:vel_H1}.

\item \textbf{Error estimate for projected distance.}  \\
We test the error equation \eqref{err} with
\[
  \phi_h := I_h\bigl(\hat e_h^{m+1}\cdot\bar n_{h,*}^m\bigr).
\]
This yields inequality \eqref{Err} but introduces two main challenges:
\begin{enumerate}[label=(\roman*)]
	\item \emph{Transfer of coercivity from $ e_h^{m+1} $ to $ \hat e_h^{m+1} $.} We establish the estimate 
	  $$
		A_{h,*}\bigl(e_h^{m+1},e_h^{m+1}\bigr)
		\gtrsim
		\bigl\|\nabla_{ \Ghsm}\hat e_h^{m+1}\bigr\|_{L^2( \Ghsm)}^2
		-\bigl\|\hat e_h^m\cdot\bar n_{h,*}^m\bigr\|_{L^2( \Ghsm)}^2
		-\tau^2
	  $$
	  by leveraging the geometric relations \eqref{eq:geo_rel_1}--\eqref{eq:geo_rel_2}, as detailed in \eqref{A_hh_hat_without_relation}.
	\item \emph{Transfer of normal-component errors.} At each time step, we convert the term
	  $$
	\bigl\|\eM\cdot\bar n_{h,*}^m\bigr\|_{L^2( \Ghsm)}^2
        \quad\mbox{to}\quad
        \bigl\|\ehM\cdot\bar n_{h,*}^{m+1}\bigr\|_{L^2(\hat\Gamma_{h,*}^{m+1})}^2
	  $$
	  using the refined geometric relations \eqref{eq:geo_rel_4}--\eqref{eq:geo_rel_6}, along with the $ L^\infty $-estimate $ \|\hat X_{h,*}^{m+1}-\hat X_{h,*}^m\|_{L^\infty}=\mathcal{O}(\tau) $ from Lemma~\ref{Xhinfty}.
\end{enumerate}
  
  \item \textbf{Norm equivalences and discrete Grönwall inequality.}\\
  By employing the norm equivalences established in Section~\ref{norm-equiv-surfaces} and applying a discrete Grönwall inequality, the desired error bounds are obtained; the uniform control of the coefficients \(\kappa_l\) and \(\kappa_{*,l}\) remains to be addressed.
  
  \item \textbf{Completion of mathematical induction and boundedness of \(\kappa_l\) and \(\kappa_{*,l}\).}\\
 By applying mathematical induction and the discrete Grönwall inequality, together with standard inverse inequalities, the uniform boundedness of \(\kappa_l\) and \(\kappa_{*,l}\) can be established, as detailed in Appendix~\ref{appendix_H}. This uniform boundedness, together with the error estimates in \eqref{eq:err_fin0}, ensures that the mathematical induction hypotheses from Section~\ref{induc ass} are satisfied. Consequently, Theorem~\ref{thm:main} is proved.
\end{enumerate}

	%.......................................................
\subsection{Estimates for $J_i^m(\phi_h)$ with $i=1,2$ and $R_i^m(\chi_h)$ with $i=1,2$}
\label{S5.2}
% For simplicity, we define:
% 	\begin{equation}\label{evm}
% 		\begin{aligned}
% 			\evm&:=\frac{X_h^{m+1}-X_h^m}{\tau}-I_hv^m,\\
% 			\ehvm&:=\frac{\eM-\ehm}{\tau}-I_h\Tsm v^m,\\
% 			\ek&:=\kappa_h^m-I_h\kappa^m.
% 		\end{aligned}
% 	\end{equation}
% 	By the relation in \eqref{kv} and \eqref{Xexrelat}, the following equations hold:
% 	\begin{equation}\label{xm}
% 		\begin{aligned}
% 			X_h^{m+1} - X_h^{m}  - \tau I_h v^m 
% 			&= X_h^{m+1} - X_h^{m}  - \tau I_h (-H^m n^m)  - \tau I_h \Tsm v^m \\
% 			&= \eM - \ehm  - \tau I_h \Tsm v^m + \tau I_h g^m, \\
% 			\evm&=\ehvm+I_hg^m.
% 		\end{aligned}
% 	\end{equation}
	
    Recall that the intermediate surfaces are given by $\hat\Gamma_{h,\theta}^{m} = (1-\theta)\, \hat\Gamma_{h,*}^{m} + \theta\, \Gamma_{h}^{m}$ for $\theta \in [0,1]$. Let $\nhtm$ denote the unit normal vector on $\hat\Gamma_{h,\theta}^{m}$, and set $\nahtm = P_{\hat\Gamma_{h,\theta}^m}\nhtm$ as the averaged normal vector on this surface.

	The function $J_1^m(\phi_h)$ defined in \eqref{J1} can be rewritten into the following form using the fundamental theorem of calculus:
	\begin{align}
		J_1^m(\phi_h)=&\int_{\Gamma_h^m}\frac{X_h^{m+1}-X_h^m}{\tau}\cdot \nahm\phi_h-\int_{\Ghsm}\frac{X_h^{m+1}-X_h^m}{\tau}\cdot \nahsm\phi_h\notag\\
		% =&\int_{\Thtm}\frac{X_h^{m+1}-X_h^m}{\tau}\cdot \nahtm\phi_h \bigg|_{\theta=0}^{\theta=1}\notag\\
		=&\int_0^1 \frac{\d}{\d\theta}\int_{\Thtm}\frac{X_h^{m+1}-X_h^m}{\tau}\cdot \nahtm\phi_h\d\theta\notag\\
		=&\int_0^1\int_{\Thtm}\frac{X_h^{m+1}-X_h^m}{\tau}\cdot \partial_\theta^\bullet\nahtm\phi_h\notag\\
		&+\int_0^1\int_{\Thtm}\frac{X_h^{m+1}-X_h^m}{\tau}\cdot\nahtm\phi_h(\nabla_{\hat\Gamma_{h,\theta}^m } \cdot \hat e_{h}^m) \quad\mbox{(Lemma \ref{lemma:ud}, item 6)} .\label{defJ1}
	\end{align}
	The term $\partial_\theta^\bullet \nahtm= \partial_\theta^\bullet P_{\hat\Gamma_{h,\theta}^m}\nhtm$ can be explicitly written using \eqref{eq:Pnhm} and \eqref{relation_material}; thus,
\begin{align}\label{def:dnahtm}
    \partial_\theta^\bullet \nahtm &= P_{\hat{\Gamma}_{h,\theta}^m} \left[ \partial_\theta^\bullet \nhtm - \big(\nahtm - \nhtm\big) \nabla_{\hat{\Gamma}_{h,\theta}^m} \cdot \hat{e}_{h}^m \right] \notag\\
    &= P_{\hat{\Gamma}_{h,\theta}^m} \left[ -(\nabla_{\hat{\Gamma}_{h,\theta}^m} \hat{e}_{h}^m) \hat{n}_{h,\theta}^m - \big(\nahtm - \nhtm\big) \nabla_{\hat{\Gamma}_{h,\theta}^m} \cdot \hat{e}_{h}^m \right].
\end{align}
% where the last equality follows from the identity \(\partial_\theta^\bullet \nhtm = -(\nabla_{\hat{\Gamma}_{h,\theta}^m} \hat{e}_{h}^m) \hat{n}_{h,\theta}^m\), as established by item 7 of Lemma \ref{lemma:ud}.

The proof of \eqref{eq:nsa1-infty} can also be used to prove $\|\bar n_{h,\theta}^m - \bar n_{h,*}^m\|_{L^\infty(\Ghsm)} \lesssim h^{-1}\|\nabla_{\Ghsm} \hat e_h^m\|_{L^2(\Ghsm)}$ uniformly holds for all \(\theta \in [0,1]\), which implies
\begin{align}\label{intermediate-normal-Linfty-bound}
	\|\bar{n}_{h,\theta}^m\|_{L^\infty(\hat{\Gamma}_{h,\theta}^m)} &\lesssim \|\bar{n}_{h,*}^m\|_{L^\infty(\hat{\Gamma}_{h,\theta}^m)} + \|\bar{n}_{h,\theta}^m-\bar{n}_{h,*}^m\|_{L^\infty(\hat{\Gamma}_{h,\theta}^m)}\notag\\
	& \lesssim \|\bar{n}_{h,*}^m\|_{L^\infty(\hat{\Gamma}_{h,\theta}^m)} + h^{-1}\|\nabla_{\Ghsm} \hat e_h^m\|_{L^2(\Ghsm)}\notag\\
	& \lesssim \|\bar{n}_{h,*}^m\|_{L^\infty(\hat{\Gamma}_{h,\theta}^m)} + h^{0.6} \lesssim 1,
\end{align}
where the second-to-last inequality uses the mathematical induction assumption \eqref{Linfty-W1infty-hat-em} and the $L^\infty$-boundedness of $\bar n_{h,*}^m$ in \eqref{nhmW-1}. From identity \eqref{def:dnahtm}, inequality \eqref{intermediate-normal-Linfty-bound} and the norm equivalences on \(\hat{\Gamma}_{h,\theta}^m\) for \(\theta \in [0,1]\) derived from \eqref{Linfty-W1infty-hat-em} and Lemma \ref{equi-MA}, we obtain 
\begin{align}\label{material-deri-normal}
    \|\partial_\theta^\bullet \nahtm\|_{L^2(\hat{\Gamma}_{h,\theta}^m)} 
    \lesssim \|\nabla_{\hat{\Gamma}_{h,*}^m} \hat{e}_h^m\|_{L^2(\hat{\Gamma}_{h,*}^m)}.
\end{align}

Using the relation $ (X_h^{m+1} - X_h^m)/\tau = I_h v^m + e_v^m = I_h v^m + \hat e_v^m + I_h g^m $, which follows from \eqref{evm} and \eqref{xm}, together with the expressions in \eqref{defJ1} and inequality \eqref{material-deri-normal}, we obtain the following estimate for $ J_1^m(\phi_h) $:
\begin{align}\label{estJ1-0}
    |J_1^m(\phi_h)| 
    &\lesssim \| \ehvm \|_{L^2(\Ghsm)} 
    \| \nabla_{\Ghsm} \ehm \|_{L^2(\Ghsm)} 
    \| \phi_h \|_{L^\infty(\Ghsm)} \notag\\
	& \quad\,\,+ \|I_h v^m + I_h g^m \|_{L^\infty(\Ghsm)} \| \nabla_{\Ghsm} \ehm \|_{L^2(\Ghsm)} 
    \| \phi_h \|_{L^2(\Ghsm)} \notag\\
    &\lesssim \| \ehvm \|_{L^2(\Ghsm)} 
    \| \nabla_{\Ghsm} \ehm \|_{L^2(\Ghsm)} 
    \| \phi_h \|_{L^\infty(\Ghsm)} + \| \nabla_{\Ghsm} \ehm \|_{L^2(\Ghsm)} 
    \| \phi_h \|_{L^2(\Ghsm)} \notag\\
	&\,\,\quad \text{(using \eqref{W1infty-g} and the $ L^\infty $ stability of the Lagrange interpolation)}.
\end{align}
Subsequently, applying the inverse inequality $\| \phi_h \|_{L^\infty(\Ghsm)}\lesssim h^{-1}\| \phi_h \|_{L^2(\Ghsm)}$ which has been shown in \eqref{inverse-ineq}, we derive the following result: 
\begin{align}\label{estJ1}
	|J_1^m(\phi_h)| \lesssim (h^{-1}\|e_v^m\|_{L^2(\Ghsm)}+1)\| \nabla_{\Ghsm} \ehm \|_{L^2(\Ghsm)} 
    \| \phi_h \|_{L^2(\Ghsm)}.
\end{align}

Using the relation $ (X_h^{m+1} - X_h^m)/\tau = I_h v^m + e_v^m = I_h v^m + \hat e_v^m + I_h g^m $ once more, along with integration by parts, geometric perturbation estimates, and the fundamental theorem of calculus, we can derive the following bound for $ |J_2^m(\phi_h)| $ (the detailed proof is omitted here and provided in Appendix~\ref{appendix_D}):
\begin{align}\label{estJ2}
    |J_2^m(\phi_h)| 
    &\lesssim \|\nabla_{\Ghsm} \ehm\|_{L^2(\Ghsm)} \|\phi_h\|_{L^2(\Ghsm)} 
    + h^{0.5} \|\nabla_{\Ghsm} \ehm\|_{L^2(\Ghsm)} \|\phi_h\|_{H^1(\Ghsm)} \notag \\
    &\quad + \big(h^{-2.1} \tau + h^{-2.1} \tau \|\ehvm\|_{L^2(\Ghsm)}\big) 
    \|\nabla_{\Ghsm} \ehm\|_{L^2(\Ghsm)} \|\phi_h\|_{H^1(\Ghsm)}.
\end{align}

By leveraging the fundamental theorem of calculus and Lemma \ref{lemma:A_iden}, and the norm equivalence established in Lemma \ref{equi-MA}, we derive the following estimate for \( R_1^m(\chi_h) \):
\begin{align}\label{estR1}
    |R_1^m(\chi_h)| 
    &= \bigg|\int_{\Gamma_h^m} \nabla_{\Gamma^m_h} I_h v^m \cdot \nabla_{\Gamma^m_h} \chi_h 
    - \int_{\Ghsm} \nabla_{\Ghsm} I_h v^m \cdot \nabla_{\Ghsm} \chi_h \bigg| \notag \\
    &\lesssim \| \nabla_{\Ghsm} I_h v^m \|_{L^\infty(\Ghsm)} 
    \| \nabla_{\Ghsm} \ehm \|_{L^2(\Ghsm)} 
    \| \nabla_{\Ghsm} \chi_h \|_{L^2(\Ghsm)} \notag \\
    &\lesssim \| \nabla_{\Ghsm} \ehm \|_{L^2(\Ghsm)} 
    \| \nabla_{\Ghsm} \chi_h \|_{L^2(\Ghsm)}.
\end{align}

Analogously, the estimate for \(| R_2^m(\chi_h) |\) is obtained by adopting the same approach as for \( J_1^m(\phi_h) \). In particular, by invoking the fundamental theorem of calculus and item 6 in Lemma \ref{lemma:ud}, the following estimate holds:
\begin{align}\label{estR2}
    |R_2^m(\chi_h)| 
    &= \Big| -\int_{\Gamma_h^m} I_h \kappa^m \nahm \cdot \chi_h 
    + \int_{\Ghsm} I_h \kappa^m \nahsm \cdot \chi_h \Big| \notag \\
    &\le \Big| \int_0^1 \int_{\hat \Gamma_{h,\theta}^m} I_h \kappa^m \partial_\theta^\bullet  \nahtm \cdot\chi_h \Big| 
    + \Big| \int_0^1 \int_{\hat \Gamma_{h,\theta}^m} I_h \kappa^m  \nahtm \cdot\chi_h (\nabla_{\hat \Gamma_{h,\theta}^m} \cdot \hat e_h^m) \Big| 
     \notag \\
    &\lesssim \| I_h \kappa^m \|_{L^\infty(\Ghsm)} 
    \| \nabla_{\Ghsm} \ehm \|_{L^2(\Ghsm)} 
    \| \chi_h \|_{L^2(\Ghsm)} \quad \text{(using \eqref{intermediate-normal-Linfty-bound} and \eqref{material-deri-normal})} \notag \\
    &\lesssim \| \nabla_{\Ghsm} \ehm \|_{L^2(\Ghsm)} 
    \| \chi_h \|_{L^2(\Ghsm)}.
\end{align}

\subsection{Stability of the tangential motion}\label{sec:stability-tangential}

In this subsection, we derive stability estimates for the velocity error $ e_v^m $ and the modified velocity error $ \hat{e}_v^m $, as defined in \eqref{evm} at the beginning of Section~\ref{outline}. To this end, we test the error equation \eqref{ekk} with $ \chi_h := I_h \Tahm e_v^m $ and use the definition of $ e_\kappa $ in \eqref{evm} to obtain the following relation:
    \begin{align}\label{eq:t1}
		\int_{\Gamma_h^m}\nabla_{\Gamma^m_h}\evm \cdot \nabla_{\Gamma^m_h} I_h\Tahm\evm
		&= \int_{\Gamma_h^m} \ek \nahm \cdot I_h\Tahm\evm 
		- R_1^m(I_h\Tahm\evm)\notag\\
		& \quad\,\, - R_2^m(I_h\Tahm\evm) - d_\kappa^m(I_h\Tahm\evm).
    \end{align}
	Since \( I_h(\Tahm (I_h\Tahm e_v^m)) = I_h\Tahm e_v^m \) and \( \Tahm\nahm = 0 \), the first term on the right-hand side of \eqref{eq:t1} can be estimated using the super-approximation property \eqref{super-bar-tangent} and the norm equivalence between \( \Gamma_h^m \) and \( \hat{\Gamma}_{h,*}^m \), as established by \eqref{Linfty-W1infty-hat-em} and Lemma \ref{equi-MA}:
	\begin{align}
		& \quad\,\,\Big|\int_{\Gamma_h^m} \ek \nahm \cdot I_h\Tahm\evm \Big|= \Big|\int_{\Gamma_h^m} \ek \nahm \cdot (I_h - 1)(\Tahm (I_h\Tahm\evm)) \Big| \notag\\
		&\lesssim  \|\ek\|_{L^2(\Ghsm)}\|\nahm\|_{L^\infty(\Ghsm)}h^2\|\bar n_h^m\|_{W^{1,\infty}(\Ghsm)} \|I_h\Tahm\evm\|_{H^1(\Ghsm)}  \notag\\
		&\lesssim (h^2 + \|\nabla_{\hat \Gamma_{h,*}^m}\hat e_h^m\|_{L^2(\hat \Gamma_{h,*}^m)}) \|\ek\|_{L^2(\Ghsm)} \|I_h\Tahm\evm\|_{H^1(\Ghsm)},
	\end{align}
	where the last inequality follows from the $L^\infty$-boundedness of $\bar n_h^m$ in \eqref{nhmW-2b} and the estimate for $\|\bar  n_h^m\|_{W^{1,\infty}(\Ghsm)}$ in \eqref{nhmW-2a}.
	
	Using the estimates in \eqref{estR1} and \eqref{estR2} with \( \chi_h = I_h\Tahm\evm \), together with the norm equivalence between \( \Gamma_h^m \) and \( \hat{\Gamma}_{h,*}^m \), we derive the following estimates:
	\begin{align}
		|R_1^m(I_h\Tahm\evm)| &\lesssim \| \nabla_{\Ghsm} \ehm \|_{L^2(\Ghsm)} \| \nabla_{\Ghsm} I_h\Tahm\evm \|_{L^2(\Ghsm)}, \\
		|R_2^m(I_h\Tahm\evm)| &\lesssim \| \nabla_{\Ghsm} \ehm \|_{L^2(\Ghsm)} \| I_h\Tahm\evm \|_{L^2(\Ghsm)}.
	\end{align}
	
For the last term on the right-hand side of \eqref{eq:t1}, the consistency error estimate \eqref{dk} in Lemma~\ref{consterr} yields
\begin{equation}
	|d_\kappa^m(I_h\Tahm\evm)| \lesssim (1 + \kappa_{*,l}) h^k \|I_h\Tahm\evm\|_{H^1(\Ghsm)}.
\end{equation}

Substituting these estimates into \eqref{eq:t1}, we obtain
\begin{align}\label{eq:XT}
	& \quad\,\,\int_{\Gamma_h^m} \nabla_{\Gamma^m_h} \evm \cdot \nabla_{\Gamma^m_h} I_h\Tahm\evm \notag\\
	&\lesssim 
	\Big[
		(1 + \kappa_{*,l})h^k 
		+ \| \nabla_{\Ghsm} \ehm \|_{L^2(\Ghsm)} 
		\notag \\
		&\quad 
		+ \big(h^2 + \|\nabla_{\hat \Gamma_{h,*}^m}\hat e_h^m\|_{L^2(\hat \Gamma_{h,*}^m)}\big)
		\|\ek\|_{L^2(\Ghsm)}
	\Big] 
	\|I_h\Tahm\evm\|_{H^1(\Ghsm)} 
	\notag \\
	&\lesssim 
	\Big[
		(1 + \kappa_{*,l})h^k 
		+ \| \nabla_{\Ghsm} \ehm \|_{L^2(\Ghsm)} 
		\notag \\
		&\quad 
		+ \big(h^2 + \|\nabla_{\hat \Gamma_{h,*}^m}\hat e_h^m\|_{L^2(\hat \Gamma_{h,*}^m)}\big)
		\|\ek\|_{L^2(\Ghsm)}
	\Big] 
	\|\nabla_{\Ghsm} I_h\Tahm\evm\|_{L^2(\Ghsm)},
\end{align}	
where the last inequality follows from the Poincaré-type inequality in \eqref{poincare4}.

To derive a stability estimate for $ \|\nabla_{\Gamma_h^m} I_h \Tbhm e_v^m\|_{L^2(\Gamma_h^m)}^2 $ in terms of $ \int_{\Gamma_h^m} \nabla_{\Gamma^m_h} \evm \cdot \nabla_{\Gamma^m_h} I_h\Tahm\evm $ from \eqref{eq:XT}, we first note the following natural decomposition:
\begin{align}
    \int_{\Gamma_h^m} \nabla_{\Gamma_h^m} I_h \Tahm e_v^m \cdot \nabla_{\Gamma_h^m} I_h \Tahm e_v^m 
    &= \int_{\Gamma_h^m} \nabla_{\Gamma_h^m} e_v^m \cdot \nabla_{\Gamma_h^m} I_h \Tahm e_v^m \notag \\
    &\quad - \int_{\Gamma_h^m} \nabla_{\Gamma_h^m} I_h \Nbhm e_v^m \cdot \nabla_{\Gamma_h^m} I_h \Tbhm e_v^m.
\end{align}
It remains to estimate the cross term $ \int_{\Gamma_h^m} \nabla_{\Gamma_h^m} I_h \Nbhm e_v^m \cdot \nabla_{\Gamma_h^m} I_h \Tbhm e_v^m $. A sharper stability estimate for this $ H^1 $ bilinear form can be derived from the orthogonality of $ \Nbhm $ and $ \Tbhm $, which enables a bound involving one fewer gradient on $I_h \Nbhm e_v^m$.

\begin{lemma}\label{lemma:NT_stab}
    Under the induction assumptions in Section \ref{induc ass}, for sufficiently small $h$, the following inequality holds:
    \begin{align}\label{eq:NT}
        &\quad\,\,\Big| \int_{\Gamma_h^m} \nabla_{\Gamma_h^m} I_h \Nbhm\evm \cdot  \nabla_{\Gamma_h^m} I_h \Tbhm\evm \Big| \notag\\
        &\lesssim  \epsilon^{-1} \Big(1 +  h^{-4.2} \| \nabla_\Ghsm \ehm \|_{L^2(\Ghsm)}^2 \Big) \| I_h \Nbhm\evm \|_{L^2(\Ghsm)}^2 \notag\\
        &\quad + (\epsilon + h^{0.6} +\epsilon^{-1} h^{1.2}) \| \nabla_\Ghsm I_h \Tbhm\evm \|_{L^2(\Ghsm)}^2, \quad \forall\, \epsilon > 0.
    \end{align}
\end{lemma}
Lemma~\ref{lemma:NT_stab} can be derived using the fundamental theorem of calculus, geometric perturbation estimates, integration by parts, and the orthogonality of $ \bar N_h^m $ and $ \bar T_h^m $. The full proof is provided in Appendix~\ref{appendix_E}.

Therefore, by combining the estimates \eqref{eq:XT}--\eqref{eq:NT}, and subsequently applying Young's inequality, we obtain
	\begin{align}
		&\quad\,\,\int_{\Gamma_h^m}\nabla_{\Gamma^m_h}I_h\Tahm\evm\cdot\nabla_{\Gamma^m_h}I_h\Tahm\evm\notag\\
		&=\int_{\Gamma_h^m}\nabla_{\Gamma^m_h}\evm\cdot\nabla_{\Gamma^m_h}I_h\Tahm\evm-\int_{\Gamma_h^m} \nabla_{\Gamma_h^m} I_h \Nbhm\evm \cdot  \nabla_{\Gamma_h^m} I_h \Tbhm\evm \notag\\
		&\lesssim \Big[(1+\kappa_{*,l})h^k+\| \nabla_{\Ghsm} \ehm \|_{L^2(\Ghsm)}+(h^2+\| \nabla_{\Ghsm}\ehm\|_{L^2(\Ghsm)}) \|\ek\|_{L^2(\Ghsm)}\Big]\notag\\
		&\quad\,\, \cdot \|\nabla_{\Ghsm}I_h\Tahm\evm\|_{L^2(\Ghsm)}+\epsilon^{-1}(1 +h^{-4.2}\| \nabla_\Ghsm \ehm \|_{L^2(\Ghsm)}^2) \| I_h \Nbhm\evm \|_{L^2(\Ghsm)}^2  \notag\\
		&\quad+ (\epsilon +h^{0.6}+\epsilon^{-1}h^{1.2})\| \nabla_\Ghsm I_h \Tbhm \evm \|_{L^2(\Ghsm)}^2 \notag\\
		&\lesssim \epsilon^{-1}\Big[(1+\kappa_{*,l})^2h^{2k}+\| \nabla_{\Ghsm} \ehm \|^2_{L^2(\Ghsm)}+(h^2+\| \nabla_{\Ghsm}\ehm\|_{L^2(\Ghsm)})^2 \|\ek\|_{L^2(\Ghsm)}^2\Big]\notag\\
		&\quad+\epsilon^{-1}(1 + h^{-4.2}\| \nabla_\Ghsm \ehm \|_{L^2(\Ghsm)}^2) \| I_h \Nbhm\evm \|_{L^2(\Ghsm)}^2  \notag\\
		&\quad+ (\epsilon +h^{0.6}+\epsilon^{-1}h^{1.2})\| \nabla_\Ghsm I_h \Tbhm \evm \|_{L^2(\Ghsm)}^2.\label{eq:grad_P_X_diff}
	\end{align}
By the norm equivalences on \(\Ghsm\) and \(\Gamma_h^m\), and by choosing \(\epsilon\) sufficiently small and then \(h\) sufficiently small (so that $\epsilon^{-1}h^{1.2}$ is sufficiently small), the last term on the right-hand side of \eqref{eq:grad_P_X_diff} can be absorbed into the left-hand side. Thus, we obtain the following inequality:
\begin{align}
    \| \nabla_{\Ghsm} I_h \Tbhm \evm \|_{L^2(\Ghsm)} 
    &\lesssim (1 + \kappa_{*,l})h^k 
    + \| \nabla_{\Ghsm} \ehm \|_{L^2(\Ghsm)} \notag \\
    &\quad + \big(h^2 + \| \nabla_{\Ghsm} \ehm \|_{L^2(\Ghsm)}\big) \|\ek\|_{L^2(\Ghsm)} \notag \\
    &\quad + \big(1 + h^{-2.1} \| \nabla_\Ghsm \ehm \|_{L^2(\Ghsm)}\big) \| I_h \Nbhm \evm \|_{L^2(\Ghsm)}. \label{ITev}
\end{align}

Using the relation \(\evm = \ehvm + I_h g^m\) from \eqref{xm}, together with stability estimate \eqref{stability3}, we derive the following estimate:
\begin{align}
    & \quad\,\,\| I_h \Nbhm \evm \|_{L^2(\Ghsm)} \notag\\
    &\lesssim \| I_h \Nahsm \ehvm \|_{L^2(\Ghsm)} 
    + \| I_h ((\Nahm - \Nahsm) \ehvm) \|_{L^2(\Ghsm)} 
    + \| I_h (\Nahm (I_h g^m)) \|_{L^2(\Ghsm)} \notag \\
    &\lesssim \| I_h \Nahsm \ehvm \|_{L^2(\Ghsm)} 
    + \| \Nahm - \Nahsm \|_{L^\infty(\Ghsm)} \| \ehvm \|_{L^2(\Ghsm)} \notag\\
    & \quad+ \| \Nahm \|_{L^\infty(\Ghsm)} \| I_h g^m \|_{L^2(\Ghsm)} \notag \\
    &\lesssim \| I_h \Nahsm \ehvm \|_{L^2(\Ghsm)} 
    + h^{0.6}  \Big( \| \nabla_{\Ghsm} I_h \Tahsm \ehvm \|_{L^2(\Ghsm)} 
    +\| I_h \Nahsm \ehvm \|_{L^2(\Ghsm)} \Big) \notag \\
    &\quad + \tau 
    \quad \text{(applying \eqref{Linfty-W1infty-hat-em}, \eqref{nhmW-2b}, \eqref{eq:poincare2}, \eqref{W1infty-g}, and \eqref{eq:nsa1-infty})} \notag \\
    &\lesssim \| I_h \Nahsm \ehvm \|_{L^2(\Ghsm)} 
    + h^{0.6} \| \nabla_{\Ghsm} I_h \Tahsm \ehvm \|_{L^2(\Ghsm)} 
    + \tau. \label{INhmev}
\end{align}
By employing the relation \(\evm = \ehvm + I_h g^m\) from \eqref{xm} once again, and applying the triangle inequality together with the super-approximation estimates \eqref{super-bar-tangent-dif} and \eqref{super-bar-tangent}, the following estimate is obtained for sufficiently small \(h\):
	\begin{align}
		&\quad\,\, \| \nabla_{\Ghsm}I_h \Tahsm \ehvm \|_{L^2(\Ghsm)}\notag\\
		&\lesssim \| \nabla_{\Ghsm} (1-I_h) ((\Tahsm-\Tahm) \ehvm) \|_{L^2(\Ghsm)} + \| \nabla_{\Ghsm} ((\Tahsm-\Tahm) \ehvm) \|_{L^2(\Ghsm)}\notag\\
		& \quad\,\,+ \| \nabla_{\Ghsm}(1-I_h) (\Tahm (I_h g^m))\|_{L^2(\Ghsm)}+ \| \nabla_{\Ghsm}\Tahm I_h g^m \|_{L^2(\Ghsm)} +\| \nabla_{\Ghsm}I_h \Tahm \evm \|_{L^2(\Ghsm)}\notag\\
		& \lesssim h \|\bar n_{h,*}^m - \bar n_h^m\|_{W^{1,\infty}(\Ghsm)} \|\hat e_v^m\|_{H^1(\Ghsm)}+ \|\nabla_{\Ghsm} (\Tahsm-\Tahm)\|_{L^2(\Ghsm)} \|\hat e_v^m\|_{L^\infty(\Ghsm)}\notag\\
		& \quad\,\, + \|\Tahsm-\Tahm\|_{L^\infty(\Ghsm)} \|\hat e_v^m\|_{H^1(\Ghsm)}+h\|\bar n_h^m\|_{W^{1,\infty}(\Ghsm)}\|I_h g^m\|_{H^1(\Ghsm)} \notag\\
		& \quad\,\,+ \| \Tahm\|_{H^1(\Ghsm)} \|I_h g^m \|_{W^{1,\infty}(\Ghsm)}+\| \nabla_{\Ghsm}I_h \Tahm \evm \|_{L^2(\Ghsm)}\quad \mbox{(using \eqref{super-bar-tangent-dif}, \eqref{super-bar-tangent})}\notag\\
		& \lesssim C_\epsilon h^{-1-\epsilon}\| \nabla_\Ghsm \ehm \|_{L^2(\Ghsm)}\|\ehvm\|_{H^1(\Ghsm)} + \|I_h g^m\|_{W^{1,\infty}(\Ghsm)} +\| \nabla_{\Ghsm}I_h \Tahm \evm \|_{L^2(\Ghsm)}\notag\\
		& \,\,\quad \text{(using \eqref{eq:nsa1}, \eqref{nhmW-3}, inverse inequality and Sobolev embedding theorem \eqref{sobolem})}\notag\\
		% &\lesssim \| \nabla_{\Ghsm}I_h ((I_h(\Tahsm-\Tahm)) \ehvm) \|_{L^2(\Ghsm)}+\| \nabla_{\Ghsm}I_h\Tahm g^m\|_{L^2(\Ghsm)}+\| \nabla_{\Ghsm}I_h \Tahm \evm \|_{L^2(\Ghsm)}\notag\\
		% & \lesssim h^{-1}\|I_h(\Tahsm-\Tahm)\|_{L^2(\Ghsm)} \|\ehvm\|_{L^\infty(\Ghsm)}\quad\text{(using inverse inequality and stability estimate \eqref{stability2})}\notag\\ nhmW-2a nhmW-3
		&\lesssim h^{-2.1}\| \nabla_\Ghsm \ehm \|_{L^2(\Ghsm)}\|I_h \Nahm \ehvm\|_{L^2(\Ghsm)}+(1+h^{0.5})\|\nabla_{\Ghsm}I_h \Tahm \ehvm\|_{L^2(\Ghsm)} + \tau\notag\\
		& \quad\,\, \text{(choosing $\epsilon=0.1$, using \eqref{inverse-ineq}, \eqref{Linfty-W1infty-hat-em}, Poincaré inequality \eqref{poincare4}, and \eqref{W1infty-g})}\notag\\
		% &\quad \mbox{(using \eqref{nhmW-2a}, \eqref{nhmW-3} and the induction assumption \eqref{Linfty-W1infty-hat-em})}\notag\\
		& \lesssim h^{-2.1}\| \nabla_\Ghsm \ehm \|_{L^2(\Ghsm)}(\| I_h \Nahsm \ehvm \|_{L^2(\Ghsm)} 
		+ h^{0.6} \| \nabla_{\Ghsm} I_h \Tahsm \ehvm \|_{L^2(\Ghsm)} 
		+ \tau)\notag\\
		&\quad\,\,+ (1+\kappa_{*,l})h^k+\| \nabla_{\Ghsm} \ehm \|_{L^2(\Ghsm)}+(h^2+\| \nabla_{\Ghsm}\ehm\|_{L^2(\Ghsm)}) \|\ek\|_{L^2(\Ghsm)}\notag\\
		&\quad\,\,+ (1 + h^{-2.1}\| \nabla_\Ghsm \ehm \|_{L^2(\Ghsm)}) (\| I_h \Nahsm \ehvm \|_{L^2(\Ghsm)} 
		+ h^{0.6} \| \nabla_{\Ghsm} I_h \Tahsm \ehvm \|_{L^2(\Ghsm)} 
		+ \tau) \notag\\
            &\quad\,\, +\tau
            \qquad \text{{(using \eqref{ITev} and \eqref{INhmev})}}\notag\\
		& \lesssim (1+ h^{-2.1}\| \nabla_\Ghsm \ehm \|_{L^2(\Ghsm)})\| I_h \Nahsm \ehvm \|_{L^2(\Ghsm)}+h^{0.1}\| \nabla_{\Ghsm}I_h \Tahsm \ehvm \|_{L^2(\Ghsm)}\notag\\
		& \quad +\tau+ (1+\kappa_{*,l})h^k+\| \nabla_{\Ghsm} \ehm \|_{L^2(\Ghsm)} +(h^2+\| \nabla_{\Ghsm}\ehm\|_{L^2(\Ghsm)}) \|\ek\|_{L^2(\Ghsm)}\notag\\
		&\,\,\quad\text{(using \eqref{Linfty-W1infty-hat-em} and $\tau \le c h^k$ with $k\ge 3$)}.\label{ITehvm}
	\end{align}
For sufficiently small mesh size $h$, the second term on the right-hand side of \eqref{ITehvm} can be absorbed into its left-hand side. Consequently, we obtain the following estimate:
\begin{align}
	\| \nabla_{\Ghsm} I_h \Tahsm \ehvm \|_{L^2(\Ghsm)} 
	&\lesssim \tau + (1 + \kappa_{*,l})h^k 
	+ \| \nabla_{\Ghsm} \ehm \|_{L^2(\Ghsm)} \notag \\
	&\quad + \big(h^2 + \| \nabla_{\Ghsm} \ehm \|_{L^2(\Ghsm)}\big) \|\ek\|_{L^2(\Ghsm)} \notag \\
	&\quad + \big(1 +  h^{-2.1} \| \nabla_\Ghsm \ehm \|_{L^2(\Ghsm)}\big) \| I_h \Nahsm \ehvm \|_{L^2(\Ghsm)}. \label{I_hTsmL}
\end{align}
	 
By decomposing $\ehvm$ into the normal and tangential components, and applying the Poincaré inequality \eqref{eq:poincare2}, as well as using the estimate in \eqref{I_hTsmL}, we obtain the following result:
\begin{align}
    \| \ehvm \|_{L^2(\Ghsm)} & \lesssim \| I_h \Nahsm \ehvm \|_{L^2(\Ghsm)} + \| \nabla_{\Ghsm} I_h \Tahsm \ehvm \|_{L^2(\Ghsm)} \notag\\
    &\lesssim \tau 
    + (1 + \kappa_{*,l})h^k 
    + \| \nabla_{\Ghsm} \ehm \|_{L^2(\Ghsm)} \notag \\
    &\quad + \big(h^2 + \| \nabla_{\Ghsm} \ehm \|_{L^2(\Ghsm)}\big) \|\ek\|_{L^2(\Ghsm)} \notag \\
    &\quad + \big(1 + h^{-2.1} \| \nabla_\Ghsm \ehm \|_{L^2(\Ghsm)}\big) \| I_h \Nahsm \ehvm \|_{L^2(\Ghsm)}. \label{evmL22}
\end{align}

Furthermore, by applying inverse inequality, the \(H^1\)-semi norm of \(\ehvm\) can be controlled as follows:
\begin{align}
    \| \nabla_{\Ghsm} \ehvm \|_{L^2(\Ghsm)} 
    &\lesssim h^{-1} \| I_h \Nahsm \ehvm \|_{L^2(\Ghsm)} 
    + \| \nabla_{\Ghsm} I_h \Tahsm \ehvm \|_{L^2(\Ghsm)} \notag \\
    &\lesssim \tau 
    + (1 + \kappa_{*,l})h^k 
    + \| \nabla_{\Ghsm} \ehm \|_{L^2(\Ghsm)} + h^{-1} \| I_h \Nahsm \ehvm \|_{L^2(\Ghsm)}\notag \\
    &\quad + \big(h^2 + \| \nabla_{\Ghsm} \ehm \|_{L^2(\Ghsm)}\big) \|\ek\|_{L^2(\Ghsm)} ,\label{evmH1}
\end{align}
where inequality \eqref{I_hTsmL} and the mathematical induction assumption \eqref{Linfty-W1infty-hat-em} are used.

	%.......................................................
\subsection{Estimates of the Velocity Error}\label{sec:velocity-estimates}
Since $ \kappa $ is an auxiliary variable, we need to bound $ \ek $ in terms of $ \ehvm $ in order to eliminate it from inequalities \eqref{I_hTsmL}--\eqref{evmH1}. An estimate for $\|e_{\kappa}^m\|_{L^2(\Ghsm)}$ is obtained by testing the error equation \eqref{ekk} with \(\chi_h := I_h(\ek\,\nahm)\), and using the inequalities \eqref{dk}, \eqref{estR1}, and \eqref{estR2}. The result is presented in the following lemma and a detailed proof can be found in Appendix \ref{appendix_F}.

\begin{lemma}\label{e_kappa_l2}
    Under the induction assumptions in Section \ref{induc ass}, for sufficiently small $h$, the following inequality holds:
    \begin{align}
    \|\ek\|_{L^2(\Ghsm)} 
    &\lesssim (1 + \kappa_{*,l})h^{k-1} + h^{-1} \| \nabla_{\Ghsm} \ehm \|_{L^2(\Ghsm)} + h^{-2}  \| I_h \Nahsm \ehvm \|_{L^2(\Ghsm)}. \label{estekL2}
    \end{align}
\end{lemma}

By substituting \eqref{estekL2} into \eqref{I_hTsmL}--\eqref{evmH1}, we obtain the following refined estimates:
\begin{align}
    \| \nabla_{\Ghsm} I_h \Tahsm \ehvm \|_{L^2(\Ghsm)} 
    &\lesssim \tau + (1 + \kappa_{*,l})h^k + \| \nabla_{\Ghsm} \ehm \|_{L^2(\Ghsm)} \notag \\
    &\quad + \big(1 +h^{-2.1} \| \nabla_\Ghsm \ehm \|_{L^2(\Ghsm)}\big) \| I_h \Nahsm \ehvm \|_{L^2(\Ghsm)}, \label{I_hTsm} \\
    \| \ehvm \|_{L^2(\Ghsm)} 
    &\lesssim \tau + (1 + \kappa_{*,l})h^k + \| \nabla_{\Ghsm} \ehm \|_{L^2(\Ghsm)} \notag \\
    &\quad + \big(1 + h^{-2.1} \| \nabla_\Ghsm \ehm \|_{L^2(\Ghsm)}\big) \| I_h \Nahsm \ehvm \|_{L^2(\Ghsm)}, \label{evmL2} \\
    \| \nabla_{\hat \Gamma_{h,*}^m} \ehvm \|_{L^2(\Ghsm)} 
    &\lesssim \tau + (1 + \kappa_{*,l})h^k + \| \nabla_{\Ghsm} \ehm \|_{L^2(\Ghsm)}  + h^{-1}  \| I_h \Nahsm \ehvm \|_{L^2(\Ghsm)},\label{evmH2}
\end{align}
where \eqref{cond1}, \eqref{Linfty-W1infty-hat-em} and $\tau \le ch^k$ with $k\ge 3$ have been used in the above estimates.

A careful examination of the right-hand sides of \eqref{I_hTsm}--\eqref{evmH2} reveals that an estimate for $ \| I_h \Nahsm \ehvm \|_{L^2(\Ghsm)} $ is essential for fully bounding these terms. This estimate is established in the following lemma.

\begin{lemma}\label{normal-control}
    Under the induction assumptions in Section \ref{induc ass}, for sufficiently small $h$, the following inequality holds:
	\begin{align}\label{INe}
    \|I_h\Nahsm\ehvm\|_{L^2(\Ghsm)}
    &\lesssim
    \|\ehvm \cdot \nahsm\|_{L^2(\Ghsm)}\notag\\
    &\lesssim (1 + \kappa_{*,l})h^{k-1}
    +h^{-1}\|\nabla_{\Ghsm}\ehm\|_{L^2(\Ghsm)}.
\end{align}
\end{lemma}

\begin{proof}
By applying stability estimate \eqref{stability4} and inequality \eqref{average-1},  the following inequality holds:
\begin{align}\label{stability-normal}
	\| I_h \Nahsm \ehvm \|_{L^2(\Ghsm)} &= \Big\| I_h \Big(\frac{\bar n_{h,*}^m}{|\bar n_{h,*}^m |^2} (\bar n_{h,*}^m\cdot\ehvm)\Big) \Big\|_{L^2(\Ghsm)}\notag\\
	&\lesssim \| \bar n_{h,*}^m\cdot\ehvm \|_{L^2(\Ghsm)}\lesssim \| \Nahsm \ehvm \|_{L^2(\Ghsm)},
\end{align}
which leads to the following estimate for \(\|\ehvm\|_{L^2(\Ghsm)}\) by utilizing \eqref{evmL2}:
\begin{align}
    \| \ehvm \|_{L^2(\Ghsm)} 
    &\lesssim \tau + (1 + \kappa_{*,l})h^k + \| \nabla_{\Ghsm} \ehm \|_{L^2(\Ghsm)} \notag\\
	& \quad+ \big(1 + h^{-2.1} \| \nabla_\Ghsm \ehm \|_{L^2(\Ghsm)}\big) \|  \ehvm \cdot\bar n_{h,*}^m \|_{L^2(\Ghsm)}.\label{evmL2-2}
\end{align}
By utilizing the mesh size assumption \eqref{cond1}, the mathematical induction assumption \eqref{Linfty-W1infty-hat-em} and the assumption that $\tau \le ch^k$ with $k\ge 3$, the above inequality \eqref{evmL2-2} further implies that
\begin{align}
    \| \ehvm \|_{L^2(\Ghsm)} 
    &\lesssim h^{1.6} + \big(1 + h^{-2.1} \| \nabla_\Ghsm \ehm \|_{L^2(\Ghsm)}\big) \|  \ehvm \cdot\bar n_{h,*}^m \|_{L^2(\Ghsm)}.\label{evmL2-3}
\end{align}
By testing the error equation \eqref{err} with $\phi_h := I_h(\ehvm \cdot \nahsm)$, and using the definition of $\hat e^m_v = (\eM - \ehm)/\tau - I_h T_*^m v^m$ in \eqref{evm}, we derive the following relation:
	\begin{align}
		&\quad\,\, \int_{\Ghsm}(\ehvm\cdot\nahsm)(\ehvm\cdot\nahsm)\notag\\
		&=\int_{\Ghsm}(\ehvm\cdot\nahsm) [(1-I_h)(\ehvm\cdot\nahsm)]+\int_{\Ghsm}(\ehvm\cdot\nahsm) I_h(\ehvm\cdot\nahsm)\notag\\
		&=\int_{\Ghsm}(\ehvm\cdot\nahsm )[(1-I_h)(\ehvm\cdot\nahsm)]-\int_{\Ghsm}I_h \Tsm v^m\cdot\nahsm I_h(\ehvm\cdot\nahsm)\notag\\
		& \quad\,\,+\int_{\Ghsm}\frac{\eM - \ehm}{\tau} \cdot\nahsm I_h(\ehvm\cdot\nahsm)
        \quad\big(\mbox{we have used $\hat e^m_v = \dfrac{\eM - \ehm}{\tau} - I_h T_*^m v^m$}\big)\notag\\
		&=\int_{\Ghsm}\ehvm\cdot\nahsm [(1-I_h)(\ehvm\cdot\nahsm)]-\int_{\Ghsm}(I_h \Tsm v^m-\Tahsm v^m)\cdot\nahsm I_h(\ehvm\cdot\nahsm)\notag\\
		&\quad \,\,\mbox{(using the orthogonality between $\Tahsm$ and $\nahsm$)}\notag\\
		&\quad\,\,-A_{h, *}^N(e_h^{m+1},\nahsm I_h(\ehvm\cdot\nahsm)) -A_{h, *}^T(e_h^{m+1} - \hat{e}_h^m,\nahsm I_h(\ehvm\cdot\nhsm))  \notag\\
		&\quad\,\,- B^m(\hat{e}_h^m,\nahsm I_h(\ehvm\cdot\nahsm))-J_1^m( I_h(\ehvm\cdot\nahsm))-J_2^m( I_h(\ehvm\cdot\nahsm))\notag\\
		& \quad\,\,-K^m(\nahsm I_h(\ehvm\cdot\nahsm))-d_v^m( I_h(\ehvm\cdot\nahsm))\label{estev},
		% &\hspace{200pt}\mbox{(error equation \eqref{err})}.\label{estev}
	\end{align}
	where the error equation \eqref{err} is applied in the last equality.

The first term on the right-hand side of \eqref{estev} can be estimated as follows, by using the super approximation property \eqref{super2} in Lemma \ref{super}, along with $\|\nahsm\|_{W^{1,\infty}(\Ghsm)}\lesssim 1$ as shown in \eqref{nhmW-1}:
\begin{align}
    & \quad \,\,\Big|\int_{\Ghsm} \ehvm \cdot \nahsm \, (1 - I_h)(\ehvm \cdot \nahsm)\Big|\notag\\ 
    &\lesssim \|\ehvm \cdot \nahsm\|_{L^2(\Ghsm)} \|(1 - I_h)(\ehvm \cdot \nahsm)\|_{L^2(\Ghsm)} \notag\\
	&\lesssim \|\ehvm \cdot \nahsm\|_{L^2(\Ghsm)} h \|\nahsm\|_{W^{1,\infty}(\Ghsm)}  \|\ehvm\|_{L^2(\Ghsm)} \notag \\
    % &\lesssim h \|\ehvm \cdot \nahsm\|_{L^2(\Ghsm)} \|\ehvm\|_{L^2(\Ghsm)}\notag\\
	& \lesssim h\big[\tau + (1+ \kappa_{*,l})h^k + \|\nabla_{\hat \Gamma_{h,*}^m} \hat e_h^m\|_{L^2(\hat \Gamma_{h,*}^m)}\big]\|\ehvm \cdot \nahsm\|_{L^2(\Ghsm)}\notag \\
	& \quad + h(1+h^{-2.1}\|\nabla_{\hat \Gamma_{h,*}^m} \hat e_h^m\|_{L^2(\Ghsm)})\|\ehvm \cdot \nahsm\|_{L^2(\Ghsm)}^2 \quad \text{(here \eqref{evmL2-2} is used)}\notag\\
	& \lesssim h\big[\tau + (1+ \kappa_{*,l})h^k + \|\nabla_{\hat \Gamma_{h,*}^m} \hat e_h^m\|_{L^2(\hat \Gamma_{h,*}^m)}\big]\|\ehvm \cdot \nahsm\|_{L^2(\Ghsm)} + h^{0.5}\|\ehvm \cdot \nahsm\|_{L^2(\Ghsm)}^2,\label{super-ev-super}
\end{align}
where the mathematical induction assumption \eqref{Linfty-W1infty-hat-em} is used in the last inequality.

The second term on the right-hand side of \eqref{estev} can be estimated by utilizing the Lagrange interpolation error estimate \eqref{Ihf}, stability estimate \eqref{stability2}, and \eqref{eq:nsa2} from Lemma \ref{Lem:normalvector}, as follows:
\begin{align}\label{norm-difference}
    & \quad \,\,\Big|-\int_{\Ghsm} (I_h \Tsm v^m - \Tahsm v^m) \cdot \nahsm \, I_h(\ehvm \cdot \nahsm)\Big|\notag\\ 
    &\lesssim \|I_h \Tsm v^m - \Tsm v^m + \Tsm v^m - \Tahsm v^m\|_{L^2(\Ghsm)} \|\ehvm \cdot \nahsm\|_{L^2(\Ghsm)} \quad \text{(using \eqref{stability2})}\notag \\
    &\lesssim (1 + \kappa_{*,l}) h^k \|\ehvm \cdot \nahsm\|_{L^2(\Ghsm)}\quad \text{(using \eqref{Ihf} and \eqref{eq:nsa2})}.
\end{align}

The terms $ d_v^m(I_h(\ehvm \cdot \nahsm)) $, $ J_1^m(I_h(\ehvm \cdot \nahsm)) $, and $ J_2^m(I_h(\ehvm \cdot \nahsm)) $ are estimated using \eqref{dv}, \eqref{estJ1}, and \eqref{estJ2}, respectively, as follows:
% These estimates are derived by applying the super-approximation properties from Lemma \ref{super} in conjunction with stability result provided in Lemma \ref{stability}:
	\begin{align}
		&\quad\,\,|d_v^m( I_h(\ehvm\cdot\nahsm))|\notag\\
		&\lesssim (1+\kappa_{*,l}h^{k-1})\tau \|I_h(\ehvm\cdot\nahsm)\|_{L^2(\Ghsm)}+(1+\kappa_{*,l})h^k\|I_h(\ehvm\cdot\nahsm)\|_{H^1(\Ghsm)}\notag\\
		&\lesssim (\tau+(1+\kappa_{*,l})h^{k-1})\|\ehvm\cdot\nahsm\|_{L^2(\Ghsm)}\notag\\
		&\,\, \quad \mbox{(using \eqref{stability2}, \eqref{cond1} and inverse inequality)},\label{d-d}\\
		&\text{and} \notag\\
		&\quad \,\,|J_1^m( I_h(\ehvm\cdot\nahsm))|\notag\\
		&\lesssim (h^{-1}\| \ehvm\|_{L^2(\Ghsm)}+1)\| \nabla_{\Ghsm} \ehm \|_{L^2(\Ghsm)} \| I_h(\ehvm\cdot\nahsm) \|_{L^2(\Ghsm)}
        \quad\mbox{(here \eqref{estJ1} is used)} \notag\\
		&\lesssim \| \nabla_{\Ghsm} \ehm \|_{L^2(\Ghsm)}\|\ehvm\cdot\nahsm\|_{L^2(\Ghsm)}+h^{-1}(1 + h^{-2.1}\| \nabla_\Ghsm \ehm \|_{L^2(\Ghsm)}) \notag\\
		& \quad \cdot\| \nabla_{\Ghsm}\ehm\|_{L^2(\Ghsm)}\|\ehvm\cdot\nahsm\|_{L^2(\Ghsm)}^2\notag\\
		& \,\,\quad\mbox{(using \eqref{cond1}, \eqref{Linfty-W1infty-hat-em}, \eqref{stability2}, \eqref{evmL2-3} and $\tau \le ch^k$ with $k\ge 3$)}\notag\\
		&\lesssim \| \nabla_{\Ghsm} \ehm \|_{L^2(\Ghsm)}\|\ehvm\cdot\nahsm\|_{L^2(\Ghsm)} 
		+h^{0.1}\|\ehvm\cdot\nahsm\|_{L^2(\Ghsm)}^2\quad\mbox{(using \eqref{Linfty-W1infty-hat-em}),}\label{J1-J1}\\
		&\text{and} \notag\\
		&\quad\,\,|J_2^m( I_h(\ehvm\cdot\nahsm))|
        \quad\mbox{(use \eqref{estJ2} here)}\notag\\
		&\lesssim \| \nabla_{\Ghsm} \ehm \|_{L^2(\Ghsm)} \| I_h(\ehvm\cdot\nahsm)\|_{L^2(\Ghsm)} + h^{0.5}\| \nabla_{\Ghsm} \ehm \|_{L^2(\Ghsm)} \| I_h(\ehvm\cdot\nahsm)\|_{H^1(\Ghsm)} \notag\\
		&\quad+ (h^{-2.1}\tau+h^{-2.1}\tau\| \ehvm\|_{L^2(\Ghsm)})\| \nabla_{\Ghsm} \ehm \|_{L^2(\Ghsm)} \| I_h(\ehvm\cdot\nahsm)\|_{H^1(\Ghsm)}\notag\\	
		&\lesssim h^{-0.5}\| \nabla_{\Ghsm} \ehm \|_{L^2(\Ghsm)}\|I_h(\ehvm\cdot\nahsm)\|_{L^2(\Ghsm)}+h^{-3.1}\tau(1 + h^{-2.1}\| \nabla_\Ghsm \ehm \|_{L^2(\Ghsm)})\notag\\
		& \quad\,\,\cdot\|\ehvm\cdot\nahsm\|_{L^2(\Ghsm)}\| \nabla_{\Ghsm}\ehm\|_{L^2(\Ghsm)} \|I_h(\ehvm\cdot\nahsm)\|_{L^2(\Ghsm)}\notag\\
		&\quad\,\,\mbox{(using inverse inequality, \eqref{cond1}, \eqref{Linfty-W1infty-hat-em}, $\tau\le  c h^k$ with $k\ge 3$, \eqref{evmL2-3})}\notag\\
		&\lesssim h^{-0.5}\| \nabla_{\Ghsm} \ehm \|_{L^2(\Ghsm)}\|\ehvm\cdot\nahsm\|_{L^2(\Ghsm)}+h\|\ehvm\cdot\nahsm\|_{L^2(\Ghsm)}^2\notag\\
		& \,\,\quad\mbox{(using \eqref{Linfty-W1infty-hat-em}, $\tau\le  c h^k$ with $k\ge 3$ and \eqref{stability2})}\label{J2-J2}.
	\end{align}

	The term $K^m\bigl(\nahsm I_h(\ehvm \cdot \nahsm)\bigr)$ on the right-hand side of \eqref{estev} can be estimated by decomposing it as in \eqref{k_m}:
\begin{align}
    K^m\bigl(\nahsm I_h(\ehvm \cdot \nahsm)\bigr) = \sum_{j=1}^5 K_j^m\bigl(\nahsm I_h(\ehvm \cdot \nahsm)\bigr),
\end{align}
where the definitions of $K_j^m(\cdot)$ for $j=1,\ldots,5$ are given in \eqref{k_m}.

By applying Lemma \ref{lemma:A_iden} and observing that
\(
\partial_\theta^\bullet\bigl(\nahsm I_h(\ehvm \cdot \nahsm)\bigr) = 0
\)
on the intermediate surface $\hat\Gamma_{h,\theta}^m$, due to the product rule and the fact that the material derivative of a finite element function with fixed nodal values vanishes, the following estimate is obtained:
\begin{align}
	\label{K_123^m_inter}
	\sum_{j=1}^3 |K_j^m(\nahsm I_h(\ehvm\cdot\nahsm))|&\lesssim \|\nabla_{\Ghsm} \hat e_h^m\|_{L^\infty(\Ghsm)} ( \|\nabla_{\Ghsm} \hat e_h^m\|_{L^2(\Ghsm)}\notag\\
	&\quad\,\, +  \|\nabla_{\Ghsm} e_h^{m+1}\|_{L^2(\Ghsm)})\|\nabla_{\Ghsm}  (\nahsm I_h(\ehvm\cdot\nahsm))\|_{L^2(\Ghsm)}
\end{align}
Furthermore, using the relation $e_h^{m+1} = \tau (\hat e_v^m + I_h \Tsm v^m) + \hat e_h^m$ as shown in \eqref{evm}, the following estimate holds:
% \begin{subequations}\label{emev}
	\begin{align}
		\|\nabla_{\Ghsm}\eM\|_{L^2(\Ghsm)}&=\|\nabla_{\Ghsm}[\tau(\ehvm+I_h\Tsm v^m)+\ehm]\|_{L^2(\Ghsm)}\notag\\
		&\lesssim h^{-1}\tau \|\ehvm\|_{L^2(\Ghsm)}+\tau+\|\nabla_{\Ghsm}\ehm\|_{L^2(\Ghsm)},\label{emev-a}
		% \|\nabla_{\Ghsm}(\eM-\ehm)\|_{L^2(\Ghsm)}&=\|\nabla_{\Ghsm}[\tau(\ehvm+I_h\Tsm v^m)]\|_{L^2(\Ghsm)}\notag\\
		% &\lesssim h^{-1}\tau \|\ehvm\|_{L^2(\Ghsm)}+\tau\label{emev-b}.
	\end{align}
% \end{subequations}
where the inverse inequality is used in deriving the above inequality. Subsequently, by substituting the above estimate \eqref{emev-a} into \eqref{K_123^m_inter}, the following estimate is obtained:
% Then the term 
% \(
% K^m\bigl(\nahsm I_h(\ehvm \cdot \nahsm)\bigr) = \sum_{j=1}^5 K_j^m\bigl(\nahsm I_h(\ehvm \cdot \nahsm)\bigr)
% \)
% defined in \eqref{k_m} can be estimated similarly as in \cite[(5.23)]{bai2024new}. Indeed, since \(\nahsm I_h(\ehvm \cdot \nahsm)\) is a polynomial of higher degree, inverse inequality can be applied to it, and moreover 
% \(
% \partial_\theta^\bullet\bigl(\nahsm I_h(\ehvm \cdot \nahsm)\bigr)=0
% \)
% on the intermediate surface $\hat\Gamma_{h,\theta}^m$ by the product rule and the fact that the material derivative of a finite element function with fixed nodal values vanishes. Hence the identity in Lemma \ref{lemma:A_iden}, with \(z_h\) replaced by the higher–degree polynomial \(\nahsm I_h(\ehvm \cdot \nahsm)\), remains valid. Applying the fundamental theorem of calculus (in the same fashion as in Lemma \ref{lemma:A_iden}) then yields the following estimates:
\begin{align}\label{K_123^m}
	& \quad \,\,\sum_{j=1}^3 |K_j^m(\nahsm I_h(\ehvm\cdot\nahsm))|\notag\\
	%  &\lesssim \|\nabla_{\Ghsm} \hat e_h^m\|_{L^\infty(\Ghsm)} ( \|\nabla_{\Ghsm} \hat e_h^m\|_{L^2(\Ghsm)} +  \|\nabla_{\Ghsm} e_h^{m+1}\|_{L^2(\Ghsm)})\|\nabla_{\Ghsm}  (\nahsm I_h(\ehvm\cdot\nahsm))\|_{L^2(\Ghsm)}\notag\\
	 & \lesssim \|\nabla_{\Ghsm} \hat e_h^m\|_{L^\infty(\Ghsm)} ( \|\nabla_{\Ghsm} \hat e_h^m\|_{L^2(\Ghsm)} + h^{-1}\tau \|\ehvm\|_{L^2(\Ghsm)}+\tau )h^{-1}\|\ehvm\cdot\nahsm\|_{L^2(\Ghsm)}\notag\\
	 & \quad\,\, \text{(using \eqref{stability2}, \eqref{nhmW-1}, \eqref{emev-a} and inverse inequality)}\notag\\
	 & \lesssim (h^{-0.4}+1)\| \nabla_{\Ghsm} \ehm \|_{L^2(\Ghsm)}\|\ehvm\cdot\nahsm\|_{L^2(\Ghsm)}+h\|\ehvm\cdot\nahsm\|_{L^2(\Ghsm)}^2,
\end{align}
where the last inequality uses inverse inequality, the estimate \eqref{evmL2-3}, the mesh size assumption \eqref{cond1}, the mathematical induction assumption \eqref{Linfty-W1infty-hat-em} and the assumption $\tau \le ch^k$ with $k\ge 3$.

By using the geometric perturbation estimates in Lemma \ref{geometric-perturbation}, the following estimate can be obtained:
\begin{align}\label{K_4^m}
	|K_4^m(\nahsm I_h(\ehvm\cdot\nahsm))| &\lesssim (1+ k_{*,l})h^k \|\nabla_{\Ghsm} \ehm\|_{L^\infty(\Ghsm)}\|\nabla_{\Ghsm}(\nahsm I_h(\ehvm\cdot\nahsm)) \|_{L^2(\Ghsm)}\notag\\
	& \lesssim (1+ k_{*,l})h^{k-2} \|\nabla_{\Ghsm} \ehm\|_{L^2(\Ghsm)}\|\ehvm\cdot\nahsm \|_{L^2(\Ghsm)}\notag\\
	& \lesssim h^{0.6}\| \nabla_{\Ghsm} \ehm \|_{L^2(\Ghsm)}\|\ehvm\cdot\nahsm\|_{L^2(\Ghsm)} \quad\text{(using \eqref{cond1})},
\end{align}
where inverse inequalities, stability estimate \eqref{stability2} and inequality \eqref{nhmW-1} are used in the second-to-last inequality.

Recall that $\hat{X}_{h, \theta}^{m} = (1 - \theta)\hat X_{h, *}^{m} + \theta X_{h}^{m}$ and $X_{h, \theta}^{m+1} = (1 - \theta)X_{h, *}^{m+1} + \theta X_{h}^{m+1}$, both defined in terms of nodal vectors. It then follows that
\begin{align}\label{Xexrelat-new}
	X_{h,\theta}^{m+1} - \hat X_{h,\theta}^m &= \theta (e_h^{m+1} - \hat e_h^m)  + X_{h,*}^{m+1}-\hat X_{h,* }^m\notag\\
	& = \theta (e_h^{m+1} - \hat e_h^m)  + \tau I_h ( -H^m n^m + g^m)\notag\\
	& = \theta (\tau \hat e_v^m +\tau I_h T_*^m v^m)  + \tau I_h ( -H^m n^m + g^m),
\end{align}
where the last equality is obtained by applying \eqref{evm}, \eqref{X-id-Hn0} and \eqref{X-id-Hn}. Consequently, by employing the relation \eqref{Xexrelat-new}, we obtain
\begin{align}\label{K_5^m}
	&\quad\,\,|K_5^m(\nahsm I_h(\ehvm\cdot\nahsm))| \notag\\
	& =\big|\int_0^1 \int_{ \hat\Gamma_{h,\theta}^m}  \nabla_{\hat\Gamma_{h,\theta}^m} (X_{h, \theta}^{m+1}-\hat X_{h, \theta}^{m}) \cdot D_{\hat\Gamma_{h,\theta}^m} \hat e_{h}^{m} \nabla_{\hat\Gamma_{h,\theta}^m} (\nahsm I_h(\ehvm\cdot\nahsm))\d\theta\big| \notag\\
	& \lesssim h^{-1}( \tau \|\hat e_v^m\|_{L^2(\Ghsm)} + \tau ) \|\nabla_{\Ghsm} \hat e_h^m\|_{L^\infty(\Ghsm)} \|\nabla_{\Ghsm}  (\nahsm I_h(\ehvm\cdot\nahsm))\|_{L^2(\Ghsm)}\notag\\
	& \quad + \tau \|I_h (-H^m n^m +g^m)\|_{W^{1,\infty}(\Ghsm)}\|\nabla_{\Ghsm} \hat e_h^m\|_{L^2(\Ghsm)} \|\nabla_{\Ghsm}  (\nahsm I_h(\ehvm\cdot\nahsm))\|_{L^2(\Ghsm)}\notag\\
	& \lesssim\| \nabla_{\Ghsm} \ehm \|_{L^2(\Ghsm)}\|\ehvm\cdot\nahsm\|_{L^2(\Ghsm)}+h^{1.1}\|\ehvm\cdot\nahsm\|_{L^2(\Ghsm)}^2,
\end{align}
where the last inequality uses inverse inequality, \eqref{cond1}, \eqref{Linfty-W1infty-hat-em}, \eqref{stability2}, \eqref{evmL2-3} and $\tau \le ch^k$ with $k\ge 3$. Therefore, by collecting the above estimates, the following result can be obtained by choosing sufficiently small $h$:
\begin{align}
\Big|K^m(\nahsm I_h(\ehvm\cdot\nahsm))\Big|
&\lesssim h^{-0.4}\| \nabla_{\Ghsm} \ehm \|_{L^2(\Ghsm)}\|\ehvm\cdot\nahsm\|_{L^2(\Ghsm)}+h \|\ehvm\cdot\nahsm\|_{L^2(\Ghsm)}^2\label{k-k}.
\end{align}
    
By the definitions of the bilinear forms $A_{h, *}^N(\cdot,\cdot)$, $A_{h, *}^T(\cdot,\cdot)$ and $B^m(\cdot,\cdot)$ in \eqref{def-As-AGs}--\eqref{def-Bm}, we have 
	\begin{align}\label{eq:AB_est}
		|A_{h, *}^N(u,v) | + |A_{h, *}^T(u,v) | + |B^m(u,v) | 
		&\lesssim \| \nabla_\Ghsm u \|_{L^2(\Ghsm)} \| v \|_{H^1(\Ghsm)} ,
	\end{align}
	for any $u, v \in H^1(\Ghsm)$.
	By applying inequality \eqref{eq:AB_est}, we derive the following estimate:
	\begin{align}\label{super-ev-super-bilinear}
		&\quad \,\,|A_{h, *}^N(e_h^{m+1},\nahsm I_h(\ehvm\cdot\nahsm))|+|A_{h, *}^T(e_h^{m+1} - \hat{e}_h^m,\nahsm I_h(\ehvm\cdot\nahsm))|\notag\\
		& \quad\,\,+|B^m(\hat{e}_h^m,\nahsm I_h(\ehvm\cdot\nahsm))|\notag\\
		&\lesssim \Big(\|\nabla_{\Ghsm}\eM\|_{L^2(\Ghsm)}+\|\nabla_{\Ghsm}\ehm\|_{L^2(\Ghsm)}\Big)\| \nahsm I_h(\ehvm\cdot\nahsm) \|_{H^1(\Ghsm)}\notag\\
		&\lesssim h^{-1}\Big(h^{-1}\tau \|\ehvm\|_{L^2(\Ghsm)}+\tau+\|\nabla_{\Ghsm}\ehm\|_{L^2(\Ghsm)}\Big)\|\ehvm\cdot\nahsm\|_{L^2(\Ghsm)}\notag\\
		& \,\,\quad\mbox{(using \eqref{emev-a}, \eqref{stability2}, \eqref{nhmW-1} and inverse inequality)}\notag\\
		&\lesssim \big[h^{-1}\tau+(1+\kappa_{*,l})h^{k+1}+h^{-1}\| \nabla_{\Ghsm} \ehm \|_{L^2(\Ghsm)}\big]\|\ehvm\cdot\nahsm\|_{L^2(\Ghsm)}\notag\\
		%	&\quad+\tau h^{-2}(h^2+\| \nabla_{\Ghsm}\ehm\|_{L^2(\Ghsm)}) \|\ek\|_{L^2(\Ghsm)}\|\ehvm\cdot\nhsm\|_{L^2(\Ghsm)}\notag\\
		&\quad+\tau h^{-2}(1 + h^{-2.1}\| \nabla_\Ghsm \ehm \|_{L^2(\Ghsm)})\|\ehvm\cdot\nahsm\|_{L^2(\Ghsm)}^2\notag\\
		&\quad\,\,\mbox{($\tau\le c h^k$ with $k\ge 3$ and \eqref{evmL2-2} are used)}\notag\\
		&\lesssim \big[h^{-1}\tau+(1+\kappa_{*,l})h^{k+1}+h^{-1}\| \nabla_{\Ghsm} \ehm \|_{L^2(\Ghsm)}\big]\|\ehvm\cdot\nahsm\|_{L^2(\Ghsm)}\notag\\
		%	&\quad+\tau h^{-2}(h^2+\| \nabla_{\Ghsm}\ehm\|_{L^2(\Ghsm)}) \|\ek\|_{L^2(\Ghsm)}\|\ehvm\cdot\nhsm\|_{L^2(\Ghsm)}\notag\\
		&\quad+h^{0.5}\|\ehvm\cdot\nahsm\|_{L^2(\Ghsm)}^2 \quad \mbox{(using $\tau\le c h^k$ with $k\ge 3$ and \eqref{Linfty-W1infty-hat-em})}.
	\end{align}
	% Then, substituting these estimates from \eqref{super-ev-super}, \eqref{norm-difference}, \eqref{d-d}, \eqref{J1-J1}, \eqref{J2-J2}, \eqref{k-k} and \eqref{super-ev-super-bilinear} into \eqref{estev}, we obtain the following result under $\tau\lesssim h^k$ and $k\ge 3$:
	
Now, substituting the estimates from \eqref{super-ev-super}--\eqref{J2-J2}, \eqref{k-k}, and \eqref{super-ev-super-bilinear} into \eqref{estev}, we obtain the following result under the condition $ \tau \le c h^k $ with $ k \geq 3 $: 
\begin{align}
    \|\ehvm \cdot \nahsm\|_{L^2(\Ghsm)}^2 
    &\lesssim \big[(1 + \kappa_{*,l})h^{k-1} + h^{-1} \| \nabla_{\Ghsm} \ehm \|_{L^2(\Ghsm)}\big] \|\ehvm \cdot \nahsm\|_{L^2(\Ghsm)} \notag \\
    &\quad + h^{0.1} \|\ehvm \cdot \nahsm\|_{L^2(\Ghsm)}^2,
\end{align}
where the last term on the right-hand side can be absorbed into the left-hand side for sufficiently small $ h $. This yields the estimate
\begin{align}\label{ev-normal}
    \|\ehvm \cdot \nahsm\|_{L^2(\Ghsm)}
    \lesssim (1 + \kappa_{*,l})h^{k-1}
    +h^{-1}\|\nabla_{\Ghsm}\ehm\|_{L^2(\Ghsm)}.
\end{align}
Furthermore, invoking inequality \eqref{stability-normal}, we obtain
\begin{align}\label{INe-2}
    \|I_h\Nahsm\ehvm\|_{L^2(\Ghsm)}
    &\lesssim
    \|\ehvm \cdot \nahsm\|_{L^2(\Ghsm)}\notag\\
	&
    \lesssim (1 + \kappa_{*,l})h^{k-1}
    +h^{-1}\|\nabla_{\Ghsm}\ehm\|_{L^2(\Ghsm)},
\end{align}
which completes the proof of Lemma~\ref{normal-control}.
\end{proof}

	By substituting the estimate \eqref{INe} from Lemma \ref{normal-control} into \eqref{I_hTsm}, we derive the following estimate:
	\begin{align}
		\label{evmTL}
		& \quad \|\nabla_{\Ghsm} I_h \Tahsm \ehvm\|_{L^2(\Ghsm)}\notag \\
		&\lesssim (1 + \kappa_{*,l})h^{k-1} + h^{-1} \| \nabla_{\Ghsm} \ehm \|_{L^2(\Ghsm)} \notag \\
		&\quad +  h^{-2.1} \| \nabla_{\Ghsm} \ehm \|_{L^2(\Ghsm)} 
		\Big[(1 + \kappa_{*,l})h^{k-1} + h^{-1} \| \nabla_{\Ghsm} \ehm \|_{L^2(\Ghsm)}\Big] \notag \\
		&\lesssim (1 + \kappa_{*,l})h^{k-1} + h^{-1} \| \nabla_{\Ghsm} \ehm \|_{L^2(\Ghsm)} 
		+  h^{-3.1} \| \nabla_{\Ghsm} \ehm \|_{L^2(\Ghsm)}^2,
	\end{align}
	where the mesh size assumption \eqref{cond1} and the assumption $\tau \le c h^k $ with $k\ge 3$ have been used in deriving the last inequality.

	Similarly, by substituting \eqref{INe} from Lemma \ref{normal-control} into \eqref{evmL2}, we obtain the following estimate for \(\|\ehvm\|_{L^2(\Ghsm)}\):
	\begin{align}
		\label{evmL}
		\|\ehvm\|_{L^2(\Ghsm)} 
		&\lesssim (1 + \kappa_{*,l})h^{k-1} + h^{-1} \| \nabla_{\Ghsm} \ehm \|_{L^2(\Ghsm)} 
		+ h^{-3.1} \| \nabla_{\Ghsm} \ehm \|_{L^2(\Ghsm)}^2.
	\end{align}
	
By decomposing the modified velocity error $ \hat e_v^m $ into its normal and tangential components, and employing estimates \eqref{INe} and \eqref{evmTL} along with the Poincaré-type inequality \eqref{eq:poincare2}, we obtain the following $ H^1 $ estimate for $ \hat e_v^m $:
\begin{align}\label{eq:vel_H1}
\|\ehvm\|_{H^1(\Ghsm)} 
&\lesssim h^{-1} \| I_h \Nahsm \ehvm \|_{L^2(\Ghsm)} + \| I_h \Tahsm \ehvm \|_{H^1(\Ghsm)} \notag \\
&\lesssim  (1 + \kappa_{*,l})h^{k-2} + h^{-2} \| \nabla_{\Ghsm} \ehm \|_{L^2(\Ghsm)} + h^{-3.1} \| \nabla_{\Ghsm} \ehm \|_{L^2(\Ghsm)}^2 \notag \\
&\lesssim  (1 + \kappa_{*,l})h^{k-2} + h^{-2} \| \nabla_{\Ghsm} \ehm \|_{L^2(\Ghsm)}.
\end{align}
Here, the final inequality follows from the induction assumption \eqref{Linfty-W1infty-hat-em}.

	%............................................................
\subsection{Estimates of $\hat e_h^{m+1}$ in terms of $e_h^{m+1}$}\label{geo-hat-wiho}
In order to establish the norm equivalences on different surfaces, we need to establish estimates for $\|e_h^{m+1}\|_{W^{1,\infty}(\Ghsm)}$ and $\|\hat e_h^{m+1}\|_{W^{1,\infty}(\Ghsm)}$. To begin with, a bound for $e_h^{m+1}$ in terms of $\hat e_h^m$ can be obtained by using the relation 
$e_h^{m+1} = \hat e_h^m + \tau(\hat e_v^m + I_h T_*^m v^m)$ in \eqref{evm}.

From \eqref{evmL}, together with \eqref{cond1}, \eqref{Linfty-W1infty-hat-em}, and \(\tau \le c h^k\) with \(k \geq 3\), the following result holds:
\begin{align}\label{evml-rough}
    \|\ehvm\|_{L^2(\Ghsm)} \lesssim h^{0.1} \lesssim 1.
\end{align}
By using the relation $e_h^{m+1} = \hat e_h^m + \tau(\hat e_v^m + I_h T_*^m v^m)$ shown in \eqref{evm}, the triangle inequality and \eqref{eq:NT_stab_L22}, as well as the estimate in \eqref{evml-rough}, we obtain  
\begin{align}\label{eq:e_NT}
    \|\eM\|_{L^2(\Ghsm)} 
    &\leq \|\ehm\|_{L^2(\Ghsm)} + \tau \|\hat e_v^m + I_h T_*^m v^m\|_{L^2(\Ghsm)} \notag \\
    &\lesssim \|\ehm \cdot \nahsm\|_{L^2(\Ghsm)} + \tau.
\end{align}
Similarly, using the $H^1$-norm estimate for the modified velocity error $\hat e_v^m$ in \eqref{eq:vel_H1}, the stepsize condition $\tau \leq c h^k$ with $k \geq 3$, as well as \eqref{eq:NT_stab_L22} and \eqref{evm}, we obtain the following estimate:
\begin{align*}
	\|\eM-\ehm\|_{H^1(\Ghsm)}&\lesssim \tau + \tau\| \ehvm\|_{H^1(\Ghsm)}  \\
	&\lesssim\tau+\tau((1+\kappa_{*,l})h^{k-2}+h^{-2}\| \nabla_{\Ghsm} \ehm \|_{L^2(\Ghsm)})\notag\\
	&\lesssim \tau + h \|\nabla_{\Ghsm}\ehm\|_{L^2(\Ghsm)} \\
    &\lesssim \tau + \|\ehm\|_{L^2(\Ghsm)} \\
    &\lesssim \tau + \|\ehm \cdot \nahsm\|_{L^2(\Ghsm)} . 
	%	&\lesssim \|\ehm\cdot\nhsm\|_{L^2(\Ghsm)}+\tau+(1+\kappa_{*,l})h^k\quad\mbox{(\eqref{eq:NT_stab_L2} and \eqref{nsmnhsm} are used)}.
\end{align*}
Consequently, the following inequality holds by utilizing triangle inequality and the estimate \eqref{eq:NT_stab_L22}:
\begin{align}\label{ehMM-copy}
\|\eM\|_{H^1(\Ghsm)}
&\lesssim \tau + \|\ehm\|_{H^1(\Ghsm)} 
\lesssim \tau + \|\hat e_h^m \cdot \bar n_{h,*}^m\|_{L^2(\Ghsm)} + \|\nabla_{\Ghsm} \ehm\|_{L^2(\Ghsm)} .
\end{align}
Therefore, the above inequality, combined with the mathematical induction hypothesis \eqref{Linfty-W1infty-hat-em}, and the assumption $\tau \le c h^k$ with $k \ge 3$, further implies that
\begin{align}\label{ehMM}
	% \|\eM\|_{L^2(\Ghsm)}&\lesssim \|\ehm\|_{L^2(\Ghsm)}+\tau,\notag\\
	\|\eM\|_{H^1(\Ghsm)}&\lesssim h^{1.6}.
\end{align}
By applying the inverse inequality, the assumption \(\tau \leq c h^k\) with $k \ge 3$, and the induction assumptions \eqref{cond2}--\eqref{Linfty-W1infty-hat-em}, the inequality \eqref{eq:e_NT} leads to the following estimates:
\begin{equation}\label{eM}
    \| \eM \|_{L^2(\Ghsm)} \lesssim h^{2.6}, \quad 
    \| \eM \|_{L^\infty(\Ghsm)} \lesssim h^{1.6}, \quad \text{and} \quad 
    \| \nabla_{\hat \Gamma_{h,*}^m} \eM \|_{L^\infty(\Ghsm)} \lesssim h^{0.6}.
\end{equation}
At each node $p \in \Gamma_h^{m+1}$, the distance between $p$ and $\Gamma^{m+1}$, as defined in \eqref{def-distance-error-0}, can be bounded by
\begin{align}
    \hat d^{m+1}(p) \le \|e_h^{m+1}\|_{L^\infty(\Ghsm)} \lesssim h^{1.6},
\end{align}
since $X_{h,*}^{m+1} \circ (X_h^{m+1})^{-1}(p)$ lies on the exact surface $\Gamma^{m+1}$ by the construction of $X_{h,*}^{m+1}$.

% For each curved triangle $K^0 \subset \Gamma_{h}^0$, let $F_{K^0}: K_{\rm f}^0 \rightarrow K^0$ denote the parametrization of the curved triangle $K^0\subset\Gamma_h^0$, where $K_{\rm f}^0$ is the flat triangle sharing the same three vertices with $K^0$.

Furthermore, from the above estimate \eqref{eM}, the norm equivalence \eqref{W1p-equiv}, the definition in \eqref{kl*}, and using the relations \eqref{X-id-Hn0}--\eqref{X-id-Hn}, we have
\begin{align}\label{X-h-m1}
    \|X_h^{m+1}\|_{W^{1,\infty}(\Gamma_h^0)} &\le \|X_h^{m+1} - \hat X_{h,*}^{m}\|_{W^{1,\infty}(\Gamma_h^0)} + \| \hat X_{h,*}^{m}\|_{W^{1,\infty}(\Gamma_h^0)} \notag \\
    & \lesssim \|e_h^{m+1}\|_{W^{1,\infty}(\Ghsm)} + \tau\| I_h(-H^m n^m + g^m)\|_{W^{1,\infty}(\Ghsm)} + 1 \notag \\
    & \lesssim h^{0.6} + \tau + 1 \lesssim 1,
\end{align}
by choosing $h$ sufficiently small and using the assumption $\tau \le ch^k$ with $k \ge 3$. For any point $q$ on $\Gamma_h^{m+1}$, which lies on some curved triangle $K$, it follows from \eqref{X-h-m1} that the distance from $p$ to any node $q$ on the corresponding curved triangle $K$ is bounded by $C_{\kappa_{m}} h$. Thus, for any $q \in \Gamma_h^{m+1}$, using \eqref{eM}, we have
\begin{align}
    \hat d^{m+1}(q) \lesssim h + \|e_h^{m+1}\|_{L^\infty(\Ghsm)} \lesssim h.
\end{align}
Therefore, by choosing $h$ sufficiently small, we have $\Gamma_h^{m+1} \subset D_\delta(\Gamma^{m+1})$. 

This guarantees that the map $a^m\circ X_h^{m+1}:\Gamma_h^0\rightarrow \Gamma^{m+1}$ is well-defined. Since \eqref{X-h-m1} implies that 
$$
\| a^m\circ X_h^{m+1} - I_h\circ a^m\circ X_h^{m+1} \|_{L^\infty(\Gamma_h^0)}
\lesssim h \| \nabla_{\Gamma_h^0} X_h^{m+1} \|_{L^\infty(\Gamma_h^0)}
\lesssim h ,
$$
it follows that the image $\hat\Gamma_{h,*}^{m+1}$ of the map $I_h\circ a^m\circ X_h^{m+1}:\Gamma_h^0\rightarrow \mathbb{R}^3$ lies in the neighborhood $D_\delta(\Gamma^{m+1})$ when $h$ is sufficiently small. 

This verifies the first mathematical induction assumption in Section~\ref{induc ass} at time level $m+1$. Consequently, $\hat e_h^{m+1}$ is well defined.

% The flow map $X_{h,*}^{m+1}$ defined on $\Ghsm$ can also be considered as the map defined on $\Gamma_h^0$ by considering the composition map $X_{h,*}^{m+1}\circ \hat X_{h,*}^{m}$. Then we have
% \begin{align}
% 	\|X_{h,*}^{m+1}\circ \hat X_{h,*}^{m}\|_{W^{1,\infty}(\Gamma_h^0)} \lesssim h^{1.6}\|
% \end{align}

We proceed to estimate the projected distance error $\hat e_h^{m+1}$ at time level $t_{m+1}$ using the bounds for $e_h^{m+1}$ given in \eqref{eq:e_NT}--\eqref{eM}, together with the geometric relations \eqref{eq:geo_rel_1} and \eqref{eq:geo_rel_2} established in Section~\ref{rekl} (which also hold at time level $t_{m+1}$ since \(e_h^{m+1}\) is small in view of \eqref{eM}). By combining these geometric relations with the $L^\infty$ stability of the Lagrange interpolation, the stability bounds \eqref{stability2}--\eqref{stability3}, and the estimate \eqref{eM}, we obtain:
\begin{subequations}
    \begin{align}
        \|\ehM\|_{L^\infty(\Ghsm)} &\lesssim \|\eM\|_{L^\infty(\Ghsm)} + \|\eM\|_{L^\infty(\Ghsm)}^2 \lesssim h^{1.6}\label{em+1-infty}.\\
		\|\ehM\|_{L^2(\Ghsm)} 
        &\lesssim \|\eM\|_{L^2(\Ghsm)} + \|\eM\|_{L^2(\Ghsm)} \|\eM\|_{L^\infty(\Ghsm)}\notag\\
		& \lesssim \|e_h^{m+1}\|_{L^2(\Ghsm)} \lesssim \|\hat e_h^m \cdot \bar n_{h,*}^m\|_{L^2(\Ghsm)}+\tau\label{em+1-2},
    \end{align}
\end{subequations}
where inequality \eqref{eq:e_NT} is used in the last inequality. 
Using the inverse inequality, the induction assumption \eqref{Linfty-W1infty-hat-em}, and the condition $\tau \le c h^k$ with $k \ge 3$, inequality \eqref{em+1-2} leads to the following estimates:
\begin{align}\label{ehM-1infty}
	\|\ehM\|_{L^2(\Ghsm)}\lesssim h^{2.6}, \quad \|\ehM\|_{H^1(\Ghsm)} \lesssim h^{1.6} \quad \text{and} \quad \|\ehM\|_{W^{1,\infty}(\Ghsm)} \lesssim h^{0.6}.
\end{align}
	It follows from \cite[Eq.~(5.31)]{bai2024new} and \eqref{eM} that
	\begin{align}\label{geo-normal}
		\nsM\circ \hat X_{h,*}^{m+1}-\nsm \circ\hat X_{h,*}^{m} &\lesssim \tau|\eM|+|\Tsm(\eM-\ehm)|+\tau\notag\\
		&\lesssim |\Tsm(\eM-\ehm)|+\tau \quad\mbox{at the nodes.}
	\end{align}
	% Analogous to \cite[Eqs. (5.32) and (5.33)]{bai2024new}, the following estimates can be established:
    
	The following estimates can then be established (here and below, we use $n_*^{m+1}$ to denote $n_*^{m+1} \circ \hat X_{h,*}^{m+1}$ and $n_*^{m}$ to denote $n_*^{m} \circ \hat X_{h,*}^{m}$ for brevity):
	\begin{align}
		\label{fhinfty}
		&\quad\,\,\|I_h([I-\nsM(\nsM)^\top]\eM)\|_{L^\infty(\Ghsm)}\notag\\
		& \lesssim \|I_h([I-\nsM(\nsM)^\top](\eM- \ehm))\|_{L^\infty(\Ghsm)} + \|I_h([I-\nsM(\nsM)^\top]\ehm)\|_{L^\infty(\Ghsm)}\notag \\
		&\lesssim \|\eM-\ehm\|_{L^\infty(\Ghsm)}+\|I_h([\nsM(\nsM)^\top-\nsm(\nsm)^\top]\ehm)\|_{L^\infty(\Ghsm)}\notag\\
		& \quad\,\,\text{(using orthogonality between $I-\nsm (\nsm)^\top$ and $\ehm$ at the nodes of $\Ghsm$)}\notag\\
		&\lesssim \|\eM-\ehm\|_{L^\infty(\Ghsm)}+(\tau+\|\eM-\ehm\|_{L^\infty(\Ghsm)})\|\ehm\|_{L^\infty(\Ghsm)} \quad \text{(using \eqref{geo-normal} )}\notag\\
		&\lesssim \tau\|\ehvm\|_{L^\infty(\Ghsm)}+\tau+(\tau+\tau\|\ehvm\|_{L^\infty(\Ghsm)})\|\ehm\|_{L^\infty(\Ghsm)}\quad\mbox{(using the relation \eqref{evm})}\notag\\
		&\lesssim\tau h^{-0.1}\|\ehvm\|_{H^1(\Ghsm)}+\tau\quad\mbox{(using Sobolev embedding \eqref{eq:poincare3} and \eqref{Linfty-W1infty-hat-em})}\notag\\
		&\lesssim \tau+\tau h^{-2.1}\| \nabla_{\Ghsm} \ehm \|_{L^2(\Ghsm)}\quad\mbox{(using \eqref{eq:vel_H1})}, \\[5pt]
		% &\quad\,\,\mbox{(using \eqref{eq:vel_H1}, \eqref{cond1}, \eqref{Linfty-W1infty-hat-em} and  $\tau \le ch^k$ with  $k\ge 3$)},\\
		&\quad\,\,\|\nabla_{\Ghsm}f_h^{m+1}\|_{L^2(\Ghsm)}
        \quad\mbox{(for the $f_h^{m+1}$ defined in \eqref{eq:geo_rel_1}--\eqref{eq:geo_rel_2})}\notag\\
		&\lesssim h^{-1}\|I_h[I_h((I-\nsM(\nsM)^\top)\eM)]^2\|_{L^2(\Ghsm)}\notag\\
		& \lesssim h^{-1}\|I_h((I-\nsM(\nsM)^\top)\eM)\|_{L^2(\Ghsm)} \|I_h((I-\nsM(\nsM)^\top)\eM)\|_{L^\infty(\Ghsm)} \notag\\
		&\quad\,\,\mbox{(using stability estimate \eqref{stability2})}\notag\\
		&\lesssim h^{1.5}\|\eM \|_{L^2(\Ghsm)} \quad\mbox{(using \eqref{fhinfty}, \eqref{stability3}, \eqref{cond1}, \eqref{Linfty-W1infty-hat-em} and  $\tau \le ch^k$ with  $k\ge 3$)}\label{nabla-f}.
	\end{align}
	{\b }Therefore, from the geometric relation $ \ehM = I_h\big[ (e_h^{m+1}\cdot n_*^{m+1}) n_*^{m+1} \big] + f_h^{m+1}$ shown in \eqref{eq:geo_rel_1}, and by combining \eqref{nabla-f} with the super-approximation estimate \eqref{super1} at time level $t_{m+1}$, we obtain
\begin{align}\label{nabla-hat-withouthat}
    \|\nabla_{\Ghsm}\ehM\|_{L^2(\Ghsm)}& \le \|\nabla_{\Ghsm}I_h [(e_h^{m+1} \cdot n_*^{m+1}) n_*^{m+1}]\|_{L^2(\Ghsm)} + \|\nabla_{\Ghsm}f_h^{m+1}\|_{L^2(\Ghsm)}\notag\\
    &\lesssim \|\nabla_{\hat \Gamma_{h,*}^{m+1}}(I_h - 1)(N_*^{m+1} e_h^{m+1})\|_{L^2(\hat \Gamma_{h,*}^{m+1})} +\|\nabla_{\hat \Gamma_{h,*}^{m+1}} (N_*^{m+1}e_h^{m+1})\|_{L^2(\hat \Gamma_{h,*}^{m+1})} \notag\\
    & \quad\,\,+ h^{1.5}\|e_h^{m+1}\|_{L^2(\hat \Gamma_{h,*}^{m+1})} \notag\\ 
    & \lesssim\|\eM\|_{L^2(\hat \Gamma_{h,*}^{m+1})}+\|\nabla_{\hat \Gamma_{h,*}^{m+1}}\eM\|_{L^2(\hat \Gamma_{h,*}^{m+1})}\notag\\
	& \lesssim\|\eM\|_{L^2(\Ghsm)}+\|\nabla_{\Ghsm}\eM\|_{L^2(\Ghsm)},
\end{align}
where the super-approximation estimate \eqref{super1} and the $W^{i,p}$ ($i=0,1$, $1\le p\le \infty$) norm equivalence between $\Ghsm$ and $\hat \Gamma_{h,*}^{m+1}$ are used in the last inequality. This equivalence is induced by Lemma \ref{equi-MA} and the following estimate for sufficiently small $h$:
\begin{align}
    &\quad\,\,\| \hat{X}_{h,*}^{m+1} - \hat{X}_{h,*}^{m} \|_{W^{1,\infty}(\Ghsm)}\notag\\
    &\le \|  \hat{X}_{h,*}^{m+1} - X_{h,*}^{m+1} \|_{W^{1,\infty}(\Ghsm)}+\| X_{h,*}^{m+1} - \hat X_{h,*}^{m} \|_{W^{1,\infty}(\Ghsm)}\notag\\
    &= \|  \hat{e}_{h}^{m+1} - e_{h}^{m+1} \|_{W^{1,\infty}(\Ghsm)}+\| X_{h,*}^{m+1} - \hat X_{h,*}^{m} \|_{W^{1,\infty}(\Ghsm)}\notag\\
    &\le \|\ehM\|_{W^{1,\infty}(\Ghsm)}+\| e_h^{m+1} \|_{W^{1,\infty}(\Ghsm)}+\tau\|I_h(-H^m n^m + g^m)\|_{W^{1,\infty}(\Ghsm)}\notag\\
    &\lesssim h^{0.6}\quad\mbox{(using \eqref{eM}, \eqref{ehM-1infty}, and $\tau \le ch^k$ with $k\ge 3$)}.\label{hxM}
\end{align}
By combining the estimates \eqref{em+1-2}, \eqref{nabla-hat-withouthat} and \eqref{ehMM-copy}, we obtain:
	\begin{align}\label{hathathat}
		\|\hat e_h^{m+1}\|_{H^1(\Ghsm)} &\lesssim \|e_h^{m+1}\|_{H^1(\Ghsm)}\lesssim \|\hat e_h^m\|_{H^1(\Ghsm)} +\tau \notag\\
		&\lesssim \|\hat e_h^m \cdot \bar n_{h,*}^m\|_{L^2(\Ghsm)} + \|\nabla_{\Ghsm}\hat e_h^m\|_{L^2(\Ghsm)} +\tau.
	\end{align}
	% where the last inequality uses the inequality \eqref{eq:NT_stab_L22}.

	\subsection{Norm equivalence on the surfaces $\Ghm$, $\GhM$, $\Ghsm$, $\GhsM$ and $\Gamma_{h,*}^{m+1}$}\label{norm-equiv-surfaces}
	From \eqref{X-id-Hn0} and \eqref{X-id-Hn}, and using $\tau \le ch^k$ with $k\ge 3$, the following estimate holds:
\begin{align}\label{eq:hat_X_s_diff3}
    \|  X_{h,*}^{m+1} - \hat X_{h,*}^{m} \|_{W^{1,\infty}(\Ghsm)} 
    &=\|I_h(X^{m+1}-{\rm id})\|_{W^{1,\infty}(\Ghsm)}\lesssim \tau \lesssim h^k.
\end{align}
By choosing $h$ sufficiently small, the $W^{i,p}$ ($i=0,1$, $1\le p\le \infty$) norm equivalence between the surfaces $\Ghsm$ and $\Gamma_{h,*}^{m+1}$ follows from Lemma \ref{equi-MA}. Furthermore, the relation \eqref{Xexrelat} implies
	\begin{align}
		&\quad\,\,\| X_{h}^{m+1} - X_{h}^{m} \|_{W^{1,\infty}(\Ghsm)} \notag\\
		&= \|\eM - \ehm - \tau I_h(H^m n^m - g^m)  \|_{W^{1,\infty}(\Ghsm)} 
		\quad\mbox{(using the relation \eqref{Xexrelat})}\notag\\
		&\lesssim  \tau +  \|\eM \|_{W^{1,\infty}(\Ghsm)} +  \|\ehm \|_{W^{1,\infty}(\Ghsm)} \notag\\
		%	&\lesssim
		%	&= h^{}\|\eM - \ehm - \tau I_h(H^m n^m - g^m)  \|_{W^{1,\infty}(\Ghsm)} \notag\\
		&\lesssim h^{0.6} \quad\mbox{(using \eqref{eM}, \eqref{Linfty-W1infty-hat-em}, $\tau \le ch^k$ with $k\ge 3$)} \label{Xm+1-Xm}.
	\end{align}
By choosing $h$ sufficiently small, the $W^{i,p}$ ($i=0,1$, $1\le p\le \infty$) norm equivalence between the surfaces $\Gamma_h^m$ and $\Gamma_h^{m+1}$ follows from Lemma \ref{equi-MA}. Similarly, using the estimate \eqref{hxM}, the $W^{i,p}$ norm equivalence between $\hat \Gamma_{h,*}^m$ and $\hat \Gamma_{h,*}^{m+1}$ is also guaranteed by Lemma \ref{equi-MA} for sufficiently small $h$. Combined with the $W^{i,p}$ norm equivalence between $\hat \Gamma_{h,*}^m$ and $\Gamma_{h}^{m}$ induced by \eqref{Linfty-W1infty-hat-em}, we conclude that the $L^p$ and $W^{1,p}$ norms of a finite element function $v_h$ (with a fixed nodal vector) on the surfaces $\Ghm$, $\GhM$, $\Ghsm$, $\GhsM$, and $\Gamma_{h,*}^{m+1}$ are all equivalent for $p \in [1,\infty]$.

	% By the norm equivalence shown in Lemma \ref{equi-MA}, for $\tau\le ch^k$ and sufficiently small $h$ satisfying the mesh size assumption in \eqref{cond1}, inequalities \eqref{eq:hat_X_s_diff3}-\eqref{hxM} imply that the $L^p$ and $W^{1,p}$ norms of a finite element function $v_h$ (with fixed nodal vector) on the surfaces $\Ghm$, $\GhM$, $\Ghsm$, $\GhsM$, $\Gamma_{h,*}^{m+1}$ are all equivalent for $p\in [1,\infty]$.
	%.....................................................
	\subsection{Stability of orthogonal projection on the error}
	In this section, we derive a stability estimate that facilitates the conversion of \(\| \eM \cdot \nbhsm \|_{L^2(\Ghsm)}^2\) to \(\| \ehM \cdot \nbhsM \|_{L^2(\hat\Gamma_{h,*}^{m+1})}^2\) at each time level. The difference between the two quantities can be decomposed into the following three components:
	\begin{align}
		& \quad \,\,\| \ehM \cdot \nbhsM \|_{L^2(\hat\Gamma_{h,*}^{m+1})}^2 - \| \eM \cdot \nahsm \|_{L^2(\Ghsm)}^2 \notag\\
		&= \| \ehM \cdot \nbhsM \|_{L^2(\hat\Gamma_{h,*}^{m+1})}^2 - \| \ehM \cdot \nbhsM \|_{L^2(\hat\Gamma_{h,*}^m)}^2 
		&\text{(change of \(\hat\Gamma_{h,*}^{m+1}\) to \(\hat\Gamma_{h,*}^m\))} \notag \\
		&\quad+ \| \ehM \cdot \nbhsM \|_{L^2(\hat\Gamma_{h,*}^m)}^2 - \| \ehM \cdot \nbhsm \|_{L^2(\Ghsm)}^2 
		&\text{(change of \(\nbhsM\) to \(\nbhsm\))} \notag \\
		&\quad+ \| \ehM \cdot \nbhsm \|_{L^2(\Ghsm)}^2 - \| \eM \cdot \nbhsm \|_{L^2(\Ghsm)}^2 
		&\text{(change of \(\ehM\) to \(\eM\))} \notag \\
		&=: M_1^m + M_2^m + M_3^m. \label{M}
	\end{align}
	To estimate the terms in \eqref{M}, we first introduce the following auxiliary result, which provides a bound on the nodal distances between \(\hat\Gamma_{h,*}^{m+1}\) and \(\hat\Gamma_{h,*}^m\) at each time step. A preliminary estimate characterizing the difference between the two consecutive projected interpolated surfaces \(\hat\Gamma_{h,*}^{m+1}\) and \(\hat\Gamma_{h,*}^m\) has already been established in \eqref{hxM}. However, this estimate is insufficient for the subsequent error analysis. Our objective is to establish an \(L^\infty\) bound for the difference between the corresponding flow maps \(\hat X_{h,*}^{m+1}\) and \(\hat X_{h,*}^m\), showing that it can be controlled by \(C_{\kappa_{m}} \tau\), which is consistent with the behavior of the exact solution, for which \(\|X^{m+1} - X^m\|_{L^\infty(\Gamma^m)} \le C_0 \tau\).
	\begin{lemma}\label{Xhinfty}
		Under the induction assumptions in Section \ref{induc ass}, the following inequality holds:
		\begin{equation}\label{Xtau}
			\| \hat{X}_{h,*}^{m+1} - \hat{X}_{h,*}^m \|_{L^\infty(\Ghsm)} \lesssim \tau.
		\end{equation}
	\end{lemma}
	
	\begin{proof}
		The following identities, established in \cite[Eqs. (A.15) and (A.17)]{bai2024new}, are independent of the numerical scheme and describe the relations governing the discrete geometry at the nodes:
\begin{align}
    \Nsm (\hat X_{h,*}^{m+1} - \hat X_{h,*}^m) 
    &= (X^{m+1} - {\rm id}) \circ a^m + \rho_h, && \text{at the nodes}, \label{eq:geo_rel_4} \\
    \text{where} \quad |\rho_h| 
    &\leq C_0 \tau^2 + C_0 |T_*^m (\hat X_{h,*}^{m+1} - \hat X_{h,*}^m)|^2, && \text{at the nodes}, \label{eq:geo_rel_5} \\
    T_*^m (\hat X_{h,*}^{m+1} - \hat X_{h,*}^m) 
    &= T_*^m (X_h^{m+1} - X_h^m) + T_*^m (N_*^{m+1} - N_*^m) \hat e_h^{m+1}, && \text{at the nodes}, \label{eq:geo_rel_6}
\end{align}
where \(C_0\) is a constant independent of \(\kappa_l\), \(\kappa_{*,l}\), \(C_*\), and \(C_{\#}\); and \(N_*^{m+1}\) and \(N_*^m\) abbreviate the compositions \(N_*^{m+1}\circ \hat X_{h,*}^{m+1}\) and \(N_*^m\circ \hat X_{h,*}^{m}\), respectively.

At each node, the difference $ n_*^{m+1}(X_{h,*}^{m+1}) - n_*^m(\hat X_{h,*}^m) $ represents the change in $ n $ along a particle trajectory of the exact flow map and is therefore $ O(\tau) $ at that node. Applying the estimate in \eqref{eq:hat_X_s_diff3}, we obtain the following result:
\begin{align}\label{nsM-nsm-nodes}
    |\nsM - \nsm| 
    &= |n_*^{m+1}(\hat X_{h,*}^{m+1}) - n_*^m(\hat X_{h,*}^m)| \notag \\
    &= |n_*^{m+1}(\hat X_{h,*}^{m+1}) - n_*^{m+1}(X_{h,*}^{m+1}) 
    + n_*^{m+1}(X_{h,*}^{m+1}) - n_*^m(\hat X_{h,*}^m)| \notag \\
    &\lesssim |\hat X_{h,*}^{m+1} - X_{h,*}^{m+1}| + \tau \quad \text{at the nodes} \notag \\
	&\lesssim |\hat X_{h,*}^{m+1} - \hat X_{h,*}^{m}| + 2|\hat X_{h,*}^m - X_{h,*}^{m+1}| + \tau \quad \text{at the nodes} \notag \\
    &\lesssim |\hat X_{h,*}^{m+1} - \hat X_{h,*}^m| + \tau \quad \text{at the nodes}.
\end{align}

Since \(\nahsm\) differs from \(\nsm\) by a small quantity, as shown in \eqref{eq:nsa2-infty}, we replace \(\Tahsm\) with \(\Tsm\) in \eqref{evmTL}, obtaining:
\begin{align}
    & \quad\,\,\|I_h \Tsm \ehvm\|_{H^1(\Ghsm)} \notag\\
    &= \|I_h \Tsm \ehvm\|_{L^2(\Ghsm)} + \|\nabla_{\Ghsm} I_h \Tsm \ehvm\|_{L^2(\Ghsm)} \notag \\
    &\lesssim \|I_h \Tahsm \ehvm\|_{L^2(\Ghsm)} 
    + \|I_h ((\Tahsm - \Tsm) \ehvm)\|_{L^2(\Ghsm)} \notag \\
    &\quad\,\, + \|\nabla_{\Ghsm} I_h \Tahsm \ehvm\|_{L^2(\Ghsm)} 
    + \|\nabla_{\Ghsm} I_h ((\Tahsm - \Tsm) \ehvm)\|_{L^2(\Ghsm)} \notag \\
    &\lesssim \|\nabla_{\Ghsm} I_h \Tahsm \ehvm\|_{L^2(\Ghsm)} 
    + h^{-1} \|\nahsm - \nsm\|_{L^\infty(\Ghsm)} \|\ehvm\|_{L^2(\Ghsm)} \quad \text{(using \eqref{eq:poincare2} and \eqref{stability3})}\notag \\
    &\lesssim (1 + \kappa_{*,l})h^{k-1} + h^{-1} \| \nabla_{\Ghsm} \ehm \|_{L^2(\Ghsm)} 
    + h^{-3.1} \| \nabla_{\Ghsm} \ehm \|_{L^2(\Ghsm)}^2\lesssim h^{0.1}, \notag \\
    &\quad \text{(using \eqref{cond1}, \eqref{Linfty-W1infty-hat-em}, \eqref{eq:nsa2-infty}, \eqref{evmTL}, \eqref{evmL} and $\tau \le ch^k$ with $k\ge 3$)}. \label{nMinfty}
\end{align}
Therefore, combining \eqref{eq:geo_rel_6}--\eqref{nMinfty} and using the relation \(X_h^{m+1} - X_h^m = \tau(\hat e_v^m + I_h v^m + I_h g^m)\) derived from \eqref{evm}--\eqref{xm11}, together with the $L^\infty$ stability of the Lagrange interpolation operator, we obtain the following result:
\begin{align}\label{tan-consecutive}
  & \quad\,\,\|I_h T_*^m (\hat X_{h,*}^{m+1} - \hat X_{h,*}^m) \|_{L^\infty(\Ghsm)}\notag\\
  & \lesssim \|I_h T_*^m ( X_{h}^{m+1} - X_{h}^m) \|_{L^\infty(\Ghsm)} + (\tau +\|\hat X_{h,*}^{m+1} - \hat X_{h,*}^m \|_{L^\infty(\Ghsm)} ) \|\hat e_h^{m+1}\|_{L^\infty(\Ghsm)} \notag\\
  &\hspace{288pt}\text{(using \eqref{eq:geo_rel_6} and \eqref{nsM-nsm-nodes})}\notag\\
  & \lesssim \tau(1+\|I_h T_*^m \hat e_v^m \|_{L^\infty(\Ghsm)}) + h^{1.6}\|\hat X_{h,*}^{m+1} - \hat X_{h,*}^m \|_{L^\infty(\Ghsm)} \quad\text{(using \eqref{xm} and \eqref{em+1-infty})}\notag\\
  & \lesssim \tau(1+ C_\epsilon h^{-\epsilon}\|I_h T_*^m \hat e_v^m \|_{H^1(\Ghsm)})+ h^{1.6}\|\hat X_{h,*}^{m+1} - \hat X_{h,*}^m \|_{L^\infty(\Ghsm)} \quad\text{(using \eqref{eq:poincare3})}\notag\\
  & \lesssim \tau+ h^{1.6}\|\hat X_{h,*}^{m+1} - \hat X_{h,*}^m \|_{L^\infty(\Ghsm)}\quad\text{(using \eqref{nMinfty} and choosing $\epsilon = 0.1$)}.
\end{align}
By using the relation \eqref{eq:geo_rel_5} and \eqref{tan-consecutive}, the following estimate holds:
\begin{align}\label{norm-equiv-discrete}
	\|I_h \rho_h\|_{L^\infty(\hat \Gamma_{h,*}^m)} &\lesssim \|I_h T_*^m (\hat X_{h,*}^{m+1} - \hat X_{h,*}^m)\|_{L^\infty(\hat \Gamma_{h,*}^m)} ^2 + \tau^2\notag\\
	& \lesssim h^{3.2}\|\hat X_{h,*}^{m+1} - \hat X_{h,*}^m \|_{L^\infty(\Ghsm)}^2+\tau^2\quad \text{(using \eqref{tan-consecutive})}\notag\\
	& \lesssim h^{3.8}\|\hat X_{h,*}^{m+1} - \hat X_{h,*}^m \|_{L^\infty(\Ghsm)}+\tau^2\quad \text{(using \eqref{hxM})}.
\end{align}
By using the relations \eqref{eq:geo_rel_4}--\eqref{eq:geo_rel_6} and the estimates in \eqref{tan-consecutive}--\eqref{norm-equiv-discrete}, we derive
	\begin{align}\label{eq:hat_X_s_diff1}
		&\quad\,\,\| \hat X_{h,*}^{m+1} - \hat X_{h,*}^{m} \|_{L^\infty(\Ghsm)} \notag\\
		&\leq \| I_h \Nsm (\hat X_{h,*}^{m+1} - \hat X_{h,*}^{m}) \|_{L^\infty(\Ghsm)} + \| I_h \Tsm (\hat X_{h,*}^{m+1} - \hat X_{h,*}^{m}) \|_{L^\infty(\Ghsm)} \notag\\
		&\lesssim \tau + \|I_h T_*^m (\hat X_{h,*}^{m+1} - \hat X_{h,*}^m) \|_{L^\infty(\Ghsm)}+\|I_h \rho_h\|_{L^\infty(\hat \Gamma_{h,*}^m)}\notag\\
		&\lesssim\tau+h^{1.6} \| \hat X_{h,*}^{m+1} - \hat X_{h,*}^{m} \|_{L^{\infty}(\Ghsm)} \quad\text{(using \eqref{tan-consecutive} and \eqref{norm-equiv-discrete})}.
	\end{align}
	 By absorbing the term \(h^{1.6} \| \hat X_{h,*}^{m+1} - \hat X_{h,*}^m \|_{L^{\infty}(\Ghsm)}\) into the left-hand side of \eqref{eq:hat_X_s_diff1}, we obtain:
	\begin{align}\label{eq:hat_X_s_diff2} 
		\|  \hat X_{h,*}^{m+1} - \hat X_{h,*}^{m}\|_{L^\infty(\Ghsm)} &\lesssim \tau . 
	\end{align} 
\end{proof}
Let \(\hat{\Gamma}_{h,*}^{m+\theta} = (1-\theta)\Ghsm + \theta\GhsM\) for \(\theta \in [0,1]\), and let \(\hat{n}^{m+\theta}_{h,*}\) denote the normal vector of \(\hat{\Gamma}_{h,*}^{m+\theta}\), with \(\bar{n}^{m+\theta}_{h,*} = P_{\hat{\Gamma}_{h,*}^{m+\theta}} \hat{n}^{m+\theta}_{h,*}\) representing the averaged normal vector on the surface \(\hat{\Gamma}_{h,*}^{m+\theta}\). In analogy with the normal vector representation in \eqref{nexpre} or \eqref{relation_material}, the difference 
\(\hat{n}^{m+\theta}_{h,*} - \hat{n}_{h,*}^m\) can be controlled by the surface gradient of
\(\hat{X}_{h,*}^{m+1} - \hat{X}_{h,*}^m\), yielding:
\begin{align}
    \label{nM22}
    \|\hat{n}^{m+\theta}_{h,*} - \hat{n}_{h,*}^m\|_{L^2(\Ghsm)} 
    &\lesssim \|\nabla_\Ghsm (\hat{X}_{h,*}^{m+1} - \hat{X}_{h,*}^m)\|_{L^2(\Ghsm)}, \\
    \|\hat{n}^{m+\theta}_{h,*} - \hat{n}_{h,*}^m\|_{L^\infty(\Ghsm)} 
    &\lesssim \|\nabla_\Ghsm (\hat{X}_{h,*}^{m+1} - \hat{X}_{h,*}^m)\|_{L^\infty(\Ghsm)}. \label{nMinft}
\end{align}

By employing an argument similar to that used in the proof of Lemma \ref{Lem:normalvector}, we obtain the following results; the proof is omitted for brevity.
\begin{lemma}\label{Lem1:normalvector}
    The following approximation properties of \(\bar{n}^{m+\theta}_{h,*}\) hold:
	\begin{subequations}\label{eq:nsa0}
		\begin{align}
        \|\bar{n}^{m+\theta}_{h,*} - \bar{n}_{h,*}^m\|_{L^2(\Ghsm)} 
        &\lesssim \|\nabla_\Ghsm (\hat{X}_{h,*}^{m+1} - \hat{X}_{h,*}^m)\|_{L^2(\Ghsm)},\label{eq:nsa0-1} \\
        \|\bar{n}^{m+\theta}_{h,*} - \nsm\|_{L^2(\Ghsm)} 
        &\lesssim \|\nabla_\Ghsm (\hat{X}_{h,*}^{m+1} - \hat{X}_{h,*}^m)\|_{L^2(\Ghsm)} + (1+\kappa_{*,l})h^k, \label{eq:nsa0-2} \\
        \|\bar{n}^{m+\theta}_{h,*} - \nhsm\|_{L^2(\Ghsm)} 
        &\lesssim \|\nabla_\Ghsm (\hat{X}_{h,*}^{m+1} - \hat{X}_{h,*}^m)\|_{L^2(\Ghsm)} + (1+\kappa_{*,l})h^k, \label{eq:nsa0-3} \\
        \|\bar{n}^{m+\theta}_{h,*} - \hat{n}^{m+\theta}_{h,*}\|_{L^2(\Ghsm)} 
        &\lesssim \|\nabla_\Ghsm (\hat{X}_{h,*}^{m+1} - \hat{X}_{h,*}^m)\|_{L^2(\Ghsm)} + (1+\kappa_{*,l})h^k \label{eq:nsa0-4}.
    \end{align}
	\end{subequations}
\end{lemma} 
\begin{lemma}\label{Lem1:normalvector-infty}
    The following approximation properties of \(\bar{n}^{m+\theta}_{h,*}\) hold:
	\begin{subequations}\label{eq:nsa0-infty}
		\begin{align}
        \|\bar{n}^{m+\theta}_{h,*} - \bar{n}_{h,*}^m\|_{L^\infty(\Ghsm)} 
        &\lesssim \|\nabla_\Ghsm (\hat{X}_{h,*}^{m+1} - \hat{X}_{h,*}^m)\|_{L^\infty(\Ghsm)},\label{eq:nsa0-1-infty} \\
        \|\bar{n}^{m+\theta}_{h,*} - \nsm\|_{L^\infty(\Ghsm)} 
        &\lesssim \|\nabla_\Ghsm (\hat{X}_{h,*}^{m+1} - \hat{X}_{h,*}^m)\|_{L^\infty(\Ghsm)} + (1+\kappa_{*,l})h^{k-1}, \label{eq:nsa0-2-infty} \\
        \|\bar{n}^{m+\theta}_{h,*} - \nhsm\|_{L^\infty(\Ghsm)} 
        &\lesssim \|\nabla_\Ghsm (\hat{X}_{h,*}^{m+1} - \hat{X}_{h,*}^m)\|_{L^\infty(\Ghsm)} + (1+\kappa_{*,l})h^{k-1}, \label{eq:nsa0-3-infty} \\
        \|\bar{n}^{m+\theta}_{h,*} - \hat{n}^{m+\theta}_{h,*}\|_{L^\infty(\Ghsm)} 
        &\lesssim \|\nabla_\Ghsm (\hat{X}_{h,*}^{m+1} - \hat{X}_{h,*}^m)\|_{L^\infty(\Ghsm)} + (1+\kappa_{*,l})h^{k-1} \label{eq:nsa0-4-infty}.
    \end{align}
	\end{subequations}
\end{lemma}
The $W^{1,\infty}$-norm boundedness of $\bar n_{h,*}^m$ in \eqref{nhmW-1} and \eqref{eq:nsa0-1-infty} imply the $W^{1,\infty}$-norm boundedness of $\bar{n}^{m+\theta}_{h,*}$, i.e.,
   \begin{align}\label{bound-n}
   	\|\bar{n}^{m+\theta}_{h,*}\|_{W^{1,\infty}(\Ghsm)}&\lesssim\|\bar n_{h,*}^m\|_{W^{1,\infty}(\Ghsm)}+h^{-1}\|\bar{n}^{m+\theta}_{h,*}-\bar n_{h,*}^m\|_{L^\infty(\Ghsm)}\notag\\
   	&\lesssim \|\bar n_{h,*}^m\|_{W^{1,\infty}(\Ghsm)}+h^{-1}\|\nabla_{\Ghsm}(\hat X_{h,*}^{m+1} - \hat X_{h,*}^{m})\|_{L^\infty(\Ghsm)} \notag\\
   	&\lesssim 1+h^{-2}\tau \lesssim 1\quad\mbox{(Lemma \ref{Xhinfty} and $\tau \le ch^k$ with $k\ge 3$ are used)}.
   \end{align}

By using the fundamental theorem of calculus, geometric perturbation estimates, integration by parts, super-approximation properties, stability results, and Lemma~\ref{Xhinfty}, we derive the following estimate for $ M_1^m + M_2^m + M_3^m $ defined in \eqref{M} (with a detailed proof presented in Appendix~\ref{appendix_G}):
\begin{lemma}\label{lemma:Mm}
The following estimate holds:
\begin{align}
	&\quad\,\, \| \ehM\cdot \nbhsM \|_{L^2(\hat\Gamma_{h,*}^{m+1})}^2 - \| \eM\cdot \nbhsm \|_{L^2(\Ghsm)}^2\notag\\
	&=M_1^m+M^m_2+M_3^m\notag\\
	&\lesssim \epsilon^{-1}\tau[\tau^2+(1+\kappa_{*,l})^2h^{2k}]+\epsilon^{-1}\tau\|\ehm\cdot\nahsm\|^2_{L^2(\Ghsm)}+\epsilon\tau\| \nabla_{\Ghsm} \ehm \|_{L^2(\Ghsm)}^2.\label{stab}
\end{align}
\end{lemma}
This result enables the conversion of $ \|\eM \cdot \nbhsm\|_{L^2(\Ghsm)}^2 $ to $ \|\ehM \cdot \nbhsM\|_{L^2(\hat\Gamma_{h,*}^{m+1})}^2 $ in the subsequent error analysis.

	\subsection{Error estimates}\label{sec:err-estimates}
	To derive the error estimate for $\|\hat e_h^m\|_{L^2(\Ghsm)}$, we note that the estimate \eqref{eq:NT_stab_L22} indicates that it suffices to control the error in the normal component, specifically $\|\hat e_h^m \cdot \bar n_{h,*}^m\|_{L^2(\Ghsm)}$. Accordingly, we estimate $\|\hat e_h^m\cdot \bar n_{h,*}^m\|_{L^2(\Ghsm)}$ by choosing $\phi_h := I_h(\eM \cdot \nahsm)$ as the test function in the error equation \eqref{err}. 
Firstly, the following inequality holds:
\begin{align}
    & \quad\,\,\frac{1}{2\tau}\Big(\| \eM \cdot \nahsm \|_{L^2(\Ghsm)}^2  - \| \ehm \cdot \nahsm \|_{L^2(\Ghsm)}^2\Big)\notag\\
    &\le \int_{\Ghsm} \Big(\frac{\eM - \ehm}{\tau} \cdot \nahsm\Big)\Big( \eM \cdot \nahsm\Big),
\end{align}
which further leads to the following estimate:
	\begin{align}\label{Err}
		& \,\,\quad\frac{1}{2\tau} (\| \ehM \cdot \nbhsM \|_{L^2(\hat\Gamma_{h,*}^{m+1})}^2 - \| \ehm \cdot \nahsm \|_{L^2(\Ghsm)}^2) + A_{h, *}(e_h^{m+1},\eM) \notag\\
		&= \frac{1}{2\tau}(\| \eM \cdot \nahsm \|_{L^2(\Ghsm)}^2 - \| \ehm \cdot \nahsm \|_{L^2(\Ghsm)}^2)  + A_{h, *}(e_h^{m+1},\eM) \notag\\
		&\quad+ \frac{1}{2\tau}(\| \ehM \cdot \nbhsM \|_{L^2(\hat\Gamma_{h,*}^{m+1})}^2 - \| \eM \cdot \nahsm \|_{L^2(\Ghsm)}^2) \notag\\
		&\le\int_{\Ghsm}\big(\frac{\eM-\ehm}{\tau}\cdot\nahsm\big)\big(\eM\cdot\nahsm\big)+A_{h, *}(e_h^{m+1},\eM) \notag\\
		&\quad+ \frac{1}{2\tau}(\| \ehM \cdot \nbhsM \|_{L^2(\hat\Gamma_{h,*}^{m+1})}^2 - \| \eM \cdot \nahsm \|_{L^2(\Ghsm)}^2) \notag\\
		&=\int_{\Ghsm}\big(\frac{\eM-\ehm}{\tau}\cdot\nahsm\big) I_h(\eM\cdot\nahsm)+A_{h, *}(e_h^{m+1},\nahsm I_h(\eM\cdot\nahsm)) \notag\\
		&\quad+\int_{\Ghsm}\big(\frac{\eM-\ehm}{\tau}\cdot\nahsm\big)\big( (1-I_h)(\eM\cdot\nahsm)\big)\notag\\
		&\quad+A_{h, *}(e_h^{m+1},\eM-\nahsm I_h(\eM\cdot\nahsm)) \notag\\
		&\quad+ \frac{1}{2\tau}(\| \ehM \cdot \nhsM \|_{L^2(\hat\Gamma_{h,*}^{m+1})}^2 - \| \eM \cdot \nhsm \|_{L^2(\Ghsm)}^2) \notag\\
		&=:\sum_{i=1}^{4}E_i^m,
	\end{align}
	where 
	%the first line on the right-hand side above can be estimated by choosing $\phi_h = I_h(\eM\cdot\nhsm)$ in the error equation \eqref{err} and using the estimates of the linear and bilinear forms developed in Sections \ref{sec:cons_err} and \ref{sec:J_stab}.
	the last term on the right-hand side above can be estimated by using \eqref{stab} : 
	\begin{align}
		E_4^m&\lesssim\epsilon^{-1}\tau^2 +\epsilon^{-1}\|\ehm\cdot\nahsm\|^2_{L^2(\Ghsm)}+\epsilon\| \nabla_{\Ghsm} \ehm \|_{L^2(\Ghsm)}^2.
	\end{align}
	The first term \(E_1^m\) can be rewritten as the following expression by selecting \(\phi_h := I_h(\eM \cdot \nahsm)\) as the test function in the error equation \eqref{err}:
	\begin{align}\label{def-E1}
		E_1^m=&A_{h, *}^T(\hat{e}_h^m,\nahsm I_h(\eM\cdot\nahsm)) - B^m(\hat{e}_h^m,\nahsm I_h(\eM\cdot\nahsm))-J_1^m(I_h(\eM\cdot\nahsm)) \notag\\
		&-J_2^m(I_h(\eM\cdot\nahsm))-K^m(\nahsm I_h(\eM\cdot\nahsm))-d_v^m(I_h(\eM\cdot\nahsm)).
	\end{align}
	To bound \(A_{h,*}^T(\hat{e}_h^m, \nahsm I_h(\eM\cdot \nahsm))\), we utilize the estimate \eqref{en}, from which the following inequality holds for any $\phi_h \in S_h(\Ghsm)$:
\begin{align}\label{estAT}
    & \quad\,\,|A_{h,*}^T(\ehm, \bar n_{h,*}^m \phi_h)| \notag\\
    &= \Big| \int_{\Ghsm} {\rm tr}\Big((\nabla_{\Ghsm} \ehm)(1 - \nhsm (\nhsm)^\top) (\nabla_{\Ghsm} (\bar n_{h,*}^m \phi_h))^\top \Big) \Big| \notag \\
    &\lesssim \Big| \int_{\Ghsm} {\rm tr}\Big((\nabla_{\Ghsm} \ehm)(1 - \nsm (\nsm)^\top) (\nabla_{\Ghsm} (\bar n_{h,*}^m \phi_h))^\top \Big) \Big| \notag \\
    &\quad + \| \nhsm - \nsm \|_{L^\infty(\Ghsm)} \| \nabla_{\Ghsm} \ehm \|_{L^2(\Ghsm)} \| \nabla_{\Ghsm} (\bar n_{h,*}^m \phi_h) \|_{L^2(\Ghsm)} \notag \\
    &\lesssim \Big( \|\nabla_{\Ghsm}((1 - \nsm (\nsm)^\top) \ehm)\|_{L^2(\Ghsm)} 
    + \|(\nabla_{\Ghsm}(1 - \nsm (\nsm)^\top)) \ehm\|_{L^2(\Ghsm)} \Big)\|\bar n_{h,*}^m \phi_h\|_{H^1(\Ghsm)} \notag \\
    &\quad + \| \nhsm - \nsm \|_{L^\infty(\Ghsm)} \| \nabla_{\Ghsm} \ehm \|_{L^2(\Ghsm)} \| \nabla_{\Ghsm} (\bar n_{h,*}^m \phi_h) \|_{L^2(\Ghsm)} \quad \text{(using Leibniz rule)} \notag \\
    &\lesssim \big(h \| \ehm \|_{H^1(\Ghsm)} + \| \ehm \|_{L^2(\Ghsm)} \big) \| \bar n_{h,*}^m \phi_h \|_{H^1(\Ghsm)} 
    \quad \text{(using inequality \eqref{en})} \notag \\
    &\quad + h^{1.6} \| \nabla_{\Ghsm} \ehm \|_{L^2(\Ghsm)} \| \nabla_{\Ghsm} (\bar n_{h,*}^m \phi_h) \|_{L^2(\Ghsm)} 
    \quad \text{(using inequality \eqref{nsm})} \notag \\
    &\lesssim \| \ehm \|_{L^2(\Ghsm)} \| \bar n_{h,*}^m \phi_h\|_{H^1(\Ghsm)} 
    \quad \text{(using inverse inequality)} .
\end{align}
Therefore, by choosing $\phi_h= I_h(\eM\cdot\nahsm)$ in the inequality above, we have 
	\begin{align}
		&\,\, A_{h, *}^T(\hat{e}_h^m,\nahsm I_h(\eM\cdot\nahsm))\notag\\
		% &\hspace{20pt}\mbox{($\nahsm I_h(\eM\cdot\nahsm)$ is a finite element function of higher order)}\notag\\
		\lesssim&\| \ehm \|_{L^2(\Ghsm)}\|\nahsm I_h(\eM\cdot\nahsm) \|_{H^1(\Ghsm)}\notag\\
		\lesssim& \| \ehm \|_{L^2(\Ghsm)}\| I_h(\eM\cdot\nahsm) \|_{H^1(\Ghsm)}\quad\text{(using \eqref{nhmW-1})}\notag\\
		\lesssim&\|\ehm\|_{L^2(\Ghsm)}\|  \eM \|_{H^1(\Ghsm)}\lesssim \epsilon^{-1}\|\ehm\|^2_{L^2(\Ghsm)}+\epsilon\|  \eM \|_{H^1(\Ghsm)}^2 ,\label{AAA}
	\end{align}
	where the stability estimate \eqref{stability7} and the $W^{1,\infty}$-boundedness of $\bar n_{h,*}^m$ in \eqref{nhmW-1} are used in the second-to-last inequality. Furthermore, we can estimate \(J_1^m(I_h(\eM \cdot \nahsm))\) using \eqref{estJ1-0}, estimate \(J_2^m(I_h(\eM \cdot \nahsm))\) using \eqref{estJ2}, and estimate \(d_v^m(I_h(\eM \cdot \nahsm))\) using \eqref{dv}, i.e., 
	\begin{align}
		&\,\,J_1^m(I_h(\eM\cdot\nahsm))\notag\\
		\lesssim&\| \ehvm\|_{L^2(\Ghsm)}\| \nabla_{\Ghsm} \ehm \|_{L^2(\Ghsm)} \| I_h(\eM\cdot\nahsm) \|_{L^\infty(\Ghsm)}\notag\\
		&+\| \nabla_{\Ghsm} \ehm \|_{L^2(\Ghsm)} \| I_h(\eM\cdot\nahsm) \|_{L^2(\Ghsm)}\notag\\
		\lesssim&C_\epsilon h^{-\epsilon}\| \ehvm\|_{L^2(\Ghsm)}\| \nabla_{\Ghsm} \ehm \|_{L^2(\Ghsm)} \| I_h(\eM\cdot\nahsm) \|_{H^1(\Ghsm)}\quad\text{(using \eqref{eq:poincare3})}\notag\\
		&+\| \nabla_{\Ghsm} \ehm \|_{L^2(\Ghsm)} \| I_h(\eM\cdot\nahsm) \|_{L^2(\Ghsm)}\notag\\
		\lesssim&h^{0.05}\| \nabla_{\Ghsm} \ehm \|_{L^2(\Ghsm)} \| I_h(\eM\cdot\nahsm) \|_{H^1(\Ghsm)}\quad\mbox{(using \eqref{evml-rough} and choosing $\epsilon=0.05$)}\notag\\
		&+\| \nabla_{\Ghsm} \ehm \|_{L^2(\Ghsm)} \| I_h(\eM\cdot\nahsm) \|_{L^2(\Ghsm)}\notag\\
		\lesssim&\epsilon^{-1}\|\eM\|^2_{L^2(\Ghsm)}+(\epsilon + h^{0.05})\big(\|\ehm\|^2_{H^1(\Ghsm)}+\|\eM\|^2_{H^1(\Ghsm)}\big) \notag\\
		&\,\,\text{(using \eqref{nhmW-1}, \eqref{stability2}, \eqref{stability7} and Young's inequality)},\label{J1J1J1}\\
		&\,\,J_2^m(I_h(\eM\cdot\nahsm))\notag\\
		\lesssim&\|\nabla_{\Ghsm}\ehm\|_{L^2(\Ghsm)}\|I_h(\eM\cdot\nahsm)\|_{L^2(\Ghsm)}+h^{0.5}\|\nabla_{\Ghsm}\ehm\|_{L^2(\Ghsm)}\|I_h(\eM\cdot\nahsm)\|_{H^1(\Ghsm)}\notag\\
		&+(h^{-2.1}\tau+h^{-2.1}\tau\|\ehvm\|_{L^2(\Ghsm)})\|\nabla_{\Ghsm}\ehm\|_{L^2(\Ghsm)}\|I_h(\eM\cdot\nahsm)\|_{H^1(\Ghsm)}\notag\\
		\lesssim&\epsilon^{-1}\|e_h^{m+1}\|^2_{L^2(\Ghsm)}+(\epsilon+h^{0.5})\big(\|\ehm\|^2_{H^1(\Ghsm)}+\|\eM\|^2_{H^1(\Ghsm)}\big),\notag\\
		& \,\,\text{(using \eqref{evml-rough}, stability estimates \eqref{stability2} and \eqref{stability7}, $\tau\le ch^k$ with $k\ge 3$)},\label{J2J2J2}\\
		&\,\,d_v^m(I_h(\eM\cdot\nahsm))\notag\\
		\lesssim &(1+\kappa_{*,l}h^{k-1})\tau \|I_h(\eM\cdot\nahsm)\|_{L^2(\Ghsm)}+(1+\kappa_{*,l})h^k\|I_h(\eM\cdot\nahsm)\|_{H^1(\Ghsm)}\notag\\
		\lesssim&\epsilon^{-1}[\tau^2+(1+\kappa_{*,l})^2h^{2k}]+\epsilon^{-1}\|\eM\|^2_{L^2(\Ghsm)}+\epsilon\| \eM \|_{H^1(\Ghsm)}^2,\notag\\
		& \,\,\text{(using the assumption \eqref{cond1}, stability estimates \eqref{stability2} and \eqref{stability7})}\label{ddd}.
	\end{align}
% The subsequent estimate follows from proof in deriving \cite[Eq.~(5.17)]{bai2024new} and \cite[Eq.~(5.22)]{bai2024new}:
The following estimate can be obtained with an identical proof to that of \cite[Eq.~(5.17)]{bai2024new}:
	\begin{align}
		& \quad\,\,B^m(\hat{e}_h^m,\nahsm I_h(\eM\cdot\nahsm))\notag\\
		&\lesssim \|\ehm\|_{L^2(\Ghsm)}\|\nahsm I_h(\eM\cdot\nahsm)\|_{H^1(\Ghsm)}\lesssim \epsilon^{-1}\|\ehm\|^2_{L^2(\Ghsm)}+\epsilon\|  \eM \|_{H^1(\Ghsm)}^2,\notag\\
		& \quad\,\,\text{(using \eqref{nhmW-1} and \eqref{stability7})}.\label{error-B}
	\end{align}
By employing arguments similar to those used in the proof of the estimates \eqref{K_123^m_inter}--\eqref{K_123^m}, we obtain
% \eqref{K_4^m} and \eqref{K_5^m}, 
\begin{align}\label{K_123^m_err}
	& \quad \,\,\sum_{j=1}^3 |K_j^m(\nahsm I_h(e_h^{m+1}\cdot\nahsm))|\notag\\
	%  &\lesssim \|\nabla_{\Ghsm} \hat e_h^m\|_{L^\infty(\Ghsm)} ( \|\nabla_{\Ghsm} \hat e_h^m\|_{L^2(\Ghsm)} +  \|\nabla_{\Ghsm} e_h^{m+1}\|_{L^2(\Ghsm)})\|\nabla_{\Ghsm}  (\nahsm I_h(\ehvm\cdot\nahsm))\|_{L^2(\Ghsm)}\notag\\
	 & \lesssim \|\nabla_{\Ghsm} \hat e_h^m\|_{L^\infty(\Ghsm)} ( \|\nabla_{\Ghsm} \hat e_h^m\|_{L^2(\Ghsm)} + h^{-1}\tau \|\ehvm\|_{L^2(\Ghsm)}+\tau )\|\nahsm I_h(e_h^{m+1}\cdot\nahsm)\|_{H^1(\Ghsm)}\notag\\
	 & \quad\,\, \text{(using \eqref{K_123^m_inter}--\eqref{emev-a} and the inverse inequality)}\notag\\
	 & \lesssim h^{0.6}\| \nabla_{\Ghsm} \ehm \|_{L^2(\Ghsm)}\|\nahsm I_h(e_h^{m+1}\cdot\nahsm)\|_{H^1(\Ghsm)}\notag\\
     & \quad\,\, \text{(using \eqref{evml-rough}, \eqref{cond1}, \eqref{Linfty-W1infty-hat-em}, $\tau \le ch^k$ with $k\ge 3$, and inverse inequality)}\notag\\
     &\lesssim h^{0.6}\big(\|\ehm\|^2_{H^1(\Ghsm)}+\|\eM\|^2_{H^1(\Ghsm)}\big), \quad \text{(using \eqref{nhmW-1} and \eqref{stability7})}.
\end{align}
By employing arguments similar to those used in the proofs of the estimates \eqref{K_4^m} and \eqref{K_5^m}, we have
\begin{align}\label{K_45^m_err}
	&\quad\,\,|K_4^m(\nahsm I_h(e_h^{m+1}\cdot\nahsm))|+|K_5^m(\nahsm I_h(e_h^{m+1}\cdot\nahsm))| \notag\\
	& \lesssim h^{0.6}\| \nabla_{\Ghsm} \ehm \|_{L^2(\Ghsm)}\|\nahsm I_h(e_h^{m+1}\cdot\nahsm)\|_{H^1(\Ghsm)}\notag\\
    & \quad+ h^{-1}( \tau \|\hat e_v^m\|_{L^2(\Ghsm)} + \tau ) \|\nabla_{\Ghsm} \hat e_h^m\|_{L^\infty(\Ghsm)} \|\nabla_{\Ghsm}  (\nahsm I_h(e_h^{m+1}\cdot\nahsm))\|_{L^2(\Ghsm)}\notag\\
	& \quad + \tau \|I_h (-H^m n^m +g^m)\|_{W^{1,\infty}(\Ghsm)}\|\nabla_{\Ghsm} \hat e_h^m\|_{L^2(\Ghsm)} \|\nabla_{\Ghsm}  (\nahsm I_h(e_h^{m+1}\cdot\nahsm))\|_{L^2(\Ghsm)}\notag\\
	& \lesssim h^{0.6}\| \nabla_{\Ghsm} \ehm \|_{L^2(\Ghsm)}\|\nahsm I_h(e_h^{m+1}\cdot\nahsm)\|_{H^1(\Ghsm)}\notag\\
    & \quad \,\,\text{(using \eqref{evml-rough}, \eqref{cond1}, \eqref{Linfty-W1infty-hat-em}, $\tau \le ch^k$ with $k\ge 3$, and inverse inequality)}\notag\\
    &\lesssim h^{0.6}\big(\|\ehm\|^2_{H^1(\Ghsm)}+\|\eM\|^2_{H^1(\Ghsm)}\big), \quad \text{(using \eqref{nhmW-1} and \eqref{stability7})}.
\end{align}
By combining the estimates \eqref{K_123^m_err} and \eqref{K_45^m_err}, we obtain
	\begin{align}
		K^m(\nahsm I_h(\eM\cdot\nahsm))\lesssim h^{0.6}\big(\|\ehm\|^2_{H^1(\Ghsm)}+\|\eM\|^2_{H^1(\Ghsm)}\big).
	\end{align}
By combining the estimates derived above and applying the inequalities
 \eqref{eq:e_NT} and \eqref{ehMM-copy}, we obtain the following estimate for the $E_1^m$ defined in \eqref{def-E1}:
	\begin{align}\label{E1E1}
		E_1^m&\lesssim\epsilon^{-1}[\tau^2+(1+\kappa_{*,l})^2h^{2k}]+\epsilon^{-1}\|\ehm\cdot\nahsm\|^2_{L^2(\Ghsm)}+(\epsilon + h^{0.05})\| \nabla_{\Ghsm} \ehm \|_{L^2(\Ghsm)}^2.
	\end{align}
	By utilizing the modified velocity error estimate in \eqref{evmL}, the orthogonality property between \(\Tahsm\) and \(\nahsm\), and the super-approximation estimate \eqref{super2} in Lemma \ref{super}, the following inequality holds:
	\begin{align}\label{E2E2E2}
		E_2^m&=\int_{\Ghsm}\Big(\frac{\eM-\ehm}{\tau}\cdot\nahsm \Big)\Big((1-I_h)(\eM\cdot\nahsm)\Big)\notag\\
		&=\int_{\Ghsm}\Big(\Big(\frac{\eM-\ehm}{\tau}-I_h\Tsm v^m\Big)\cdot\nahsm\Big)\Big( (1-I_h)(\eM\cdot\nahsm)\Big)\notag\\
		&\quad+\int_{\Ghsm}\Big((I_h\Tsm v^m-\Tahsm v^m)\cdot\nahsm\Big)\Big( (1-I_h)(\eM\cdot\nahsm)\Big)\quad\text{(using \(\Tahsm \perp \nahsm\))}\notag\\
		&\lesssim \| \nahsm\|_{L^\infty(\Ghsm)}\Big(\| \ehvm \|_{L^2(\Ghsm)}+\|I_h\Tsm v^m -\Tsm v^m \|_{L^2(\Ghsm)}\notag\\
		& \quad+\| \Tsm v^m-\Tahsm v^m\|_{L^2(\Ghsm)}\Big)h\| \nahsm\|_{W^{1,\infty}(\Ghsm)}\| \eM \|_{L^2(\Ghsm)}\quad \text{(using \eqref{super2})}\notag\\
		% &\hspace{200pt}\mbox{(Lemma \ref{super} is used)}\notag\\
		&\lesssim (h^{-1}\tau+(1+\kappa_{*,l})h^{k-1}+h^{-1}\| \nabla_{\Ghsm} \ehm \|_{L^2(\Ghsm)}+h^{-3.1}\| \nabla_\Ghsm \ehm \|_{L^2(\Ghsm)}^2)\notag\\
		& \quad \cdot h\| \eM \|_{L^2(\Ghsm)}\quad \mbox{(using \eqref{Ihf}, \eqref{eq:nsa2},  \eqref{nhmW-1} and \eqref{evmL})}\notag\\
		&\lesssim\epsilon^{-1}[\tau^2+(1+\kappa_{*,l})^2h^{2k}]+\epsilon^{-1}\|\ehm\cdot\nahsm\|^2_{L^2(\Ghsm)}+(\epsilon+h^{0.5})\| \nabla_{\Ghsm} \ehm \|_{L^2(\Ghsm)}^2,
	\end{align} 
	where the last inequality follows from the estimate \eqref{eq:e_NT} and the inequality 
	\begin{align*}
		h^{-3.1} h \|e_h^{m+1}\|_{L^2(\Ghsm)} \lesssim h^{0.5}, \quad\text{(using \eqref{eM})}.
	\end{align*}

	The term \(E_3^m\) can be decomposed into the following distinct components:
	\begin{align}
		E_3^m&=A_{h, *}(e_h^{m+1},\eM-\nahsm I_h(\eM\cdot\nahsm)) \notag\\
		&=A_{h, *}(e_h^{m+1},I_h\Tahsm\eM+I_h\Nahsm\eM-\nahsm I_h(\eM\cdot\nahsm)) \notag\\
		&=A_{h, *}(e_h^{m+1},I_h\Tahsm\eM) +A_{h, *}(e_h^{m+1},(I_h-1)(\bar N_{h,*}^m e_h^{m+1})) \notag\\
		& \quad\,\,+ A_{h,*}(e_h^{m+1}, (\frac{1}{|\bar n_{h,*}^m|^2} - 1)\bar n_{h,*}^m (e_{h}^{m+1} \cdot \bar n_{h,*}^m)) \notag\\
		& \quad\,\,+ A_{h,*}(e_h^{m+1}, \bar n_{h,*}^m((e_h^{m+1} \cdot \bar n_{h,*}^m) - I_h(e_h^{m+1} \cdot \bar n_{h,*}^m)))\notag\\
		& =: E_{31}^m +E_{32}^m+E_{33}^m+E_{34}^m.\label{est-E_3m}
	\end{align}
	By using the definition of the bilinear form $A_{h,*}(\cdot,\cdot)$ and employing the super-approximation properties \eqref{super2_projected-surface} and \eqref{super2}, together with the $W^{1,\infty}$-boundeness of $\bar n_{h,*}^m$ in \eqref{nhmW-1}, we obtain
	\begin{align}\label{est-E_324m}
		&\quad\,\, E_{32}^m  + E_{34}^m\notag\\
		& \lesssim \|\nabla_{\Ghsm} e_h^{m+1}\|_{L^2(\Ghsm)} h \|e_h^{m+1}\|_{H^1(\Ghsm)}\quad\text{(using \eqref{nhmW-1})}\notag\\
		& \lesssim h(\tau^2+ \|\hat e_h^m \cdot \bar n_{h,*}^m\|_{L^2(\Ghsm)}^2) + h \|\nabla_{\Ghsm}\hat e_h^m\|_{L^2(\Ghsm)}^2 \quad\text{(using \eqref{ehMM-copy})}.
	\end{align}
	Furthermore, by using the definition of the bilinear form $A_{h,*}(\cdot,\cdot)$ and the following relation inferred from $1 = |n^m_*|$, 
	\begin{align*}
		\frac{1}{|\bar n_{h,*}^m|^2} - 1 = \frac{1}{|\bar n_{h,*}^m|^2} - \frac{1}{|n_*^m|^2} = \frac{|\bar n_{h,*}^m|^2 - |n_*^m|^2}{|\bar n_{h,*}^m|^2 |n_*^m|^2} = \frac{(\bar n_{h,*}^m - n_*^m)\cdot (\bar n_{h,*}^m + n_*^m)}{|\bar n_{h,*}^m|^2},
	\end{align*}
	together with the lower bound of $\||\bar n_{h,*}^m|\|_{L^\infty(\Ghsm)}$ in \eqref{average-1}, and the inequality \eqref{nhmW-1}, the following estimate holds:
	\begin{align}\label{est-E_33}
		E_{33}^m &\lesssim  \|\nabla_{\Ghsm} e_h^{m+1}\|_{L^2(\Ghsm)} \|\bar n_{h,*}^m - n_*^m\|_{W^{1,\infty}(\Ghsm)}\|e_h^{m+1}\|_{H^1(\Ghsm)}\notag\\
		& \lesssim h^{0.6} \|e_h^{m+1}\|_{H^1(\Ghsm)}^2\quad \text{(using \eqref{eq:nsa2-infty-1} and the assumption \eqref{cond1})}\notag\\
		& \lesssim h^{0.6}(\tau^2+ \|\hat e_h^m \cdot \bar n_{h,*}^m\|_{L^2(\Ghsm)}^2) + h^{0.6} \|\nabla_{\Ghsm}\hat e_h^m\|_{L^2(\Ghsm)}^2\quad\text{(using \eqref{ehMM-copy})}.
	\end{align}
	By using the relation \eqref{evm}, we obtain
	\begin{align}\label{est-E_31}
		E_{31}^m
		&=A_{h, *}(\tau\ehvm+\tau I_h\Tsm v^m,I_h\Tahsm\eM) +A^T_{h, *}(\ehm,I_h\Tahsm\eM)+A^N_{h, *}(\ehm,I_h\Tahsm\eM) \notag\\
		& =: E_{311}^m + E_{312}^m + E_{313}^m.
	\end{align}
	By employing \eqref{eq:vel_H1} and the super-approximation estiamte \eqref{super1_projected-surface}, the following result can be obtained:
	\begin{align}\label{est-E_311}
		E_{311}^m
		&\lesssim \tau(\|\ehvm\|_{H^1(\Ghsm)}+1)(\|(I_h -1)\bar T_{h,*}^m \eM\|_{H^1(\Ghsm)}+\|\bar T_{h,*}^m \eM\|_{H^1(\Ghsm)})\notag\\
		& \lesssim\tau(\|\ehvm\|_{H^1(\Ghsm)}+1) (h\|\eM\|_{H^1(\Ghsm)}+\|\eM\|_{H^1(\Ghsm)})\quad\text{(using \eqref{super1_projected-surface} and \eqref{nhmW-1})}\notag\\
			&\lesssim \tau(h^{-2} \tau + (1 + \kappa_{*,l})h^{k-2} + h^{-2} \| \nabla_{\Ghsm} \ehm \|_{L^2(\Ghsm)}+1)\|\eM\|_{H^1(\Ghsm)} \quad \text{(using \eqref{eq:vel_H1})}\notag\\
			& \lesssim \epsilon^{-1} \tau^2 + (h+\epsilon)\|e_h^{m+1}\|_{H^1(\Ghsm)}^2 + h\|\nabla_{\Ghsm} \hat e_h^m\|_{L^2(\Ghsm)}^2 \notag\\
			&\quad\text{(using $\tau \le c h^k$ with $k\ge 3$ and \eqref{cond1})}\notag\\
			& \lesssim \epsilon^{-1}(\tau^2+ \|\hat e_h^m \cdot \bar n_{h,*}^m\|_{L^2(\Ghsm)}^2) + (h+\epsilon) \|\nabla_{\Ghsm}\hat e_h^m\|_{L^2(\Ghsm)}^2\quad\text{(using \eqref{ehMM-copy})}.
		\end{align}
		Following the same argument as in the proof of \eqref{estAT}, we obtain
	\begin{align}\label{est-E_312}
		E_{312}^m & \lesssim \|\hat e_h^m\|_{L^2(\Ghsm)}\|I_h \bar T_{h,*}^m e_h^{m+1}\|_{H^1(\Ghsm)}\notag\\
		& \lesssim \|\hat e_h^m\|_{L^2(\Ghsm)}(\|\bar T_{h,*}^m e_h^{m+1}\|_{H^1(\Ghsm)} + h\|e_h^{m+1}\|_{H^1(\Ghsm)})\quad\text{(using \eqref{super1_projected-surface})}\notag\\
		& \lesssim \|\hat e_h^m\|_{L^2(\Ghsm)}\|e_h^{m+1}\|_{H^1(\Ghsm)} \quad\text{(using \eqref{nhmW-1})}\notag\\
		& \lesssim \epsilon^{-1}(\tau^2+ \|\hat e_h^m \cdot \bar n_{h,*}^m\|_{L^2(\Ghsm)}^2) + \epsilon\|\nabla_{\Ghsm}\hat e_h^m\|_{L^2(\Ghsm)}^2\quad\text{(using \eqref{ehMM-copy})}.
	\end{align}

	By using the orthogonality relation between $\bar T_{h,*}^m$ and $ \bar N_{h,*}^m$, the following estimate holds:
	\begin{align}\label{noraml}
		&\quad\,\,|A_{h,*}^N(\ehm, \Tahsm e_h^{m+1})| \notag\\
		&= \Big|  \int_{\Ghsm} [(\nabla_{\Ghsm}\hat e_h^m) \hat n_{h,*}^m]\cdot [(\nabla_{\Ghsm}(\bar T_{h,*}^m e_h^{m+1}))\hat n_{h,*}^m]  \Big|\notag\\
		& \lesssim  \Big|  \int_{\Ghsm} [(\nabla_{\Ghsm}\hat e_h^m) \hat n_{h,*}^m]\cdot [(\nabla_{\Ghsm}(\bar T_{h,*}^m e_h^{m+1}))(\hat n_{h,*}^m - \bar n_{h,*}^m)]  \Big|\notag\\
		% & \quad\,\, + \Big|  \int_{\Ghsm} [(\nabla_{\Ghsm}\hat e_h^m) \hat n_{h,*}^m]\cdot [(\nabla_{\Ghsm}(\bar T_{h,*}^m e_h^{m+1}))\bar n_{h,*}^m]  \Big|\notag\\
		& \quad\,\,  + \Big|  \int_{\Ghsm} [(\nabla_{\Ghsm}\hat e_h^m) \hat n_{h,*}^m]\cdot [\nabla_{\Ghsm}(\bar T_{h,*}^m e_h^{m+1}\cdot\bar n_{h,*}^m)]  \Big|\notag\\
		& \quad\,\, + \Big|  \int_{\Ghsm} [(\nabla_{\Ghsm}\hat e_h^m) \hat n_{h,*}^m]\cdot [(\nabla_{\Ghsm}\bar n_{h,*}^m)\bar T_{h,*}^m e_h^{m+1}]  \Big|\quad\text{(using Leibniz rule)}\notag\\
		& \lesssim \|\nabla_{\Ghsm}\hat e_h^m\|_{L^2(\Ghsm)}\|\nabla_{\Ghsm}\bar T_{h,*}^m e_h^{m+1}\|_{L^2(\Ghsm)} \|\hat n_{h,*}^m - \bar n_{h,*}^m\|_{L^\infty(\Ghsm)} \notag\\
		& \quad\,\,+  \|\nabla_{\Ghsm}\hat e_h^m\|_{L^2(\Ghsm)}\|\bar T_{h,*}^m e_h^{m+1}\|_{L^2(\Ghsm)} \quad\text{(using the orthogonality between $\bar T_{h,*}^m$ and $ \bar n_{h,*}^m$) }\notag\\
		& \lesssim (1+\kappa_{*,l})h^{k-1}\|\nabla_{\Ghsm}\hat e_h^m\|_{L^2(\Ghsm)}\| e_h^{m+1}\|_{H^1(\Ghsm)} + \|\nabla_{\Ghsm}\hat e_h^m\|_{L^2(\Ghsm)}\| e_h^{m+1}\|_{L^2(\Ghsm)}\notag\\
		&\,\,\quad\text{(using \eqref{eq:nsa4-infty} and \eqref{nhmW-1})}\notag\\
		& \lesssim h^{1.6} \|\nabla_{\Ghsm}\hat e_h^m\|_{L^2(\Ghsm)}\| e_h^{m+1}\|_{H^1(\Ghsm)} + \|\nabla_{\Ghsm}\hat e_h^m\|_{L^2(\Ghsm)}\| e_h^{m+1}\|_{L^2(\Ghsm)}\quad\text{(using \eqref{cond1})}\notag\\
		& \lesssim \epsilon^{-1}(\tau^2+ \|\hat e_h^m \cdot \bar n_{h,*}^m\|_{L^2(\Ghsm)}^2) + \epsilon\|\nabla_{\Ghsm}\hat e_h^m\|_{L^2(\Ghsm)}^2,
	\end{align}
	where the inequality \eqref{eq:e_NT} and the inverse inequality \eqref{inverse-ineq} are used in deriving the last inequality. Furthermore, the above inequality \eqref{noraml} and the super-approximation estimate \eqref{super1_projected-surface} together yield the following estimate for $E_{313}^m$:
	\begin{align}
		&\quad\,\,E_{313}^m = A_{h,*}^N(\ehm, I_h\Tahsm e_h^{m+1}) \notag\\
		&\le |A_{h,*}^N(\ehm, (I_h-1)\Tahsm e_h^{m+1})| + |A_{h,*}^N(\ehm, \Tahsm e_h^{m+1})|  \notag\\
		& \lesssim \|\nabla_{\Ghsm}\hat e_h^m\|_{L^2(\Ghsm)}  h \|e_h^{m+1}\|_{H^1(\Ghsm)} +\epsilon^{-1}(\tau^2+ \|\hat e_h^m \cdot \bar n_{h,*}^m\|_{L^2(\Ghsm)}^2) + \epsilon\|\nabla_{\Ghsm}\hat e_h^m\|_{L^2(\Ghsm)}^2\notag\\
		& \lesssim \epsilon^{-1}(\tau^2+ \|\hat e_h^m \cdot \bar n_{h,*}^m\|_{L^2(\Ghsm)}^2) + \epsilon\|\nabla_{\Ghsm}\hat e_h^m\|_{L^2(\Ghsm)}^2,
	\end{align}
	where the above inequality follows from \eqref{eq:e_NT} and the inverse inequality \eqref{inverse-ineq}.

By collecting the estimates for \(E_{31j}^m\), where \(j = 1, 2,3\), we obtain
\begin{align}\label{est-E_31-est}
	|E_{31}^m| \lesssim \epsilon^{-1}(\tau^2+ \|\hat e_h^m \cdot \bar n_{h,*}^m\|_{L^2(\Ghsm)}^2) + (\epsilon + h)\|\nabla_{\Ghsm}\hat e_h^m\|_{L^2(\Ghsm)}^2.
\end{align}
By aggregating the estimates for \(E_{3j}^m\), where \(j = 1, 2,3,4\), we obtain
\begin{align}\label{est-E_3-est}
	|E_{3}^m| \lesssim \epsilon^{-1}(\tau^2+ \|\hat e_h^m \cdot \bar n_{h,*}^m\|_{L^2(\Ghsm)}^2) + (\epsilon + h^{0.6})\|\nabla_{\Ghsm}\hat e_h^m\|_{L^2(\Ghsm)}^2.
\end{align}
Moreover, by invoking the definition of the bilinear form \(A_{h,*}(\cdot,\cdot)\), and the estimate \eqref{nabla-hat-withouthat}, we deduce:
\begin{align}\label{A_hh_hat_without_relation}
	\|\nabla_{\Ghsm} \hat e_h^{m+1}\|_{L^2(\Ghsm)}&\lesssim \|\nabla_{\Ghsm} e_h^{m+1}\|_{L^2(\Ghsm)} + \|e_h^{m+1}\|_{L^2(\Ghsm)}\notag\\
	&=A_{h,*}(e_h^{m+1}, e_h^{m+1}) + \|\hat e_h^m \cdot \bar n_{h,*}^m\|_{L^2(\Ghsm)} +\tau,
\end{align}
where the last inequality uses \eqref{eq:e_NT}.

By aggregating the estimates for \(E_j^m\) (with \(j = 1,2,3,4\)) and applying \eqref{A_hh_hat_without_relation} to the inequality \eqref{Err}, we obtain the following inequality for $0\le m\le l$:
	\begin{align}
		& \quad\,\,\frac{1}{2\tau} \big(\| \ehM \cdot \nbhsM \|_{L^2(\hat\Gamma_{h,*}^{m+1})}^2 - \| \ehm \cdot \nahsm \|_{L^2(\Ghsm)}^2\big) 
		+ C^{-1}\| \nabla_{\hat \Gamma_{h,*}^{m}} \ehM \|_{L^2(\Ghsm)}^2 \notag\\
		&\lesssim \epsilon^{-1} \big[\tau^2 + (1+\kappa_{*,l})^2 h^{2k}\big] 
		+ \epsilon^{-1} \|\ehm \cdot \nahsm\|_{L^2(\Ghsm)}^2 
		+ (\epsilon + h^{0.05}) \| \nabla_{\Ghsm} \ehm \|_{L^2(\Ghsm)}^2, \label{Err1}
	\end{align}
	where \(\epsilon\) is an arbitrarily small positive constant. It is important to note that the constant \(\kappa_{*,l}\) on the right-hand side of \eqref{Err1} can be replaced by \(\kappa_{*,m}\), as the preceding analysis is based entirely on the surface \(\Ghsm\) rather than \(\hat{\Gamma}^l_{h,*}\). By applying the discrete Grönwall inequality, together with the $H^1$-norm equivalence between \(\Ghsm\) and \(\hat\Gamma_{h,*}^{m+1}\), we derive the following error estimate for a constant \(C_{\kappa_l}\), which may depend on \(\kappa_l\) as defined in \eqref{kl*}:
	\begin{align}\label{eq:err_fin0}
		\max_{0 \leq m \leq l} \| \hat{e}_h^{m+1} \|_{L^2(\Ghsm)}^2 
		+ \sum_{m = 0}^{l} \tau \| \nabla_{\Ghsm} \hat{e}_h^{m+1} \|_{L^2(\Ghsm)}^2 
		\leq C_{\kappa_l} \Big(\tau + \Big(1 + \sum_{m=0}^{l} \tau \kappa_{*,m}^2\Big)^{\frac{1}{2}} h^k\Big)^2,
	\end{align}
	where \eqref{eq:e_NT} is used in deriving the above inequality.
	% Furthermore, in light of \eqref{eq:e_NT} and \eqref{ehMM-copy}, we can also establish the following result:
	% \begin{align}\label{eq:err_fin1}
	% 	\max_{0 \leq m \leq l} \| e_h^{m+1} \|_{L^2(\Ghsm)}^2 
	% 	+ \sum_{m = 0}^{l} \tau \| \nabla_{\Ghsm} e_h^{m+1} \|_{L^2(\Ghsm)}^2 
	% 	\leq C_{\kappa_l} \Big(\tau + \Big(1 + \sum_{m=0}^{l} \tau \kappa_{*,m}^2\Big)^{\frac{1}{2}} h^k\Big)^2.
	% \end{align}
	
    The uniform boundedness of the constants \(\kappa_{l}\), \(\kappa_{*,l}\), \(C_\#\), and \(C_*\), independent of \(\tau\), \(h\), and \(l\), is established in Appendix~\ref{appendix_H} under the stepsize condition \(\tau \le c h^{k}\), based on the error estimates \eqref{eq:err_fin0}. Consequently, the mathematical induction assumption \eqref{cond2} is verified for \(m = l + 1\) for sufficiently small \(h \le h_0\), using the uniform boundedness of these constants and the error estimates \eqref{eq:err_fin0}. In parallel, the mesh size condition \eqref{cond1} is also fulfilled for sufficiently small \(h \le h_0\), again relying on the uniform boundedness of \(\kappa_{l}\), \(\kappa_{*,l}\), \(C_\#\), and \(C_*\). Therefore, the following error estimate holds:
	\begin{align}\label{eror-inbter}
		\max_{0 \leq m \leq \lfloor {T}/{\tau} \rfloor} \| \hat{e}_h^{m} \|_{L^2(\Ghsm)}^2 
		+ \sum_{m = 0}^{\lfloor {T}/{\tau} \rfloor} \tau \| \nabla_{\Ghsm} \hat{e}_h^{m} \|_{L^2(\Ghsm)}^2 
		\leq C_0 (\tau^2 + h^{2k}),
	\end{align}
	where the constant $C_0$ is independent of $\tau$, $h$ (but may depend on $T$ and $\kappa_0$ defined in \eqref{P0}).

	% This completes the proof of Theorem \ref{thm:main}. 

	% Moreover, the error estimate for the distance error from the numerically computed surface \(\Gamma_h^m\) to the smooth surface \(\Gamma^m\), defined by \(\hat d^m(x) := \min_{y \in \Gamma^m} |x - y|\) for \(x \in \Gamma_h^m\) in \eqref{def-distance-error-0}, can be derived. The distance error \(\hat d^m(x)\), defined on \(\Gamma_h^m\), can be controlled as follows:
	% \begin{align}
	% 	|\hat d^m(x)| \le |\hat X_{h,*}^m \circ (X_h^m)^{-1}(x)- x| + |a^m(\hat X_{h,*}^m \circ (X_h^m)^{-1}(x)) - \hat X_{h,*}^m \circ (X_h^m)^{-1}(x)|.
	% \end{align}
	% The above inequality holds due to the fact that \(a^m(\hat X_{h,*}^m \circ (X_h^m)^{-1}(x))\) lies on the exact smooth surface \(\Gamma^m\), and \(|\hat d^m(x)|\) characterizes the shortest distance from \(x \in \Gamma_h^m\) to the smooth surface \(\Gamma^m\).

	% \begin{align}
	% 	\hat e^m(x) &  = a^m(x) -a^m(\hat X_{h,*}^m \circ (X_h^m)^{-1}(x)) + a^m(\hat X_{h,*}^m \circ (X_h^m)^{-1}(x)) - \hat X_{h,*}^m \circ (X_h^m)^{-1}(x) \notag\\
	% 	& \quad\,\,+\hat X_{h,*}^m \circ (X_h^m)^{-1}(x)- x, \quad \text{for }\forall x \in \Gamma_h^m.
	% \end{align}
	% Subsequently, using the equivalence of the \(W^{i,p}\)-norms for \(i = 0, 1\) and \(1 \le p \le \infty\) between the numerically computed surface \(\Gamma_h^m\) and the interpolated surface \(\hat{\Gamma}_h^m\), as established in \eqref{Linfty-W1infty-hat-em} and Lemma~\ref{equi-MA}, we obtain

\subsection{Distance error from $\Ghm$ to $\Gamma^m$}

For the discrete flow maps $ X_{h}^m:\Gamma_{h,{\rm f}}^0\rightarrow\Ghm$, $m=0,\cdots, \lfloor {T}/{\tau} \rfloor$, we denote
	\begin{equation}\label{kl}
		\begin{aligned}
			\hat C_1&:=\max\limits_{0\leq m\leq \lfloor {T}/{\tau} \rfloor}(\|{X}^m_{h}\|_{H^{k-1}_h(\Gamma^0_{h,{\rm f}})}+\|{X}^m_{h}\|_{W^{k-2,\infty}_h(\Gamma^0_{h,{\rm f}})}+\|({X}^m_{h})^{-1}\|_{W^{1,\infty}(\Gamma^m_{h})}),\\
			\hat C_2&:=\max\limits_{0\leq m\leq \lfloor {T}/{\tau} \rfloor}(\|{X}^m_{h}\|_{H^{k}_h(\Gamma^0_{h,{\rm f}})}+\|{X}^m_{h}\|_{W^{k-1,\infty}_h(\Gamma^0_{h,{\rm f}})}).
		\end{aligned}
	\end{equation}
From the error estimate \eqref{eror-inbter}, the uniform boundedness of \(\kappa_{l}\), \(\kappa_{*,l}\), \(C_\#\), and \(C_*\), and the use of inverse inequalities, we conclude that the constants \(\hat C_1\) and \(\hat C_2\) are uniformly bounded, independent of \(\tau\) and \(h\). 
    
Recall that $\hat e^m(x) = a^m(x) - x$ is the shortest-distance error defined on $\Ghm$; see its definition in \eqref{def-distance-error}. We decompose it into three parts on $\Ghm$, i.e., 
    $$
    \hat e^m = (a^m - a^m\circ \hat X_{h,*}^m \circ (X_h^m)^{-1} ) 
    + (a^m\circ \hat X_{h,*}^m \circ (X_h^m)^{-1} - \hat X_{h,*}^m \circ (X_h^m)^{-1})
    + (\hat X_{h,*}^m \circ (X_h^m)^{-1} - {\rm id}) ,
    $$
    and apply the triangle inequlaity: 
	\begin{subequations}
		\begin{align}
			\|\hat e^m\|_{L^2(\Ghm)}
			&\le C_0 \|\hat e_h^m\|_{L^2(\Ghsm)} + C_0\|a^m- {\rm id}\|_{L^2(\Ghsm)} + C_0 \|\hat e_h^m\|_{L^2(\Ghsm)} \notag\\
			% &\quad\,\,+ \|\hat X_{h,*}^m \circ (X_h^m)^{-1}-{\rm id}\|_{L^2(\Ghm)}\notag\\
			& \le C_0\|\hat e_h^m\|_{L^2(\Ghsm)} + C_0 h^{k+1},\label{ioio-1}\\
			\|\nabla_{\Ghm}\hat e^m\|_{L^2(\Ghm)}
            &\le C_0 \|\hat e_h^m\|_{H^1(\Ghsm)} + C_0\|a^m- {\rm id}\|_{H^1(\Ghsm)} + C_0 \|\hat e_h^m\|_{H^1(\Ghsm)} \notag\\
			% &\quad\,\,+ \|\nabla_{\Ghm}(\hat X_{h,*}^m \circ (X_h^m)^{-1}-{\rm id})\|_{L^2(\Ghm)}\notag\\
			& \le C_0\|\nabla_{\Ghsm}\hat e_h^m\|_{L^2(\Ghsm)} + C_0 h^{k},\label{ioio-2}
		\end{align}
	\end{subequations}
	where the Lagrange interpolation error estimate \eqref{Ihf} together with the uniform boundedness of \(\kappa_{l}\), \(\kappa_{*,l}\), \(C_\#\), \(C_*\), $\hat C_1$ and $\hat C_2$ are used in deriving the above inequality. By using the above estimates \eqref{ioio-1}--\eqref{ioio-2} and the error estimate \eqref{eror-inbter}, the following estimate holds:
	\begin{align}\label{eror-inbter-1}
		\max_{0 \leq m \leq \lfloor {T}/{\tau} \rfloor} \| \hat{e}^m \|_{L^2(\Ghm)}^2 
		+ \sum_{m = 0}^{\lfloor {T}/{\tau} \rfloor} \tau \| \nabla_{\Ghm} \hat{e}^m \|_{L^2(\Ghm)}^2 
		\leq C_0 (\tau^2 + h^{2k}),
	\end{align}
	where the constant $C_0$ is independent of $\tau$, $h$ (but may depend on $\kappa_0$ and $T$). 

Since we have established that the numerically computed surface \(\Gamma_h^m\) lies within the neighborhood \(D_\delta(\Gamma^m)\), provided that \(h\) is sufficiently small, it follows that the lift \((\hat e^m)^{\hat l}\) is well defined on the smooth surface \(\Gamma^m\) as in \eqref{lift-lift}. {By a similar argument as \eqref{ioio-2}, using the inverse inequality, we can see that 
$$
\|\hat e^m\|_{W^{1,\infty}(\Ghm)}
\le C_0h^{-1}\|\hat e_h^m\|_{H^1(\Ghsm)} + C_0 h^{k-1} 
\le C_0h^{k-2} \le C_0h . 
$$
Therefore, for sufficiently small $h$, the $L^p$ and $W^{1,p}$ norms of $\hat e^m$ on $\Ghm$ are equivalent to the $L^p$ and $W^{1,p}$ norms of its lift \((\hat e^m)^{\hat l}\) on $\Gamma^m$.} Consequently, \eqref{eror-inbter-1} leads to the following result:
\begin{align}\label{eq:err_fin1}
	\max_{0 \leq m \leq \lfloor {T}/{\tau} \rfloor} \| (\hat{e}^m)^{\hat l} \|_{L^2(\Gamma^m)}^2 
	+ \sum_{m = 0}^{\lfloor {T}/{\tau} \rfloor} \tau \| \nabla_{\Gamma^m} ((\hat{e}^m)^{\hat l} )\|_{L^2(\Gamma^m)}^2 
	\leq C (\tau^2 + h^{2k}).
\end{align}
This completes the proof of Theorem \ref{thm:main}.
\hfill\qed

\section{Numerical experiments}\label{sec:num}
In this section, we present several numerical experiments to support the theoretical result established in Theorem~\ref{thm:main} and to illustrate the performance of the proposed MDR method~\eqref{NBGN}. Example~\ref{Example1} demonstrates the convergence rates of the method, confirming consistency with the theoretical predictions. Examples~\ref{Example2} and \ref{Example3} compare the mesh quality during the evolution process among four schemes: the proposed MDR method, Dziuk’s method, the BGN method, and the evolution–equation–based MDR approach. 

    \begin{example}[Convergence rates]\label{Example1}
		In this example, numerical experiments are presented to illustrate the convergence rates of the MDR numerical scheme \eqref{NBGN} for the evolution of an initially spherical surface of radius~\(2\) under mean curvature flow. In this setting, the sphere remains a self‐shrinker, and its exact radius at time \(t\) is given by
		\[
		R(t) = \sqrt{4 - 4t}, 
		\qquad 
		t\in[0,1],
		\]
		thus collapsing to a point at \(t=1\).
		
		We test the convergence rates of the MDR scheme on a sphere with initial radius \(2\) up to the final time \(T = 0.125\), during which the surface remains smooth and the curvature is uniformly bounded. Although Theorem~\ref{thm:main} establishes convergence only for finite element spaces of degree \(k \geq 3\), we also present results for \(k = 1, 2, 3\). Errors are measured in the \(L^\infty(0,T;L^2)\) norm and the \(L^2(0,T;H^1)\) norm:
			\[
			\max_{0 \le m \le \lfloor T/\tau\rfloor} \bigl\|\hat e_h^m\bigr\|_{L^2(\Ghsm)}, 
			\qquad
			\Big( \sum_{m=0}^{\lfloor T/\tau\rfloor} \tau \,\bigl\|\nabla_{\Gamma_{h,*}^m}\hat e_h^m\bigr\|_{L^2(\Ghsm)}^2 \Big)^{\frac12}.
			\]

		The spatial discretization errors in the $L^\infty(0,T;L^2)$ and $L^2(0,T;H^1)$ norms with $T = 0.125$ and $k = 1, 2, 3$ are presented in Figures \ref{Example1a}-\ref{Example1c}. The temporal discretization errors in the $L^\infty(0,T;L^2)$ and $L^2(0,T;H^1)$ norms with $T = 0.125$ are computed for a sufficiently small mesh size $h_{\text{ref}} = 0.025, 0.2, 0.2$ for $k = 1, 2, 3$, as shown in Figure \ref{Example1d}.

		From the numerical results in Figures~\ref{Example1a}–\ref{Example1d}, we observe convergence of order 
		\[
			\max_{0 \le m \le \lfloor T/\tau\rfloor} \bigl\|\hat e_h^m\bigr\|_{L^2(\Ghsm)} = \mathcal{O}(\tau + h^{k+1}), 
		\quad
		\Big( \sum_{m=0}^{\lfloor T/\tau\rfloor} \tau \,\bigl\|\nabla_{\Gamma_{h,*}^m}\hat e_h^m\bigr\|_{L^2(\Ghsm)}^2 \Big)^{\frac12} = \mathcal{O}(\tau + h^k).
		\]
		  Although the proof of Theorem~\ref{thm:main} imposes the stepsize restriction $\tau \le ch^k$, this constraint does not appear to be necessary in the numerical experiments.

		\begin{figure}[htbp]
			\begin{subfigure}[b]{0.45\textwidth}
				\includegraphics[width=\textwidth]{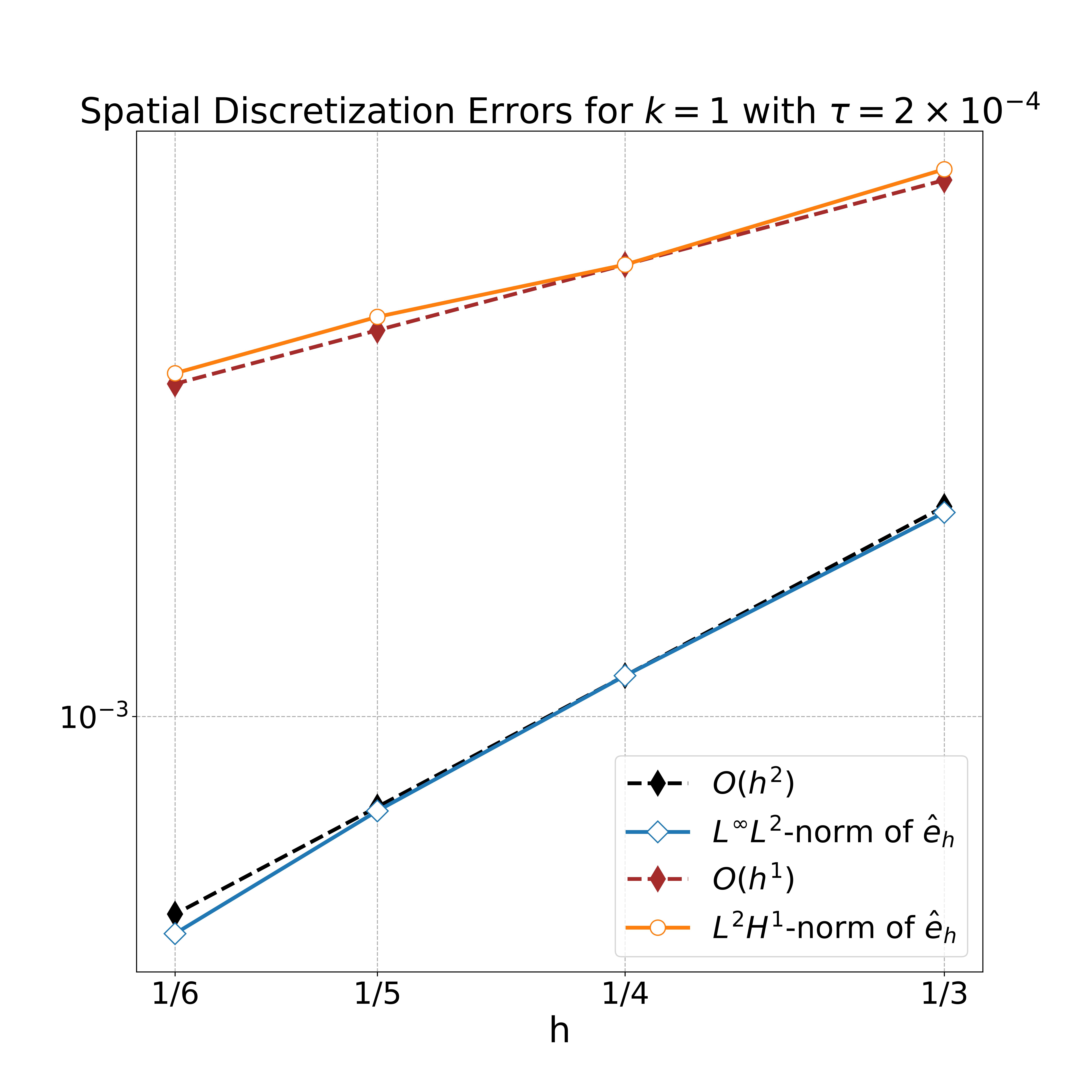}
				\caption{Spatial discretization errors with $k=1$}
				\label{Example1a}
			\end{subfigure}
			\hspace{5pt}
			\begin{subfigure}[b]{0.45\textwidth}
				\includegraphics[width=\textwidth]{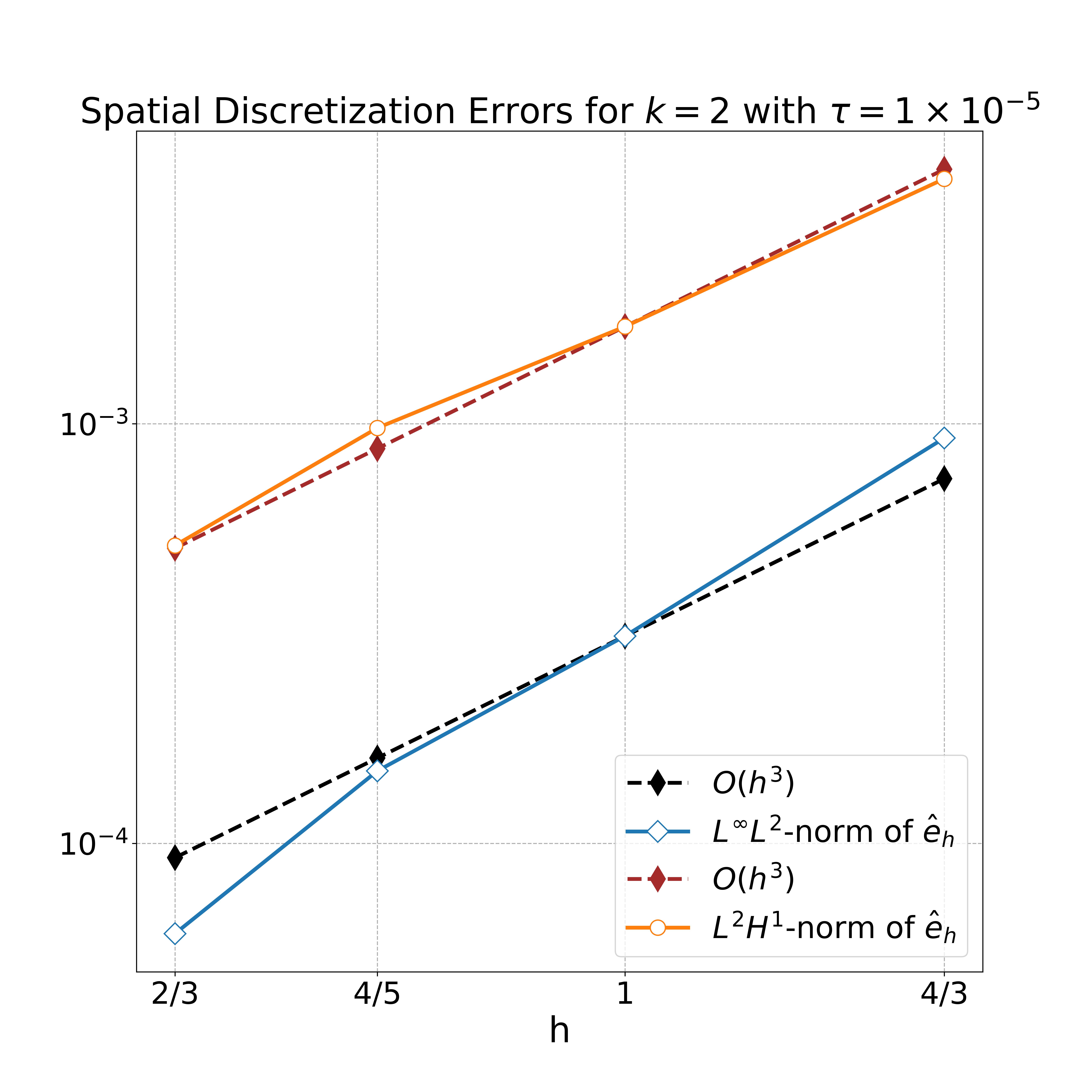}
				\caption{Spatial discretization errors with $k=2$}
				\label{Example1b}
			\end{subfigure}
			\vspace{5pt}
			\begin{subfigure}[b]{0.45\textwidth}
				\includegraphics[width=\textwidth]{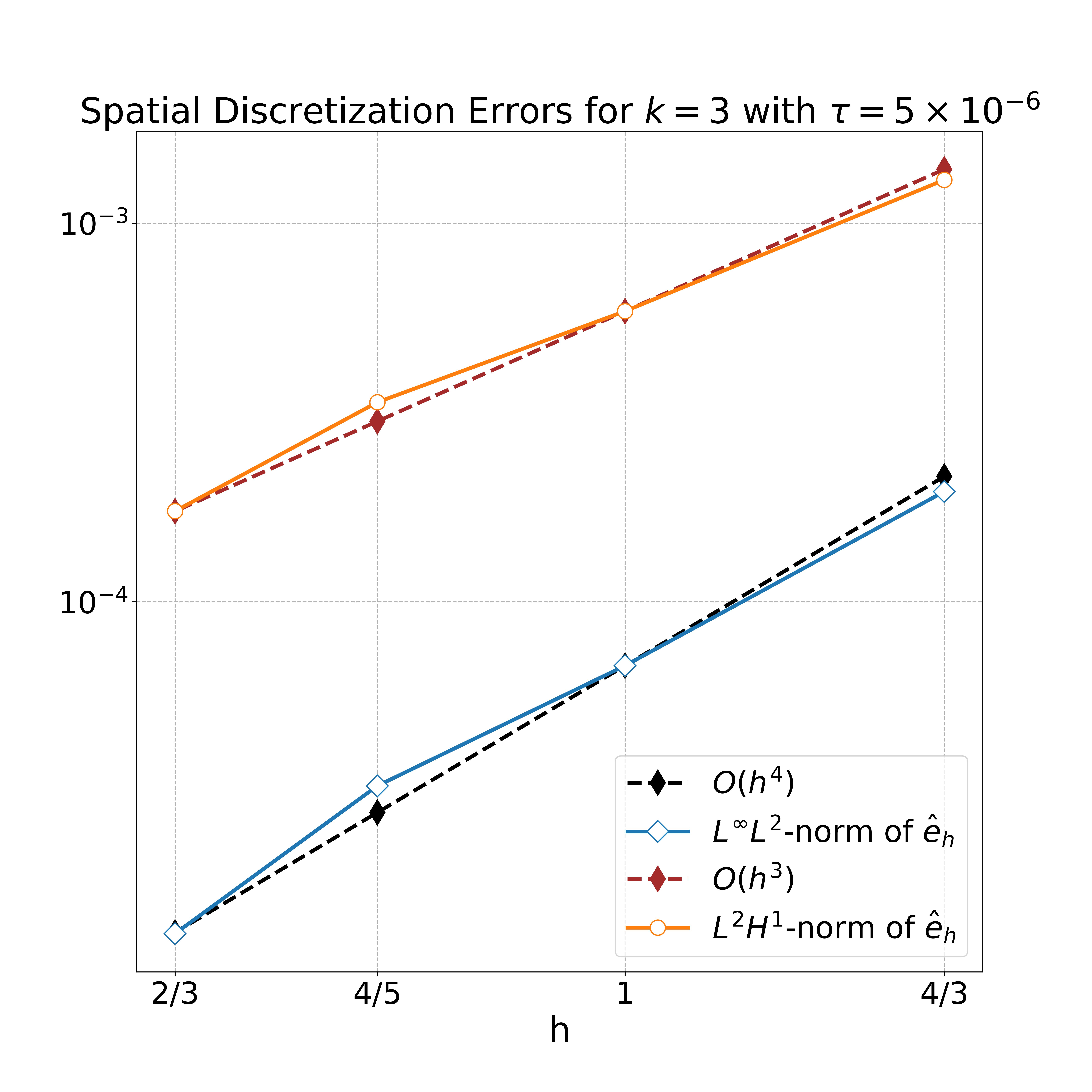}
				\caption{Spatial discretization errors with $k=3$}
				\label{Example1c}
			\end{subfigure}
			\hspace{5pt}
			\begin{subfigure}[b]{0.45\textwidth}
				\includegraphics[width=\textwidth]{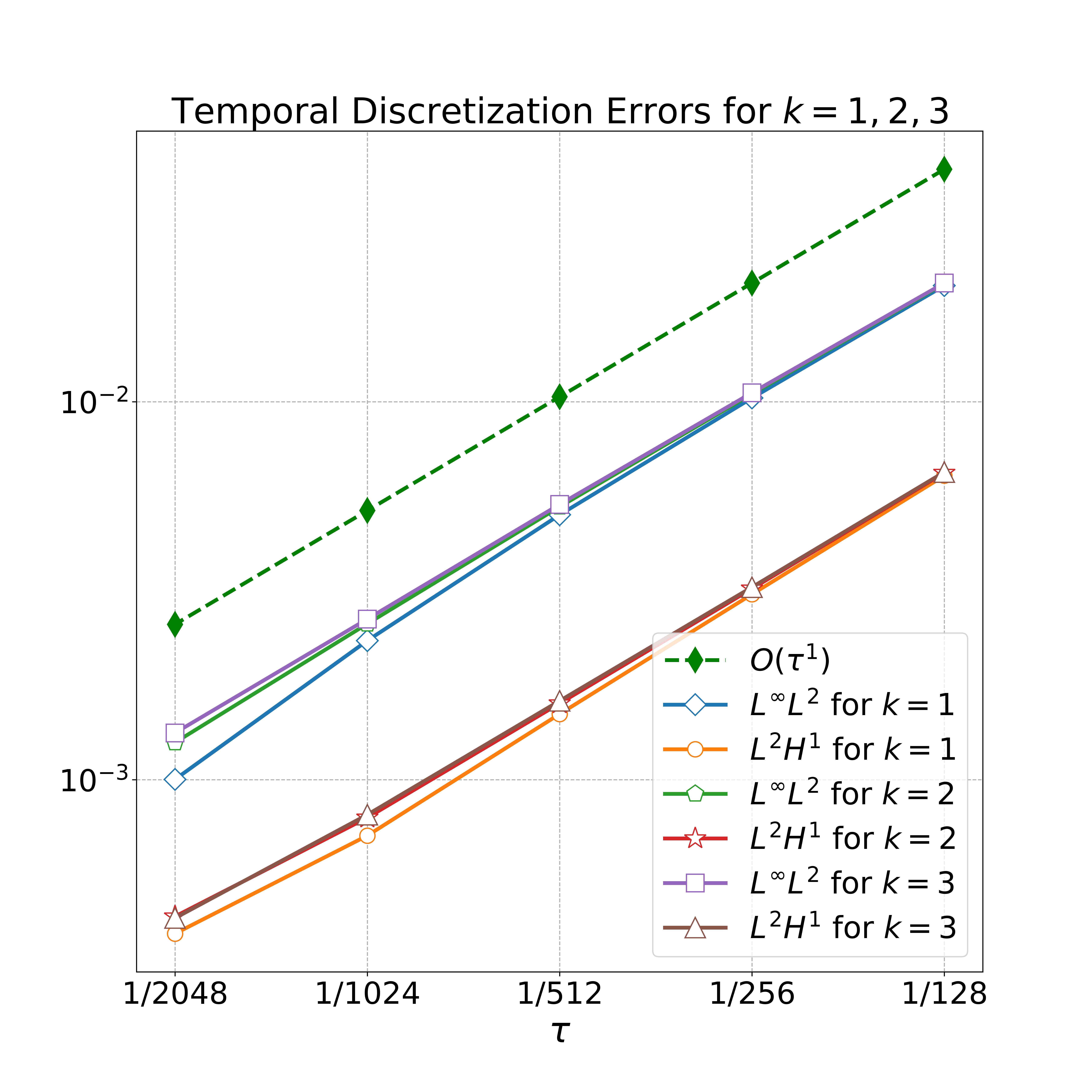}
				\caption{Temporal discretization errors}
				\label{Example1d}
			\end{subfigure}
		
			\caption{Errors and convergence rates (Example \ref{Example1}).}
			\label{eg:Example1}
		\end{figure}

	\end{example}

\begin{example}[Dumbbell surface]\label{Example2}
	In this example, we present numerical simulations of surface evolution in mean curvature flow with the initial surface being a smooth dumbbell-shaped surface defined by the parametrization
	\[
	F(\theta, \varphi) :=
	\begin{pmatrix}
	\cos \varphi\\
	(0.6\cos^2\varphi + 0.4)\cos \theta \sin \varphi\\
	(0.6\cos^2\varphi + 0.4)\sin \theta \sin \varphi
	\end{pmatrix},
	\quad
	\theta\in[0,2\pi),\ \varphi\in[0,\pi].
	\]
	Numerical simulations are conducted using four schemes: the MDR method as formulated in~\eqref{NBGN}, the BGN method, Dziuk’s method, and the MDR method based on the evolution equations for geometric quantities as \cite{hu2022evolving}. To achieve accurate resolution near the blow-up time, an adaptive time-stepping strategy is employed, whereby the time step size is reduced from \(\tau=1\times10^{-5}\) to \(\tau=1\times10^{-7}\) for \(t\ge0.09\). Figures~\ref{Example2b}–\ref{Example2h} show that mesh quality deteriorates progressively for Dziuk’s method and the evolution-equation-based MDR method, resulting in failure to capture the final blow-up. The degradation observed in Dziuk’s method is attributed to the absence of a tangential velocity component to counteract mesh distortion, whereas the evolution-equation MDR method suffers from error accumulation in the geometric quantities (no re-initialization is used here; see the discussions in \cite{hu2022evolving}). In contrast, the BGN method successfully approximates the blow-up time. The proposed MDR method also successfully approximates the blow-up time and maintains good mesh quality throughout the evolution, as illustrated in Figures~\ref{Example2d}–\ref{Example2g}.

		\begin{figure}[!htbp]
			\centering
			\vspace{5pt}
		
			% 第一行
			\begin{subfigure}[b]{0.4\textwidth}
			\centering
				\includegraphics[width=0.7\textwidth]{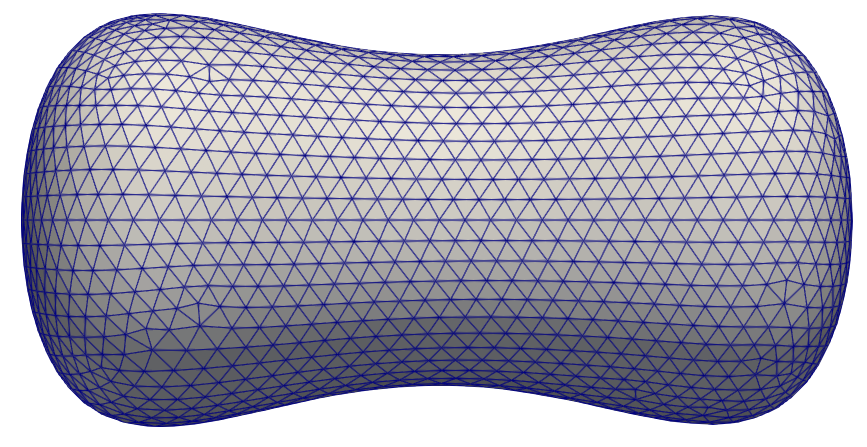}
				\caption{Initial Surface}
				\label{Example2a}
			\end{subfigure}
			\vspace{15pt} % 调整行间距
		
			% 第二行
			\begin{subfigure}[b]{0.4\textwidth}
			\centering
				\includegraphics[width=0.7\textwidth]{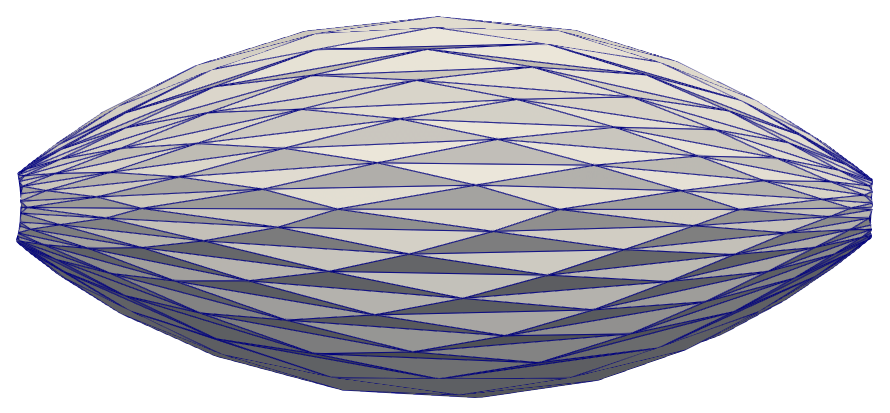}
				\caption{Dziuk scheme at $T = 0.09$}
				\label{Example2b}
			\end{subfigure}
			\hspace{10pt}
			\begin{subfigure}[b]{0.4\textwidth}
			\centering
				\includegraphics[width=0.7\textwidth]{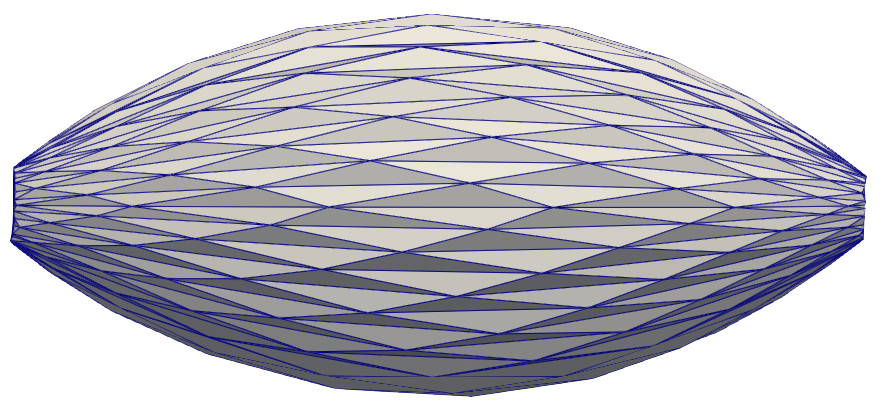}
				\caption{Dziuk scheme at $T = 0.09002624$}
				\label{Example2e}
			\end{subfigure}
			\vspace{18pt} % 调整行间距

			\begin{subfigure}[b]{0.4\textwidth}
				\centering
					\includegraphics[width=0.7\textwidth]{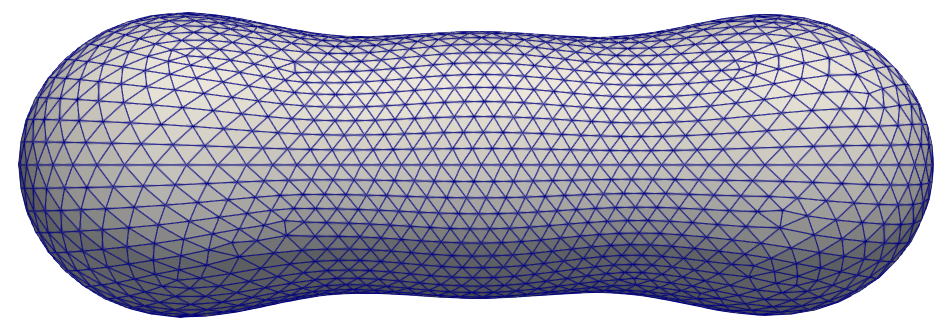}
					\caption{Evolution-eq MDR at $T = 0.072$}
					\label{Example2i}
				\end{subfigure}
				\hspace{10pt}
				\begin{subfigure}[b]{0.4\textwidth}
				\centering
					\includegraphics[width=0.7\textwidth]{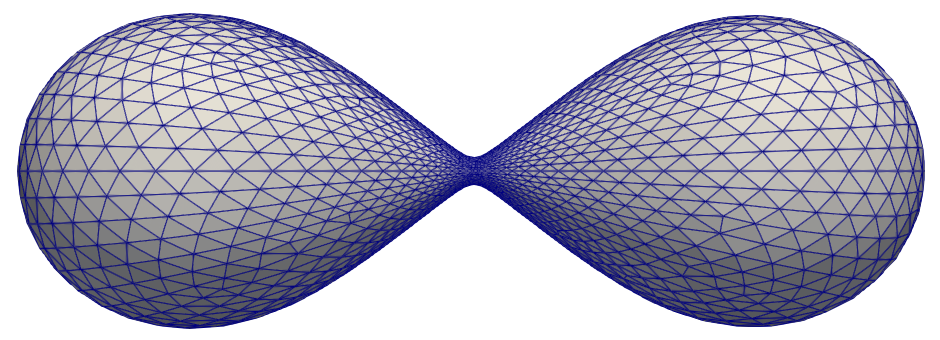}
					\caption{Evolution-eq MDR at $T = 0.0852$}
					\label{Example2h}
				\end{subfigure}
				\vspace{22pt} % 调整行间距
			
			% 第三行
			\vspace{-10pt}
			\begin{subfigure}[b]{0.4\textwidth}
			\centering
				\includegraphics[width=0.6\textwidth]{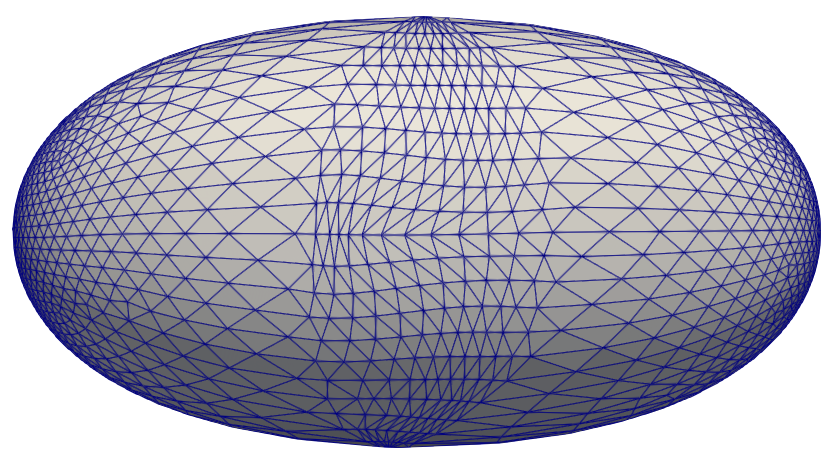}
				\caption{BGN scheme at $T = 0.09$}
				\label{Example2c}
			\end{subfigure}
			\hspace{10pt}
			\begin{subfigure}[b]{0.4\textwidth}
			\centering
				\includegraphics[width=0.35\textwidth]{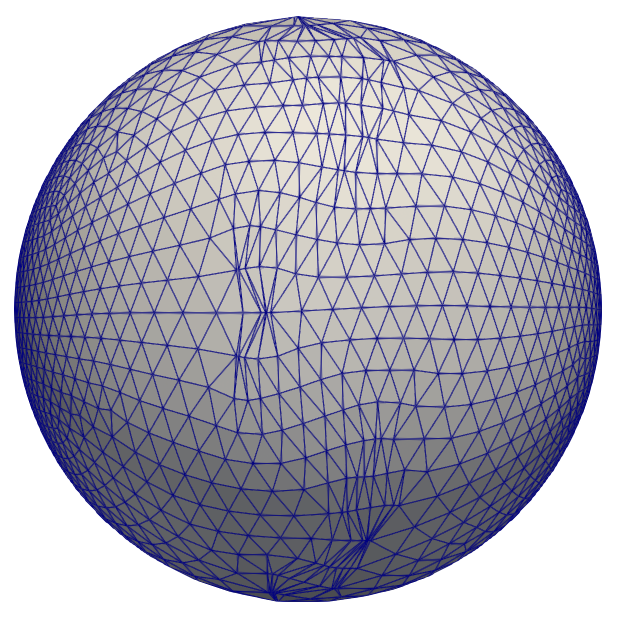}
				\caption{BGN scheme at $T = 0.0906507$}
				\label{Example2f}
			\end{subfigure}
			\vspace{12pt} % 调整行间距
			
			% 第四行
			\begin{subfigure}[b]{0.4\textwidth}
			\centering
				\includegraphics[width=0.6\textwidth]{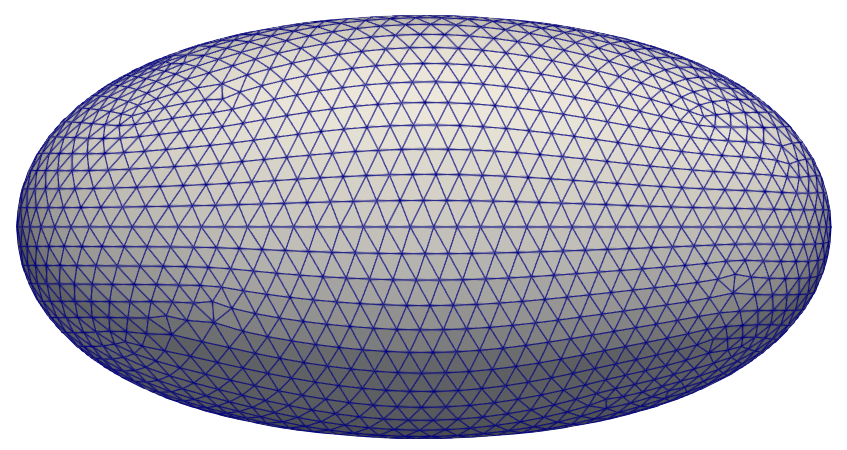}
				\caption{MDR scheme for $T=0.09$}
				\label{Example2d}
			\end{subfigure}
			\hspace{10pt}
			\begin{subfigure}[b]{0.4\textwidth}
			\centering
				\includegraphics[width=0.35\textwidth]{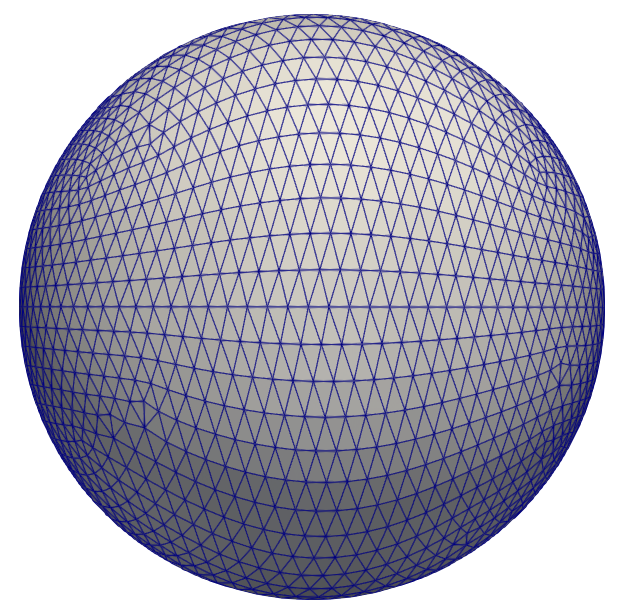}
				\caption{MDR scheme at $T = 0.0907558$}
				\label{Example2g}
			\end{subfigure}
		
			\vspace{5pt}
			\caption{Mean curvature flow for Dumbbell surface  in Example \ref{Example2}}
			\label{eg:Example2}
		\end{figure}
		
\end{example}

\begin{example}[Rectangular box]\label{Example3}
	We consider the mean curvature flow with the initial surface given by a box centered at \((0,0,0)\) with dimensions \(1\times3\times1\) in the \(x\)-, \(y\)-, and \(z\)-directions, respectively. Numerical simulations are carried out with mesh size \(h=0.1\). To evolve the surface up to the intermediate time \(T_{\rm intermediate}=0.157\), we employ a time step size \(\tau_{1}=10^{-4}\). In order to approximate the final collapse time more accurately as the surface shrinks to a point, we continue the simulation beyond \(T=0.157\) using a finer time step size \(\tau_{2}=10^{-6}\). The surfaces closest to the collapse for the BGN and MDR schemes are shown in Figures \ref{Example3d} and \ref{Example3f}, respectively; both methods preserve good mesh quality.
  
	Following the approach in \cite{hu2022evolving}, one must first compute the numerical normal vector and numerical mean curvature by solving the evolution equations. The initial value \(H_h^0\) for the discrete mean curvature on the interpolated surface \(\Gamma_h^0\) is then defined by the weak formulation
	\begin{align*}
	  \int_{\Gamma_h^0} H_h^0\,\phi_h
	  &:= \int_{\Gamma_h^0} \bigl(\nabla_{\Gamma_h^0}\!\cdot\bar n_h^0\bigr)\,\phi_h,
	  \quad \forall\,\phi_h\in S_h(\Gamma_h^0),
	\end{align*}
	where \(\bar n_h^0 := P_{\Gamma_h^0}n_h^0\) denotes the \(L^2\)-projection of the piecewise-defined normal vector \(n_h^0\) on \(S_h(\Gamma_h^0)^3\). Without reinitialization, computing the numerical mean curvature and normal vector directly from the evolution equations can lead to inaccuracies. In particular, the surface generated by the MDR method (using evolution equations) at \(T=0.15\) deviates significantly, as shown in Figure \ref{Example3b}. For comparison, the surfaces at \(T=0.15\) computed by the BGN and MDR numerical scheme \eqref{NBGN} are depicted in Figures \ref{Example3c} and \ref{Example3e}. 
	% These figures have been proportionally resized to enhance visualization.

	\begin{figure}[!htbp]
		\centering
		\vspace{5pt}
	
		% 第一行
		\begin{subfigure}[b]{0.45\textwidth}
		\centering
			\includegraphics[width=0.8\textwidth]{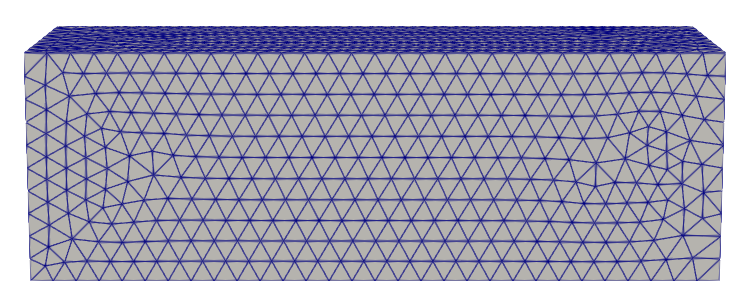}
			\caption{Initial Surface}
			\label{Example3a}
		\end{subfigure}
		\hspace{10pt}
		\begin{subfigure}[b]{0.45\textwidth}
		\centering
			\includegraphics[width=0.8\textwidth]{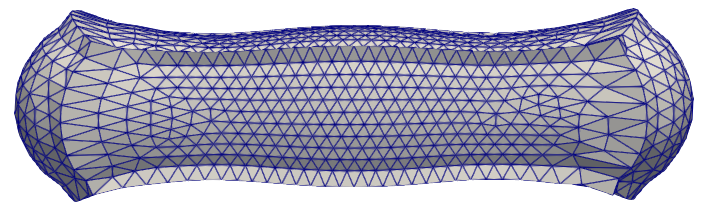}
			\caption{MDR (evolution) scheme at $T = 0.15$}
			\label{Example3b}
		\end{subfigure}
		\vspace{15pt} % 调整行间距
		% 第二行
        
		\begin{subfigure}[b]{0.45\textwidth}
		\centering
			\includegraphics[width=0.7\textwidth]{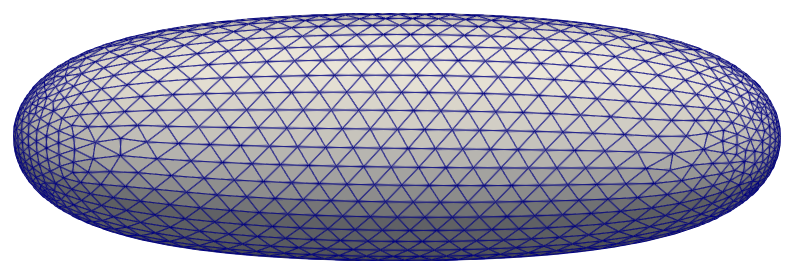}
			\caption{BGN scheme at $T = 0.15$}
			\label{Example3c}
		\end{subfigure}
		\hspace{10pt}
		\begin{subfigure}[b]{0.45\textwidth}
				\centering
					\includegraphics[width=0.7\textwidth]{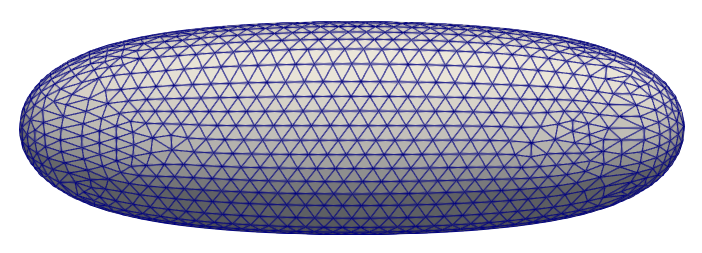}
					\caption{MDR scheme at $T = 0.15$}
					\label{Example3e}
				\end{subfigure}
				\vspace{10pt} % 调整行间距
		% 第三行
                
		\begin{subfigure}[b]{0.45\textwidth}
			\centering
				\includegraphics[width=0.3\textwidth]{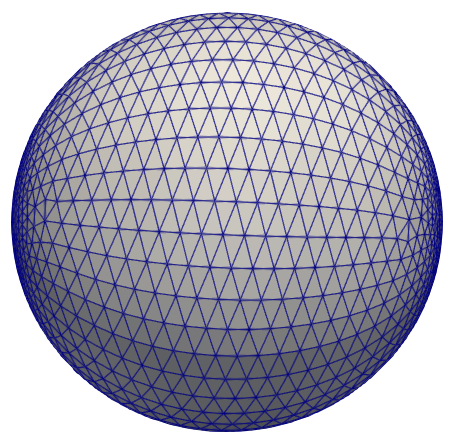}
				\caption{BGN scheme at $T = 0.157140$}
				\label{Example3d}
			\end{subfigure}
				\hspace{5pt}
		\begin{subfigure}[b]{0.45\textwidth}
		\centering
			\includegraphics[width=0.3\textwidth]{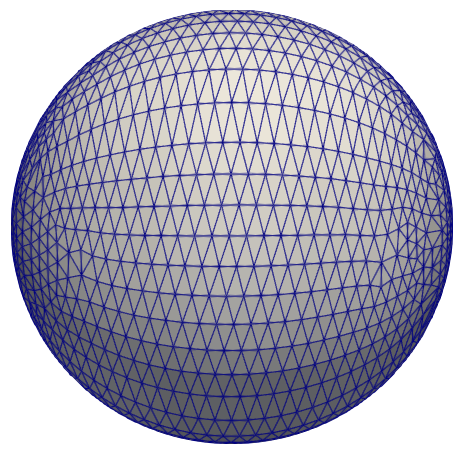}
			\caption{MDR scheme at $T = 0.157776$}
			\label{Example3f}
		\end{subfigure}
		
		% \vspace{10pt} % 调整行间距

		\caption{Mean curvature flow for rectangular box in Example \ref{Example3}}
		\label{eg:Example3}
	\end{figure}
	
\end{example}

From the numerical experiments we can see that a notable advantage of the present MDR scheme is its direct applicability to non-smooth initial surfaces—such as rectangular boxes—without requiring reinitialization, in contrast to the evolution–equation–based MDR scheme (see Example~\ref{Example3}). In addition, the method accurately captures the blow-up time of surfaces shrinking to a point under mean curvature flow (Example~\ref{Example2}). Moreover, as shown in Example~\ref{Example2}, when the time step $\tau$ is sufficiently small, the MDR scheme maintains good mesh quality compariable to the BGN method.

\section{Conclusion}
We have proposed a parametric FEM for mean curvature flow based on the minimization of the deformation‐rate energy (MDR), extending the work of Hu and Li~\cite{hu2022evolving} without relying on the evolution equations for the mean curvature and the normal vector. The continuous formulation \eqref{kv} using the MDR tangential motion enables the derivation of stability estimates for the tangential velocity while maintaining mesh quality comparable to the BGN method. A key feature of the proposed numerical scheme is the use of an $L^2$–projected averaged normal vector, which is motivated by the convergence analysis and plays an important role in establishing the stability of the scheme. Numerical experiments confirm the theoretical convergence rates and demonstrate that the method maintains excellent mesh quality, similar to the BGN method. Overall, this work provides a successful application of the projected–distance framework in combination with the averaged–normal construction. The resulting methodology offers a flexible foundation for analyzing parametric FEMs with artificial tangential motions for geometric curvature flows.

	%Next, using velocity estimate of \eqref{nMinfty}, substitute the result of \eqref{Wjj}, \eqref{Hjj} into \eqref{qWj}, \eqref{qHj} respectively. The rest of proof is similar to that in \cite[Appendix: Part I]{bai2024new}, and therefore omitted.
	%.....................................................
% 	\clearpage        % 清空所有未处理的浮动体，再换页
% % \cleardoublepage % 若是双面排版，用这个
% \phantomsection    % 让 hyperref 的链接锚点正确
% \addcontentsline{toc}{section}{References}
	\renewcommand{\refname}{\bf References}
	\bibliographystyle{plain}
	\bibliography{doc.bib}

	\newpage
	% \appendix
	
% 无编号主标题，但加入目录
% \section*{Supplementary material}
% \addcontentsline{toc}{section}{Supplementary material}

% % 目录与编号层级
% \setcounter{secnumdepth}{2}
% \setcounter{tocdepth}{2}

% % 从 A 开始给 subsection 按字母编号
% \setcounter{subsection}{0}
% \renewcommand{\thesubsection}{\Alph{subsection}}

% 公式随 subsection 重置并带上字母
% （在导言区 \usepackage{chngcntr}）
% \counterwithin{equation}{subsection}
% \renewcommand{\theequation}{\thesubsection.\arabic{equation}}
\section{Appendix}\label{appendix}
% \addcontentsline{toc}{section}{Supplementary material} contents无编号

% 小节用字母编号
\setcounter{subsection}{0}
\renewcommand{\thesubsection}{\Alph{subsection}}

% 公式编号按小节走，例如 (A.1), (B.1)
\counterwithin{equation}{subsection}
\renewcommand{\theequation}{\thesubsection.\arabic{equation}}

This supplementary material provides the detailed proofs of the technical results used in the paper, which are omitted from the main exposition for the sake of clarity, and presents them here in the appendix.
% \subsection{Proof of Lemma \ref{stability-L2}}
% Some text. Here is an equation:
% \begin{equation}
% a^2+b^2=c^2
% \end{equation}

% \subsection{Proof of Super-Approximation Properties}\label{try}
% Some more text.
% \begin{equation}
% e^{i\pi}+1=0
% \end{equation}

\subsection{Proof of Lemma \ref{stability-L2}: ${\bf L^p}$ stability of the ${\bf L^2}$ projection operator} 
\label{appendix_0}

\begin{proof}
It is well known that the $ L^2 $ projection operator $ P_{\hat{\Gamma}_{h,*}^m} $ is stable in the $ L^2 $ norm, i.e.,
\begin{align}
\|P_{\hat{\Gamma}_{h,*}^m} v\|_{L^2(\Ghsm)} \le \|v\|_{L^2(\Ghsm)},
\end{align}
with a stability constant of 1. We now establish its $ L^\infty $ stability on the surface $ \hat{\Gamma}_{h,*}^m $, following an argument analogous to that in Lemma~6.1 of \cite{Thomee2006}, for completeness.

First, we prove the following surface analogue of (6.14) in \cite{Thomee2006}. Let $ K_0 $ be a curved triangle on $ \Ghsm $, $ \Omega_0 \subset \Ghsm $ a subregion disjoint from $ K_0 $, and $ v $ a function on $ \Ghsm $ with $ \mathrm{supp}(v) \subset K_0 $. Then
\begin{align}\label{decay}
	\|P_{\Ghsm} v\|_{L^2(\Omega_0)} \le C e^{-c \, \mathrm{dist}(\Omega_0, K_0)/h} \, \|v\|_{L^2(K_0)} ,
\end{align}
where $ c $ and $ C $ are constants depending on $ \|\hat X_{h,*}^m\|_{W^{1,\infty}(\Gamma_{h,\rm f}^0)} $ and $ \|(\hat X_{h,*}^m)^{-1}\|_{W^{1,\infty}(\Ghsm)} $ for the map $ \hat X_{h,*}^m: \Gamma_{h,\rm f}^0 \to \Ghsm $.

To prove \eqref{decay}, we set $ R_0 = K_0 $ and recursively define $ \{R_j\}_{j\ge0} $ by letting $ R_j $ ($ j \ge 1 $) be the union of closed curved triangles on $ \Ghsm $ that are not contained in $ \bigcup_{l<j} R_l $ but share an edge or vertex with at least one triangle from this set. The norms $ \|\hat X_{h,*}^m\|_{W^{1,\infty}(\Gamma_{h,\rm f}^0)} $ and $ \|(\hat X_{h,*}^m)^{-1}\|_{W^{1,\infty}(\Ghsm)} $ control the Lipschitz constants of the mappings $ \hat X_{h,*}^m $ and $ (\hat X_{h,*}^m)^{-1} $, thereby determining the relative sizes of the triangles on $ \Ghsm $ compared to those in the triangulation of $ \Gamma_{h,\rm f}^0 $. Consequently, there exist constants $ c, C > 0 $ (depending on these norms) such that, for every $ x \in R_j $,
\[
  c\,(j-1)h \;\le\; \mathrm{dist}(x,K_0) \;\le\; C\,jh.
\]

Let $ D_j := \bigcup_{l>j} R_l $. If $ \mathrm{supp}(v) \subset K_0 $, then $ (P_{\Ghsm} v, \tilde\chi_h) = 0 $ for any $ \tilde\chi_h \in S_h(\Ghsm)^3 $ with $ \mathrm{supp}\,\tilde\chi_h \subset D_{j-1} = D_j \cup R_j $, $ j \ge 1 $. In particular, we choose $ \tilde\chi_h $ such that
\[
\tilde\chi_h = P_{\Ghsm} v \quad \text{in } D_j ,  
\qquad \tilde\chi_h = 0 \quad \text{in } \Ghsm \setminus D_{j-1}.
\]
On each triangle of $ R_j $, let $ \tilde\chi_h $ coincide with $ P_{\Ghsm} v $ at the nodes shared with $ D_j $ and vanish at the remaining nodes. Then
\[
0 = (P_{\Ghsm} v, \tilde\chi_h) = \|P_{\Ghsm} v\|_{L^2(D_j)}^{2} + \int_{R_j} (P_{\Ghsm} v)\,\tilde\chi_h,
\]
and by the Cauchy--Schwarz inequality, we obtain
\[
\|P_{\Ghsm} v\|_{L^2(D_j)}^{2}
\le \|P_{\Ghsm} v\|_{L^2(R_j)}\,\|\tilde\chi_h\|_{L^2(R_j)}.
\]
By the definition of $ \tilde\chi_h $, for each $ K \subset R_j $,
\[
\|\tilde\chi_h\|_{L^2(K)} 
\le |K|^{\frac12} \|\tilde\chi_h\|_{L^\infty(K)}
\lesssim |K|^{\frac12} \|P_{\Ghsm} v\|_{L^\infty(K)} 
\lesssim \|P_{\Ghsm} v\|_{L^2(K)}, 
\]
where we have used the equivalence of norms for finite element functions on the triangle $ K $. Summing these inequalities over $ K \subset R_j $ yields the same bound with $ R_j $ in place of $ K $, giving
\begin{align*}
\|P_{\Ghsm} v\|_{L^2(D_j)}
\le C\|P_{\Ghsm} v\|_{L^2(R_j)}.
\end{align*}
Since $ R_j = D_{j-1} \setminus D_j $, it follows that
\begin{align*}
\|P_{\Ghsm} v\|_{L^2(D_j)}^2
\le C(\|P_{\Ghsm} v\|_{L^2(D_{j-1})}^2-\|P_{\Ghsm} v\|_{L^2(D_j)}^2).
\end{align*}
Therefore,
\begin{align*}
\|P_{\Ghsm} v\|_{L^2(D_j)}^2
\le \frac{C}{C+1} \|P_{\Ghsm} v\|_{L^2(D_{j-1})}^2 .
\end{align*}
Iterating this inequality yields 
\begin{align*}
\|P_{\Ghsm} v\|_{L^2(D_j)}^2
\le \Big(\frac{C}{C+1}\Big)^j \|P_{\Ghsm} v\|_{L^2(\Ghsm)}^2
\le e^{-cj} \|v\|_{L^2(K_0)}^2
\le e^{-c\,{\rm dist}(D_j,K_0)/h} \|v\|_{L^2(K_0)}^2 . 
\end{align*}
This proves \eqref{decay}.

We now prove the $ L^\infty $ stability of $ P_{\hat{\Gamma}_{h,*}^m} $ using \eqref{decay}. Let $ K_0 $ be the curved triangle where $ P_{\hat{\Gamma}_{h,*}^m} v $ attains its maximum. Decompose $ v = \sum_j v_j $, where $ v_j $ coincides with $ v $ on $ K_j $ and vanishes elsewhere. Then
\[
\|P_{\hat{\Gamma}_{h,*}^m} v\|_{L^\infty(K_0)}
\le \sum_j \|P_{\hat{\Gamma}_{h,*}^m} v_j\|_{L^\infty(K_0)}.
\]
Using the norm equivalence \eqref{W1p-equiv} and the inverse inequality, we have
\begin{align}\label{surf-local-inv}
	\|P_{\hat{\Gamma}_{h,*}^m} v_j\|_{L^\infty(K_0)} 
	&\le C\|(P_{\hat{\Gamma}_{h,*}^m} v_j) \circ F_{K_0}\|_{L^\infty(K_{\rm f}^0)}\notag\\
	&\le C h^{-1} \|(P_{\hat{\Gamma}_{h,*}^m} v_j )\circ F_{K_0}\|_{L^2(K_{\rm f}^0)}
	\le C h^{-1} \|P_{\hat{\Gamma}_{h,*}^m} v_j \|_{L^2(K_0)}.
\end{align}
Combining \eqref{decay} and \eqref{surf-local-inv} gives
\[
\|P_{\hat{\Gamma}_{h,*}^m} v_j\|_{L^\infty(K_0)}
\le C h^{-1} e^{-c\,\mathrm{dist}(K_0,K_j)/h}
    \|v_j\|_{L^2(K_j)}
\le C e^{-c\,\mathrm{dist}(K_0,K_j)/h} 
    \|v\|_{L^\infty(\hat{\Gamma}_{h,*}^m)}.
\]
By the quasi-uniformity of the triangulation, the number of triangles in $ R_j $ is bounded by $ C j $, where $ C $ depends only on the aforementioned $ W^{1,\infty} $ norms. Summing over $ j $ yields
\[
\|P_{\hat{\Gamma}_{h,*}^m} v\|_{L^\infty(\hat{\Gamma}_{h,*}^m)}
\le C \sum_{j} j e^{-c j} \|v\|_{L^\infty(\hat{\Gamma}_{h,*}^m)}
\le C \|v\|_{L^\infty(\hat{\Gamma}_{h,*}^m)}.
\]
This proves that $ P_{\hat{\Gamma}_{h,*}^m} $ is bounded on $ L^\infty(\hat{\Gamma}_{h,*}^m)^3 $.

The $ L^p $ boundedness of $ P_{\Ghsm} $ for $ 2\le p\le\infty $ follows from Riesz--Thorin interpolation between $ L^2(\hat{\Gamma}_{h,*}^m)^3 $ and $ L^\infty(\hat{\Gamma}_{h,*}^m)^3 $. The $ L^p $ boundedness for $ 1\le p\le 2 $ follows from the self-adjointness of $ P_{\Ghsm} $. This completes the proof of Lemma~\ref{stability-L2}.
\end{proof}

	\subsection{Proof of Super-Approximation Properties and Stability Results} \label{appendix_B}
	\renewcommand{\theequation}{B.\arabic{equation}}
	\textbf{Proof of Lemma \ref{super}}:\\[-2ex]
	\begin{enumerate}
		\item {\textbf{Proof of \eqref{super2}.}} For each curved triangle $K \subset \Ghsm$ with parametrization $F_K: K_{\rm f}^0 \rightarrow K$, the norm equivalence \eqref{W1p-equiv} implies that
		\begin{align*}
			\|(1-I_h)(v_h w_h)\|_{L^2(K)} \sim \|((1-I_h)(v_h w_h))\circ F_K\|_{L^2(K_{\rm f}^0)},\\
			\|(1-I_h)(v_h w_h)\|_{H^1(K)} \sim \|((1-I_h)(v_h w_h))\circ F_K\|_{H^1(K_{\rm f}^0)},
		\end{align*}
		where $\sim$ represents the norm equivalence with the ratio bounds only depend on $\kappa_m$. From the Lagrange interpolation error estimates, we have
		\begin{align}
			&\| (1 - I_h)(v_h w_h) \|_{L^2(K)} \sim\| ((1 - I_h)(v_h w_h))\circ F_K \|_{L^2(K_{\rm f}^0)} \notag\\
                &=\|(1 - I_{K_{\rm f}^0})[(v_h\circ F_K) (w_h\circ F_K)] \|_{L^2(K_{\rm f}^0)} \notag\\
			&\lesssim h^{k+1} \|(v_h\circ F_K) (w_h \circ F_K)\|_{H^{k+1}(K_{\rm f}^0)}\notag\\
			& \lesssim h^{k+1}\sum_{\substack{0 \le i,j \le k \\ i + j \le k+1}} \|\nabla_{K_{\rm f}^0}^i(v_h \circ F_K)\|_{L^\infty(K_{\rm f}^0)}  \|\nabla_{K_{\rm f}^0}^j(w_h \circ F_K)\|_{L^2(K_{\rm f}^0)}\notag\\
			& \lesssim h^{k+1}\sum_{\substack{1 \le i,j \le k \\ i + j \le k+1}} h^{1-i} h^{1-j} \|\nabla_{K_{\rm f}^0}(v_h \circ F_K)\|_{L^\infty(K_{\rm f}^0)}  \|\nabla_{K_{\rm f}^0}(w_h \circ F_K)\|_{L^2(K_{\rm f}^0)}\notag\\
			& \quad\, + h^{k+1} h^{1-k} \|\nabla_{K_{\rm f}^0}(v_h \circ F_K)\|_{L^\infty(K_{\rm f}^0)}  \|w_h \circ F_K\|_{L^2(K_{\rm f}^0)}\notag\\
			& \quad\,+ h^{k+1} h^{1-k} \|v_h \circ F_K\|_{L^\infty(K_{\rm f}^0)} \|\nabla_{K_{\rm f}^0}(w_h \circ F_K)\|_{L^2(K_{\rm f}^0)}\notag\\
			& \lesssim h^2 \|v_h \circ F_K \|_{W^{1,\infty}(K_{\rm f}^0)}\|w_h \circ F_K\|_{H^1(K_{\rm f}^0)}\notag\\
			& \lesssim h^2 \|v_h \|_{W^{1,\infty}(K)}\|w_h \|_{H^1(K)}\quad\text{(by the norm equivalence in \eqref{W1p-equiv})},\label{inverse12}
		\end{align}
		where the inverse inequality $\|\nabla_{K_{\rm f}^0}^i(v_h \circ F_K)\|_{L^p(K_{\rm f}^0)} \le C_0 h^{1-i} \|\nabla_{K_{\rm f}^0}(v_h \circ F_K)\|_{L^p(K_{\rm f}^0)}$ is applied in the third-to-last inequality. By aggregating the above estimates \eqref{inverse12} over all triangles $K \subset \Ghsm$, we obtain
		\begin{align}
			\|(1-I_h)(v_h w_h)\|_{L^2(\Ghsm)} \lesssim h^2 \|v_h \|_{W^{1,\infty}(\Ghsm)}\|w_h \|_{H^1(\Ghsm)}.
		\end{align}
		In a same manner, we obtain 
		\begin{align}
			\|(1-I_h)(v_h w_h)\|_{H^1(\Ghsm)} \lesssim h \|v_h \|_{W^{1,\infty}(\Ghsm)}\|w_h \|_{H^1(\Ghsm)}.
		\end{align}
		\item {\textbf{Proof of \eqref{super6}.}} For each curved triangle $K \subset \Ghsm$ with parametrization $F_K: K_{\rm f}^0 \rightarrow K$, the norm equivalence \eqref{W1p-equiv} implies that
		\begin{align}
			&\| (1-I_h)(u_h v_h w_h)\|_{L^1(K)} \sim  \| ((1-I_h)(u_h v_h w_h))\circ F_K\|_{L^1(K_{\rm f}^0)}\notag\\
                &= \| ((1-I_{K_{\rm f}^0})[(u_h\circ F_K) (v_h\circ F_K) (w_h\circ F_K)]\|_{L^1(K_{\rm f}^0)}\notag\\
			& \lesssim h^{k+1} \|(u_h\circ F_K)(v_h\circ F_K)(w_h\circ F_K)\|_{W^{k+1,1}(K_{\rm f}^0)}\notag\\
			& \lesssim h^{k+1}\sum_{\substack{0 \le i,j,l \le k \\ i + j +l\le k+1}}\|u_h\circ F_K\|_{W^{i,\infty}(K_{\rm f}^0)}\|v_h\circ F_K\|_{H^j(K_{\rm f}^0)}\|w_h\circ F_K\|_{H^l(K_{\rm f}^0)}\notag\\
			& \lesssim h^{k+1} h^{1-i} h^{1-j}h^{1-l} \sum_{\substack{1 \le i,j,l \le k \\ i + j +l\le k+1}}\|u_h\circ F_K\|_{W^{1,\infty}(K_{\rm f}^0)}\|v_h\circ F_K\|_{H^1(K_{\rm f}^0)}\|w_h\circ F_K\|_{H^1(K_{\rm f}^0)}\notag\\
			&\quad\, + h^{k+1} \|u_h\circ F_K\|_{W^{k,\infty}(K_{\rm f}^0)}\|v_h\circ F_K\|_{H^1(K_{\rm f}^0)}\|w_h\circ F_K\|_{H^1(K_{\rm f}^0)}\notag\\
			& \quad\, + h^{k+1} \|u_h\circ F_K\|_{W^{1,\infty}(K_{\rm f}^0)}\|v_h\circ F_K\|_{H^k(K_{\rm f}^0)}\|w_h\circ F_K\|_{H^1(K_{\rm f}^0)}\notag\\
			& \quad\, + h^{k+1} \|u_h\circ F_K\|_{W^{1,\infty}(K_{\rm f}^0)}\|v_h\circ F_K\|_{H^1(K_{\rm f}^0)}\|w_h\circ F_K\|_{H^k(K_{\rm f}^0)}\notag\\
			& \lesssim h^2 \|u_h\circ F_K\|_{W^{1,\infty}(K_{\rm f}^0)}\|v_h\circ F_K\|_{H^1(K_{\rm f}^0)}\|w_h\circ F_K\|_{H^1(K_{\rm f}^0)}\notag\\
			& \lesssim h^2 \|u_h\|_{W^{1,\infty}(K)}\|v_h\|_{H^1(K)}\|w_h\|_{H^1(K)} \quad \text{(by norm equivalences \eqref{W1p-equiv})}, \label{super2-prove}
		\end{align}
		where the inverse inequality $\|\nabla_{K_{\rm f}^0}^i(v_h \circ F_K)\|_{L^p(K_{\rm f}^0)} \le C_0 h^{1-i} \|\nabla_{K_{\rm f}^0}(v_h \circ F_K)\|_{L^p(K_{\rm f}^0)}$ is applied in the derivation of the above estimate. By aggregating the above estimates \eqref{super2-prove} over all curved triangles $K \subset \Ghsm$, and applying Cauchy-Schwarz inequality, we obtain the desired estimate \eqref{super6}:
		\begin{align*}
			\| (1-I_h)(u_h v_h w_h)\|_{L^1(\Ghsm)}\lesssim h^2 \|u_h\|_{W^{1,\infty}(\Ghsm)}\|v_h\|_{H^1(\Ghsm)}\|w_h\|_{H^1(\Ghsm)}.
\end{align*}
\end{enumerate}
\textbf{Proof of Lemma \ref{super-bar-nhm}}:\\[-2ex]
	\begin{enumerate}
		\item {\textbf{Proof of \eqref{super1} and \eqref{super0}.}} For each curved triangle $K \subset \Ghsm$ with parametrization $F_K: K_{\rm f}^0\rightarrow K$ mapping the flat triangle $K_{\rm f}^0$ onto the curved triangle $K$, the Lagrange interpolation approximation estimate provides the following bound:
			\begin{align}	
				& \|(1-I_h) (T_*^m v_h)\|_{L^2(K)} \notag\\
				&\lesssim \|(1-I_{K_{\rm f}^0})( (T_*^m \circ F_K) (v_h \circ F_K))\|_{L^2(K_{\rm f}^0)} \notag\\
				& \lesssim h^{k+1} \|\nabla_{K_{\rm f}^0}^{k+1}( (T_*^m \circ F_K) (v_h \circ F_K))\|_{L^2(K_{\rm f}^0)} \notag\\
				& \lesssim h^{k+1} \sum_{i=0}^k\big(\|v_h \circ F_K\|_{H^{k-i}(K_{\rm f}^0)} \sum_{\substack{j_1+\cdots+j_l = i+1\\ j_1,\cdots,j_l\ge 1 \\ l\ge 1\,\,\mbox{\scriptsize when}\,\,i\le k-1\\
                l\ge 2\,\,\mbox{\scriptsize when}\,\,i=k}}\|\nabla_{K_{\rm f}^0}^{j_1} F_K\|_{L^\infty(K_{\rm f}^0)}\times\cdots\notag\\
				& \quad\,\, \times  \|\nabla_{K_{\rm f}^0}^{j_l} F_K\|_{L^\infty(K_{\rm f}^0)}\big) \notag\\
				& \lesssim h^{k+1}\sum_{i=0}^{k-1} (h^{1-k+i}\|v_h \circ F_K\|_{H^1(K_{\rm f}^0)}\sum_{1\le l\le i+1} h^{l-(i+1)}\|F_K\|_{W^{1,\infty}(K_{\rm f}^0)}^l) \notag\\
				& \quad\,\, + h^{k+1}\|v_h \circ F_K\|_{L^2(K_{\rm f}^0)} \,  \sum_{2\le l\le k+1}h^{l-(k+1)}\|F_K\|_{W^{1,\infty}(K_{\rm f}^0)}^l \notag \\
				& \lesssim h^2 \|v_h\circ F_K\|_{H^1(K_{\rm f}^0)} \notag\\
                    &\lesssim h^2 \|v_h\|_{H^1(K)} \quad \text{(by the norm equivalence in \eqref{W1p-equiv})}, \label{supersper-nm}
			\end{align}
			where the third-to-last inequality follows from the inverse inequality and the fact that $l \geq 2$ when $i=k$ (since $F_K$ is a polynomial of degree $k$, its $(k+1)$th-order partial derivatives must vanish). In the same way, we can obtain 
            \begin{align*}	
			& h\|(1-I_h) (T_*^m v_h)\|_{H^1(K)} \lesssim h^2 \|v_h\|_{H^1(K)} .
		\end{align*}
            By summing these estimates over all curved elements $ K \subset \Ghsm $, we obtain \eqref{super1}. The result in \eqref{super0} can be proved in the same way.

		\item {\textbf{Proof of \eqref{super1_projected-surface} and \eqref{super2_projected-surface}.}} Analogous to the previous proof, the Lagrange interpolation approximation estimate yields the following bound for each curved element $K \subset \Ghsm$:
			\begin{align}
				&\|(1-I_h) (\bar T_{h,*}^m v_h)\|_{L^2(K)} \sim \|[(1-I_h) (\bar T_{h,*}^m v_h)] \circ F_K \|_{L^2(K_{\rm f}^0)}\notag\\
				& = \|(1-I_{K_{\rm f}^0})( (\bar T_{h,*}^m \circ F_K) (v_h \circ F_K))\|_{L^2(K_{\rm f}^0)} \notag\\
				& \lesssim h^{k+1} \|\nabla_{K_{\rm f}^0}^{k+1}( (\bar T_{h,*}^m \circ F_K) (v_h \circ F_K))\|_{L^2(K_{\rm f}^0)} \notag\\
				& \lesssim h^{k+1} \sum_{i=0}^k\big(\|v_h \circ F_K\|_{H^{k-i}(K_{\rm f}^0)} \sum_{\substack{j_1+\cdots+j_l = i+1\\ j_1,\cdots,j_l\ge 1 \\ l\ge 1\,\,\mbox{\scriptsize when}\,\,i\le k-1\\
                l\ge 2\,\,\mbox{\scriptsize when}\,\,i=k}}\|\nabla_{K_{\rm f}^0}^{j_1} (\bar n_{h,*}
				^m \circ F_K)\|_{L^\infty(K_{\rm f}^0)}\times\cdots\notag\\
				& \quad\,\, \times \|\nabla_{K_{\rm f}^0}^{j_l} (\bar n_{h,*}^m \circ F_K)\|_{L^\infty(K_{\rm f}^0)}\big) \,\,\text{(using the $L^\infty$ lower bound of $\bar n_{h,*}^m$ in \eqref{average-1})}\notag\\
				& \lesssim h^2 \|v_h \circ F_K\|_{H^1(K_{\rm f}^0)}
                \sum_{1\le l\le k+1}\|\bar n_{h,*}^m \circ F_K\|_{W^{1,\infty}(K_{\rm f}^0)}^l\,\,\text{(using inverse inequalities)}\notag\\
				&\lesssim h^2 \|v_h\|_{H^1(K)} \,\, \text{(by the norm equivalence in \eqref{W1p-equiv})},
			\end{align}
			where the $W^{1,\infty}$-boundedness of $\bar n_{h,*}^m$ in \eqref{nhmW-1} is used in the last step. The estimate of $\|(1-I_h) (\bar T_{h,*}^m v_h)\|_{H^1(K)} $ can be done in the same way.
			Therefore, we obtain the desired estimate \eqref{super1_projected-surface} by summing the preceding estimate over all curved elements $K \subset \Ghsm$. 
            The result in \eqref{super2_projected-surface} can be proved in the same way.  

		\item {\textbf{Proof of \eqref{super-bar-tangent} and \eqref{super-bar-normal}.}} Analogous to the previous proof, the Lagrange interpolation approximation estimate yields the following bound for each curved element $K \subset \Ghsm$:
			\begin{align}
				&\|(1-I_h) (\bar T_{h}^m v_h)\|_{L^2(K)} 
                \sim \|[(1-I_h) (\bar T_{h}^m v_h)]\circ F_K\|_{L^2(K_{\rm f}^0)} \notag\\
				& = \|(1-I_{K_{\rm f}^0})[ (\bar T_{h}^m \circ F_K) (v_h \circ F_K)]\|_{L^2(K_{\rm f}^0)} \notag\\
				& \lesssim h^{k+1} \|\nabla_{K_{\rm f}^0}^{k+1}( (\bar T_{h}^m \circ F_K) (v_h \circ F_K))\|_{L^2(K_{\rm f}^0)} \notag\\
				& \lesssim h^{k+1} \sum_{i=0}^{k-1}\big(\|v_h \circ F_K\|_{H^{k-i}(K_{\rm f}^0)}\sum_{\substack{j_1+\cdots+j_l = i+1\\ j_1,\cdots,j_l\ge 1\,\,\&\,\, l\ge 1}}\|\nabla_{K_{\rm f}^0}^{j_1} (\bar n_{h}
				^m \circ F_K)\|_{L^\infty(K_{\rm f}^0)}\times\cdots\notag\\
				& \quad\,\, \times \|\nabla_{K_{\rm f}^0}^{j_l} (\bar n_{h}^m \circ F_K)\|_{L^\infty(K_{\rm f}^0)}\big) \,\,\text{(using the $L^\infty$ lower bound of $\bar n_{h}^m$ in \eqref{1nhmf})}\notag\\
				& \quad\,\,+ h^{k+1} \|v_h \circ F_K\|_{L^{10}(K_{\rm f}^0)} \sum_{\substack{j_1+\cdots + j_l=k+1\\ j_1,\cdots,j_l\ge 1\,\,\&\,\,l\ge 2}}\|\nabla_{K_{\rm f}^0}^{j_1} (\bar n_{h}
				^m \circ F_K)\|_{L^5(K_{\rm f}^0)}\times\cdots\notag\\
				& \quad\,\, \times \|\nabla_{K_{\rm f}^0}^{j_l} (\bar n_{h}^m \circ F_K)\|_{L^5(K_{\rm f}^0)} \notag\\
				& \quad \,\,\text{(where $\cdots$ denotes the product of $\|\nabla_{K_{\rm f}^0}^{j_p} (\bar n_{h}^m \circ F_K)\|_{L^\infty(K_{\rm f}^0)}$, $p= 2, \cdots,l-1$)}\notag\\
				& \lesssim h^{k+1}\sum_{i=0}^{k-1} (h^{1-k+i}\|v_h \circ F_K\|_{H^1(K_{\rm f}^0)}\notag\\
				& \quad\,\,\times \sum_{j_1+\cdots+j_l = i+1}h^{1-(i+1)}\|\bar n_h^m \circ F_K\|_{W^{1,\infty}(K_{\rm f}^0)}\|\bar n_h^m \circ F_K\|_{L^\infty(K_{\rm f}^0)}^{l-1}) \notag\\
				& \quad\,\, +h^{k+1}\|v_h \circ F_K\|_{L^{10}(K_{\rm f}^0)} \, h^{2-(k+1)} \|\bar n_h^m \circ F_K\|_{W^{1,5}(K_{\rm f}^0)}^2\|\bar n_h^m \circ F_K\|_{L^{\infty}(K_{\rm f}^0)}^{k-1} \notag \\
				& \lesssim h^2 \|v_h \circ F_K\|_{H^1(K_{\rm f}^0)}\|\bar n_{h}^m \circ F_K\|_{W^{1,\infty}(K_{\rm f}^0)}\notag\\
				& \quad\,\, + h^2 \|v_h \circ F_K\|_{L^{10}(K_{\rm f}^0)}\|\bar n_h^m \circ F_K\|_{W^{1,5}(K_{\rm f}^0)}^2\notag\\
				&\lesssim h^2 \|v_h\|_{H^1(K)}\|\bar n_{h}^m \|_{W^{1,\infty}(K)} + h^2 \|v_h \|_{L^{10}(K)}\|\bar n_h^m \|_{W^{1,5}(K)}^2, \label{super-sper-nhm}
			\end{align}
			where the $L^\infty$-boundedness of $\bar n_{h}^m$ in \eqref{1nhmf} is used in the above inequality. The estimate of $h\|(1-I_h) (\bar T_{h}^m v_h)\|_{H^1(K)} $ can be done in the same way. By aggregating the preceding estimate \eqref{super-sper-nhm} over all curved elements $K \subset \Ghsm$, together with the discrete H\"older inequality, we obtain 
			\begin{align}
				& \quad\,\,\|(1-I_h) (\bar T_{h}^m v_h)\|_{L^2(\Ghsm)} +h\|(1-I_h) (\bar T_{h}^m v_h)\|_{H^1(\Ghsm)}\notag\\
				&\lesssim h^2 \|\bar n_h^m\|_{W^{1,\infty}(\Ghsm)} \|v_h\|_{H^1(\Ghsm)}+ h^2 \|\bar n_h^m\|_{W^{1,5}(\Ghsm)}^2\|v_h\|_{L^{10}(\Ghsm)}\notag\\
				& \lesssim h^2 \|\bar n_h^m\|_{W^{1,\infty}(\Ghsm)} \|v_h\|_{H^1(\Ghsm)} + h^2 \|v_h\|_{H^1(\Ghsm)},\label{ibter}
			\end{align}
			where the following inequality is used in the second-to-last inequality: 
			\begin{align}
				 &\|\bar n_h^m\|_{W^{1,5}(\Ghsm)}\notag\\
				 & \lesssim \|\bar n_{h,*}^m\|_{W^{1,5}(\Ghsm)} + \|\bar n_h^m - \bar n_{h,*}^m\|_{W^{1,5}(\Ghsm)}\notag\\
				 &\lesssim \|\bar n_{h,*}^m\|_{W^{1,5}(\Ghsm)} +h^{-1+0.4-1} \|\bar n_{h,*}^m - \bar n_h^m\|_{L^2(\Ghsm)}\notag\\
				 & \lesssim 1 + h^{-1.6} \|\nabla_{\Ghsm}\hat e_h^m\|_{L^2(\Ghsm)} \lesssim 1,\,\,\text{(using \eqref{Linfty-W1infty-hat-em}, \eqref{eq:nsa1} and \eqref{nhmW-1})}
			\end{align}
			From inequality \eqref{ibter} and \eqref{nhmW-2a}, we obtain
			\begin{align}\label{ibter2}
				&\|(1-I_h) (\bar T_{h}^m v_h)\|_{L^2(\Ghsm)} + h\|(1-I_h) (\bar T_{h}^m v_h)\|_{H^1(\Ghsm)}\notag\\
                &\lesssim (h^2 + \|\nabla_{\Ghsm}\hat e_h^m\|_{L^2(\Ghsm)}) \|v_h\|_{H^1(\Ghsm)}\notag\\
				& \lesssim h^{1.6}\|v_h\|_{H^1(\Ghsm)}.
			\end{align}
			Therefore, the desired estimate \eqref{super-bar-tangent} follows from \eqref{ibter} and \eqref{ibter2}. The estimate \eqref{super-bar-normal} can be proved in the same way.
	\end{enumerate}
	\textbf{Proof of Lemma \ref{stability}}:\\[-3ex]
	\begin{enumerate}
		\item  {\textbf{Proof of \eqref{stability2} and \eqref{stability7}.}} For each curved triangle $K \subset \Ghsm$ with parametrization $F_K: K_{\rm f}^0\rightarrow K$ mapping the flat triangle $K_{\rm f}^0$ onto the curved triangle $K$, Lagrange interpolation approximation estimate provides the following bounds:
		\begin{align}
		&\| I_h(v_h w_h) \|_{L^2(K)} \sim \| I_{K_{\rm f}^0}[(v_h w_h) \circ F_K] \|_{L^2(K_{\rm f}^0)}\notag\\
            &\le \| I_{K_{\rm f}^0}[(v_h w_h) \circ F_K] -(v_h w_h) \circ F_K\|_{L^2(K_{\rm f}^0)} + \| (v_h w_h) \circ F_K\|_{L^2(K_{\rm f}^0)} \notag\\
		&\lesssim h^2 \|(v_h w_h) \circ F_K\|_{H^2(K_{\rm f}^0)}  + \| (v_h w_h)\circ F_K \|_{L^2(\Ghsm)} \notag\\
            &\lesssim \|(v_h w_h)\circ F_K\|_{L^2(K_{\rm f}^0)} \quad\mbox{(inverse inequality, since $(v_h w_h)\circ F_K$ is a polynomial)} \notag\\
		& \lesssim \|v_h w_h\|_{L^2(K)} \lesssim \|v_h\|_{L^\infty(K)}\|w_h\|_{L^2(K)} .
		% &\quad\,\,\| I_h(v_h w_h) \|_{H^1(K)} \lesssim \| I_h(v_h w_h)\circ F_K \|_{H^1(K_{\rm f}^0)}\notag\\
		% &\lesssim h\|(v_h w_h)\circ F_K\|_{H^2(K_{\rm f}^0)}+\| (v_h w_h)\circ F_K \|_{H^1(K_{\rm f}^0)} \notag\\
		% &\lesssim \|(v_h w_h)\circ F_K\|_{H^1(K_{\rm f}^0)}\lesssim 
	\end{align}
	By summing these inequalities over all curved triangles $K \subset \Ghsm$, we obtain the desired estimate \eqref{stability2}. The estimate \eqref{stability7} can be proved in the same way.
    
	\item {\textbf{Proof of \eqref{stability3} and \eqref{stability4}.}} For each curved triangle $K \subset \Ghsm$ with parametrization $F_K: K_{\rm f}^0\rightarrow K$ mapping the flat triangle $K_{\rm f}^0$ onto the curved triangle $K$, since $\Gamma_{h,\rm f}^0$ is a piecewise flat triangular surface with a shape-regular and quasi-uniform triangulation, the following estimate holds:
		\begin{align}
		&\|I_h(u v_h)\|_{L^2(K)} \sim \|I_{K_{\rm f}^0}[(u v_h)\circ F_K]\|_{L^2(K_{\rm f}^0)}\notag\\
		&\lesssim h \|I_{K_{\rm f}^0}[(u v_h)\circ F_K]\|_{L^\infty(K_{\rm f}^0)} \lesssim h\|u\circ F_K \|_{L^\infty(K_{\rm f}^0)}\|v_h\circ F_K\|_{L^\infty(K_{\rm f}^0)} \notag\\
		&\lesssim \|u\circ F_K\|_{L^\infty(K_{\rm f}^0)}\|v_h \circ F_K\|_{L^2(K_{\rm f}^0)}\lesssim \|u\|_{L^\infty(K)}\|v_h\|_{L^2(K)},
	\end{align}
	where the $L^\infty$ stability of the Lagrange interpolation operator and the inverse inequality are used. By summing these inequalities over all curved triangles $ K \subset \Ghsm $, we obtain the desired estimate \eqref{stability3}. The estimate \eqref{stability4} can be proved in the same way.

	\item {\textbf{Proof of \eqref{stability1}, \eqref{stability11} and \eqref{stability12}.}}
	By utilizing \eqref{stability3} together with the uniform bound \(\|T_*^m\|_{L^\infty(\Ghsm)}\lesssim1\), we derive the desired estimate \eqref{stability1}. For the proof of \eqref{stability11} and \eqref{stability12}, we proceed in the same way by using the $L^\infty$ bound of \(\|\bar T_{h,*}^m\|_{L^\infty(\Ghsm)}\) which follows from \eqref{average-1}, and the $L^\infty$ bound of \(\|\bar T_{h}^m\|_{L^\infty(\Ghsm)}\) which follows from \eqref{1nhmf}.
	
	\end{enumerate}

	\subsection{Proof of Lemma \ref{consterr} (Consistency Estimates)} \label{appendix_C}
	\renewcommand{\theequation}{C.\arabic{equation}}
	\begin{enumerate}
		\item \textbf{Proof of \eqref{dv}:}
	From inequality \eqref{nhmW-1}, it follows that
		\[
		\|\nahsm \phi_h\|_{H^1(\Ghsm)} \leq \|\nahsm\|_{W^{1,\infty}(\Ghsm)} \|\phi_h\|_{H^1(\Ghsm)} \lesssim \|\phi_h\|_{H^1(\Ghsm)}.
		\]
		Then the estimate \eqref{dv} follows directly from the proof of \cite[Lemma 4.3]{bai2024new}.
		\item \textbf{Proof of \eqref{dk}:}
		We begin by decomposing the first term on the right-hand side of \eqref{er2} as follows:
		\begin{align*}
            %\label{dk1}
			d^m_{\kappa1}(\chi_h) 
			&= \int_{\Ghsm} \nabla_{\Ghsm} I_h v^m \cdot \nabla_{\Ghsm} \chi_h - \int_{\Gamma^m} \nabla_{\Gamma^m} v^m \cdot \nabla_{\Gamma^m} (\chi_h)^l \notag \\
			&= \int_{\Ghsm} \nabla_{\Ghsm} (I_h v^m - v^{m,-l}) \cdot \nabla_{\Ghsm} \chi_h \notag \\
			&\quad + \int_{\Ghsm} \nabla_{\Ghsm} v^{m,-l} \cdot \nabla_{\Ghsm} \chi_h - \int_{\Gamma^m} \nabla_{\Gamma^m} v^m \cdot \nabla_{\Gamma^m} (\chi_h)^l \notag \\
			&=: d^m_{\kappa11}(\chi_h) + d^m_{\kappa12}(\chi_h).
		\end{align*}
		Using the approximation property of the Lagrange interpolation in \eqref{Ihf} and the geometric perturbation estimate \eqref{perturbation2} from Lemma \ref{geometric-perturbation}, we have:
		\begin{align*}
			|d^m_{\kappa1}(\chi_h)| 
			&\lesssim \|\nabla_{\Ghsm} (I_h v^m - v^{m,-l})\|_{L^2(\Ghsm)} \|\nabla_{\Ghsm} \chi_h\|_{L^2(\Ghsm)} \\
			&\quad + (1 + \kappa_{*,l}) h^{k+1} \|\nabla_{\Ghsm} v^{m,-l}\|_{L^\infty(\Ghsm)} \|\nabla_{\Ghsm} \chi_h\|_{L^2(\Ghsm)} \\
			&\lesssim (1 + \kappa_{*,l}) h^k \|\chi_h\|_{H^1(\Ghsm)},
		\end{align*}
		where the norm equivalence bewtween $\Ghsm$ and $\Gamma^m$ is used. The second term on the right-hand side of \eqref{er2} is decomposed as follows:
		\begin{align*}
			d^m_{\kappa2}(\chi_h) 
			&= \int_{\Gamma^m} \kappa^m n^m \cdot (\chi_h)^l - \int_{\Ghsm} I_h \kappa^m \nahsm \cdot \chi_h \\
			&= \int_{\Gamma^m} \kappa^m (n_*^m \cdot \chi_h)^l - \int_{\Ghsm} \kappa^{m,-l} n_*^m \cdot \chi_h \\
			&\quad + \int_{\Ghsm} (\kappa^{m,-l} - I_h \kappa^m) n_*^m \cdot \chi_h  + \int_{\Ghsm} I_h \kappa^m (n_*^m - \nahsm) \cdot \chi_h. 
		\end{align*}
	% 	where the final inequality follows from the estimate
	% \[
	%   \|\nabla_{\Ghsm} a^m\|_{L^\infty(\Ghsm)}
	%    \lesssim 
	%   1 + \|\nabla_{\Ghsm}(a^m - \mathrm{id})\|_{L^\infty(\Ghsm)}
	%    \lesssim 
	%   1,
	% \]
	% which is a direct consequence of the Lagrange interpolation approximation property \eqref{Ihf} and the condition \eqref{cond1}.
	 Using the geometric perturbation estimate \eqref{perturbation1} from Lemma \ref{geometric-perturbation}, the Lagrange interpolation approximation property \eqref{Ihf} and inequality \eqref{eq:nsa2} from Lemma \ref{Lem:normalvector}, we estimate $|d^m_{\kappa2}(\chi_h)|$ as follows:
		\begin{align*}
			|d^m_{\kappa2}(\chi_h)| 
			&\lesssim (1 + \kappa_{*,l}) h^{k+1} \|\kappa^{m,-l}\|_{L^\infty(\Ghsm)} \|n_*^m \cdot \chi_h\|_{L^2(\Ghsm)} \\
			&\quad + \|\kappa^{m,-l} - I_h \kappa^m\|_{L^2(\Ghsm)} \|n_*^m \cdot \chi_h\|_{L^2(\Ghsm)} \\
			&\quad + \|I_h \kappa^m\|_{L^\infty(\Ghsm)} \|n_*^m - \nahsm\|_{L^2(\Ghsm)} \|\chi_h\|_{L^2(\Ghsm)} \\
			&\lesssim (1 + \kappa_{*,l}) h^k \|\chi_h\|_{L^2(\Ghsm)}.
		\end{align*}
		Combining the estimates above, we conclude \eqref{dk}.
	\end{enumerate}

		\subsection{Proof of inequality \eqref{estJ2} (Bound for \( |J_2^m(\phi_h)| \))} \label{appendix_D}
		\renewcommand{\theequation}{D.\arabic{equation}}
		\begin{proof}
	By utilizing the relation $(X_h^{m+1} - X_h^m)/\tau = I_h v^m + e_v^m = I_h v^m + \hat e_v^m + I_h g^m$, which is derived from \eqref{evm} and \eqref{xm}, the term \( J_2^m(\phi_h) \) can be reformulated and estimated as follows (note that $X_h^m$ is changed to ${\rm id}$ here as it is considered as a finite element function on $\Gamma_h^m$):
\begin{align}
    J_2^m(\phi_h)
    &= \int_{\Gamma^m_h} \nabla_{\Gamma^m_h} X^{m+1}_h \cdot \nabla_{\Gamma^m_h} [(\nahm - \nahsm)\phi_h] \notag \\
    &= \int_{\Gamma^m_h} \nabla_{\Gamma^m_h} [{\rm id} + \tau (\ehvm + I_h v^m + I_h g^m)] \cdot \nabla_{\Gamma^m_h} [(\nahm - \nahsm)\phi_h] \notag \\
    &= \int_{\Gamma^m_h} [I - \nhm (\nhm)^\top] \cdot \nabla_{\Gamma^m_h} [(\nahm - \nahsm)\phi_h] \quad\mbox{(here we use $\nabla_{\Gamma^m_h} {\rm id} = I - \nhm (\nhm)^\top$)} \notag \\
    &\quad + \int_{\Gamma^m_h} \tau \nabla_{\Gamma^m_h} (\ehvm + I_h v^m + I_h g^m) \cdot \nabla_{\Gamma^m_h} [(\nahm - \nahsm)\phi_h] \notag \\
    &= \int_{\Gm} [I - n^m (n^m)^\top] \cdot \nabla_{\Gm} [(\bar{n}_h^{m,l} - \bar{n}_{h,*}^{m,l})\phi_h^l] \notag \\
    &\quad + \Big( \int_{\Gm} [I - \bar{n}_h^{m,l} (\bar{n}_h^{m,l})^\top] \cdot \nabla_{\Gm} [(\bar{n}_h^{m,l} - \bar{n}_{h,*}^{m,l})\phi_h^l] \notag\\
    &\qquad - \int_{\Gm} [I - n^m (n^m)^\top] \cdot \nabla_{\Gm} [(\bar{n}_h^{m,l} - \bar{n}_{h,*}^{m,l})\phi_h^l] \Big) \notag \\
    &\quad + \Big( \int_{\Ghsm} [I - \nahm (\nahm)^\top] \cdot \nabla_{\Ghsm} [(\nahm - \nahsm)\phi_h] \notag\\
    &\qquad 
    - \int_{\Gm} [I - \bar{n}_h^{m,l} (\bar{n}_h^{m,l})^\top] \cdot \nabla_{\Gm} [(\bar{n}_h^{m,l} - \bar{n}_{h,*}^{m,l})\phi_h^l] \Big) \notag \\
    &\quad + \Big(\int_{\Gamma^m_h} [I - \nahm (\nahm)^\top] \cdot \nabla_{\Gamma^m_h} [(\nahm - \nahsm)\phi_h] \notag\\
    &\qquad 
    - \int_{\Ghsm} [I - \nahm (\nahm)^\top] \cdot \nabla_{\Ghsm} [(\nahm - \nahsm)\phi_h] \Big) \notag \\
    &\quad + \int_{\Gamma^m_h} [\nahm (\nahm)^\top - \nhm (\nhm)^\top] \cdot \nabla_{\Gamma^m_h} [(\nahm - \nahsm)\phi_h] \notag \\
    &\quad + \int_{\Gamma^m_h} \tau \nabla_{\Gamma^m_h} (\ehvm + I_h v^m + I_h g^m) \cdot \nabla_{\Gamma^m_h} [(\nahm - \nahsm)\phi_h] \notag \\
    &=: \sum_{i=1}^{6} J_{2i}^m(\phi_h).
\end{align}

Using integration by parts and Lemma \ref{lemma:ud} (item 3), the following result holds:
\begin{align}
    |J_{21}^m(\phi_h)|
    &= \left|\int_{\Gm} \nabla_{\Gm} \cdot [(\bar{n}_h^{m,l} - \bar{n}_{h,*}^{m,l})\phi_h^l]\right| 
    = \left|\int_{\Gm} (\bar{n}_h^{m,l} - \bar{n}_{h,*}^{m,l}) \phi_h^l \cdot H^m n^m \right| \notag \\
    &\lesssim \|\nahm - \nahsm\|_{L^2(\Ghsm)} \|\phi_h\|_{L^2(\Ghsm)} \quad \text{(using norm equivalence between \( \hat \Gamma_{h,*}^m \) and \( \Gamma^m \))}\notag \\
    &\lesssim \|\nabla_{\Ghsm} \ehm\|_{L^2(\Ghsm)} \|\phi_h\|_{L^2(\Ghsm)} 
    \quad \text{(using \eqref{eq:nsa1})}.
\end{align}

Applying \eqref{eq:nsa1}, \eqref{eq:nsa3}, the norm equivalence between $\hat\Gamma_{h,*}^m$ and $\Gamma^m$, and inverse inequality, the term $J_{22}^m(\phi_h)$ can be estimated as follows:
\begin{align}
    |J_{22}^m(\phi_h)| 
    &\lesssim  \|\bar{n}_h^m - n_*^m\|_{L^2(\hat \Gamma_{h,*}^m)} (\|\bar{n}_h^m - \bar{n}_{h,*}^m\|_{H^1(\hat \Gamma_{h,*}^m)} \|\phi_h\|_{L^\infty(\hat \Gamma_{h,*}^m)}\notag\\
	& \quad\,\, + \|\bar{n}_h^m - \bar{n}_{h,*}^m\|_{L^2(\hat \Gamma_{h,*}^m)} \|\phi_h\|_{W^{1,\infty}(\hat \Gamma_{h,*}^m)} ) \notag \\
    &\lesssim h^{-1} \big((1 + \kappa_{*,l}) h^k + \|\nabla_{\hat \Gamma_{h,*}^m} \hat{e}_h^m\|_{L^2(\hat \Gamma_{h,*}^m)}\big) 
    \|\nabla_{\hat \Gamma_{h,*}^m} \hat{e}_h^m\|_{L^2(\hat \Gamma_{h,*}^m)} \|\phi_h\|_{L^\infty(\hat \Gamma_{h,*}^m)} \notag \\
    &\lesssim C_{\epsilon}h^{0.6-\epsilon}\|\nabla_{\hat \Gamma_{h,*}^m} \hat{e}_h^m\|_{L^2(\hat \Gamma_{h,*}^m)} \|\phi_h\|_{H^1(\hat \Gamma_{h,*}^m)} \lesssim h^{0.5} \|\nabla_{\hat \Gamma_{h,*}^m} \hat{e}_h^m\|_{L^2(\hat \Gamma_{h,*}^m)} \|\phi_h\|_{H^1(\hat \Gamma_{h,*}^m)},
\end{align}
where the second-to-last inequality follows from \eqref{eq:poincare3} and the mathematical induction assumption in \eqref{cond1}, which holds for arbitrarily small $ \epsilon $ at the cost of enlarging the constant $C_\epsilon$. The last inequality is obtained by choosing a sufficiently small $ \epsilon $. 

By employing the geometric perturbation estimate \eqref{perturbation3} from Lemma \ref{geometric-perturbation}, \( J_{23}^m(\phi_h) \) can be estimated as follows:
\begin{align}
    &\quad\,\,|J_{23}^m(\phi_h)| \notag\\
    &\lesssim (1 + \kappa_{*,l}) h^{k} \|\bar n_h^m\|_{L^\infty(\Ghsm)}\|\nabla_{\Ghsm}[(\nahm - \nahsm)\phi_h]\|_{L^2(\Ghsm)} \notag \\
    &\lesssim (1 + \kappa_{*,l}) h^{k-2} \|\nahm - \nahsm\|_{L^2(\Ghsm)} \|\phi_h\|_{L^2(\Ghsm)} \quad \text{(using \eqref{nhmW-2b} and the inverse inequality)}\notag \\
    &\lesssim h^{0.6} \|\nabla_{\Ghsm} \ehm\|_{L^2(\Ghsm)} \|\phi_h\|_{L^2(\Ghsm)} 
    \quad\text{(using \eqref{eq:nsa1} and \eqref{cond1})}.
\end{align}

By utilizing H\"older's inequality, inverse inequality, and inequality \eqref{eq:nsa5} in Lemma \ref{Lem:normalvector}, the terms \( J_{25}^m(\phi_h) \) and \( J_{26}^m(\phi_h) \) can be estimated as follows:
\begin{align}
    & \quad\,\,|J_{25}^m(\phi_h)| + |J_{26}^m(\phi_h)| \notag\\
    &\lesssim h^{-1} \|\nahm - \nhm\|_{L^2(\Ghsm)} \|\nahm - \nahsm\|_{L^2(\Ghsm)} \|\phi_h\|_{L^\infty(\Ghsm)} \notag \\
    &\quad + \big(h^{-1} \tau + h^{-1} \tau \|\ehvm\|_{L^2(\Ghsm)}\big) h^{-1} \|\nahm - \nahsm\|_{L^2(\Ghsm)} \|\phi_h\|_{L^\infty(\Ghsm)} \notag \\
    &\lesssim C_\epsilon h^{-1-\epsilon} \big(\|\nabla_{\Ghsm} \ehm\|_{L^2(\Ghsm)} + (1 + \kappa_{*,l}) h^k\big) 
    \|\nabla_{\Ghsm} \ehm\|_{L^2(\Ghsm)} \|\phi_h\|_{H^1(\Ghsm)} 
    \quad\mbox{(\eqref{eq:nsa5} is used)}\notag \\
    &\quad +C_\epsilon h^{-\epsilon} \big(h^{-2} \tau + h^{-2} \tau \|\ehvm\|_{L^2(\Ghsm)}\big) 
    \|\nabla_{\Ghsm} \ehm\|_{L^2(\Ghsm)} \|\phi_h\|_{H^1(\Ghsm)}
    \quad\mbox{(here \eqref{eq:poincare3} is used)}\notag\\
	& \lesssim  h^{0.5} \|\nabla_{\Ghsm} \ehm\|_{L^2(\Ghsm)} \|\phi_h\|_{H^1(\Ghsm)} \notag\\
    &\quad + \big(h^{-2.1} \tau + h^{-2.1} \tau \|\ehvm\|_{L^2(\Ghsm)}\big) \|\nabla_{\Ghsm} \ehm\|_{L^2(\Ghsm)} \|\phi_h\|_{H^1(\Ghsm)} ,
 \end{align}
where the last inequality is obtained by choosing a sufficiently small $ \epsilon $, applying the Sobolev embedding inequality in \eqref{eq:poincare3}, and using the bounds
\[ 
(1 + \kappa_{*,l}) h^k\lesssim h^{2.6} \quad\mbox{and}\quad \|\nabla_{\Ghsm} \ehm\|_{L^2(\Ghsm)} \lesssim h^{1.6},
\] 
which follow from \eqref{cond1} and \eqref{Linfty-W1infty-hat-em}, respectively.

Using the fundamental theorem of calculus and item 6 in Lemma \ref{lemma:ud}, the term \( J_{24}^m(\phi_h) \) can be rewritten as follows:
\begin{align}
    & \quad\,\,\, J_{24}^m(\phi_h)\notag\\ 
    &= \int_{\Gamma_h^m} [I - \nahm (\nahm)^\top] \cdot \nabla_{\Gamma_h^m} \big[(\nahm - \nahsm)\phi_h\big] 
    - \int_{\Ghsm} [I - \nahm (\nahm)^\top] \cdot \nabla_{\Ghsm} \big[(\nahm - \nahsm)\phi_h\big] \notag \\
    &= \int_{0}^{1} \frac{\d}{\d \theta} \int_{\hat{\Gamma}_{h,\theta}^m} [I - \nahm (\nahm)^\top] 
    \cdot \nabla_{\hat{\Gamma}_{h,\theta}^m} \big[(\nahm - \nahsm)\phi_h\big] \d \theta \notag \\
    &= \int_{0}^{1} \int_{\hat{\Gamma}_{h,\theta}^m} [I - \nahm (\nahm)^\top] 
    \cdot \partial_\theta^\bullet \Big(\nabla_{\hat{\Gamma}_{h,\theta}^m} \big[(\nahm - \nahsm)\phi_h\big] \Big) \d \theta \notag \\
    &\quad\,\, + \int_{0}^{1} \int_{\hat{\Gamma}_{h,\theta}^m} [I - \nahm (\nahm)^\top] 
    \cdot \nabla_{\hat{\Gamma}_{h,\theta}^m} \big[(\nahm - \nahsm)\phi_h\big] 
    \big(\nabla_{\hat{\Gamma}_{h,\theta}^m} \cdot \ehm \big) \d \theta. \notag
\end{align}
By utilizing Lemma \ref{lemma:ud} (item 5), we have the following estimate of $\partial_\theta^\bullet \Big(\nabla_{\hat{\Gamma}_{h,\theta}^m} \big[(\nahm - \nahsm)\phi_h\big] \Big)$:
\begin{equation*}
    \|\partial_\theta^\bullet \Big(\nabla_{\hat{\Gamma}_{h,\theta}^m} \big[(\nahm - \nahsm)\phi_h\big]\Big)\|_{L^1(\Ghsm)} 
    \lesssim \|\nabla_{\Ghsm} \big[(\nahm - \nahsm)\phi_h\big]\|_{L^2(\Ghsm)} 
    \|\nabla_{\Ghsm} \ehm\|_{L^2(\Ghsm)}.
\end{equation*}
Therefore, the following estimate for \(J_{24}^m(\phi_h)\) holds:
\begin{align}
   |J_{24}^m(\phi_h)| 
    &\lesssim \|\nabla_{\Ghsm} \big[(\nahm - \nahsm)\phi_h\big]\|_{L^2(\Ghsm)} 
    \|\nabla_{\Ghsm} \ehm\|_{L^2(\Ghsm)} \notag \\
    &\lesssim h^{-1} \|\nahm - \nahsm\|_{L^2(\Ghsm)} 
    \|\phi_h\|_{L^\infty(\Ghsm)} \|\nabla_{\Ghsm} \ehm\|_{L^2(\Ghsm)} \notag \\
    &\lesssim C_\epsilon h^{-1-\epsilon} \|\phi_h\|_{H^1(\Ghsm)} 
    \|\nabla_{\Ghsm} \ehm\|_{L^2(\Ghsm)}^2 
    \quad \text{(using \eqref{eq:nsa1} and \eqref{eq:poincare3})} \notag \\
    &\lesssim h^{0.5} \|\phi_h\|_{H^1(\Ghsm)} \|\nabla_{\Ghsm} \ehm\|_{L^2(\Ghsm)} \qquad \text{(using \eqref{Linfty-W1infty-hat-em})}.
\end{align}
% In the above derivation, the following inequality from Lemma \ref{lemma:ud} (item 5) is utilized:

By combining the preceding estimates for $|J_{2i}^m(\phi_h)|$ with \(i = 1, \cdots, 6\), we derive the desired estimate \eqref{estJ2}.
\end{proof}

\subsection{Proof of Lemma \ref{lemma:NT_stab} (Estimate for \texorpdfstring{${\int_{\Gamma_h^m} \nabla_{\Gamma_h^m} I_h \Nbhm\evm \cdot  \nabla_{\Gamma_h^m} I_h \Tbhm\evm }$})} \label{appendix_E}
		\renewcommand{\theequation}{E.\arabic{equation}}

\begin{proof}
    Using the fundamental theorem of calculus, the geometric perturbation estimates \eqref{perturbation2} in Lemma \ref{geometric-perturbation}, we have
    \begin{align}\label{inter-orthogonal}
        &\quad\,\,\Big| \int_{\Gamma_h^m} \nabla_{\Gamma_h^m} I_h \Nbhm\evm \cdot  \nabla_{\Gamma_h^m} I_h \Tbhm\evm \Big| \notag\\
        &\leq \Big| \int_{\Gamma_h^m} \nabla_{\Gamma_h^m} I_h \Nbhm\evm \cdot  \nabla_{\Gamma_h^m} I_h \Tbhm\evm 
        - \int_{\Ghsm} \nabla_{\Ghsm} I_h \Nbhm\evm \cdot  \nabla_{\Ghsm} I_h \Tbhm\evm \Big| \notag\\
        &\quad + \Big| \int_{\Ghsm} \nabla_{\Ghsm} I_h \Nbhm\evm \cdot  \nabla_{\Ghsm} I_h \Tbhm\evm 
        - \int_{\Gm} \nabla_{\Gm} (I_h \Nbhm\evm)^l \cdot  \nabla_{\Gm} (I_h \Tbhm\evm)^l \Big| \notag\\
        &\quad + \Big| \int_{\Gm} \nabla_{\Gm} (I_h \Nbhm\evm)^l \cdot  \nabla_{\Gm} (I_h \Tbhm\evm)^l \Big| \notag\\
        &\lesssim \| \nabla_\Ghsm \ehm \|_{L^\infty(\Ghsm)} \| \nabla_{\Ghsm} I_h \Nbhm\evm \|_{L^2(\Ghsm)} \| \nabla_{\Ghsm} I_h \Tbhm\evm \|_{L^2(\Ghsm)} \notag\\
        &\quad + (1 + \kappa_{*,l}) h^{k+1} \| \nabla_{\Ghsm} I_h \Nbhm\evm \|_{L^\infty(\Ghsm)} \| \nabla_{\Ghsm} I_h \Tbhm\evm \|_{L^2(\Ghsm)} \notag\\
        &\quad + \Big| \int_{\Gm} \nabla_{\Gm} (I_h \Nbhm\evm)^l \cdot  \nabla_{\Gm} (I_h \Tbhm\evm)^l \Big|.
    \end{align}
    Applying inverse inequality and Young's inequality, from \eqref{inter-orthogonal}, we further obtain:
    \begin{align}
        &\quad\,\,\Big| \int_{\Gamma_h^m} \nabla_{\Gamma_h^m} I_h \Nbhm\evm \cdot  \nabla_{\Gamma_h^m} I_h \Tbhm\evm \Big| \notag\\
        &\lesssim \epsilon \| \nabla_\Ghsm I_h \Tbhm\evm \|_{L^2(\Ghsm)}^2+ \epsilon^{-1} \Big( (1 + \kappa_{*,l}) h^{k-1} + h^{-2} \| \nabla_\Ghsm \ehm \|_{L^2(\Ghsm)} \Big)^2 \| I_h \Nbhm\evm \|_{L^2(\Ghsm)}^2 \notag\\
        &\quad + \Big| \int_{\Gm} \nabla_{\Gm} (I_h \Nbhm\evm)^l \cdot  \nabla_{\Gm} (I_h \Tbhm\evm)^l \Big|\notag\\
		& \lesssim \epsilon \| \nabla_\Ghsm I_h \Tbhm\evm \|_{L^2(\Ghsm)}^2+ \epsilon^{-1} \Big( h^{3.2} + h^{-4} \| \nabla_\Ghsm \ehm \|_{L^2(\Ghsm)}^2 \Big) \| I_h \Nbhm\evm \|_{L^2(\Ghsm)}^2 \notag\\
		&\quad + \Big| \int_{\Gm} \nabla_{\Gm} (I_h \Nbhm\evm)^l \cdot  \nabla_{\Gm} (I_h \Tbhm\evm)^l \Big|,\label{intermediate1-ortho}
    \end{align}
	where the mathematical induction assumption on \eqref{cond1} is used in the last inequality.
	By utilizing the super-approximation estimates \eqref{super-bar-tangent} and \eqref{super-bar-normal} from Lemma \ref{super-bar-nhm}, together with the bound for $\|\bar n_h^m\|_{W^{1,\infty}(\Ghsm)}$ given in \eqref{nhmW-2a}, as well as the norm equivalence between $\Gamma^m$ and $\Ghsm$ established in \eqref{norm-equiv-lift}, the following estimate is obtained:
	\begin{align}\label{eq:tan_stab212}
		&\quad\,\,\Big| \int_{\Gm} \nabla_{\Gm} (I_h \Nbhm \evm)^l \cdot  \nabla_{\Gm} (I_h (\Tbhm (I_h \Tbhm  \evm)))^l \Big| \notag\\
		&\leq \Big| \int_{\Gm} \nabla_{\Gm} (\Nbhm \evm)^l \cdot  \nabla_{\Gm} (\Tbhm (I_h \Tbhm \evm))^l \Big| \notag\\
		&\quad + \Big| \int_{\Gm} \nabla_{\Gm} ((1 - I_h)\Nbhm \evm)^l \cdot  \nabla_{\Gm} (I_h \Tbhm \evm)^l \Big| \notag\\
		&\quad + \Big| \int_{\Gm} \nabla_{\Gm} (I_h\Nbhm \evm)^l \cdot  \nabla_{\Gm} ((1 - I_h) \Tbhm( I_h \Tbhm  \evm))^l \Big| \notag\\
		&\quad + \Big| \int_{\Gm} \nabla_{\Gm} ((1 - I_h) \Nbhm \evm)^l \cdot  \nabla_{\Gm} ((1 - I_h) \Tbhm (I_h \Tbhm  \evm))^l \Big| \notag\\
		&\lesssim \Big| \int_{\Gm} \nabla_{\Gm} (\Nbhm \evm)^l \cdot  \nabla_{\Gm} (\Tbhm (I_h \Tbhm  \evm))^l \Big| \notag\\
		&\quad + (h + h^{-1} \|\nabla_{\Ghsm} \hat e_h^m\|_{L^2(\Ghsm)})  \|\evm\|_{H^1(\Ghsm)} \| \nabla_{\Ghsm} I_h \Tbhm \evm \|_{L^2(\Ghsm)} \notag\\
		&\quad + (h + h^{-1} \|\nabla_{\Ghsm} \hat e_h^m\|_{L^2(\Ghsm)}) \| \nabla_{\Ghsm} I_h\Nbhm \evm \|_{L^2(\Ghsm)} \| I_h \Tbhm \evm \|_{H^1(\Ghsm)} \notag\\
		&\quad + (h + h^{-1} \|\nabla_{\Ghsm} \hat e_h^m\|_{L^2(\Ghsm)})^2 \| \evm \|_{H^1(\Ghsm)} \| I_h\Tbhm \evm \|_{H^1(\Ghsm)}.
	\end{align}
	% where the following estimates, derived similarly as the proof of Lemma \ref{super-bar-nhm} by the Lagrange interpolation approximation properties and the norm equivalences beween the exact surface $\Gamma^m$ and the interpolated surface $\Ghsm$, are used in the last inequality:
	% \begin{subequations}\label{I-Ih-bNh-bTh}
	% 	\begin{align}
	% 		\| \nabla_{\Gm} ((1 - I_h)\Nbhm v_h)^l \|_{L^2(\Gm)} 
	% 		&\lesssim (h + h^{-1} \|\nabla_{\Ghsm} \hat e_h^m\|_{L^2(\Ghsm)}) \|v_h\|_{H^1(\Ghsm)},\label{I-Ih-bNh-bTh-a} \\
	% 		\| \nabla_{\Gm} ((1 - I_h) \Tbhm v_h)^l \|_{L^2(\Gm)} 
	% 		&\lesssim (h + h^{-1} \|\nabla_{\Ghsm} \hat e_h^m\|_{L^2(\Ghsm)})  \|v_h\|_{H^1(\Ghsm)},\label{I-Ih-bNh-bTh-b}
	% 	\end{align}
	% \end{subequations}
	% for any finite element function \( v_h\in S_h(\Ghsm)^3 \).

	By decomposing \( e_v^m \) on the right-hand side of \eqref{eq:tan_stab212} into \( I_h \Nbhm e_v^m \) and \( I_h \Tbhm e_v^m \), and subsequently applying the inverse inequality to \( \| I_h \Nbhm e_v^m \|_{H^1(\Ghsm)} \) together with the Poincaré-type inequality \eqref{poincare4}, and the mathematical induction assumption \eqref{Linfty-W1infty-hat-em}, the following bounds are obtained from \eqref{eq:tan_stab212}:
	\begin{align}\label{eq:tan_stab2}
		&\quad\,\,\Big| \int_{\Gm} \nabla_{\Gm} (I_h \Nbhm \evm)^l \cdot \nabla_{\Gm} (I_h (\Tbhm (I_h \Tbhm \evm)))^l \Big| \notag \\
		&\lesssim \Big| \int_{\Gm} \nabla_{\Gm} (\Nbhm \evm)^l \cdot \nabla_{\Gm} (\Tbhm (I_h \Tbhm \evm))^l \Big| + \big(h + h^{-1}\| \nabla_{\Ghsm}\ehm \|_{L^2(\Ghsm)}\big)\notag\\
		&\quad\,\,\cdot \big( \| \nabla I_h \Tbhm \evm \|_{L^2(\Ghsm)} + h^{-1} \| I_h \Nbhm \evm \|_{L^2(\Ghsm)} \big) \| \nabla I_h \Tbhm \evm \|_{L^2(\Ghsm)} \notag \\
		&\lesssim \Big| \int_{\Gm} \nabla_{\Gm} (\Nbhm \evm)^l \cdot \nabla_{\Gm} (\Tbhm (I_h \Tbhm \evm))^l \Big| + \big(h + h^{-1}\| \nabla_{\Ghsm}\ehm \|_{L^2(\Ghsm)}\big) \| \nabla I_h \Tbhm \evm \|^2_{L^2(\Ghsm)} \notag \\
		&\quad + \epsilon \| \nabla I_h \Tbhm \evm \|_{L^2(\Ghsm)}^2 + \epsilon^{-1} \big(1 + h^{-2}\| \nabla_{\Ghsm}\ehm \|_{L^2(\Ghsm)}\big)^2 \| I_h \Nbhm \evm \|_{L^2(\Ghsm)}^2 \notag \\
		&\lesssim \Big| \int_{\Gm} \nabla_{\Gm} (\Nbhm \evm)^l \cdot \nabla_{\Gm} (\Tbhm (I_h \Tbhm \evm))^l \Big| + (\epsilon + h^{0.6}) \| \nabla I_h \Tbhm \evm \|^2_{L^2(\Ghsm)} \notag\\
		& \quad\,\,+ \epsilon^{-1} \big(1 + h^{-2}\| \nabla_{\Ghsm}\ehm \|_{L^2(\Ghsm)}\big)^2 \| I_h \Nbhm \evm \|_{L^2(\Ghsm)}^2,
	\end{align}
	where we have utilized the mathematical induction assumption \eqref{Linfty-W1infty-hat-em} in the last inequality. 
	The first term on the right-hand side of \eqref{eq:tan_stab2} can be decomposed as
	\begin{align}
		&\quad\,\, \Big| \int_{\Gm} \nabla_{\Gm} (\Nbhm \evm)^l \cdot \nabla_{\Gm} (\Tbhm (I_h\Tbhm \evm))^l \Big| \notag \\
		&\leq \Big| \int_{\Gm} \nabla_{\Gm} (N_*^m\Nbhm \evm)^l \cdot \nabla_{\Gm} (T_*^m\Tbhm (I_h\Tbhm \evm))^l \Big| \notag \\
		&\quad + \Big|\int_{\Gm} \nabla_{\Gm} ((\Nahm - N_*^m)\Nbhm \evm)^l \cdot \nabla_{\Gm} (T_*^m \Tbhm (I_h\Tbhm \evm))^l \Big|\notag \\
		&\quad + \Big|\int_{\Gm} \nabla_{\Gm} (\Nahm\Nbhm \evm)^l \cdot \nabla_{\Gm} ((\Tahm - T_*^m)\Tbhm (I_h\Tbhm \evm))^l\Big|. \notag
	\end{align}
	Using the product rule of differentiation, Sobolev embedding inequality \eqref{eq:poincare3}, Lemma \ref{Lem:normalvector} and the norm equivalence between $\Gamma^m$ and $\Ghsm$ established in \eqref{norm-equiv-lift}, the above terms are bounded as follows:
	\begin{align}
		& \quad\,\,\Big| \int_{\Gm} \nabla_{\Gm} (\Nbhm \evm)^l \cdot \nabla_{\Gm} (\Tbhm (I_h\Tbhm \evm))^l \Big| \notag \\
		&\lesssim \Big| \int_{\Gm} \nabla_{\Gm} (N_*^m\Nbhm \evm)^l \cdot \nabla_{\Gm} (T_*^m \Tbhm (I_h\Tbhm \evm))^l \Big|  +\Big( \|\bar N_h^m  - N_*^m\|_{L^\infty(\Ghsm)} \|\bar N_h^m e_v^m\|_{H^1(\Ghsm)} \notag\\
		& \quad\,\,+ \|\bar N_h^m  - N_*^m\|_{H^1(\Ghsm)} \|\bar N_h^m e_v^m\|_{L^\infty(\Ghsm)} \Big)\| \Tbhm (I_h\Tbhm \evm) \|_{H^1(\Ghsm)} + \Big(\|\bar T_h^m - T_*^m\|_{L^\infty(\Ghsm)} \notag\\
		& \quad\,\,\cdot\|\bar T_h^m (I_h \bar T_h^m e_v^m)\|_{H^1(\Ghsm)} + \|\bar T_h^m - T_*^m\|_{H^1(\Ghsm)}\|\bar T_h^m (I_h \bar T_h^m e_v^m)\|_{L^\infty(\Ghsm)}\Big)\|\bar N_h^m e_v^m\|_{H^1(\Ghsm)}\notag\\
		&\lesssim \Big| \int_{\Gm} \nabla_{\Gm} (N_*^m\Nbhm \evm)^l \cdot \nabla_{\Gm} (T_*^m \Tbhm (I_h\Tbhm \evm))^l \Big|  + h^{-1.1}\Big((1+\kappa_{*,l})h^k \notag\\
		& \quad\,\,+ \|\nabla_\Ghsm \ehm \|_{L^2(\Ghsm)}\Big) \| \Nbhm \evm \|_{H^1(\Ghsm)} \| \Tbhm (I_h\Tbhm \evm) \|_{H^1(\Ghsm)}.\label{eq:tan_stab3-intermediate}
	% 	& \lesssim \Big| \int_{\Gm} \nabla_{\Gm} (N_*^m\Nbhm \evm)^l \cdot \nabla_{\Gm} (T_*^m \Tbhm (I_h\Tbhm \evm))^l \Big| \notag \\
	% 	&\quad + \Big[\epsilon + h^{-1-\epsilon}\Big((1+\kappa_{*,l})h^k + \|\nabla_\Ghsm \ehm \|_{L^2(\Ghsm)}\Big)\Big] \| I_h\Tbhm \evm \|_{H^1(\Ghsm)}^2 \notag \\
	% 	&\quad + \epsilon^{-1} h^{-2-\epsilon}\Big((1+\kappa_{*,l})h^k + \|\nabla_\Ghsm \ehm \|_{L^2(\Ghsm)}\Big)^2 \| I_h \Nbhm \evm \|_{H^1(\Ghsm)}^2\notag\\
	% 	& \lesssim  \Big| \int_{\Gm} \nabla_{\Gm} (N_*^m\Nbhm \evm)^l \cdot \nabla_{\Gm} (T_*^m \Tbhm I_h\Tbhm \evm)^l \Big|+ (\epsilon + h^{0.5})\|\nabla_{\Ghsm}I_h\Tbhm \evm\|_{L^2(\Ghsm)}^2 \notag \\
	% 	& \quad  + \epsilon
	% 	^{-1}h^{-4-\epsilon} \Big((1+\kappa_{*,l})h^k + \|\nabla_\Ghsm \ehm \|_{L^2(\Ghsm)}\Big)^2 \| I_h \Nbhm \evm \|_{L^2(\Ghsm)}^2 ,\label{eq:tan_stab3}
	\end{align}
	% where the last inequality uses the mesh size condition \eqref{cond1}, the induction assumption \eqref{Linfty-W1infty-hat-em}, the Poincaré inequality \eqref{poincare4} and inverse inequalities.
	
	By employing the super-approximation estimates \eqref{super-bar-tangent} and \eqref{super-bar-normal}, the triangle inequality, the estimate \eqref{nhmW-2a} and the mathematical induction assumption \eqref{Linfty-W1infty-hat-em}, we derive
	\begin{subequations}\label{TN}
	\begin{align}
		\| \Tbhm (I_h\Tbhm \evm) \|_{H^1(\Ghsm)} &\lesssim \| I_h(\Tbhm (I_h \Tbhm \evm)) \|_{H^1(\Ghsm)} + h\|\bar n_h^m\|_{W^{1,\infty}(\Ghsm)}  \|I_h \Tbhm \evm\|_{H^1(\Ghsm)}\notag \\
		&\lesssim \|I_h \Tbhm \evm\|_{H^1(\Ghsm)}+h^{0.6} \|I_h \Tbhm \evm\|_{H^1(\Ghsm)} \notag\\
		&\lesssim  \|I_h \Tbhm \evm\|_{H^1(\Ghsm)}, \label{TN-a}\\
		\| \Nbhm \evm \|_{H^1(\Ghsm)} &\lesssim \|I_h \Nbhm \evm \|_{H^1(\Ghsm)} + h  \|\bar n_h^m\|_{W^{1,\infty}(\Ghsm)} \|\evm\|_{H^1(\Ghsm)} \notag\\
		& \lesssim  \|I_h \Nbhm \evm \|_{H^1(\Ghsm)} + h^{0.6}(\|I_h \Nbhm \evm \|_{H^1(\Ghsm)} + \|I_h \Tbhm \evm \|_{H^1(\Ghsm)}) \notag\\
		&\lesssim \|I_h \Nbhm \evm\|_{H^1(\Ghsm)} + h^{0.6}\|\nabla_{\Ghsm}I_h \Tbhm \evm\|_{L^2(\Ghsm)}\label{TN-b},
	\end{align}
	\end{subequations}
	where the Poincaré-type inequality \eqref{poincare4} is used in the last inequality.

	By substituting \eqref{TN} into \eqref{eq:tan_stab3-intermediate}, we obtain 
	\begin{align}
		& \quad\,\,\Big| \int_{\Gm} \nabla_{\Gm} (\Nbhm \evm)^l \cdot \nabla_{\Gm} (\Tbhm (I_h\Tbhm \evm))^l \Big| \notag \\
		&\lesssim \Big| \int_{\Gm} \nabla_{\Gm} (N_*^m\Nbhm \evm)^l \cdot \nabla_{\Gm} (T_*^m \Tbhm (I_h\Tbhm \evm))^l \Big|  + h^{-1.1}\Big((1+\kappa_{*,l})h^k \notag\\
		& \quad\,\,+ \|\nabla_\Ghsm \ehm \|_{L^2(\Ghsm)}\Big) (\|I_h \Nbhm \evm \|_{H^1(\Ghsm)} + h^{0.6}\|\nabla_{\Ghsm}I_h \Tbhm \evm\|_{L^2(\Ghsm)}) \| I_h\Tbhm \evm \|_{H^1(\Ghsm)}\notag\\
			& \lesssim \Big| \int_{\Gm} \nabla_{\Gm} (N_*^m\Nbhm \evm)^l \cdot \nabla_{\Gm} (T_*^m \Tbhm (I_h\Tbhm \evm))^l \Big| \notag \\
		&\quad + \Big[\epsilon+h^{-0.5}\Big((1+\kappa_{*,l})h^k + \|\nabla_\Ghsm \ehm \|_{L^2(\Ghsm)}\Big)\Big] \| \nabla_{\Ghsm} I_h\Tbhm \evm \|_{L^2(\Ghsm)}^2 \notag \\
		&\quad + \epsilon^{-1} h^{-2.2}\Big((1+\kappa_{*,l})h^k + \|\nabla_\Ghsm \ehm \|_{L^2(\Ghsm)}\Big)^2 \| I_h \Nbhm \evm \|_{H^1(\Ghsm)}^2\notag\\
		& \lesssim  \Big| \int_{\Gm} \nabla_{\Gm} (N_*^m\Nbhm \evm)^l \cdot \nabla_{\Gm} (T_*^m \Tbhm (I_h\Tbhm \evm))^l \Big|+ (\epsilon + h^{1.1})\|\nabla_{\Ghsm}I_h\Tbhm \evm\|_{L^2(\Ghsm)}^2 \notag \\
		& \quad  + \epsilon
		^{-1}\Big(1+  h^{-4.2} \|\nabla_\Ghsm \ehm \|_{L^2(\Ghsm)}^2\Big) \| I_h \Nbhm \evm \|_{L^2(\Ghsm)}^2 ,\label{eq:tan_stab3}
	\end{align}
	where the Poincaré-type inequality \eqref{poincare4}, the mathematical induction assumption \eqref{Linfty-W1infty-hat-em}, the mesh size assumption \eqref{cond1} are used in deriving the above inequality.
		
	Since $N_*^m=N^m=N^mN^m$ and $T_*^m=T^m=T^mT^m$ on $\Gamma^m$, the first term on the right-hand side of \eqref{eq:tan_stab3} can be estimated as follows:
	\begin{align*}
		&\quad\,\, \Big| \int_{\Gm} \nabla_{\Gm} [N^mN^m(\Nbhm \evm)^l] \cdot \nabla_{\Gm} [T^mT^m(\Tbhm (I_h\Tbhm \evm))^l] \Big| \notag \\
		&\leq \Big| \int_{\Gm} (\nabla_{\Gm} N^m) N^m (\Nbhm \evm)^l \cdot (\nabla_{\Gm} T^m) T^m (\Tbhm (I_h\Tbhm \evm))^l \Big| \notag \\
		&\quad + \Big| \int_{\Gm} (\nabla_{\Gm} [N^m(\Nbhm \evm)^l]) N^m \cdot (\nabla_{\Gm} [T^m(\Tbhm (I_h\Tbhm \evm))^l]) T^m \Big| \notag \\
		&\quad + \Big| \int_{\Gm} (\nabla_{\Gm} N^m) N^m (\Nbhm \evm)^l \cdot (\nabla_{\Gm} [T^m(\Tbhm (I_h\Tbhm \evm))^l]) T^m \Big| \notag \\
		&\quad + \Big| \int_{\Gm} (\nabla_{\Gm} [N^m(\Nbhm \evm)^l]) N^m \cdot (\nabla_{\Gm} T^m) T^m (\Tbhm (I_h\Tbhm \evm))^l \Big|,
	\end{align*}
	where the second term vanishes due to the orthogonality between the projections \(N^m\) and \(T^m\). For the last term, we can remove the gradient acting on \(N^m(\Nbhm \evm)^l\) by using integration by parts. This results in the following estimate:
	\begin{align}\label{eq:tan_stab51}
		&\quad\,\,\Big| \int_{\Gm} \nabla_{\Gm} [N^m(\Nbhm \evm)^l] \cdot \nabla_{\Gm} [T^m(\Tbhm (I_h\Tbhm \evm))^l] \Big| \notag \\
		&\lesssim \epsilon^{-1} \| \Nbhm \evm \|_{L^2(\Ghsm)}^2 + \epsilon \| \Tbhm (I_h\Tbhm \evm) \|_{H^1(\Ghsm)}^2 \notag \\
		&\lesssim \epsilon^{-1} \| I_h \Nbhm \evm \|_{L^2(\Ghsm)}^2 
		+ (\epsilon + \epsilon^{-1}h^{1.2}) \| \nabla_\Ghsm I_h \Tbhm \evm \|_{L^2(\Ghsm)}^2,
	\end{align}
	where the last inequality uses \eqref{TN} and the Poincaré-type inequality \eqref{poincare4}. 
    
    By combining the estimates \eqref{intermediate1-ortho}, \eqref{eq:tan_stab2}, \eqref{eq:tan_stab3} and \eqref{eq:tan_stab51}, we obtain the desired result in \eqref{eq:NT}.
	\end{proof}

	\subsection{Proof of Lemma \ref{e_kappa_l2} (Estimate for $\|\ek\|_{L^2(\Ghsm)}$)} \label{appendix_F}
	\renewcommand{\theequation}{F.\arabic{equation}}
\begin{proof}
	By testing the error equation \eqref{ekk} with \(\chi_h := I_h(\ek \nahm)\), we derive
\begin{align}
    \int_{\Ghm} |\ek|^2 
    &= \int_{\Ghm} |\ek|^2 \big(1 - |\nahm|^2\big) 
    + \int_{\Ghm} \ek \nahm \cdot (1 - I_h)(\ek \nahm) 
    + \int_{\Ghm} \ek \nahm \cdot I_h(\ek \nahm) \notag \\
    &= \int_{\Ghm} |\ek|^2 \big(1 - |\nahm|^2\big) 
    + \int_{\Ghm} \ek \nahm \cdot (1 - I_h)(\ek \nahm) 
    + \int_{\Gamma_h^m} \nabla_{\Gamma^m_h} \evm \cdot \nabla_{\Gamma^m_h} I_h(\ek \nahm) \notag \\
    &\quad + R_1^m(I_h(\ek \nahm)) + R_2^m(I_h(\ek \nahm)) + d_\kappa^m(I_h(\ek \nahm)). \label{estek}
\end{align}
The first term on the right-hand side of \eqref{estek} can be estimated by using the mesh size assumption \eqref{cond1}, the mathematical induction assumption \eqref{Linfty-W1infty-hat-em}, and \eqref{1nhmf} as follows:
\begin{align}
    \int_{\Ghm} |\ek|^2 \big(1 - |\nahm|^2\big) 
    &\lesssim \|1 - |\nahm|\|_{L^\infty(\Ghsm)}\|1 + |\nahm|\|_{L^\infty(\Ghsm)} \|\ek\|^2_{L^2(\Ghsm)} \notag \\
    &\lesssim \big((1 + \kappa_{*,l})h^{k-1} + h^{-1} \| \nabla_{\Ghsm} \ehm \|_{L^2(\Ghsm)}\big) \|\ek\|^2_{L^2(\Ghsm)} \notag \\
    &\lesssim h^{0.6} \|\ek\|^2_{L^2(\Ghsm)}\label{kappa-intermediate-1}.
\end{align}
The second term on the right-hand side of \eqref{estek} can be estimated by applying \eqref{super2} in Lemma \ref{super}, together with \eqref{nhmW-2b}, \eqref{nhmW-2a}, and inverse inequality, as follows:
\begin{align}
    \Big|\int_{\Ghm} \ek \nahm \cdot (1 - I_h)(\ek \nahm)\Big| 
    &\lesssim h \|\nahm\|_{W^{1,\infty}(\Ghsm)} \|\nahm\|_{L^\infty(\Ghsm)} \|\ek\|^2_{L^2(\Ghsm)} \notag \\
    &\lesssim \big(h + h^{-1} \|\nabla_{\Ghsm} \ehm\|_{L^2(\Ghsm)}\big) \|\ek\|^2_{L^2(\Ghsm)} \notag \\
    &\lesssim h^{0.6} \|\ek\|^2_{L^2(\Ghsm)}\label{kappa-intermediate-2},
\end{align}
where the last inequality follows from the mathematical induction assumption \eqref{Linfty-W1infty-hat-em}. 
The third term on the right-hand side of \eqref{estek} can be estimated by applying inverse inequality and stability estimate \eqref{stability2}, as follows:
\begin{align}
   \Big|\int_{\Gamma_h^m} \nabla_{\Gamma^m_h} \evm \cdot \nabla_{\Gamma^m_h} I_h(\ek \nahm)\Big|&\lesssim h^{-1} \|\nabla_{\Ghsm} \evm\|_{L^2(\Ghsm)} \|\ek \nahm\|_{L^2(\Ghsm)} 
   \notag \\
    &\lesssim h^{-1} \|\nabla_{\Ghsm} \evm\|_{L^2(\Ghsm)} \|\ek\|_{L^2(\Ghsm)}\label{kappa-intermediate-3},
\end{align}
where the last inequality uses \eqref{nhmW-2b}.
By applying inequalities \eqref{dk}, \eqref{estR1}, and \eqref{estR2}, we derive the following estimate:
\begin{align}
    &\quad\,\,|R_1^m(I_h(\ek \nahm))| + |R_2^m(I_h(\ek \nahm))| + |d_\kappa^m(I_h(\ek \nahm))| \notag \\
    &\lesssim \|\nabla_{\Ghsm} \ehm \|_{L^2(\Ghsm)} \|\nabla_{\Ghsm} I_h(\ek \nahm)\|_{L^2(\Ghsm)} 
    + \|\nabla_{\Ghsm} \ehm \|_{L^2(\Ghsm)} \|I_h(\ek \nahm)\|_{L^2(\Ghsm)} \notag \\
    &\quad + (1 + \kappa_{*,l}) h^k \|I_h(\ek \nahm)\|_{H^1(\Ghsm)} \notag \\
    &\lesssim (h^{-1} \|\nabla_{\Ghsm} \ehm \|_{L^2(\Ghsm)} + (1 + \kappa_{*,l}) h^{k-1}) \|\ek\|_{L^2(\Ghsm)}\label{kappa-intermediate-4},
\end{align}
where stability result \eqref{stability2}, inequality \eqref{nhmW-2b} and inverse inequality are used in deriving the above estimate.

Substituting the estimates \eqref{kappa-intermediate-1}-\eqref{kappa-intermediate-4} into \eqref{estek}, and using the norm equivalence between \(\Ghm\) and \(\Ghsm\), we obtain the following estimate:
\begin{align}
    \|\ek\|^2_{L^2(\Ghsm)} 
    &\lesssim h^{0.6} \|\ek\|^2_{L^2(\Ghsm)} 
    + h^{-1} \|\nabla_{\Ghsm} \evm\|_{L^2(\Ghsm)} \|\ek\|_{L^2(\Ghsm)} \notag \\
    &\quad + (h^{-1} \|\nabla_{\Ghsm} \ehm\|_{L^2(\Ghsm)} + (1 + \kappa_{*,l}) h^{k-1}) \|\ek\|_{L^2(\Ghsm)}. \label{estek1}
\end{align}

For sufficiently small \(h\), the first term on the right-hand side of \eqref{estek1} can be absorbed into the left-hand side, resulting in the following bound:
\begin{align}
    \|\ek\|_{L^2(\Ghsm)} 
    &\lesssim h^{-1} \|\nabla_{\Ghsm} e_v^m\|_{L^2(\Ghsm)} 
    + h^{-1} \|\nabla_{\Ghsm} \ehm\|_{L^2(\Ghsm)} 
    + (1 + \kappa_{*,l}) h^{k-1}. \label{estekL}
\end{align}
By using the relation $e_v^m = \hat e_v^m + I_h g^m$ in \eqref{xm}, together with the \(W^{1,\infty}\) stability of the Lagrange interpolation and the inequality \eqref{W1infty-g}, we derive:
\begin{align}
    \|\nabla_{\Ghsm} \evm\|_{L^2(\Ghsm)} 
    &\lesssim \|\nabla_{\Ghsm} \ehvm\|_{L^2(\Ghsm)} + \|\nabla_{\Ghsm} I_h g^m\|_{L^2(\Ghsm)} \notag\\
    &\lesssim \|\nabla_{\Ghsm} \ehvm\|_{L^2(\Ghsm)} + \|g^m\|_{W^{1,\infty}(\Gamma^m)} \notag\\
    &\lesssim \|\nabla_{\Ghsm} \ehvm\|_{L^2(\Ghsm)} + \tau. \label{intermediate-intermediate}
\end{align}
By substituting the above estimate \eqref{intermediate-intermediate} into \eqref{estekL}, and further applying the estimate \eqref{evmH1} together with the induction assumption \eqref{Linfty-W1infty-hat-em}, the following estimate is obtained:
\begin{align}\label{kappa-temp}
    \|\ek\|_{L^2(\Ghsm)} 
    &\lesssim h^{-1} \tau + (1 + \kappa_{*,l})h^{k-1} + h^{-1} \| \nabla_{\Ghsm} \ehm \|_{L^2(\Ghsm)} \notag \\
    &\quad + \big(h + h^{-1} \| \nabla_{\Ghsm} \ehm \|_{L^2(\Ghsm)}\big) \|\ek\|_{L^2(\Ghsm)}+h^{-2}  \| I_h \Nahsm \ehvm \|_{L^2(\Ghsm)} \notag \\
    &\lesssim h^{-1} \tau + (1 + \kappa_{*,l})h^{k-1} + h^{-1} \| \nabla_{\Ghsm} \ehm \|_{L^2(\Ghsm)} 
    + h^{0.6} \|\ek\|_{L^2(\Ghsm)} \notag \\
    &\quad + h^{-2} \| I_h \Nahsm \ehvm \|_{L^2(\Ghsm)}.
\end{align}
For sufficiently small \(h\), the fourth term on the right-hand side of \eqref{kappa-temp} can be absorbed into the left-hand side, thereby yielding the estimate \eqref{estekL2} as stated in Lemma \ref{e_kappa_l2}.

\end{proof}

\subsection{Proof of \eqref{stab} (Converting \(\| \eM \cdot \nbhsm \|_{L^2(\Ghsm)}^2\) to \(\| \ehM \cdot \nbhsM \|_{L^2(\hat\Gamma_{h,*}^{m+1})}^2\))} \label{appendix_G}
		\renewcommand{\theequation}{G.\arabic{equation}}
\begin{proof}
Analogous to the results in \cite[Eqs. (5.54), (5.55), (5.61), (5.62)]{bai2024new}, and by applying Lemma \ref{Xhinfty} along with the inequality \eqref{bound-n} to control \(\|\bar{n}^{m+1}_{h,*}\|_{W^{1,\infty}(\Ghsm)}\), the term \(M_1^m\) can be estimated using integration by parts as follows:
\begin{equation}
	M_1^m \lesssim \epsilon \tau \|\ehM\|_{H^1(\Ghsm)}^2 + \epsilon^{-1} \tau \|\ehM\|_{L^2(\Ghsm)}^2.
\end{equation}

Decompose $M_2^m$ into several parts as follows:
\begin{align}\label{def:M2im}
	M_2^m&= \| \ehM\cdot \nbhsM \|_{L^2(\hat\Gamma_{h,*}^{m})}^2  - \| \ehM\cdot  \nbhsm \|_{L^2(\Ghsm)}^2\notag\\
	 &=\int_{\Ghsm}(\ehM\cdot(\nbhsM+\nbhsm))(\ehM\cdot(\nbhsM-\nahsm))\notag\\
	 &=\int_{\Ghsm}(1-I_h)\big[(\ehM\cdot(\nbhsM+\nbhsm))\ehM\big]\cdot(\nbhsM-\nahsm)\notag\\
	 &\quad+\int_{\Ghsm}I_h\big[(\ehM\cdot(\nbhsM+\nbhsm))\ehM\big]\cdot(\nbhsM-\nbhsm)\notag\\
	 &=\int_{\Ghsm}(1-I_h)\big[(\ehM\cdot(\nbhsM+\nbhsm))\ehM\big]\cdot(\nbhsM-\nahsm)\notag\\
	 &\quad+\int_{0}^{1}\int_{\Ghsm}I_h\big[(\ehM\cdot(\nbhsM+\nbhsm))\ehM\big]\cdot\partial_\theta^\bullet \bar{n}^{m+\theta}_{h,*}\d\theta\notag\\
	 &=\int_{\Ghsm}(1-I_h)\big[(\ehM\cdot(\nbhsM+\nbhsm))\ehM\big]\cdot(\nbhsM-\nahsm)\notag\\
	 &\quad+\int_{0}^{1}\int_{\hat{\Gamma}_{h,*}^{m+\theta}}I_h\big[(\ehM\cdot(\nbhsM+\nbhsm))\ehM\big]\cdot\partial_\theta^\bullet \bar{n}^{m+\theta}_{h,*}\d\theta\notag\\
	 &\quad+\Big[\int_{0}^{1}\int_{\Ghsm}I_h\big[(\ehM\cdot(\nbhsM+\nbhsm))\ehM\big]\cdot\partial_\theta^\bullet \bar{n}^{m+\theta}_{h,*}\d\theta\notag\\
	 &\quad-\int_{0}^{1}\int_{\hat{\Gamma}_{h,*}^{m+\theta}}I_h\big[(\ehM\cdot(\nbhsM+\nbhsm))\ehM\big]\cdot\partial_\theta^\bullet \bar{n}^{m+\theta}_{h,*}\d\theta\Big]\notag\\
	 &=\int_{\Ghsm}(1-I_h)\big[(\ehM\cdot(\nbhsM+\nbhsm))\ehM\big]\cdot(\nbhsM-\nahsm)
     \quad\mbox{(below we use \eqref{def:dnM})}\notag\\
	 &\quad-\int_{0}^{1}\int_{\hat{\Gamma}_{h,*}^{m+\theta}}I_h\big[(\ehM\cdot(\nbhsM+\nbhsm))\ehM\big]\cdot\big(\nabla_{\hat{\Gamma}_{h,*}^{m+\theta} } \big(\hat X_{h,*}^{m+1} - \hat X_{h,*}^{m}\big)\bar{n}^{m+\theta}_{h,*}\big)\d\theta\notag\\
	 &\quad +\int_{0}^{1}\int_{\hat{\Gamma}_{h,*}^{m+\theta}}I_h\big[(\ehM\cdot(\nbhsM+\nbhsm))\ehM\big]\cdot \big(\nabla_{\hat{\Gamma}_{h,*}^{m+\theta} }\big(\hat X_{h,*}^{m+1} - \hat X_{h,*}^{m}\big)\big(\bar{n}^{m+\theta}_{h,*}-\hat{n}^{m+\theta}_{h,*}\big)\big)\d\theta\notag\\
	 &\quad -\int_{0}^{1}\int_{\hat{\Gamma}_{h,*}^{m+\theta}}I_h\big[(\ehM\cdot(\nbhsM+\nbhsm))\ehM\big]\cdot\big(\bar{n}^{m+\theta}_{h,*}-\hat{n}^{m+\theta}_{h,*}\big)\nabla_{\hat{\Gamma}_{h,*}^{m+\theta} }\cdot\big(\hat X_{h,*}^{m+1} - \hat X_{h,*}^{m}\big)\d\theta\notag\\
	 &\quad+\Big[\int_{0}^{1}\int_{\Ghsm}I_h\big[(\ehM\cdot(\nbhsM+\nbhsm))\ehM\big]\cdot\partial_\theta^\bullet \bar{n}^{m+\theta}_{h,*}\d\theta\notag\\
	 &\quad-\int_{0}^{1}\int_{\hat{\Gamma}_{h,*}^{m+\theta}}I_h\big[(\ehM\cdot(\nbhsM+\nbhsm))\ehM\big]\cdot\partial_\theta^\bullet \bar{n}^{m+\theta}_{h,*}\d\theta\Big]\notag\\
	 &=:\sum_{i=1}^{5}M^m_{2i},
\end{align}
where the following relation is employed in the above derivation:
\begin{equation}\label{def:dnM}
	\partial_\theta^\bullet \bar{n}^{m+\theta}_{h,*} = P_{\hat{\Gamma}_{h,*}^{m+\theta}} 
	\big[
		-\nabla_{\hat{\Gamma}_{h,*}^{m+\theta}} \big(\hat{X}_{h,*}^{m+1} - \hat{X}_{h,*}^m\big) \hat{n}^{m+\theta}_{h,*} 
		- \big(\bar{n}^{m+\theta}_{h,*} - \hat{n}^{m+\theta}_{h,*}\big) 
		\nabla_{\hat{\Gamma}_{h,*}^{m+\theta}} \cdot \big(\hat{X}_{h,*}^{m+1} - \hat{X}_{h,*}^m\big)
	\big].
\end{equation}
This identity is the same as the expression of \(\partial_\theta^\bullet \nahtm\) given in \eqref{def:dnahtm}.

Using \eqref{def:dnM}, along with the \(L^p\)-stability of the \(L^2\)-projection, the norm equivalence between \(\hat \Gamma_{h,*}^{m+\theta}\) and \(\Ghsm\) inferred from \eqref{hxM}, and the estimate in \eqref{eq:nsa0-4-infty}, we obtain:
\begin{subequations}
 \begin{align}
 \|\partial_\theta^\bullet \bar{n}^{m+\theta}_{h,*}\|_{L^2(\Ghsm)} 
 &\lesssim \|\nabla_\Ghsm (\hat{X}_{h,*}^{m+1} - \hat{X}_{h,*}^m)\|_{L^2(\Ghsm)} \big(1 + \|\nabla_\Ghsm (\hat{X}_{h,*}^{m+1} - \hat{X}_{h,*}^m)\|_{L^\infty(\Ghsm)} \notag \\
 &\quad + (1 + \kappa_{*,l})h^{k-1}\big) \notag \\
 &\lesssim\|\nabla_\Ghsm (\hat{X}_{h,*}^{m+1} - \hat{X}_{h,*}^m)\|_{L^2(\Ghsm)}, \label{def:dnMresult}\\
 \|\partial_\theta^\bullet \bar{n}^{m+\theta}_{h,*}\|_{L^\infty(\Ghsm)} 
 &\lesssim \|\nabla_\Ghsm (\hat{X}_{h,*}^{m+1} - \hat{X}_{h,*}^m)\|_{L^\infty(\Ghsm)} \big(1 + \|\nabla_\Ghsm (\hat{X}_{h,*}^{m+1} - \hat{X}_{h,*}^m)\|_{L^\infty(\Ghsm)} \notag \\
 &\quad + (1 + \kappa_{*,l})h^{k-1}\big) \notag \\
 &\lesssim\|\nabla_\Ghsm (\hat{X}_{h,*}^{m+1} - \hat{X}_{h,*}^m)\|_{L^\infty(\Ghsm)}, \label{def:dnMresult-infty}
\end{align}
\end{subequations}
where the above inequalities utilize inverse inequality, \eqref{Xtau}, \eqref{cond1}, and $\tau \le c h^k$ with $k\ge 3$.

 By applying the super-approximation estimate \eqref{super6} from Lemma \ref{super}, employing Lemma \ref{Xhinfty}, together with the bound provided in \eqref{bound-n}, the term \(M^m_{21}\) can be estimated as follows:
 \begin{align}
	 M^m_{21} 
	 &\lesssim h^2 \big(\|\nbhsM\|_{W^{1,\infty}(\Ghsm)} + \|\nbhsm\|_{W^{1,\infty}(\Ghsm)}\big) 
	 \|\ehM\|_{H^1(\Ghsm)}^2 \|\nabla_\Ghsm (\hat{X}_{h,*}^{m+1} - \hat{X}_{h,*}^m)\|_{L^\infty(\Ghsm)} \notag \\
	 &\lesssim \epsilon \tau \|\ehM\|_{H^1(\Ghsm)}^2 + \epsilon^{-1} \tau \|\ehM\|_{L^2(\Ghsm)}^2.
 \end{align}
 
 By applying Lemma \ref{Xhinfty} and \eqref{eq:nsa0-4}, along with stability estimate \eqref{stability4}, we derive the following estimate for the terms \(M^m_{23}\) and \(M^m_{24}\):
\begin{align}
 M^m_{23} + M^m_{24} 
 &\lesssim \big((1+\kappa_{*,l})h^{k} + \|\nabla_{\hat{\Gamma}_{h,*}^m} (\hat{X}_{h,*}^{m+1} - \hat{X}_{h,*}^m)\|_{L^2(\Ghsm)}\big) \|\ehM\|_{L^2(\Ghsm)} \|\ehM\|_{L^\infty(\Ghsm)} \notag \\
 &\quad \cdot \|\nabla_{\hat{\Gamma}_{h,*}^m} (\hat{X}_{h,*}^{m+1} - \hat{X}_{h,*}^m)\|_{L^\infty(\Ghsm)} \notag \\
 &\lesssim h^{-3} \tau^2 \|\ehM\|_{L^2(\Ghsm)}^2  + (1+\kappa_{*,l})h^{k-2}\tau\|\ehM\|_{L^2(\Ghsm)}^2 \quad \mbox{(using inverse inequality)}\notag \\
 &\lesssim \tau \|\ehM\|_{L^2(\Ghsm)}^2,
\end{align}
where the last inequality uses \(\tau \le c h^k\) with \(k \geq 3\) and the mesh size condition \eqref{cond1}.

By utilizing Lemma \ref{lemma:ud} (item 6), together with Lemma \ref{Xhinfty}, the inequality \eqref{def:dnMresult}, the stability estimate \eqref{stability4}, the bound in \eqref{bound-n}, and the inverse inequality, the term \(M^m_{25}\) can be estimated as follows:
\begin{align}
 M^m_{25} 
 &= -\int_{0}^{1}\int_{0}^{\theta} \frac{\d}{\d\alpha} \int_{\hat{\Gamma}_{h,*}^{m+\alpha}} 
 I_h\big[\ehM\cdot(\nbhsM+\nbhsm)\ehM\big]\cdot\partial_\theta^\bullet \bar{n}^{m+\theta}_{h,*} \d\alpha \d\theta \notag \\
 &= -\int_{0}^{1}\int_{0}^{\theta} \int_{\hat{\Gamma}_{h,*}^{m+\alpha}} 
 I_h\big[\ehM\cdot(\nbhsM+\nbhsm)\ehM\big]\cdot\partial_\theta^\bullet \bar{n}^{m+\theta}_{h,*} 
 \nabla_{\hat{\Gamma}_{h,*}^{m+\alpha}} \cdot \theta\big(\hat{X}_{h,*}^{m+1} - \hat{X}_{h,*}^m\big) \d\alpha \d\theta \notag \\
 &\lesssim \|\ehM\|_{L^2(\Ghsm)} \|\ehM\|_{L^\infty(\Ghsm)} 
 \|\nabla_{\hat{\Gamma}_{h,*}^m} (\hat{X}_{h,*}^{m+1} - \hat{X}_{h,*}^m)\|_{L^2(\Ghsm)} \notag\\
 &\quad\,\, \cdot\|\nabla_{\hat{\Gamma}_{h,*}^m} (\hat{X}_{h,*}^{m+1} - \hat{X}_{h,*}^m)\|_{L^\infty(\Ghsm)} \notag \\
 &\lesssim h^{-3} \tau^2 \|\ehM\|_{L^2(\Ghsm)}^2\lesssim \tau\|\ehM\|_{L^2(\Ghsm)}^2\quad \mbox{($\tau \le c h^k$ with $k\ge 3$ is used)} \label{M25}.
\end{align}

The second term \(M^m_{22}\) can be expressed in the following equivalent form:
 \begin{align}
	 M^m_{22}&=-\int_{\Gm}\big(I_h\big[(\ehM\cdot(\nbhsM+\nbhsm))\ehM\big]\big)^l\cdot\nabla_{\Gm } \big(\hat X_{h,*}^{m+1,l} - \hat X_{h,*}^{m,l}\big)\bar{n}^{m,l}_{h,*}\notag\\
	 &\quad\,-\Bigg[\int_{\hat{\Gamma}_{h,*}^{m}}I_h\big[(\ehM\cdot(\nbhsM+\nbhsm))\ehM\big]\cdot\nabla_{\hat{\Gamma}_{h,*}^{m} } \big(\hat X_{h,*}^{m+1} - \hat X_{h,*}^{m}\big)\bar{n}^{m}_{h,*}\notag\\
	 &\quad\,\quad -\int_{\Gm}\big(I_h\big[(\ehM\cdot(\nbhsM+\nbhsm))\ehM\big]\big)^l\cdot\nabla_{\Gm } \big(\hat X_{h,*}^{m+1,l} - \hat X_{h,*}^{m,l}\big)\bar{n}^{m,l}_{h,*}\Bigg]\notag\\
	 &\quad\,-\Bigg[\int_{0}^{1}\int_{\hat{\Gamma}_{h,*}^{m+\theta}}I_h\big[(\ehM\cdot(\nbhsM+\nbhsm))\ehM\big]\cdot\nabla_{\hat{\Gamma}_{h,*}^{m+\theta} } \big(\hat X_{h,*}^{m+1} - \hat X_{h,*}^{m}\big)\bar{n}^{m+\theta}_{h,*}\d\theta\notag\\
	 &\quad\,\quad-\int_{0}^{1}\int_{\hat{\Gamma}_{h,*}^{m}}I_h\big[(\ehM\cdot(\nbhsM+\nbhsm))\ehM\big]\cdot\nabla_{\hat{\Gamma}_{h,*}^{m} } \big(\hat X_{h,*}^{m+1} - \hat X_{h,*}^{m}\big)\bar{n}^{m}_{h,*}\d\theta\Bigg]\notag\\
	 &=:\sum_{j=1}^{3}M^m_{22j}.
 \end{align}
 By performing integration by parts and employing Lemma \ref{lemma:ud} (item 3), in conjunction with stability result in Lemma \ref{stability}, Lemma \ref{Xhinfty}, and the boundedness of \(\|\bar{n}_{h,*}^{m+\theta}\|_{W^{1,\infty}(\hat{\Gamma}_{h,*}^m)}\) for \(0 \leq \theta \leq 1\) as established in \eqref{bound-n}, the term \(M^m_{221}\) can be estimated as follows:
\begin{align}\label{M221}
 |M^m_{221}|&=\Big| \int_{\Gm}\Big(\nabla_{\Gm}\cdot\big(I_h\big[(\ehM\cdot(\nbhsM+\nbhsm))\ehM\big]\big)^l\Big) \big(\hat X_{h,*}^{m+1,l} - \hat X_{h,*}^{m,l}\big)\cdot\bar{n}^{m,l}_{h,*}\notag\\
 &\quad+\int_{\Gm}\big(I_h\big[(\ehM\cdot(\nbhsM+\nbhsm))\ehM\big]\big)^l\cdot\big(\big(\nabla_{\Gm}\bar{n}^{m,l}_{h,*}\big)\big(\hat X_{h,*}^{m+1,l} - \hat X_{h,*}^{m,l}\big)\big)\notag\\
 &\quad-\int_{\Gm}H^mn^m\cdot\big(I_h\big[(\ehM\cdot(\nbhsM+\nbhsm))\ehM\big]\big)^l\big(\big(\hat X_{h,*}^{m+1,l} - \hat X_{h,*}^{m,l}\big)\cdot\bar{n}^{m,l}_{h,*}\big)\Big|\notag\\
 &\lesssim\|\ehM\|_{L^2{(\Ghsm)}}\|\ehM\|_{H^1{(\Ghsm)}}\|\hat X_{h,*}^{m+1} - \hat X_{h,*}^{m}\|_{L^\infty(\Ghsm)}\notag\\
 &\lesssim \epsilon\tau \|\ehM\|_{H^1(\Ghsm)}^2+\epsilon^{-1}\tau\|\ehM\|_{L^2(\Ghsm)}^2 \quad\text{(using \eqref{Xtau})}, 
\end{align}
where the following stability estimates are used in the second-to-last inequality:
\begin{align}\label{stab-L1semi}
 &\quad\,\,\|\nabla_{\Gm}\big(I_h\big[(\ehM\cdot(\nbhsM+\nbhsm))\ehM\big]\big)^l\|_{L^1(\Gm)}\notag\\ & \lesssim  \|\nabla_{\Ghsm}\big(I_h\big[(\ehM\cdot(\nbhsM+\nbhsm))\ehM\big]\big)\|_{L^1(\Ghsm)}\notag\\
 & \,\,\quad\text{(using norm equivalence between $\Ghsm$ and $\Gm$)}\notag\\
 & \lesssim \|(\ehM\cdot(\nbhsM+\nbhsm))\ehM\|_{W^{1,1}(\Ghsm)} \quad\text{(using \eqref{stability-final})}\notag\\
 &\lesssim \|\nbhsM+\nbhsm\|_{W^{1,\infty}(\Ghsm)}\|\ehM\|_{L^2(\Ghsm)}\|\ehM\|_{H^1(\Ghsm)} \notag\\
 & \lesssim  \|\ehM\|_{L^2(\Ghsm)}\|\ehM\|_{H^1(\Ghsm)} \quad\text{(using \eqref{bound-n})},
\end{align}
and 
\begin{align}\label{stab-L1}
 &\quad\,\,\|I_h\big[(\ehM\cdot(\nbhsM+\nbhsm))\ehM\big]^l\|_{L^1(\Gm)}\notag\\ & \lesssim  \|I_h\big[(\ehM\cdot(\nbhsM+\nbhsm))\ehM\big]\|_{L^1(\Ghsm)}\notag\\
 & \,\,\quad\text{(using norm equivalence between $\Ghsm$ and $\Gm$)}\notag\\
 & \lesssim \|\ehM\cdot(\nbhsM+\nbhsm)\ehM\|_{L^1(\Ghsm)}\quad\text{(using \eqref{stability-final})}\notag\\
%  & \,\,\quad\text{(using Lagrange interpolation error estimates and inverse inequalities)}\notag\\
 & \lesssim \|\ehM\|_{L^2(\Ghsm)}^2\quad\text{(using \eqref{bound-n})}.
\end{align}
By using the geometric perturbation estimate \eqref{perturbation3} in Lemma \ref{geometric-perturbation}, together with the stability estimate \eqref{stability-final} and the bound in \eqref{bound-n}, $M^m_{222}$ can be estimated as follows: 
 \begin{align}
	 |M^m_{222}|&\lesssim (1+\kappa_{*,l})h^{k}\|\ehM\|_{L^2{(\Ghsm)}}\|\ehM\|_{L^\infty{(\Ghsm)}}\|\nabla_\Ghsm\big(\hat X_{h,*}^{m+1} - \hat X_{h,*}^{m}\big)\|_{L^\infty(\Ghsm)}\notag\\
	 &\lesssim(1+\kappa_{*,l})h^{k-2}\tau\|\ehM\|_{L^2(\Ghsm)}^2 \quad\text{(using inverse inequalities and \eqref{Xtau})}\notag\\
	 &\lesssim \tau \|\ehM\|_{L^2(\Ghsm)}^2 \quad \mbox{(using the mesh size condition \eqref{cond1})}.
 \end{align}
Using the fundamental theorem of calculus, $M^m_{223}(\phi_{h})$ can be rewritten into the following form:
 \begin{align}\label{M223}
	 M^m_{223}&=-\int_{0}^{1}\int_{0}^{\theta}\frac{\d }{\d\alpha}\int_{\hat{\Gamma}_{h,*}^{m+\alpha}}I_h\big[(\ehM\cdot(\nbhsM+\nbhsm))\ehM\big]\cdot\big(\nabla_{\hat{\Gamma}_{h,*}^{m+\alpha} } \big(\hat X_{h,*}^{m+1} - \hat X_{h,*}^{m}\big)\bar{n}^{m+\alpha}_{h,*}\big)\d\alpha\d\theta\notag\\
	 &=-\int_{0}^{1}\int_{0}^{\theta}\int_{\hat{\Gamma}_{h,*}^{m+\alpha}}I_h\big[(\ehM\cdot(\nbhsM+\nbhsm))\ehM\big]\cdot\big(\partial_\alpha^\bullet\big(\nabla_{\hat{\Gamma}_{h,*}^{m+\alpha} } \big(\hat X_{h,*}^{m+1} - \hat X_{h,*}^{m}\big)\big)\bar{n}^{m+\alpha}_{h,*}\big)\d\alpha\d\theta\notag\\
	 &\quad -\int_{0}^{1}\int_{0}^{\theta}\int_{\hat{\Gamma}_{h,*}^{m+\alpha}}I_h\big[(\ehM\cdot(\nbhsM+\nbhsm))\ehM\big]\cdot\big(\nabla_{\hat{\Gamma}_{h,*}^{m+\alpha} } \big(\hat X_{h,*}^{m+1} - \hat X_{h,*}^{m}\big)\partial_\alpha^\bullet\bar{n}^{m+\alpha}_{h,*}\big)\d\alpha\d\theta\notag\\
	 &\quad -\int_{0}^{1}\int_{0}^{\theta}\int_{\hat{\Gamma}_{h,*}^{m+\alpha}}\Big\{I_h\big[(\ehM\cdot(\nbhsM+\nbhsm))\ehM\big]\notag\\
	 &\quad \quad\quad\quad\quad\quad\quad\quad\,\,\cdot\big(\nabla_{\hat{\Gamma}_{h,*}^{m+\alpha} } \big(\hat X_{h,*}^{m+1} - \hat X_{h,*}^{m}\big)\bar{n}^{m+\alpha}_{h,*}\big)\big(\nabla_{\hat{\Gamma}_{h,*}^{m+\alpha}}\cdot\theta\big(\hat X_{h,*}^{m+1} - \hat X_{h,*}^{m}\big)\big)\d\alpha\d\theta\Big\}\notag\\
	 &\quad\,\,\mbox{(Lemma \ref{lemma:ud}, item 6 is used)}.
 \end{align}
 Based on the expression derived in \eqref{M223}, and by applying the stability estimate \eqref{stab-L1}, the inequalities \eqref{nhmW-1} and \eqref{bound-n}, together with Lemma \ref{Xhinfty}, the term \(M^m_{223}(\phi_h)\) can be estimated as follows:
\begin{align}
	|M^m_{223}(\phi_h)| 
	&\lesssim \|\ehM\|_{L^2(\Ghsm)}^2 
	\|\nabla_\Ghsm (\hat{X}_{h,*}^{m+1} - \hat{X}_{h,*}^m)\|_{L^\infty(\Ghsm)}^2\notag \\
	&\lesssim \tau^2 h^{-2} \|\ehM\|_{L^2(\Ghsm)}^2 \quad \text{(by inverse inequalities and \eqref{Xtau})}\notag \\
	&\lesssim \tau \|\ehM\|_{L^2(\Ghsm)}^2 \quad \text{(since $\tau \leq c h^k$ with $k \geq 3$)},
\end{align}
where we have bounded 
\(
\|\partial_\alpha^\bullet \bar{n}^{m+\alpha}_{h,*}\|_{L^\infty(\Ghsm)}
\)
using \eqref{def:dnMresult-infty}, and employed the following estimate from Lemma \ref{lemma:ud} (item~5):
\begin{equation*}
	\|\partial_\alpha^\bullet \big(\nabla_{\hat{\Gamma}_{h,*}^{m+\alpha}} (\hat{X}_{h,*}^{m+1} - \hat{X}_{h,*}^m)\big)\|_{L^\infty(\Ghsm)} 
	\lesssim \|\nabla_\Ghsm (\hat{X}_{h,*}^{m+1} - \hat{X}_{h,*}^m)\|_{L^\infty(\Ghsm)}^2.
\end{equation*}
 By collecting the estimates for \(M_{22i}^m\), where \(i = 1, 2, 3\), we derive the following bound for \(M_{22}^m\):
 \begin{align*}
	 |M_{22}^m| &\lesssim \epsilon \tau \|\ehM\|_{H^1(\Ghsm)}^2 + \epsilon^{-1} \tau \|\ehM\|_{L^2(\Ghsm)}^2.
 \end{align*}
 Similarly, by aggregating the estimates for \(M_{2i}^m\), where \(i = 1, 2,3, 4,5\), the term \(M_2^m\) is estimated as follows:
 \begin{align}
	 |M_2^m| \lesssim \epsilon \tau \|\ehM\|_{H^1(\Ghsm)}^2 + \epsilon^{-1} \tau \|\ehM\|_{L^2(\Ghsm)}^2.
 \end{align}

 We decompose $M_3^m$ defined in the expression \eqref{M} into several parts as follows:
 \begin{align}
	 M_3^m=&\| \ehM\cdot \nahsm \|_{L^2(\hat\Gamma_{h,*}^{m})}^2 - \| \eM\cdot \nahsm \|_{L^2(\hat\Gamma_{h,*}^{m})}^2 \notag\\
	 =&\int_{\Ghsm}((\ehM-\eM)\cdot\nahsm)((\ehM+\eM)\cdot\nahsm)\notag\\
	 =&\int_{\Ghsm}([-I_h(\TsM\eM)+f_h^{m+1}]\cdot\nahsm)((\ehM+\eM)\cdot\nahsm) \notag\\
	 &\quad \mbox{(using relation \eqref{eq:geo_rel_1} at time \(t_{m+1}\), with $T_*^{m+1}$ standing for $T_*^{m+1}\circ X_{h,*}^{m+1}$)}\notag\\
	 =&-\int_{\Ghsm}\big(I_h[(\TsM-\Tsm)\eM]\cdot\nahsm\big)\big((\ehM+\eM)\cdot\nahsm\big)\notag\\
%		&-\int_{\Ghsm}I_h\Tsm\eM\cdot(\nhsm-\nsm)(\ehM+\eM)\cdot\nhsm\notag\\
	 &-\int_{\Ghsm}\big(I_h(\Tsm\eM)\cdot\nahsm\big)\big((\ehM+\eM)\cdot\nahsm\big)\notag\\
	 &+\int_{\Ghsm}\big(f_h^{m+1}\cdot\nahsm\big)\big((\ehM+\eM)\cdot\nahsm\big)\notag\\
	 =&-\int_{\Ghsm}\big(I_h[(\TsM-\Tsm)\eM]\cdot\nahsm\big)\big((\ehM+\eM)\cdot\nahsm\big)\notag\\
%		&-\int_{\Ghsm}I_h\Tsm(\eM-\ehm)\cdot(\nhsm-\nsm)(\ehM+\eM)\cdot\nhsm\notag\\
	 &-\int_{\Ghsm}\big(I_h[\Tsm(\eM-\ehm-\tau I_h\Tsm v^m)]\cdot\nahsm\big)\big((\ehM+\eM)\cdot\nahsm\big)\notag\\
	 &\quad\,\,\mbox{(using the orthogonality between $\ehm$ and $\Tsm$ at all the nodes of $\Ghsm$)}\notag\\
	 &-\tau\int_{\Ghsm}\big(I_h(\Tsm v^m)\cdot\nahsm\big)\big((\ehM+\eM)\cdot\nahsm\big)\notag\\
	 &+\int_{\Ghsm}\big(f_h^{m+1}\cdot\nahsm\big)\big((\ehM+\eM)\cdot\nahsm\big)=:\sum_{i=1}^{4}M_{3i}^m.
 \end{align}
At each node of $\Ghsm$, the difference $ n_*^{m+1}(X_{h,*}^{m+1}) - n_*^m$ represents the change in $ n $ along a particle trajectory of the exact flow map and is therefore $ O(\tau) $ at that node. It follows that $|T_*^{m+1} - T_*^m |\lesssim \tau$ at the nodes and therefore, by utilizing the stability estimate \eqref{stability2}, 
\begin{align}
	M_{31}^m &\lesssim \tau \|\eM\|_{L^2(\Ghsm)} ( \|\ehM\|_{L^2(\Ghsm)} + \|\eM\|_{L^2(\Ghsm)} ) \notag\\
	&\lesssim \tau ( \|\ehM\|_{L^2(\Ghsm)}^2 + \|\eM\|_{L^2(\Ghsm)}^2 ).
\end{align}

By using the definition of $\hat e_v^m$ in \eqref{evm}, and applying the estimates \eqref{eq:vel_H1} and \eqref{nMinfty}, the term \(M_{32}^m\) can be estimated as follows:
 \begin{align}\label{M32}
	 M_{32}^m=&-\int_{\Ghsm}\big(I_h\Tsm(\eM-\ehm-\tau I_h\Tsm v^m)\cdot\nahsm\big)\big((\ehM+\eM)\cdot\nahsm\big)\notag\\
	 =&-\tau\int_{\Ghsm}\big(I_h\Tsm\ehvm\cdot\nahsm\big)\big((\ehM+\eM)\cdot\nahsm\big)\notag\\
	 =&-\tau\int_{\Ghsm}\big(((I_h-1)\Tsm\ehvm)\cdot\nsm\big)\big((\ehM+\eM)\cdot\nahsm \big)\notag\\
	 &\,\,\mbox{(using the orthogonality between $\nsm$ and $\Tsm$)}\notag\\
	 &-\tau\int_{\Ghsm}I_h\Tsm\ehvm\cdot(\nahsm-\nsm)(\ehM+\eM)\cdot\nahsm\notag\\
	 \lesssim&\tau h^2\|\ehvm\|_{H^1(\Ghsm)}(\|\ehM\|_{L^2(\Ghsm)}+\|\eM\|_{L^2(\Ghsm)})\quad \mbox{(using \eqref{super1})}\notag\\
	 &+\tau \|I_h\Tsm\ehvm\|_{L^\infty(\Ghsm)}\|\nahsm-\nsm\|_{L^2(\Ghsm)}(\|\ehM\|_{L^2(\Ghsm)}+\|\eM\|_{L^2(\Ghsm)})\notag\\
	 \lesssim&\tau h^2((1+\kappa_{*,l})h^{k-2}+h^{-2}\| \nabla_{\Ghsm} \ehm \|_{L^2(\Ghsm)})(\|\ehM\|_{L^2(\Ghsm)}+\|\eM\|_{L^2(\Ghsm)})\notag\\
	 &+\tau C_{\epsilon_0}h^{0.1-\epsilon_0}(1+\kappa_{*,l})h^k(\|\ehM\|_{L^2(\Ghsm)}+\|\eM\|_{L^2(\Ghsm)})\notag\\
	 &\,\, \mbox{(using \eqref{eq:nsa2}, \eqref{eq:vel_H1}, \eqref{nMinfty} and Sobolev embedding \eqref{eq:poincare3})}\notag\\
	 \lesssim&\epsilon^{-1}\tau(1+\kappa_{*,l})^2h^{2k}+\epsilon\tau\| \nabla_{\Ghsm} \ehm \|_{L^2(\Ghsm)}^2+\epsilon^{-1}\tau(\|\ehM\|^2_{L^2(\Ghsm)}+\|\eM\|^2_{L^2(\Ghsm)})\notag\\
	 & \,\, \mbox{(choosing $\epsilon_0=0.1$ and applying Young's inequality)}.
 \end{align}

 The term \(M_{33}^m\) can be estimated by utilizing the approximation property of the Lagrange interpolation, as shown in \eqref{Ihf}, and can be expressed as follows:
 \begin{align}\label{M33}
	 M_{33}^m=&-\tau\int_{\Ghsm}\big(I_h\Tsm v^m\cdot\nahsm\big)\big((\ehM+\eM)\cdot\nahsm\big)\notag\\
	 =&-\tau\int_{\Ghsm}\big(((I_h-1)\Tsm v^m)\cdot\nsm\big)\big((\ehM+\eM)\cdot\nahsm\big) \notag\\
	 &\,\, \mbox{(using the orthogonality between $\nsm$ and $\Tsm$ )}\notag\\
	 &-\tau\int_{\Ghsm}\big(I_h\Tsm v^m\cdot(\nahsm-\nsm)\big)\big((\ehM+\eM)\cdot\nahsm\big)\notag\\
	 \lesssim &\tau\|((I_h-1)\Tsm v^m)\cdot\nsm\|_{L^2(\Ghsm)}(\|\ehM\|_{L^2(\Ghsm)}+\|\eM\|_{L^2(\Ghsm)})
     \quad \mbox{(using \eqref{nhmW-1})}\notag\\
	 &+\tau\|I_h\Tsm v^m\|_{L^\infty(\Ghsm)}\|\nahsm-\nsm\|_{L^2(\Ghsm)}(\|\ehM\|_{L^2(\Ghsm)}+\|\eM\|_{L^2(\Ghsm)})\notag\\
	 \lesssim& \tau(1+\kappa_{*,l})h^k(\|\ehM\|_{L^2(\Ghsm)}+\|\eM\|_{L^2(\Ghsm)})\quad\mbox{(using \eqref{Ihf} and \eqref{eq:nsa2})}\notag\\
	 \lesssim&\epsilon\tau(1+\kappa_{*,l})^2h^{2k}+\epsilon^{-1}\tau(\|\ehM\|^2_{L^2(\Ghsm)}+\|\eM\|^2_{L^2(\Ghsm)}).
 \end{align}
 Finally, by applying the estimate provided in \eqref{fhinfty} and leveraging the \(L^\infty\)-stability property of the Lagrange interpolation operator, the term \(M_{34}^m\) can be estimated as follows:
 \begin{align}\label{M34}
	 M_{34}^m=&\, \int_{\Ghsm}\big(f_h\cdot\nahsm\big)\big((\ehM+\eM)\cdot\nahsm\big)\notag\\
	 \lesssim &\, \|f_h\|_{L^\infty(\Ghsm)}(\|\ehM\|_{L^2(\Ghsm)}+\|\eM\|_{L^2(\Ghsm)})\notag\\
	 \lesssim &\, \|I_h([I-\nsM(\nsM)^\top]\eM)\|_{L^\infty(\Ghsm)}^2(\|\ehM\|_{L^2(\Ghsm)}+\|\eM\|_{L^2(\Ghsm)})\notag\\
	 &\,\,\mbox{(using relation \eqref{eq:geo_rel_2} at time $t_{m+1}$)}\notag\\
	 \lesssim &\, (\tau+\tau h^{-2.1}\| \nabla_{\Ghsm} \ehm \|_{L^2(\Ghsm)})^2(\|\ehM\|_{L^2(\Ghsm)}+\|\eM\|_{L^2(\Ghsm)})\notag\\
	 &\,\,\mbox{(using inequality \eqref{fhinfty})}\notag\\
	 \lesssim &\,(\tau^2+\tau^2 h^{-4.2}\| \nabla_{\Ghsm} \ehm \|_{L^2(\Ghsm)}^2)(\|\ehM\|_{L^2(\Ghsm)}+\|\eM\|_{L^2(\Ghsm)})\notag\\
	 \lesssim &\, (\tau^2+\tau\| \nabla_{\Ghsm} \ehm \|_{L^2(\Ghsm)})(\|\ehM\|_{L^2(\Ghsm)}+\|\eM\|_{L^2(\Ghsm)})\notag\\
	 &\,\,\mbox{(using \eqref{cond1}, \eqref{Linfty-W1infty-hat-em} and \(\tau \leq c h^k\) with \(k \geq 3\))}\notag\\
	 \lesssim &\, \epsilon\tau^3 +\epsilon\tau\| \nabla_{\Ghsm} \ehm \|_{L^2(\Ghsm)}^2 +\epsilon^{-1}\tau(\|\ehM\|^2_{L^2(\Ghsm)}+\|\eM\|^2_{L^2(\Ghsm)}) \notag\\
     &\,\,\mbox{(using Young's inequality)}.
 \end{align}

 By collecting the estimates of $M_{3j}^m$, $j=1,\dots,4$, we obtain the following estimate:
 \begin{align}
	 M_3^m
     &\lesssim\epsilon^{-1}\tau^3 +\epsilon^{-1}\tau(\|\ehM\|^2_{L^2(\Ghsm)}+\|\eM\|^2_{L^2(\Ghsm)}) +\epsilon\tau\| \nabla_{\Ghsm} \ehm \|_{L^2(\Ghsm)}^2.
 \end{align}
 By collecting the estimates of $M_{j}^m$, $j=1,2,3$ and using inequalities  
 % \eqref{fhinfty}
 \eqref{eq:e_NT}, \eqref{em+1-2} and \eqref{hathathat}, the following inequality holds:
 \begin{align}
	&\quad\,\, M_1^m + M_2^m  +M_3^m \notag\\
        &\lesssim \epsilon^{-1}\tau^3 +\epsilon^{-1}\tau(\|\ehM\|^2_{L^2(\Ghsm)}+\|\eM\|^2_{L^2(\Ghsm)})
        +\epsilon\tau\| \nabla_{\Ghsm} \ehm \|_{L^2(\Ghsm)}^2 + \epsilon \tau \|\hat e_h^{m+1}\|_{H^1(\Ghsm)}^2\notag\\
	& \lesssim \epsilon^{-1}\tau^3 + \epsilon^{-1} \tau \|\hat e_h^m \cdot \bar n_{h,*}^m\|_{L^2(\Ghsm)}^2 + \epsilon \tau \|\nabla_{\Ghsm}\hat e_h^{m}\|_{L^2(\Ghsm)}^2.
 \end{align}
 
This completes the proof of \eqref{stab}.
\end{proof}

\subsection{Uniform boundedness of $\kappa_l$, $\kappa_{*,l}$, $C_\#$ and $C_*$} \label{appendix_H}
 \renewcommand{\theequation}{H.\arabic{equation}}

In this section, we denote by $C_0$ a constant which is independent of $\kappa_l$, $\kappa_{l,*}$, $C_\#$ and $C_*$, as mentioned in Section \ref{induc ass}. For the simplicity of notation, we use the same notations $\Nsm$, $\Tsm$ amd $v_*^m$ to denote their pull-back functions onto $\Gamma_{h,{\rm f}}^0$, i.e.,  
\begin{align}\label{NTv}
\Nsm=(n^m (n^m)^\top\circ a^m\circ \hat{X}^m_{h,*},\quad
\Tsm=(I-(n^m (n^m)^\top)\circ a^m\circ \hat{X}^m_{h,*},\quad
v_*^m: = v^m \circ a^m\circ \hat{X}^m_{h,*} ,
\end{align}
which are smooth functions composed with $\hat{X}^m_{h,*}$. Moreover, we use $I_hT_*^mv_*^m$ instead of $I_h(T_*^mv_*^m)$ to denote the interpolation of $T_*^mv_*^m$ onto $\Gamma_{h,{\rm f}}^0$. 

As introduced in Section \ref{3.2}, we denote by \(\hat{X}_{h,*}^m: \Gamma_{h,{\rm f}}^0 \rightarrow \Ghsm\) the unique piecewise polynomial of degree \(k\), which parametrizes \(\Ghsm\). Then the following decomposition holds:
$$
\hat{X}^{q+1}_{h,*} = \hat{X}^0_{h,*} + \sum_{m=0}^{q} (\hat{X}^{m+1}_{h,*} - \hat{X}^m_{h,*} - \tau I_h\Tsm v_*^m) + \sum_{m=0}^{q} \tau I_h\Tsm v_*^m .
$$

For $1\le j\le k-1$, using the $W_h^{j,\infty}$ stability of the Lagrange interpolation operator on $\Gamma^0_{h,{\rm f}}$, and the notations in \eqref{NTv}, we have 
\begin{align}
&\|I_h\Tsm v_*^m\|_{W^{j,\infty}_h(\Gamma^0_{h,{\rm f}})} 
\le C_0+C_0\|\Tsm v_*^m\|_{W^{j,\infty}_h(\Gamma^0_{h,{\rm f}})} \notag\\
&\le C_0+C_0\sum_{\substack{j_1+\cdots+j_r\le j\\ 1\le j_1,\cdots,j_r \le j}}\|\hat{X}_{h,*}^m \|_{W^{j_1,\infty}_h(\Gamma^0_{h,{\rm f}})} \cdots \|\hat{X}_{h,*}^m \|_{W^{j_r,\infty}_h(\Gamma^0_{h,{\rm f}})} \notag\\
&\le C_0 + C_0\|\hat{X}_{h,*}^m \|_{W^{j,\infty}_h(\Gamma^0_{h,{\rm f}})} 
+ C_0\sum_{\substack{j_1+\cdots+j_r\le j\\ 1\le j_1,\cdots,j_r\le j-1}}\|\hat{X}_{h,*}^m \|_{W^{j_1,\infty}_h(\Gamma^0_{h,{\rm f}})} \cdots \|\hat{X}_{h,*}^m \|_{W^{j_r,\infty}_h(\Gamma^0_{h,{\rm f}})} \notag\\
&\le C_0\|\hat{X}_{h,*}^m \|_{W^{j,\infty}_h(\Gamma^0_{h,{\rm f}})} 
+C_0[1+(j-1)\|\hat{X}_{h,*}^m \|_{W^{j-1,\infty}_h(\Gamma^0_{h,{\rm f}})}^j] .
\end{align}
where the factor $j-1$ in the last term indicates that the corresponding term should vanish when $j=1$. Similarly, for $2\le j\le k$, using the $H_h^{j}$ stability of the Lagrange interpolation operator on $\Gamma^0_{h,{\rm f}}$, we have 
\begin{align}
&\|I_h\Tsm v_*^m\|_{H^{j}_h(\Gamma^0_{h,{\rm f}})} \notag\\
&\le C_0\|\Tsm v_*^m\|_{H^{j}_h(\Gamma^0_{h,{\rm f}})} \notag\\
&\le C_0\sum_{\substack{j_1+\cdots+j_r\le j\\ 1\le j_2,\cdots,j_r \le j_1}}\|\hat{X}_{h,*}^m \|_{H^{j_1}_h(\Gamma^0_{h,{\rm f}})} \|\hat{X}_{h,*}^m \|_{W^{j_2,\infty}_h(\Gamma^0_{h,{\rm f}})}\cdots \|\hat{X}_{h,*}^m \|_{W^{j_r,\infty}_h(\Gamma^0_{h,{\rm f}})} \notag\\
&\le C_0\|\hat{X}_{h,*}^m \|_{H^{j}_h(\Gamma^0_{h,{\rm f}})} 
+ C_0\big[ 1+\|\hat{X}_{h,*}^m \|_{H^{j-1}_h(\Gamma^0_{h,{\rm f}})}\|\hat{X}_{h,*}^m \|_{W^{1,\infty}_h(\Gamma^0_{h,{\rm f}})}
+ (j-2)\|\hat{X}_{h,*}^m\|_{W^{j-2,\infty}_h(\Gamma^0_{h,{\rm f}})}^j\big] ,
\end{align}
where the factor $j-2$ indicates that the corresponding term should vanish when $j=2$. 
% By using the approximation properties of the Lagrange interpolation on the initial triangulated surface $\Gamma_{h,\rm f}^0$, we have 
% \begin{align*}
% \|I_h[(\Tsm v_*^m) \circ \hat{X}_{h,*}^m] \|_{H^1(K_{\rm f}^0)} 
% &\le C_0h\|I_h[(\Tsm v_*^m) \circ \hat{X}_{h,*}^m] \|_{W^{1,\infty}(K_{\rm f}^0)} \\
% &\le C_0h\|(\Tsm v_*^m) \circ \hat{X}_{h,*}^m \|_{W^{1,\infty}(K_{\rm f}^0)} \\
% &\le C_0h(\|\Tsm v_*^m\|_{L^\infty(D_\delta(\Gamma^m))}+\|\nabla_{K_{\rm f}^0}\hat{X}_{h,*}^m \|_{L^{\infty}(K_{\rm f}^0)}) \\
% &\le C_0(h+\|\hat{X}_{h,*}^m \|_{H^{1}(K_{\rm f}^0)}) 
% \end{align*}
% and therefore, by summing this inequality over all triangles $K_{\rm f}^0\subset\Gamma^0_{h,{\rm f}}$, we obtain 
% \begin{align} 
% \|I_h[(\Tsm v_*^m) \circ \hat{X}_{h,*}^m] \|_{H^1(\Gamma^0_{h,{\rm f}})}  
% &\le C_0(1+\|\hat{X}_{h,*}^m \|_{H^{1}(\Gamma^0_{h,{\rm f}})}) . 
% \end{align}

Then, using the triangle inequality and the shape-regularity of the initial triangulation at \(t=0\), as established in \eqref{P0}, the following estimates hold:
\begin{align}
 \label{qWj}
 \|\hat{X}^{q+1}_{h,*}\|_{W^{j,\infty}_h(\Gamma^0_{h,{\rm f}})} 
 &\leq C_0 + \sum_{m=0}^{q} \|\hat{X}^{m+1}_{h,*} - \hat{X}^m_{h,*} - \tau I_h\Tsm v_*^m\|_{W^{j,\infty}_h(\Gamma^0_{h,{\rm f}})}  + C_0 \tau\sum_{m=0}^{q} \|\hat{X}_{h,*}^m \|_{W^{j,\infty}_h(\Gamma^0_{h,{\rm f}})} 
 \notag\\
 &\quad + C_0 \tau \sum_{m=0}^{q} [1+(j-1)\|\hat{X}_{h,*}^m \|_{W^{j-1,\infty}_h(\Gamma^0_{h,{\rm f}})}^j]  
 \quad\,\, \text{for}\,\,\, 1 \leq j \leq k-1, \\
 \label{qHj}
 \|\hat{X}^{q+1}_{h,*}\|_{H^j_h(\Gamma^0_{h,{\rm f}})} 
 &\leq C_0 + \sum_{m=0}^{q} \|\hat{X}^{m+1}_{h,*} - \hat{X}^m_{h,*} - \tau I_h\Tsm v_*^m\|_{H^j_h(\Gamma^0_{h,{\rm f}})}  + C_0 \tau \sum_{m=0}^{q} \|\hat{X}_{h,*}^m \|_{H^{j}_h(\Gamma^0_{h,{\rm f}})} \notag\\
 &\quad
+ C_0 \tau \big[ 1 + \|\hat{X}_{h,*}^m \|_{H^{j-1}_h(\Gamma^0_{h,{\rm f}})} \|\hat{X}_{h,*}^m \|_{W^{1,\infty}_h(\Gamma^0_{h,{\rm f}})} +(j-2)\|\hat{X}_{h,*}^m \|_{W^{j-2,\infty}_h(\Gamma^0_{h,{\rm f}})}^j \big] \notag\\
&\quad\,\, \text{for}\,\,\, 2 \leq j \leq k .
\end{align} 

% and utilizing the relation \eqref{xm} to replace the corresponding expression in \cite[A.24]{bai2024new} as follows:
% \begin{align}\label{TXX}
%     &\quad \,\|I_h \Tsm (X_h^{m+1} - X_h^m - \tau \Tsm v^m)\|_{H^1(\Gamma^0_{h,{\rm f}})} \notag\\
%     &= \|I_h \Tsm (\eM - \ehm - \tau I_h \Tsm v^m + \tau I_h g^m)\|_{H^1(\Gamma^0_{h,{\rm f}})} \notag\\
%     &\leq \|I_h \Tsm (\eM - \ehm - \tau I_h \Tsm v^m)\|_{H^1(\Ghsm)} 
%     + C_0 \tau^2, \quad \text{(using \eqref{W1infty-g})}.
% \end{align}
% Analogous to the derivation in \cite[A.18]{bai2024new}, 
Therefore, for $1 \leq j \leq k-1$, applying the triangle inequality and decomposing 
\(
\hat{X}^{m+1}_{h,*} - \hat{X}^m_{h,*} - \tau I_h\Tsm v_*^m
\)
into its normal and tangential components, we obtain the following estimate:
 \begin{align}
	 &\quad\,\,\,\|\hat{X}^{m+1}_{h,*}-\hat{X}^m_{h,*}-\tau I_h\Tsm v_*^m\|_{W^{j,\infty}_h(\Gamma^0_{h,{\rm f}})}\notag\\
	 &\le \|I_h\Nsm(\hat{X}^{m+1}_{h,*}- \hat{X}^m_{h,*})\|_{W^{j,\infty}_h(\Gamma^0_{h,{\rm f}})}+\|I_h\Tsm(\hat{X}^{m+1}_{h,*}-\hat{X}^m_{h,*}-\tau \Tsm v_*^m) \|_{W^{j,\infty}_h(\Gamma^0_{h,{\rm f}})}\notag\\
	 & \quad \text{(using the orthogonality between $\Nsm$ and $\Tsm$)}\notag\\
	 &\le \|I_h[(X^{m+1}-{\rm id})\circ a^m\circ\hat{X}^m_{h,*}+\rho_h\circ\hat{X}^m_{h,*}]\|_{W^{j,\infty}_h(\Gamma^0_{h,{\rm f}})}\quad\mbox{(relation \eqref{eq:geo_rel_4} is used)}\notag\\
	 &\quad+ C_0h^{-j+1}\|I_h\Tsm(\hat{X}^{m+1}_{h,*}-\hat{X}^m_{h,*}-\tau \Tsm v_*^m)\|_{W^{1,\infty}(\Gamma^0_{h,{\rm f}})}\quad\text{(inverse inequality is used)}\notag\\
	 &\le C_0\|X^{m+1}-{\rm id}\|_{W^{j,\infty}(\Gm)}\Big(1+\sum_{\underset{j_1,\cdots,j_r\ge1}{j_1+\cdots+j_r\le j}}\|\hat{X}^m_{h,*}\|_{W^{j_1,\infty}_h(\Gamma^0_{h,{\rm f}})}\cdots\|\hat{X}^m_{h,*}\|_{W^{j_r,\infty}_h(\Gamma^0_{h,{\rm f}})}\Big)\notag\\
	 &\quad\text{(using the product rule of differentiation and the $W^{j,\infty}$-stability of $I_h$)}\notag\\
	 &\quad+C_0h^{-j}\|I_h(\rho_h\circ \hat X_{h,*}^m)\|_{L^\infty(\Gamma^0_{h,{\rm f}})}+ C_0h^{-j}\|I_h\Tsm(\hat{X}^{m+1}_{h,*}-\hat{X}^m_{h,*}-\tau \Tsm v_*^m)\|_{H^1(\Gamma^0_{h,{\rm f}})}\notag\\
	 &\le C_0\tau[1+(j-1)\|\hat{X}^m_{h,*}\|^j_{W^{j-1,\infty}_h(\Gamma^0_{h,{\rm f}})}]+C_0\tau\|\hat{X}^m_{h,*}\|_{W^{j,\infty}_h(\Gamma^0_{h,{\rm f}})}\notag\\
	 &\quad+C_0h^{-j}(\tau^2+\|I_h\Tsm(\hat{X}^{m+1}_{h,*}-\hat{X}^m_{h,*})\|_{L^\infty(\Gamma^0_{h,{\rm f}})}^2)\quad\mbox{(relation \eqref{eq:geo_rel_5} is used)}\notag\\
	 &\quad+ C_0h^{-j}\|I_h\Tsm(\hat{X}^{m+1}_{h,*}-\hat{X}^m_{h,*}-\tau \Tsm v_*^m)\|_{H^1(\Gamma^0_{h,{\rm f}})}\notag\\
	 &\le C_0\tau[1+(j-1)\|\hat{X}^m_{h,*}\|^j_{W^{j-1,\infty}_h(\Gamma^0_{h,{\rm f}})}]+C_0\tau\|\hat{X}^m_{h,*}\|_{W^{j,\infty}_h(\Gamma^0_{h,{\rm f}})}+C_0 h^{-j}\tau^2\notag\\
	 &\quad+C_0 h^{-j}\|I_h\Tsm(\hat{X}^{m+1}_{h,*}-\hat{X}^m_{h,*}-\tau \Tsm v_*^m)\|_{L^\infty(\Gamma^0_{h,{\rm f}})}^2 \notag\\
	 & 	\quad + C_0h^{-j}\|I_h\Tsm(\hat{X}^{m+1}_{h,*}-\hat{X}^m_{h,*}-\tau \Tsm v_*^m)\|_{H^1(\Gamma^0_{h,{\rm f}})}.\label{XXTW}
 \end{align}
%  where \eqref{Xtau} is used in the last inequality.
 To account for the special case \(j=1\), a factor of \((j-1)\) is introduced in front of the term \(\|\hat{X}^m_{h,*}\|^j_{W^{j-1,\infty}_h(\Gamma^0_{h,{\rm f}})}\), ensuring that the this term vanishes when \(j=1\).
 %In the case $j=0$, we have
 %\begin{align}
 %	&\|\hat{X}^{m+1}_{h,*}-\hat{X}^m_{h,*}-\tau \Tsm v_*^m\|_{L^\infty(\Gamma^0_{h,{\rm f}})}\notag\\
 %	&\le C_0\tau+C_0\|I_h\Tsm(\hat{X}^{m+1}_{h,*}-\hat{X}^m_{h,*}-\tau \Tsm v_*^m)\|_{L^\infty(\Gamma^0_{h,{\rm f}})}^2+\|I_h\Tsm(\hat{X}^{m+1}_{h,*}-\hat{X}^m_{h,*}-\tau \Tsm v_*^m)\|_{L^\infty(\Gamma^0_{h,{\rm f}})}\notag\\
 %	&\le C_0\tau+C_0\|I_h\Tsm(\hat{X}^{m+1}_{h,*}-\hat{X}^m_{h,*}-\tau \Tsm v_*^m)\|_{L^\infty(\Gamma^0_{h,{\rm f}})},
 %\end{align}
 %where the last inequality follows from the estimates in \eqref{hxM}, which implies
 %\begin{align}
 %	&\|\hat{X}^{m+1}_{h,*}-\hat{X}^m_{h,*}-\tau \Tsm v_*^m\|_{L^\infty(\Gamma^0_{h,{\rm f}})}+\|I_h\Tsm(\hat{X}^{m+1}_{h,*}-\hat{X}^m_{h,*}-\tau \Tsm v_*^m)\|_{L^\infty(\Gamma^0_{h,{\rm f}})}\notag\\
 %	&\le C_{\kappa_l}h^{0.6}+C_0\tau\le 1,
 %\end{align}
 %when $h\le h_{\kappa_l}$ and $\tau\le ch_{\kappa_l}^k$ (for some constant $h_{\kappa_l}$ which may depend on $\kappa_l$).
 
 We shall estimate the $L^\infty$ and $H^1$ norm of $I_h\Tsm(\hat{X}^{m+1}_{h,*}-\hat{X}^m_{h,*}-\tau \Tsm v_*^m)$, which appears on the right-hand side of \eqref{XXTW}, by utilizing relation \eqref{eq:geo_rel_6}, which gives us the following decomposition: 
 \begin{align*}
 & I_h\Tsm(\hat{X}^{m+1}_{h,*}-\hat{X}^m_{h,*}-\tau \Tsm v_*^m) \\
 =\, & I_h\Tsm(X_h^{m+1}-X_h^m-\tau \Tsm v_*^m)+I_h(\Tsm(\NsM-\Nsm)\ehM) ,
 \end{align*}
and therefore (using the triangle inequality) 
 \begin{align}
	 &\,\,\quad\|I_h\Tsm(\hat{X}^{m+1}_{h,*}-\hat{X}^m_{h,*}-\tau \Tsm v_*^m)\|_{H^1(\Gamma^0_{h,{\rm f}})}\notag\\
	 &\le \|I_h\Tsm(X_h^{m+1}-X_h^m-\tau \Tsm v_*^m)\|_{H^1(\Gamma^0_{h,{\rm f}})}+\|I_h(\Tsm(\NsM-\Nsm)\ehM)\|_{H^1(\Gamma^0_{h,{\rm f}})}.\label{TNN}
 \end{align}
 The first term on the right-hand side of \eqref{TNN} can be estimated as follows:
 \begin{align}\label{TXX}
	 &\quad\, \,\|I_h \Tsm (X_h^{m+1} - X_h^m - \tau \Tsm v_*^m)\|_{H^1(\Gamma^0_{h,{\rm f}})} \notag\\
	 & =\|I_h \Tsm (X_h^{m+1} - X_h^m - \tau I_h v_*^m)\|_{H^1(\Gamma^0_{h,{\rm f}})} \quad\text{(using the relation $I_h \Tsm \Tsm v_*^m =I_h \Tsm I_h v_*^m$)} \notag\\
	 &= \|I_h \Tsm (\eM - \ehm - \tau I_h \Tsm v_*^m + \tau I_h (g^m \circ a^m \circ \hat X_{h,*}^m))\|_{H^1(\Gamma^0_{h,{\rm f}})}\quad\text{(using relation \eqref{xm11})} \notag\\
	&\leq \|I_h \Tsm (\eM - \ehm - \tau I_h \Tsm v_*^m)\|_{H^1(\Gamma^0_{h,{\rm f}})} 
	 + \tau  \|I_h [T_*^m (g^m \circ a^m)\circ \hat X_{h,*}^m]\|_{H^1(\Gamma^0_{h,{\rm f}})}\notag\\
	 &\leq \|I_h \Tsm (\eM - \ehm - \tau I_h \Tsm v_*^m)\|_{H^1(\Gamma^0_{h,{\rm f}})} \notag\\
     &\quad\,
	 + \tau \Big(\sum_{K_{\rm f}^0\subset \Gamma^0_{h,{\rm f}}} |K_{\rm f}^0| \|I_hT_*^m (g^m \circ a^m)\circ \hat X_{h,*}^m\|_{W^{1,\infty}(K_{\rm f}^0)}^2\Big)^{\frac12} \notag\\
	 &\leq \|I_h \Tsm (\eM - \ehm - \tau I_h \Tsm v_*^m)\|_{H^1(\Gamma^0_{h,{\rm f}})} \notag\\
     &\quad\,
	 + C_0\tau \|g^m\|_{W^{1,\infty}(\Gamma^m)} \Big(\sum_{K_{\rm f}^0\subset \Gamma^0_{h,{\rm f}}} h^2
     (1+\|\nabla_{K_{\rm f}^0}\hat X_{h,*}^m\|_{L^\infty(K_{\rm f}^0)}^2)\Big)^{\frac12} 
     \quad\mbox{(using $W^{1,\infty}$-stability of $I_h$)}\notag\\
	 &\leq \|I_h \Tsm (\eM - \ehm - \tau I_h \Tsm v_*^m)\|_{H^1(\Gamma^0_{h,{\rm f}})} 
	 + C_0 \tau \|g^m\|_{W^{1,\infty}(\Gamma^m)}(1+\|\hat X_{h,*}^m\|_{H^1(\Gamma^0_{h,{\rm f}})}) \notag\\
      &\quad\,\,\mbox{(using the inverse inequality)}\notag\\
	 &\le \|I_h \Tsm (\eM - \ehm - \tau I_h \Tsm v_*^m)\|_{H^1(\Gamma^0_{h,{\rm f}})} 
	 + C_0 \tau^2(1+\|\hat X_{h,*}^m\|_{H^1(\Gamma^0_{h,{\rm f}})}) \quad \text{(using \eqref{W1infty-g})} .
 \end{align}
The second term on the right-hand side of \eqref{TNN} can be estimated as follows:
 \begin{align}
	 &\quad\,\,\|I_h(\Tsm(\NsM-\Nsm)\ehM)\|_{H^1(\Gamma^0_{h,{\rm f}})}\notag\\
	 &\le C_0 h^{-1}\|I_h (\Tsm(\NsM-\Nsm)\ehM)\|_{L^2(\Gamma^0_{h,{\rm f}})}\quad\text{(using inverse inequality)}\notag\\
	 &\le C_{\k_m}h^{-1}\|I_h (\Tsm(\NsM-\Nsm)\ehM)\|_{L^2(\Ghsm)} \quad\text{(using \eqref{W1p-equiv})}\notag\\
    	 &\le C_{\k_m}h^{-1}\|I_h\Tsm(\NsM-\Nsm)\|_{L^\infty(\Ghsm)}\|\ehM\|_{L^2(\Ghsm)}\quad\text{(using \eqref{stability3})}\notag\\
	%  & \,\,\quad\mbox{(equivalence of discrete norm is used)}\notag\\
	 & \le C_{\k_m} h^{-1} (\|\hat X_{h}^{m+1} - \hat X_h^m\|_{L^\infty(\Ghsm)}+\tau) \|\ehM\|_{L^2(\Ghsm)} \quad\text{(using \eqref{nsM-nsm-nodes})}\notag\\
	 & \le C_{\k_m} h^{-1} \tau \|\ehM\|_{L^2(\Ghsm)} \quad\text{(using \eqref{Xtau} in Lemma \ref{Xhinfty})}\notag\\
	 % &\le C_{\k_m}h^{-1}\|I_h[\nsM\circ \hat{X}^{m+1}_{h,*}-\nsM\circ X^{m+1}_{h,*}]\|_{L^\infty(\Gamma^0_{h,{\rm f}})}\|\ehM\|_{L^2(\Gamma^0_{h,{\rm f}})}\notag\\
	 % &\quad +C_{\k_m}h^{-1}\|I_h[\nsM\circ X^{m+1}_{h,*}-\nsm\circ\hat{X}^{m}_{h,*}]\|_{L^\infty(\Gamma^0_{h,{\rm f}})}\|\ehM\|_{L^2(\Gamma^0_{h,{\rm f}})}\quad\mbox{(Lemma \ref{stability} is used)}\notag\\
	 % % &\le C_{\k_m}h^{-1}\|I_h[\nsM\circ \hat{X}^{m+1}_{h,*}-\nsm\circ X^{m+1}_{h,*}]\|_{L^\infty(\Gamma^0_{h,{\rm f}})}\|\ehM\|_{H^1_h(\Gamma^0_{h,{\rm f}})}\notag\\
	 % % &\quad +C_{\k_m}h^{-1}\|I_h[\nsM\circ X^{m+1}_{h,*}-\nsm\circ\hat{X}^{m}_{h,*}]\|_{L^\infty(\Gamma^0_{h,{\rm f}})}\|\ehM\|_{H^1_h(\Gamma^0_{h,{\rm f}})}\notag\\
	 % &\le C_{\k_m}(\|I_h\TsM(\hat{X}^{m+1}_{h,*}-\hat{X}^m_{h,*})\|_{L^\infty(\Gamma^0_{h,{\rm f}})}+\|\hat{X}^{m+1}_{h,*}-\hat{X}^m_{h,*}\|_{L^\infty(\Gamma^0_{h,{\rm f}})}^2+\tau)h^{-1}\|\ehM\|_{L^2(\Gamma^0_{h,{\rm f}})}\notag\\
	 % &\le C_{\k_m}(\|I_h\TsM(\hat{X}^{m+1}_{h,*}-\hat{X}^m_{h,*})-\tau I_h\Tsm v_*^m\|_{L^\infty(\Gamma^0_{h,{\rm f}})}+\|\hat{X}^{m+1}_{h,*}-\hat{X}^m_{h,*}\|_{L^\infty(\Gamma^0_{h,{\rm f}})}^2+\tau)h^{-1}\|\ehM\|_{L^2(\Gamma^0_{h,{\rm f}})}\notag\\
	 % &\le C_{\k_m}(\|I_h\TsM(\hat{X}^{m+1}_{h,*}-\hat{X}^m_{h,*})-\tau I_h\Tsm v_*^m\|_{L^\infty(\Gamma^0_{h,{\rm f}})}+\tau)h^{-1}\|\ehM\|_{L^2(\Gamma^0_{h,{\rm f}})}\quad\mbox{(\eqref{Xtau} is used)}\notag\\
	 % &\le C_{\k_m}h^{-1}\|I_h\Tsm(\hat{X}^{m+1}_{h,*}-\hat{X}^m_{h,*})-\tau I_h\Tsm v_*^m\|_{L^\infty(\Gamma^0_{h,{\rm f}})}\|\ehM\|_{L^2(\Gamma^0_{h,{\rm f}})}+C_{\k_m}h^{-1}\tau \|\ehM\|_{L^2(\Gamma^0_{h,{\rm f}})}\notag\\
	 &\le C_{\kappa_m}(1+\kappa_{*,m})\tau h^{-1}(\tau+h^k) , \label{NN}
 \end{align}
 %\begin{align}
 %	&\|I_h\Tsm(\NsM-\Nsm)\ehM\|_{L^\infty(\Gamma^0_{h,{\rm f}})}\notag\\
 %	&\le C_{\kappa_m}(\|I_h\TsM(\hat{X}^{m+1}_{h,*}-\hat{X}^m_{h,*})\|_{L^\infty(\Gamma^0_{h,{\rm f}})}+\|\hat{X}^{m+1}_{h,*}-\hat{X}^m_{h,*}\|_{L^\infty(\Gamma^0_{h,{\rm f}})}^2+\tau)\|\ehM\|_{L^\infty(\Ghsm)}\notag\\
 %	&\le C_{\kappa_m}(\|I_h\TsM(\hat{X}^{m+1}_{h,*}-\hat{X}^m_{h,*}-\tau\Tsm v_*^m)\|_{L^\infty(\Gamma^0_{h,{\rm f}})}+\|\hat{X}^{m+1}_{h,*}-\hat{X}^m_{h,*}\|_{L^\infty(\Gamma^0_{h,{\rm f}})}^2+\tau)\|\ehM\|_{L^\infty(\Ghsm)}\notag\\
 %	&\le C_{\kappa_m}(\|I_h\TsM(\hat{X}^{m+1}_{h,*}-\hat{X}^m_{h,*}-\tau\Tsm v_*^m)\|_{L^\infty(\Gamma^0_{h,{\rm f}})}+\tau)\|\ehM\|_{L^\infty(\Ghsm)}\quad\mbox{(\eqref{Xtau} is used)}\notag\\
 %	&\le C_{\kappa_m}h^{0.6}\|I_h\TsM(\hat{X}^{m+1}_{h,*}-\hat{X}^m_{h,*}-\tau\Tsm v_*^m)\|_{L^\infty(\Gamma^0_{h,{\rm f}})}+C_{\kappa_m}\tau h^{-1}\|\ehM\|_{L^2(\Ghsm)}\notag\\
 %	&\le C_{\kappa_m}h^{0.6}\|I_h\Tsm(\hat{X}^{m+1}_{h,*}-\hat{X}^m_{h,*}-\tau\Tsm v_*^m)\|_{L^\infty(\Gamma^0_{h,{\rm f}})}+C_{\kappa_m}(1+\kappa_{*,l})\tau h^{-1}(\tau+h^k), 
 %\end{align}
 where the last inequality follows from the error estimate for \(\|\ehM\|_{L^2(\Ghsm)}\) in \eqref{eq:err_fin0}.

 By substituting \eqref{TXX} and \eqref{NN} into \eqref{TNN}, we derive the following estimate:
 \begin{align}
	 &\quad\,\,\|I_h\Tsm(\hat{X}^{m+1}_{h,*}-\hat{X}^m_{h,*}-\tau \Tsm v_*^m)\|_{H^1(\Gamma^0_{h,{\rm f}})}\notag\\
	 &\le \|I_h\Tsm(\eM-\ehm-\tau I_h\Tsm v_*^m)\|_{H^1(\Gamma^0_{h,{\rm f}})} +C_0\tau^2(1+\|\hat X_{h,*}^m\|_{H^1(\Gamma^0_{h,{\rm f}})})\notag\\
	 &\quad+C_{\kappa_m}(1+\kappa_{*,m})\tau h^{-1}(\tau+h^k)\notag\\
	 & \le C_{\kappa_m}\|I_h\Tsm(\eM-\ehm-\tau I_h\Tsm v_*^m)\|_{H^1(\Ghsm)}+C_{\kappa_m}(1+\kappa_{*,m})\tau h^{-1}(\tau+h^k) , \label{TXXT}
 \end{align}
where we have used the norm equivalence relation in \eqref{W1p-equiv}. 
From the estimate in \eqref{nMinfty} and the definition of $\hat e_v^m=(\eM-\ehm-\tau I_h\Tsm v_*^m )/\tau $ in \eqref{evm}, together with the conditions \(\tau \leq c h^k\) and $k \geq 3$, we obtain 
 \begin{align}
	 &\quad\,\,\|I_h\Tsm(\eM-\ehm-\tau I_h\Tsm v_*^m)\|_{H^1(\Ghsm)}\notag\\
	 &\le C_{\kappa_m}(1+\kappa_{*,m})h^{k-1}\tau+C_{\k_m}h^{-1}\tau\| \nabla_{\Ghsm} \ehm \|_{L^2(\Ghsm)}+C_{\k_m}h^{-3.1}\tau\| \nabla_{\Ghsm} \ehm \|_{L^2(\Ghsm)}^2.\label{hjH1}
 \end{align}
 By substituting \eqref{hjH1} into \eqref{TXXT}, we obtain,
 \begin{align}\label{inter-bound}
	& \quad\,\,\|I_h\Tsm(\hat{X}^{m+1}_{h,*}-\hat{X}^m_{h,*}-\tau \Tsm v_*^m)\|_{H^1(\Gamma^0_{h,{\rm f}})} \\
	 &\le C_{\kappa_m}(1+\kappa_{*,m})\tau h^{-1}(\tau + h^k) +C_{\k_m}h^{-1}\tau\| \nabla_{\Ghsm} \ehm \|_{L^2(\Ghsm)}+C_{\k_m}h^{-3.1}\tau\| \nabla_{\Ghsm} \ehm \|_{L^2(\Ghsm)}^2. \notag
 \end{align}
 By using \eqref{inter-bound} and the Sobolev embedding \eqref{eq:poincare3}, we have
 \begin{align}\label{TXXT-infty}
	&\quad\,\,\|I_h\Tsm(\hat{X}^{m+1}_{h,*}-\hat{X}^m_{h,*}-\tau \Tsm v_*^m)\|_{L^\infty(\Gamma^0_{h,{\rm f}})} \\
	&\le C_{\kappa_m}(1+\kappa_{*,m})\tau h^{-1.1}(\tau+h^k)+C_{\k_m}h^{-1.1}\tau\| \nabla_{\Ghsm} \ehm \|_{L^2(\Ghsm)} +C_{\k_m}h^{-3.2}\tau\| \nabla_{\Ghsm} \ehm \|_{L^2(\Ghsm)}^2. \notag
 \end{align}
%  by choosing \(\epsilon =0.1\) in deriving the last inequality.
Under the stepsize condition \(\tau \leq c h^k\), by substituting \eqref{inter-bound}--\eqref{TXXT-infty} into \eqref{XXTW} and noting that the square of the right-hand side of \eqref{TXXT-infty} is bounded by the right-hand side of \eqref{inter-bound}, we obtain the following result, provided that $h$ is sufficiently small (\(h \leq h_{\kappa_l, \kappa_{*,l}}\)):
 \begin{align}
	& \quad\,\,\|\hat{X}^{m+1}_{h,*}-\hat{X}^m_{h,*}-\tau \Tsm v_*^m\|_{W^{j,\infty}_h(\Gamma^0_{h,{\rm f}})}\notag\\
	&\le C_0\tau\|\hat{X}^m_{h,*}\|_{W^{j,\infty}_h(\Gamma^0_{h,{\rm f}})} + C_0\tau[1+(j-1)\|\hat{X}^m_{h,*}\|^j_{W^{j-1,\infty}_h(\Gamma^0_{h,{\rm f}})}] + C_0 h^{-j}\tau^2 \notag \\
	 &\quad +C_{\k_m}(1+\kappa_{*,m})\tau h^{k-j-1} + C_{\k_m}h^{-j-1}\tau\| \nabla_{\Ghsm} \ehm \|_{L^2(\Ghsm)}\notag\\
	 & \quad\,\,+C_{\k_m}h^{-j-3.1}\tau\| \nabla_{\Ghsm} \ehm \|_{L^2(\Ghsm)}^2.\label{Wjh}
	 %	&\le C_0\tau[1+(j-1)\|\hat{X}^m_{h,*}\|^j_{W^{j-1,\infty}_h(\Gamma^0_{h,{\rm f}})}]+C_0\tau\|\hat{X}^m_{h,*}\|_{W^{j,\infty}_h(\Gamma^0_{h,{\rm f}})}\notag\\
	 %	&\quad +C_0h^{k-j}\tau+C_{\k^m}(1+\kappa_{*,l})h^{k-j-1}\tau\notag\\
	 %	&\quad +C_{\k^m}\lg(1/h)h^{-j}\tau[h^{-1}\tau+(1+\kappa_{*,l})h^{k-1}+h^{-1}\| \nabla_{\Ghsm} \ehm \|_{L^2(\Ghsm)}+\lg(1/h)h^{-3}\| \nabla_\Ghsm \ehm \|_{L^2(\Ghsm)}^2]\notag\\
	 %	&\hspace{250pt}\mbox{(\eqref{nMinfty} is used)}\notag\\
	 %	&\le C_0\tau[1+(j-1)\|\hat{X}^m_{h,*}\|^j_{W^{j-1,\infty}_h(\Gamma^0_{h,{\rm f}})}]+C_0\tau\|\hat{X}^m_{h,*}\|_{W^{j,\infty}_h(\Gamma^0_{h,{\rm f}})}\notag\\
	 %	&\quad +C_0h^{k-j}\tau+C_{\k^m}(1+\kappa_{*,l})\lg(1/h)h^{k-j-1}\tau\notag\\
	 %	&\quad+C_{\k^m}\lg(1/h)h^{-j-1}\tau\| \nabla_{\Ghsm} \ehm \|_{L^2(\Ghsm)}+C_{\k^m}\lg(1/h)h^{-j-3}\tau\| \nabla_{\Ghsm} \ehm \|_{L^2(\Ghsm)}^2.\label{Wjj}
 \end{align}
 
The following result is similar as \eqref{XXTW} and \eqref{Wjh}, with only the norm being modified from $\|\cdot\|_{W^{j,\infty}_h(\Gamma^0_{h,{\rm f}})}$ to $\|\cdot\|_{H^j_h(\Gamma^0_{h,{\rm f}})}$:
\begin{align}
 &\quad\,\,\|\hat{X}^{m+1}_{h,*} - \hat{X}^m_{h,*} - \tau \Tsm v_*^m\|_{H^j_h(\Gamma^0_{h,{\rm f}})} \notag \\
 &\leq C_0 \tau \|\hat{X}^m_{h,*}\|_{H^{j}_h(\Gamma^0_{h,{\rm f}})} 
 + C_0 \tau \big[ 1 
 + \|\hat{X}^m_{h,*}\|_{H^{j-1}_h(\Gamma^0_{h,{\rm f}})} \|\hat{X}^m_{h,*}\|_{W^{1,\infty}_h(\Gamma^0_{h,{\rm f}})}
 + (j-2)\|\hat{X}^m_{h,*}\|^j_{W^{j-2,\infty}_h(\Gamma^0_{h,{\rm f}})} \big] \notag \\
 &\,\,      +C_0 h^{-j}\tau^2+ C_0 h^{-j}\|I_h\Tsm(\hat{X}^{m+1}_{h,*}-\hat{X}^m_{h,*}-\tau \Tsm v_*^m)\|_{L^\infty(\Gamma^0_{h,{\rm f}})}  \notag\\
&\,\,\cdot \|I_h\Tsm(\hat{X}^{m+1}_{h,*}-\hat{X}^m_{h,*}-\tau \Tsm v_*^m)\|_{L^2(\Gamma^0_{h,{\rm f}})} + C_0h^{-j+1}\|I_h\Tsm(\hat{X}^{m+1}_{h,*}-\hat{X}^m_{h,*}-\tau \Tsm v_*^m)\|_{H^1(\Gamma^0_{h,{\rm f}})}\notag\\
& \leq C_0 \tau \|\hat{X}^m_{h,*}\|_{H^{j}_h(\Gamma^0_{h,{\rm f}})}  
+ C_0 \tau \big[ 1+ \|\hat{X}^m_{h,*}\|_{H^{j-1}_h(\Gamma^0_{h,{\rm f}})}\|\hat{X}^m_{h,*}\|_{W^{1,\infty}_h(\Gamma^0_{h,{\rm f}})} 
+ (j-2)\|\hat{X}^m_{h,*}\|^j_{W^{j-2,\infty}_h(\Gamma^0_{h,{\rm f}})} \big] 
  \notag \\
 &\,\, + C_0 h^{-j}\tau^2+C_{\k_m}(1+\kappa_{*,m})\tau h^{k-j} + C_{\k_m}h^{-j}\tau\| \nabla_{\Ghsm} \ehm \|_{L^2(\Ghsm)}\notag\\
 & \,\,+C_{\k_m}h^{-j-2.1}\tau\| \nabla_{\Ghsm} \ehm \|_{L^2(\Ghsm)}^2. \label{Hjh}
\end{align}
%  Therefore, by substituting the result from \eqref{hjH1} into \eqref{Wjj} and \eqref{Hjj}, and further employing Sobolev's embedding theorem, we obtain the following estimates under the condition $\tau  \le ch^k$:
%  \begin{align}
% 	 &\quad\,\,\|\hat{X}^{m+1}_{h,*}-\hat{X}^m_{h,*}-\tau \Tsm v_*^m\|_{W^{j,\infty}_h(\Gamma^0_{h,{\rm f}})}\notag\\
% 	 &\le C_0\tau[1+(j-1)\|\hat{X}^m_{h,*}\|^j_{W^{j-1,\infty}_h(\Gamma^0_{h,{\rm f}})}]+C_0\tau\|\hat{X}^m_{h,*}\|_{W^{j,\infty}_h(\Gamma^0_{h,{\rm f}})}+C_{\k_m}(1+\kappa_{*,m})h^{k-j-1}\tau\notag\\
% 	 &\quad+C_{\k_m}h^{-j-1}\tau\| \nabla_{\Ghsm} \ehm \|_{L^2(\Ghsm)}+C_{\k_m}h^{-j-3-\epsilon}\tau\| \nabla_{\Ghsm} \ehm \|_{L^2(\Ghsm)}^2.\label{Wjh}\\
% 	 &\quad\,\,\|\hat{X}^{m+1}_{h,*}-\hat{X}^m_{h,*}-\tau \Tsm v_*^m\|_{H^j_h(\Gamma^0_{h,{\rm f}})}\notag\\
% 	 &\le C_0\tau[1+(j-1)(j-2)\|\hat{X}^m_{h,*}\|^j_{W^{j-2,\infty}_h(\Gamma^0_{h,{\rm f}})}+(j-1)\|\hat{X}^m_{h,*}\|_{W^{1,\infty}_h(\Gamma^0_{h,{\rm f}})}\|\hat{X}^m_{h,*}\|^j_{H^{j-1}_h(\Gamma^0_{h,{\rm f}})}]\notag\\
% 	 &\quad+C_0 \tau \|\hat{X}^m_{h,*}\|^j_{H^{j}_h(\Gamma^0_{h,{\rm f}})}+C_{\k_m}(1+\kappa_{*,m}) h^{k-j}\tau\notag\\
% 	 &\quad+C_{\k_m}h^{-j}\tau\| \nabla_{\Ghsm} \ehm \|_{L^2(\Ghsm)}+C_{\k_m}h^{-j-2-\epsilon}\tau\| \nabla_{\Ghsm} \ehm \|_{L^2(\Ghsm)}^2.\label{Hjh}
%  \end{align}

Now, by setting $j=1$ in \eqref{Wjh} and summing \eqref{Wjh} from $m = 0$ to $q$, we obtain the following estimate for $1 \leq q \leq l$:
 \begin{align}
	 &\quad\,\,\sum_{m=0}^{q}\|\hat{X}^{m+1}_{h,*}-\hat{X}^m_{h,*}-\tau \Tsm v_*^m\|_{W^{1,\infty}(\Gamma^0_{h,{\rm f}})}\notag\\
	 &\le C_0+\sum_{m=0}^{q}C_0\tau\|\hat{X}^m_{h,*}\|_{W^{1,\infty}(\Gamma^0_{h,{\rm f}})}+\sum_{m=0}^{q}C_{\k_m}\tau(1+\kappa_{*,m})h^{k-2}\notag\\
	 &\quad+ \sum_{m=0}^{q}C_{\k_m}\Big[h^{-2}\tau\| \nabla_{\Ghsm} \ehm \|_{L^2(\Ghsm)}+h^{-4.1}\tau\| \nabla_{\Ghsm} \ehm \|_{L^2(\Ghsm)}^2\Big]\notag\\
	 &\le C_0+C_{\k_l}(1+\kappa_{*,l})h^{k-2}+C_{\k_l}(1+\kappa_{*,l})^2h^{2k-4.1}+\sum_{m=0}^{q}C_0\tau\|\hat{X}^m_{h,*}\|_{W^{1,\infty}(\Gamma^0_{h,{\rm f}})} , 
 \end{align}
where we have utilized the error estimate from \eqref{eq:err_fin0}, along with the condition \(\tau \leq c h^k\) and Cauchy-Schwarz inequality in the last inequality. Since \(k \geq 3\), for sufficiently small mesh size \(h \leq h_{\kappa_l, \kappa_{*,l}}\), substituting this inequality into \eqref{qWj} leads to the following result:
 \begin{align}
	 \|\hat{X}^{q+1}_{h,*}\|_{W^{1,\infty}(\Gamma^0_{h,{\rm f}})}
	 &\leq C_0 + \sum_{m=0}^{q} C_0 \tau \|\hat{X}^m_{h,*}\|_{W^{1,\infty}(\Gamma^0_{h,{\rm f}})}.
 \end{align}
 Subsequently, by applying the discrete Grönwall inequality, we obtain:
 \begin{align}
	 \max_{0 \leq q \leq l} \|\hat{X}^{q+1}_{h,*}\|_{W^{1,\infty}(\Gamma^0_{h,{\rm f}})} \leq C_0.
 \end{align}
 
 To establish the claim by mathematical induction, we aim to prove that if \(j \leq k-2\) and \(\displaystyle\max_{0 \leq q \leq l} \|\hat{X}^{q+1}_{h,*}\|_{W^{j-1,\infty}_h(\Gamma^0_{h,{\rm f}})} \leq C_0\), then 
 \[
 \max_{0 \leq q \leq l} \|\hat{X}^{q+1}_{h,*}\|_{W^{j,\infty}_h(\Gamma^0_{h,{\rm f}})} \leq C_0.
 \]
 Indeed, under the assumption \(\displaystyle\max_{0 \leq q \leq l} \|\hat{X}^{q+1}_{h,*}\|_{W^{j-1,\infty}_h(\Gamma^0_{h,{\rm f}})} \leq C_0\), summing \eqref{Wjh} over \(m = 0, \dots, q\), yields:
 \begin{align}
	 &\quad\,\,\sum_{m=0}^{q}\|\hat{X}^{m+1}_{h,*}-\hat{X}^m_{h,*}-\tau \Tsm v_*^m\|_{W^{j,\infty}_h(\Gamma^0_{h,{\rm f}})}\notag\\
	 &\le C_0+\sum_{m=0}^{q}C_0\tau\|\hat{X}^m_{h,*}\|_{W^{j,\infty}_h(\Gamma^0_{h,{\rm f}})}+\sum_{m=0}^{q}C_{\k_m}\tau(1+\kappa_{*,m})h^{k-j-1}\notag\\
	 &\quad+ \sum_{m=0}^{q}C_{\k_m}\Big[h^{-j-1}\tau\| \nabla_{\Ghsm} \ehm \|_{L^2(\Ghsm)}+h^{-j-3.1}\tau\| \nabla_{\Ghsm} \ehm \|_{L^2(\Ghsm)}^2\Big]\notag\\
	 % &\le C_0+C_{\k_l}(1+\kappa_{*,l})h^{k-j-1} + C_{\k_l}h^{-j-1} \sum_{m=0}^q \tau\| \nabla_{\Ghsm} \ehm \|_{L^2(\Ghsm)} \notag\\
	 % & \quad +C_{\k_l}(1+\kappa_{*,l})^2h^{2k-j-3-\epsilon}+\sum_{m=0}^{q}C_0\tau\|\hat{X}^m_{h,*}\|_{W^{j,\infty}_h(\Gamma^0_{h,{\rm f}})}\notag\\
	 &\le C_0+C_{\k_l}(1+\kappa_{*,l})h^{k-j-1}+C_{\k_l}(1+\kappa_{*,l})^2h^{2k-j-3.1}+\sum_{m=0}^{q}C_0\tau\|\hat{X}^m_{h,*}\|_{W^{j,\infty}_h(\Gamma^0_{h,{\rm f}})},
 \end{align}
 where we have utilized the error estimate provided in \eqref{eq:err_fin0}, along with the assumption that \(\tau \leq ch^k\). 
 Since \(k \geq 3\) and $j \leq k-2$, for sufficiently small mesh size \(h \leq h_{\kappa_l, \kappa_{*,l}}\), substituting the last inequality into \eqref{qWj} yields the following result:
\begin{align}
 \|\hat{X}^{q+1}_{h,*}\|_{W^{j,\infty}_h(\Gamma^0_{h,{\rm f}})} 
 &\leq C_0 + \sum_{m=0}^{q} C_0 \tau \|\hat{X}^m_{h,*}\|_{W^{j,\infty}_h(\Gamma^0_{h,{\rm f}})},
\end{align}
where the generic constant $C_0$ may depend on $1\le j \le k-2$ but independent of $\k_l,\k_{l,*}$.
Then, by applying the discrete Grönwall inequality, we obtain the following bound:
\begin{align}
 \max_{0 \leq q \leq l} \|\hat{X}^{q+1}_{h,*}\|_{W^{j,\infty}_h(\Gamma^0_{h,{\rm f}})} \leq C_0\quad\mbox{for}\,\,\, 1 \leq j \leq k-2. 
\end{align}

Similarly, by employing the estimates \eqref{Hjh} and \eqref{qHj}, we can establish the following result for \(1 \leq j \leq k-1\) by mathematical induction:
\begin{align}
 \max_{0 \leq q \leq l} \|\hat{X}^{q+1}_{h,*}\|_{H^j_h(\Gamma^0_{h,{\rm f}})} \leq C_0.
\end{align}

Analogous arguments can be applied to show that 
\(\displaystyle \max_{0 \leq q \leq l} \|(\hat{X}^{q+1}_{h,*})^{-1}\|_{W^{1,\infty}_h(\Gamma^0_{h,{\rm f}})} \leq C_0\), the proof is omitted here for brevity. 

Consequently, we conclude that under the step size condition \(\tau \leq ch^k\) and the mesh size condition \(h \leq h_{\kappa_l, \kappa_{*,l}}\), it holds that \(\kappa_{l+1} \leq C_0\).

 Therefore, we can replace $C_{\k_m}$ by $C_0$ in \eqref{Wjh}--\eqref{Hjh} to obtain the following results for $0\le q \le l$ in the same way as above, under the conditions $\tau\le ch^k$ and $h\le h_{\k_l,\kappa_{*,l}}$:
 \begin{align}
	 \|\hat{X}^{q+1}_{h,*}\|_{W^{k-1,\infty}_h(\Gamma^0_{h,{\rm f}})}&\le\ C_0+\sum_{m=0}^{q}C_0\tau\kappa_{*,m}+\sum_{m=0}^{q}C_0\tau\|\hat{X}^m_{h,*}\|_{W^{k-1,\infty}_h(\Gamma^0_{h,{\rm f}})},\\
	 \|\hat{X}^{q+1}_{h,*}\|_{H^k_h(\Gamma^0_{h,{\rm f}})}&\le\ C_0+\sum_{m=0}^{q}C_0\tau\kappa_{*,m}+\sum_{m=0}^{q}C_0\tau\|\hat{X}^m_{h,*}\|_{H^k_h(\Gamma^0_{h,{\rm f}})}.
 \end{align}
 In regard to the definition of $\kappa_{*,m}$ in \eqref{kl*}, we replace $\kappa_{*,m}$ by $ \|\hat{X}^m_{h,*}\|_{W^{k-1,\infty}_h(\Gamma^0_{h,{\rm f}})}+\|\hat{X}^m_{h,*}\|_{H^k_h(\Gamma^0_{h,{\rm f}})}$ and then sum up the two inequalities above. This yields that
 \begin{align}
	 &\|\hat{X}^{q+1}_{h,*}\|_{W^{k-1,\infty}_h(\Gamma^0_{h,{\rm f}})}+\|\hat{X}^{q+1}_{h,*}\|_{H^k_h(\Gamma^0_{h,{\rm f}})}\notag\\
	 &\le C_0+\sum_{m=0}^{q}C_0\tau(\|\hat{X}^m_{h,*}\|_{W^{k-1,\infty}_h(\Gamma^0_{h,{\rm f}})}+\|\hat{X}^m_{h,*}\|_{H^k_h(\Gamma^0_{h,{\rm f}})}).
 \end{align}
 By applying Gr\"onwall's inequality, we obtain
 \begin{equation}
	 \max_{0\le q\le l} (\|\hat{X}^{q+1}_{h,*}\|_{W^{k-1,\infty}_h(\Gamma^0_{h,{\rm f}})}+\|\hat{X}^{q+1}_{h,*}\|_{H^k_h(\Gamma^0_{h,{\rm f}})})\le C_0.
 \end{equation}
 This proves that $\kappa_{*,l+1}\le C_0$.

%  By mathematical induction under conditions $\tau\le ch^k$ and $h\le h_{\k_l,\kappa_{*,l}}$, for this constant $C_0$ (which is independent of $l$) we get the following result: If $\k_l\le C_0$ and $\kappa_{*,l}\le C_0$ with $h\le h_{C_0,C_0}$, then
%  \begin{equation}
% 	 \kappa_{l+1}\le C_{\#}\quad\mbox{and}\quad \kappa_{*,l+1}\le C_*.
%  \end{equation}
As a result, the quantities $\k_l$ and $\kappa_{*,l}$ defined in \eqref{kl*} are uniformly bounded with respect to $\tau$ and $h$, and are bounded above by some constants $C_{\#}$ and $C_*$, respectively.

\end{document}